%% file: main.tex
\documentclass[11pt]{amsart}
\usepackage[left=1.25in,right=1.25in]{geometry}
\usepackage{amsmath,amsthm,amsfonts,amssymb}
\usepackage{tikz}
\usetikzlibrary{calc,shapes,backgrounds,arrows,positioning,decorations.pathmorphing,cd,decorations.markings,patterns}
\usepackage{subcaption}
\usepackage{hyperref}
\usepackage{cleveref}
\usepackage[style=alphabetic,doi=false,isbn=false,url=false,eprint=false]{biblatex}
\usepackage{comment}

\newcommand{\C}{\mathbb{C}}
\newcommand{\R}{\mathbb{R}}

\newcommand{\Z}{\mathbb{Z}}

\renewcommand{\P}{\mathbb{P}}
\newcommand{\F}{\mathbb{F}}

\newcommand{\cT}{\mathcal{T}}
\newcommand{\cR}{\mathcal{R}}
\newcommand{\cD}{\mathcal{D}}
\newcommand{\cH}{\mathcal{H}}

\DeclareMathOperator{\rank}{rank}
\DeclareMathOperator{\inv}{Inv}
\DeclareMathOperator{\wt}{wt}
\DeclareMathOperator{\depth}{depth}
\DeclareMathOperator{\codim}{codim}
\DeclareMathOperator{\gr}{Gr}
\DeclareMathOperator{\ev}{ev}
\DeclareMathOperator{\fl}{Fl}
\DeclareMathOperator{\Des}{Des}
\DeclareMathOperator{\Asc}{Asc}

\DeclareFontEncoding{LS1}{}{}
\DeclareFontSubstitution{LS1}{stix}{m}{n}
\DeclareSymbolFont{stixletters}{LS1}{stix}{m}{it}
\DeclareMathAccent{\cev}{\mathord}{stixletters}{"91}
\DeclareMathOperator{\flatten}{\cev{f}}
\DeclareMathOperator{\jump}{\cev{j}}
\DeclareMathOperator{\Jump}{\cev{J}}

\newcommand{\jmin}{\jump_{\min}}
\newcommand{\jmax}{\jump_{\max}}
\newcommand{\id}{\mathrm{id}}
\newcommand{\hpi}{\pi_{[n]\setminus i}}
\newcommand{\p}{\Lambda}
\newcommand{\op}{\mathrm{op}}
\newcommand{\pt}{\mathrm{pt}}
\newcommand{\ba}{\mathbf{a}}
\newcommand{\tilt}{\mathrm{tilt}}
\newcommand{\word}{\mathrm{word}}
\newcommand{\sign}{\mathrm{sign}}
\newcommand{\cyclic}{\tau}
\newcommand{\size}[1]{\#\left\{#1\right\}}
\newcommand{\bflat}{\cev{\mathrm{f}}}
\newcommand{\jasize}{\left|\Jump_\ba\right|}
\newcommand{\lesssimdot}{\mathrel{\ooalign{$\lesssim$\cr\hidewidth\raise0.45ex\hbox{$\cdot\mkern1mu$}\cr}}}

\DeclareTextFontCommand{\emph}{\color{blue}\em}

\numberwithin{equation}{section}

\newtheorem{theorem}{Theorem}[section]
\newtheorem{prop}[theorem]{Proposition}
\newtheorem{conj}[theorem]{Conjecture}
\newtheorem{lemma}[theorem]{Lemma}
\newtheorem{cor}[theorem]{Corollary}

\theoremstyle{definition}

\newtheorem{prob}[theorem]{Problem}
\newtheorem{defin}[theorem]{Definition}

\theoremstyle{remark}
\newtheorem{remark}[theorem]{Remark}
\newtheorem{ex}[theorem]{Example}
\newtheorem{notation}[theorem]{Notation}
\newtheorem{construction}[theorem]{Construction}

\theoremstyle{plain}

\title{Tilted Richardson Varieties}

\author{Jiyang Gao}
\address{Department of Mathematics, Harvard University, Cambridge, MA 02138}
\email{\href{mailto:jgao.math@alumni.harvard.edu}{{\tt jgao.math@alumni.harvard.edu}}}

\author{Shiliang Gao}
\address{Department of Mathematics, Cornell University, Ithaca, NY 14850}
\email{\href{mailto:shiliang.gao@cornell.edu}{{\tt shiliang.gao@cornell.edu}}}

\author{Yibo Gao}
\address{Beijing International Center for Mathematical Research, Peking University, Beijing 100871}
\email{\href{mailto:gaoyibo@bicmr.pku.edu.cn}{{\tt gaoyibo@bicmr.pku.edu.cn}}}

\date{\today}

\begin{document}

\begin{abstract}
The study of the flag variety $\mathrm{Fl}_n$ and its subvarieties, including Schubert and Richardson varieties, plays a fundamental role in algebraic geometry and algebraic combinatorics. In this paper, we introduce and develop the theory of tilted Richardson varieties $\mathcal{T}_{u,v}$, a new family of subvarieties of the flag variety that provides a geometric framework for the quantum Bruhat graphs. These varieties are defined for all pairs of permutations $u$ and $v$, extending the classical Richardson varieties in the case where $u\leq v$ in the Bruhat order. We establish their fundamental geometric properties, proving irreducibility and providing explicit dimension formulas. Moreover, we show that they have a well-defined stratification indexed by tilted Bruhat intervals, a generalization of classical Bruhat intervals previously introduced by Brenti, Fomin, and Postnikov. Additionally, we introduce a tilted generalization of the classical Deodhar decomposition of Richardson varieties, which leads to a combinatorial formula for tilted Kazhdan--Lusztig R-polynomials, a notion that arises naturally in our framework.

We further develop a theory of total positivity for tilted Richardson varieties. In particular, we define and study the totally nonnegative parts of tilted Richardson varieties, proving they form a CW complex. This generalizes earlier results on the totally nonnegative flag variety and answers Bj\"orner's questions regarding geometric realizations of tilted Bruhat intervals.

Finally, we establish explicit connections between tilted Richardson varieties and quantum Schubert calculus. Specifically, we prove that $\mathcal{T}_{u,v}$ coincides with minimal-degree two-point curve neighborhoods. As a result, we compute their cohomology classes and derive new relationships among Gromov--Witten invariants of the flag variety.
\end{abstract}

\maketitle
\tableofcontents

\input{tex/1-introduction}
\input{tex/2-preliminaries}
\input{tex/3-tilted-order}
\input{tex/4-richardson-def}

\input{tex/5-geometric}

\input{tex/6-deodhar}
\input{tex/7-curve-neighborhood}

\input{tex/8-projections}

\section*{Acknowledgements}
We are grateful to Alex Postnikov for introducing us to quantum Bruhat graphs. We thank Thomas Lam for pointing out the connection between R-polynomials and Hecke algebras. We thank Allen Knutson for explaining the relation between positroid varieties and curve neighborhoods in the Grassmannian. We are also grateful to Anders Buch and Leonardo Mihalcea for discussions on path Schubert polynomials and the cohomology classes of curve neighborhoods. We thank Sergey Fomin for introducing us to the Fomin--Shapiro conjecture \cite{fomin-shapiro}, and Lauren Williams for detailed explanations of \cite{rietsch-williams}. We also thank Grant Barkley, Sara Billey, Melissa Sherman-Bennett, Weihong Xu, and Alexander Yong for many helpful conversations. SG is partially supported by the NSF MSPRF under grant No. DMS-2402285. YG is partially supported by NSFC Grant No. 12471309.

\printbibliography

\end{document}

%% file: tex/1-introduction.tex
\section{Introduction}\label{sec:intro}

Hilbert's fifteenth problem \cite{hilbert}, \emph{Schubert calculus}, studies the \emph{(complete) flag variety} $\fl_n$, the space of nested linear subspaces in $\C^n$, and a family of subvarieties in $\fl_n$ known as the \emph{Schubert varieties} $X_w$, indexed by permutations $w \in S_n$. The main goal of Schubert calculus is to understand how Schubert varieties intersect, which is encoded in the cohomology ring $H^*(\fl_n)$. The ring has a natural linear basis given by the \emph{Schubert classes} $\{\sigma_w\}_{w\in S_n}$. The corresponding structure constants $c_{u,v}^w$ of $H^\ast(\fl_n)$, known as the \emph{(generalized) Littlewood--Richardson coefficients}, are nonnegative integers that enumerate points in the transverse intersection of three Schubert varieties $X_{w_0u}$, $X_{w_0v}$, and $X_{w}$ in general position. Finding a combinatorial interpretation of these numbers has been a long standing open problem. The study of flag varieties, Schubert varieties and their associated structure constants is central in algebraic geometry and algebraic combinatorics.

\emph{Quantum Schubert calculus} extends classical Schubert calculus by studying not only the intersection structure of Schubert varieties, but also rational curves of a fixed degree passing through them. This data is encoded in the \emph{(small) quantum cohomology ring} $QH^*(\fl_n)$, which can be obtained from the cohomology ring $H^\ast(\fl_n)$ by introducing quantum parameters $q_1,q_2,\dots,q_{n-1}$ and deforming the product structure accordingly. The structure constants of $QH^*(\fl_n)$ with respect to the Schubert basis, known as the \emph{3-point genus-0 Gromov--Witten invariants} $c_{u,v}^{w,d}$, are nonnegative integers that enumerate rational curves of degree $d$ passing through three Schubert varieties $X_{w_0u}$, $X_{w_0v}$, and $X_{w}$ in general position. These invariants recover the Littlewood--Richardson coefficients when $d=0$.

Various combinatorial and geometric objects play a central role in Schubert calculus, including the \emph{(strong) Bruhat order}, \emph{Bruhat intervals}, \emph{Schubert varieties}, and \emph{Richardson varieties}. Similarly, quantum Schubert calculus is associated with structures such as \emph{quantum Bruhat graphs}, \emph{tilted Bruhat intervals} \cite{BFP-tilted-Bruhat}, and \emph{two-point curve neighborhoods} \cite{BCMP}. While significant progress has been made in understanding their combinatorial \cite{BFP-tilted-Bruhat, Postnikov-quantum-Bruhat-graph} and geometric \cite{curveneighborhood, LiMihalcea, BM15, BCMP, bclm2020-quantumK} properties, a more explicit description of these ``quantum" objects would provide further insight.

\begin{prob}\label{prob:curve-neighborhood}
    Give explicit descriptions of tilted Bruhat orders, tilted Bruhat intervals, and two-point curve neighborhoods.
\end{prob}

The combinatorial structure of Bruhat orders and Bruhat intervals is also closely connected to the geometry of \emph{totally positive spaces}. In \cite{Bjorner-poset}, Bj\"orner showed that every Bruhat interval forms the face poset of a regular CW-complex. However, his construction of the CW-complex was entirely \textit{synthetic}, constructed through a succession of cell attachments. This led to a fundamental question posed in the same paper: does there exist a \textit{natural} realization of this CW-complex within the framework of complex algebraic geometry or representation theory? Fomin and Shapiro \cite{fomin-shapiro} took an initial step toward answering this question by proposing a geometric construction and formulating the Fomin--Shapiro conjecture. Williams \cite{williams-shelling} later generalized this conjecture, connecting it to the \emph{totally nonnegative flag variety} and its cell decomposition. These conjectures have since driven significant progress in the study of totally positive spaces and their regularity in recent years \cite{Lusztig98b,rietsch99,fomin-shapiro,williams-shelling,rietsch-williams,lauren-morse,hersh,GKL1,GKL2,GKL3,bao-he}.

It has been shown that every tilted Bruhat interval also forms the face poset of a regular CW-complex \cite{BFP-tilted-Bruhat}, so a natural step is to develop a tilted counterpart to this story.

\begin{prob}\label{prob:cw-complex}
    Find natural (regular) CW-complexes in the flag variety $\fl_n$ whose face posets correspond to tilted Bruhat intervals.
\end{prob}

Motivated by the two problems above, we introduce the \emph{tilted Richardson varieties} $\cT_{u,v}$ and the \emph{open tilted Richardson varieties} $\cT_{u,v}^{\circ}$, which are subvarieties of $\fl_n$ indexed by a pair of permutations $u,v\in S_n$. These varieties, explicitly defined through rank conditions on certain submatrices, recover classical Richardson varieties when $u\leq v$ in the Bruhat order. We provide multiple equivalent definitions of this family of varieties, further demonstrating the naturality of our construction.

\begin{defin}
    We provide four equivalent definitions of \emph{tilted Richardson varieties} $\cT_{u,v}$ and \emph{open tilted Richardson varieties} $\cT_{u,v}^{\circ}$ for any pair of permutations $u,v\in S_n$, using:
    \begin{enumerate}
    \item rank conditions on certain submatrices;
    \item intersections of cyclically rotated Richardson varieties;
    \item vanishing loci of Pl\"ucker coordinates;
    \item intersections of two opposite tilted Schubert cells.
    \end{enumerate}
\end{defin}

In this paper, we initiate the study of tilted Richardson varieties to answer the two aforementioned problems.
We develop a tilted analogue of the Deodhar decomposition, defining the totally nonnegative parts of tilted Richardson varieties (addressing \Cref{prob:cw-complex}), and establish connections to two-point curve neighborhoods in the minimal degree case (addressing \Cref{prob:curve-neighborhood}). Collectively, our results build a theory of tilted Richardson varieties that closely parallels the classical framework of Bruhat orders, Richardson varieties, and related Schubert geometry.

Our first main theorem relates the geometry of tilted Richardson varieties to the poset structure of tilted Bruhat intervals, and shows that these varieties possess remarkably elegant geometric properties, closely resembling those of classical Richardson varieties.

\begin{theorem}
    The (open) tilted Richardson varieties $\cT_{u,v}$ and $\cT^\circ_{u,v}$ satisfy the following geometric properties:
    \begin{enumerate}
    \item $\cT_{u,v}$ is a closed subvariety of $\fl_n$, and $\cT^\circ_{u,v}\subseteq \cT_{u,v}$ is an open subvariety;
    \item The torus fixed points in $\cT_{u,v}$ corresponds exactly to the elements of the tilted Bruhat interval $[u,v]$;
    \item $\cT_{u,v}=\bigsqcup_{[x,y]\subseteq[u,v]}\cT_{x,y}^{\circ}$ is stratified by open tilted Richardson varieties indexed by subintervals of the tilted Bruhat interval $[u,v]$;
    \item The dimensions of $\cT_{u,v}$ and $\cT_{u,v}^{\circ}$ are given by the length of the shortest path from $u$ to $v$ in the quantum Bruhat graph;
    \item The closure relation $\overline{\cT_{u,v}^{\circ}}=\cT_{u,v}$ holds.
    \item Both $\cT_{u,v}^\circ$ and $\cT_{u,v}$ are irreducible varieties. 
    \end{enumerate}
\end{theorem}

Our proof differs substantially from Richardson’s original proof of these results for Richardson varieties in \cite{richardson}, primarily due to the lack of a Borel orbit intersection structure in the tilted setting. To overcome this issue, we introduce \emph{tilted reduced words} and \emph{tilted Deodhar decomposition}. These new tools are essential both in proving geometric properties of tilted Richardson varieties and in establishing connections with Kazhdan--Lusztig theory and total positivity.


In \cite{Deodhar}, Deodhar introduced the \emph{Deodhar decomposition} to study the \emph{Kazhdan–Lusztig polynomials} \cite{KL79} (see \cite{combinatorial-invariance} for an exposition). These polynomials can be computed recursively using \emph{Kazhdan–Lusztig R-polynomials} $R_{u,v}(q)$, the $\F_q$-point counts of the open Richardson variety $\cR_{u,v}^\circ$. The Deodhar decomposition expresses $\cR^\circ_{u,v}$ as disjoint union of simple pieces, each isomorphic to $\C^a\times (\C^\ast)^b$ for some $a,b\in \mathbb{Z}_{\geq 0}$.Analogously, our second main theorem extends the Deodhar decomposition to open tilted Richardson varieties.


\begin{theorem}[tilted Deodhar decomposition]
    The open tilted Richardson variety $\cT^{\circ}_{u,v}$ admits a decomposition into simple pieces, each isomorphic to $\C^a\times (\C^\ast)^b$, indexed by tilted distinguished subwords of a regular tilted reduced word. 
\end{theorem}

The decomposition also provides a combinatorial formula for the $\F_q$-point count of $\cT_{u,v}^\circ$, which we called
\emph{tilted Kazhdan--Lusztig R-polynomial} $R^\tilt_{u,v}(q)$. In \Cref{sec:tilted-r-poly}, we further explore the connection of $R^\tilt_{u,v}(q)$ to the Hecke algebra, extending the results of \cite{KL79,GL24}, and to the combinatorial invariance conjecture of Lusztig and Dyer \cite{Dyer-thesis}.

The study of the \emph{totally nonnegative flag variety} $\fl_n^{\geq 0}$ has gained significant interest in recent years \cite{Lusztig98,MarshRietsch,Pos06,kodama-williams2,tsukerman-williams,Lusztig20,Boretsky,Bloch-Karp}. First introduced by Lusztig \cite{Lusztig98}, $\fl_n^{\geq 0}$ is defined as the closure of the set of flags that can be represented by matrices with nonnegative minors. The \emph{totally nonnegative parts} of Richardson varieties $\cR_{u,v}^{>0}$ and $\cR_{u,v}^{\geq 0}$ are then defined as the semi-algebraic intersections of $\fl_n^{\geq 0}$ with the Richardson varieties $\cR_{u,v}$ and $\cR_{u,v}^\circ$, respectively. A series of work \cite{williams-shelling,rietsch-williams,GKL3} has shown that the totally nonnegative part of a Richardson variety forms a regular CW-complex, whose face poset corresponds to the interval poset of the associated Bruhat interval. Through the tilted Deodhar decomposition, we extend this result to the setting of tilted Bruhat intervals by introducing the \emph{totally nonnegative parts} of tilted Richardson varieties. These spaces are conjectured to form a regular CW-complex whose face poset coincides with the interval poset of tilted Bruhat intervals, thus providing a potential solution to \Cref{prob:cw-complex}. Our third main theorem makes substantial progress toward this conjecture.

\begin{theorem}
    We define the totally nonnegative parts of tilted Richardson varieties $\cT_{u,v}^{>0}$ and $\cT_{u,v}^{\geq 0}$. They satisfy the following properties:
    \begin{enumerate}
        \item Each $\cT_{u,v}^{>0}$ is homeomorphic to an open ball;
        \item The space $\cT_{u,v}^{\geq 0}=\bigsqcup_{[x,y]\subseteq[u,v]}\cT_{x,y}^{>0}$ forms a CW-complex, whose face poset is conjecturally equal to the interval poset of the tilted Bruhat interval $[u,v]$.
    \end{enumerate}
\end{theorem}

We relate tilted Richardson varieties to quantum Schubert calculus. To study the quantum cohomology ring $QH^\ast(\fl_n)$ geometrically, Buch-Chaput-Mihalcea-Perrin \cite{curveneighborhood,BCMP} introduced the \emph{two-point curve neighborhood} $\Gamma_{d}(\Omega_u,X_v)$, defined as the closure of the union of all degree-$d$ rational curves intersecting the Schubert variety $X_v$ and the opposite Schubert variety $\Omega_u$ in the flag variety $\fl_n$. The cohomology classes of these varieties encode Gromov--Witten invariants of degree $d$, and their geometric and cohomological properties have been extensively studied \cite{curveneighborhood, LiMihalcea, BM15, BCMP, bclm2020-quantumK}. However, explicit descriptions of these curve neighborhoods are known only in special cases: when $v=\id$ \cite{BM15} or when $|d|\leq 1$ \cite{LiMihalcea}, where they coincide with classical Schubert or Richardson varieties and thus do not yield genuinely new families of subvarieties. Our fourth main theorem establishes an equivalence between tilted Richardson varieties and a \textit{new} class of curve neighborhoods, answering \Cref{prob:curve-neighborhood} in the \emph{minimal degree} case, and further identifies their cohomology classes.

\begin{theorem}
The tilted Richardson variety $\cT_{u,v}$ coincides with the minimal-degree two-point curve neighborhood $\Gamma_{d_{\min}}(\Omega_u,X_v)$. Furthermore, their cohomology classes $[\cT_{u,v}]$ and $[\Gamma_{d_{\min}}(\Omega_u,X_v)]$ in the cohomology ring $H^\ast(\fl_n)$ both equal the minimal quantum degree component of the quantum product $\sigma_{u}\star \sigma_{w_0v}$.
\end{theorem}

As an application, we prove new relations among Gromov--Witten invariants of the flag variety that generalize the descent-cycling formula from \cite{knutson-cycling}.

Finally, we study the image of tilted Richardson varieties under the natural projection $\pi_k:\fl_n \rightarrow \mathrm{Gr}(k,n)$. The images of Richardson varieties are studied by Knutson-Lam-Speyer in \cite{KLSjuggling,KLSprojection}. We show that the set of projection images of tilted Richardson varieties coincides with the set of images of Richardson varieties. We also introduce \emph{k-tilted Bruhat order}, denoted as $\leq_\ba^k$, extending the \emph{k-Bruhat order} \cite{BS98,LS82a}, and use it to characterize when $\pi_k$ is birational on a tilted Richardson variety.

\begin{theorem}
    For any $u,v\in S_n$, the projection image $\Pi_{u,v}:= \pi_k(\cT_{u,v})$ is a positroid variety. Moreover, the map $\pi_k:\cT_{u,v} \rightarrow \Pi_{u,v}$ is birational if and only if $u\leq_{\ba}^k v$. 
\end{theorem}

The structure of this paper is as follows. In \Cref{sec:preliminaries}, we present the necessary preliminaries. In \Cref{sec:graph}, we introduce the $\ba$-tilted Bruhat orders on $S_n$, use them to characterize tilted Bruhat intervals, and develop a theory of tilted reduced words. In \Cref{sec:tilted-richardson-definition}, we provide four equivalent definitions of tilted Richardson varieties. In \Cref{sec:tilted-richardson-geometric}, we establish the fundamental geometric properties of these varieties, including their stratification, dimension formula, and closure relations. In \Cref{sec:deodhar}, we develop the theory of tilted Deodhar decomposition, and present two applications of this decomposition: the definition of tilted Kazhdan--Lusztig R-polynomials, and a construction of the totally nonnegative parts of tilted Richardson varieties, which form a CW-complex whose face poset is closely related to the tilted Bruhat intervals. In \Cref{sec:curve-neighborhood-sec}, we establish an equivalence between tilted Richardson varieties and minimal-degree two-point curve neighborhoods, compute their cohomology classes, and prove an analogue of the descent-cycling formula. Finally, in \Cref{sec:projection}, we show the projection images of tilted Richardson varieties onto $\mathrm{Gr}(k,n)$ are all positroid varieties, and provide a characterization of when the projection is birational. 

%% file: tex/2-preliminaries.tex
\section{Preliminaries}\label{sec:preliminaries}

\subsection{Combinatorics of the Symmetric Group}\label{sec:prelim-1}
The \emph{symmetric group} $S_n$ consists of all permutations of the set $[n]:=\{1,2,\dots,n\}$. A permutation $w \in S_n$ is represented by its one-line notation $w(1)w(2)\cdots w(n)$, or simplified as $w_1w_2\cdots w_n$. Each permutation $w\in S_n$ also corresponds to an $n\times n$ \emph{permutation matrix}, having an $1$ in each position $(w_k,k)$ and $0$ elsewhere. For instance, $w=2314$ corresponds to the matrix
\[\begin{bmatrix}
    0 & 0 & 1 & 0\\1 & 0 & 0 & 0\\0 & 1 & 0 & 0\\0 & 0 & 0 & 1
\end{bmatrix}.\]

Let $\{t_{ij}=(i\; j):1\leq i<j\leq n\}$ be the set of \emph{transpositions}, and let $\{s_{i}=(i
\; i{+}1):1\leq i\leq n-1\}$ be the set of \emph{simple transpositions}. The \emph{identity permutation} is denoted by $\id:=12\cdots n$ (also written as $1$ when the context is clear), and the \emph{longest permutation} is denoted by $w_0:=n\cdots21$. The set of \emph{inversions} of $w\in S_n$ is defined as
\[\inv(w):=\{(i,j)\in[n]^2:i<j,\ w_i>w_j\}.\]
The \emph{length} of $w$, denoted $\ell(w)$, is defined as the number of inversions $\ell(w):=|\inv(w)|$. The \emph{descent set} and \emph{ascent set} of $w$ are defined by
\[\Des(w):=\{i\in[n-1]:w_i>w_{i+1}\},\quad\Asc(w):=\{i\in[n-1]:w_i<w_{i+1}\}.\]

The \emph{(strong) Bruhat order} is a partial order on $S_n$ defined by the transitive closure of the relations $w<wt_{ij}$ whenever $w_i<w_j$, for all $1\leq i<j\leq n$ and $w\in S_n$. It is a graded poset with rank function given by the length.  A \emph{Bruhat interval} $[u,v]$ is defined, for permutations $u\leq v$, as the set of all permutations between $u$ and $v$ in the Bruhat order. \emph{Ehresmann criterion} \cite{ehresmann} characterizes the Bruhat order efficiently. 
\begin{defin}\label{def:gale-order}
Given two subsets $A,B\in \binom{[n]}{k}$ with $A=\{a_1<a_2<\cdots<a_k\}$ and $B=\{b_1<b_2<\cdots<b_k\}$, we say that $A\leq B$ in the \emph{Gale order} if $a_i\leq b_i\text{ for all }i\in [k]$.
\end{defin}

\begin{prop}[Ehresmann Criterion \cite{ehresmann}]\label{prop:ehresmann}
    Let $u, v \in S_n$ be permutations. For each $k \in [n]$, define $u[k] := \{u_1, u_2, \dots, u_k\}$ and $v[k] := \{v_1, v_2, \dots, v_k\}$. Then
    \[u\leq v\text{ in the Bruhat order}\iff u[k]\leq v[k]\text{ in the Gale order for all }k\in [n].\]
\end{prop}

As a Coxeter group, $S_n$ is generated by the simple transpositions $\{s_i:i\in [n-1]\}$, subject to the following relations, last two of which known as the \emph{braid relations}:
\begin{enumerate}
    \item $s_i^2=1$,
    \item $s_is_j=s_js_i$ if $|i-j|>1$,
    \item $s_is_{i+1}s_i=s_{i+1}s_is_{i+1}$.
\end{enumerate}

A \emph{word} for $w\in S_n$, denoted $\mathbf{w}=s_{i_1}s_{i_2}\cdots s_{i_\ell}$, is a sequence of simple transpositions whose product equals $w$. The \emph{length} of $\mathbf{w}$ is $\ell$. If $\ell$ is minimal among all such words, then $\mathbf{w}$ is called a \emph{reduced word}. A classical fact is that the length of a reduced word equals $\ell(w)$.

A \emph{subword} $\mathbf{v}$ of a word $\mathbf{w}$, corresponding to $v\in S_n$, is obtained by replacing some of the factors in $\mathbf{w}$ with the identity element $1$, such that the resulting product equals $v$. For example, $1s_211s_1$ is a subword of $s_3s_2s_3s_2s_1$. When clear from context, we may omit the identity elements and simply write $s_2s_1$ as a subword of $s_3s_2s_3s_2s_1$.

Reduced words and their subwords exhibit many fundamental properties. We highlight four of them below. For detailed proofs and a more thorough introduction, see \cite{bjorner-brenti}.
\begin{enumerate}
    \item \textbf{Length Property:} The length $\ell(w)$, which is equal to the length of any reduced word for $w$, serves as a rank function for the Bruhat order.
    \item \textbf{Subword Property:} Let $u,v\in S_n$ and let $\mathbf{v}$ be a reduced word for $v$. Then, $u\leq v$ if and only if there exists a (reduced) subword $\mathbf{u}$ of $\mathbf{v}$ corresponding to $u$. 
    \item \textbf{Word Property:} Any two reduced words for the same permutation $w$ are connected by a sequence of moves involving braid relations.
    \item \textbf{Lifting Property:} If $u<v$ and $i\in \Des(v)\setminus \Des(u)$, then $us_i\leq v$ and $u\leq vs_i$.
\end{enumerate}

The \emph{root system} of type $A_{n-1}$ consists of $\Phi=\{e_i-e_j: 1\leq i,\neq j\leq n\}$, with \emph{positive roots} $\Phi^+=\{e_i-e_j: 1\leq i<j\leq n\}$ and \emph{simple roots} $\Delta=\{\alpha_i:=e_i-e_{i+1}:i\in[n-1]\}$. Its corresponding \emph{Weyl group} is identified with the symmetric group $S_n$, where the reflection across the hyperplane normal to $e_i-e_j$ is identified with the transposition $t_{ij}$.

An ordering of $\Phi^+$, denoted $\gamma_1,\ldots,\gamma_{\binom{n}{2}}$, is called a \emph{reflection ordering} if the root $e_i-e_k$ appears (not necessarily consecutively) between $e_i-e_j$ and $e_j-e_k$ in the ordering for all $i<j<k$. The following classical lemma gives a correspondence between reflection orderings and reduced words for the longest permutation $w_0\in S_n$.

\begin{lemma}[{\cite[Proposition~3]{bjorner1984orderings}}]\label{lemma:bjorner-reflection-order}
Reflection orderings are in bijection with reduced words for the longest permutation $w_0$. Specifically, given a reduced word $s_{i_1}s_{i_2}\cdots s_{i_\ell}$ for $w_0$, the corresponding reflection ordering $\gamma_1,\ldots,\gamma_{\ell}$ is constructed via $\gamma_j=s_{i_1}\cdots s_{i_{j-1}}(\alpha_{i_j})$ for $j=1,\ldots,\ell=\binom{n}{2}$.
\end{lemma}

\begin{ex}
Consider the reduced word $4321=s_3s_1s_2s_1s_3s_2$, which corresponds to the reflection ordering written on top of the arrows in the diagram below. Each root $e_i-e_j$ records the pair of numbers $i$ and $j$ that are swapped at each step.
\[
\begin{tikzcd}
1234 \arrow[r, "e_3-e_4"]& 1243 \arrow[r, "e_1-e_2"] & 2143 \arrow[r, "e_1-e_4"] & 2413 \arrow[r, "e_2-e_4"] & 4213 \arrow[r, "e_1-e_3"] & 4231 \arrow[r, "e_2-e_3"] & 4321.
\end{tikzcd}
\]
\end{ex}


\subsection{Schubert Varieties and Richardson Varieties}\label{sec:prelim-2}

\subsubsection{The Flag Variety}\label{sec:prelim-2-1}

Let $n\in \Z_{>0}$. A \emph{(complete) flag}, denoted $F_\bullet$, is a sequence of nested linear subspaces $0\subseteq F_1\subseteq F_2\subseteq \cdots \subseteq F_n=\C^n$, where $\dim(F_i)=i$ for all $i\in [n]$. The \emph{(complete) flag variety} $\fl_n$ is defined as the set of all complete flags $F_\bullet$ in $\C^n$.

There is an equivalent description of the flag variety from the perspective of Lie theory. Let $G = \mathrm{GL}_n(\C)$ denote the \emph{general linear group} of invertible $n\times n$ complex matrices. Let and $B,B_{-}\subseteq G$ be the \emph{Borel subgroup} and \emph{opposite Borel subgroup}, consisting of upper and lower triangular matrices, respectively. Every flag $F_\bullet$ can be represented by an invertible matrix $M_F$ in $G$, called a \emph{matrix representative} of $F_\bullet$, such that for each $k\in [n]$, the subspace $F_k$ is the column span of the first $k$ columns of $M_F$. Since two matrix representatives of the same flag lie in the same left coset of $B$ in $G$, the flag variety $\fl_n$ can be identified with the quotient space $G/B$. We will denote an element of the flag variety either as $F_\bullet\in \fl_n$, emphasizing the complete flag, or as $gB\in \fl_n$, emphasizing its matrix representative $g\in G$.

\subsubsection{Pl\"ucker Coordinates}\label{sec:prelim-2-2}

Given a flag $F_\bullet$ and any matrix representative $M_F$, the \emph{Pl\"ucker coordinate} $\Delta_I(F_\bullet)$ for a subset $I\in \binom{[n]}{k}$ is the determinant of the submatrix of $M_F$ consisting of rows indexed by $I$ and the first $k$ columns. These Pl\"ucker coordinates are not intrinsically well-defined, as they depend on the choice of the matrix representative, but the collection $\left(\Delta_I(F_\bullet)_{I\in\binom{[n]}{k}}\right)$ is well-defined up to a common scalar multiple and thus determines a point in the projective space. These coordinates naturally give rise to the \emph{Pl\"ucker embedding}
\[\iota:\fl_n\hookrightarrow \prod_{k=1}^n\P^{\binom{[n]}{k}-1},\]
a closed embedding that maps a flag $F_\bullet$ to the tuple $\left((\Delta_I(F_\bullet))_{I\in\binom{[n]}{1}},\dots,(\Delta_I(F_\bullet))_{I\in\binom{[n]}{n}}\right)$.

We write  $\Delta_{i_1, \dots, i_k}(F_\bullet)$ as the determinant of the submatrix of a matrix representative $M_F$, formed by selecting the rows $i_1, \dots, i_k$ in that order and the first $k$ columns. This notation is alternating in the indices: for any permutation $\sigma \in S_k$, 
\[
\Delta_{i_1, \dots, i_k}(F_\bullet) = (-1)^{\ell(\sigma)} \Delta_{i_{\sigma(1)}, \dots, i_{\sigma(k)}}(F_\bullet).
\]
In particular, $\Delta_{i_1, \dots, i_k}(F_\bullet) = 0$ if the indices $i_1,\dots,i_k$ are not all distinct.

When the context suggests a natural ordering on a subset $I=\{i_1,\dots,i_k\}\subseteq [n]$, we may write $\Delta_I(F_\bullet):=\Delta_{i_1,i_2,\dots,i_k}(F_\bullet)$ without requiring the indices $i_1,\dots,i_k$ to be sorted. For example, given a permutation $w \in S_n$, we use the shorthand $\Delta_{w[k]}(F_\bullet) := \Delta_{w_1, \dots, w_k}(F_\bullet)$.

Now for $w\in S_n$, define the \emph{multi-Pl\"ucker coordinate} of the flag
$F_\bullet$ as
\[\Delta_w(F_\bullet):=\prod_{k=1}^{n}\Delta_{w[k]}(F_\bullet).\]

For a set $I=\{i_1,i_2,\dots,i_k\}$, whether ordered or unordered, define
\begin{align*}
    \Delta_{I+i}(F_\bullet) &:= \Delta_{i_1,\dots,i_k,i}(F_\bullet)&\text{ for all }i\in [n],\\
    \Delta_{I-i_j}(F_\bullet) &:= (-1)^{k-j}\Delta_{i_1,\dots,\widehat{i_j},\dots,i_k}(F_\bullet)&\text{ for all }j\in [k].
\end{align*}

The flag variety is cut out in projective space by multi-homogeneous polynomials in the Pl\"ucker coordinates, known as the \emph{incidence Pl\"ucker relations}. Their general form and related proofs can be found in \cite[Section~9.1]{youngtableaux}. Since the full generality is rather involved, we state three special cases of the incidence Pl\"ucker relations that will be particularly useful.
\begin{enumerate}
    \item For any $I\in \binom{[n]}{k}$ and $J \in \binom{[n]}{k-1}$, with $k\in [n]$, we have
    \begin{equation}\label{eqn:incidenceplucker1}
    \Delta_I\Delta_J(F_\bullet) = \sum_{i\in I}\Delta_{I-i}\Delta_{J+i}(F_\bullet).
    \end{equation}
    \item For any $I\in \binom{[n]}{r}$ and $J\in \binom{[n]}{s}$, with $r,s\in [n]$ and $r-s\geq 2$, we have
    \begin{equation}\label{eqn:incidenceplucker2}
     \sum_{i\in I} \Delta_{I-i}\Delta_{J+i}(F_\bullet) = 0.
    \end{equation}
    \item For any $I\in \binom{[n]}{r}$, $J\in \binom{[n]}{s}$, and $j\in [n]$, with $r,s\in [n]$ and $r-s\geq 1$, we have
    \begin{equation}\label{eqn:incidenceplucker3}
    \Delta_{I}\Delta_{J+j}(F_\bullet) = \sum_{i\in I}\Delta_{I-i+j}\Delta_{J+i}(F_\bullet).
    \end{equation}
\end{enumerate}

\subsubsection{Schubert Varieties and Richardson Varieties}\label{sec:prelim-2-3}

For a permutation $w\in S_n$, the \emph{Schubert cell} $X_w^\circ := B e_w$ (equivalently, $X_w^\circ= BwB$) is defined as the $B$-orbit of the $T$-fixed point $e_w$ in $\fl_n$. Similarly, the \emph{opposite Schubert cell} is defined as the $B_{-}$-orbit $\Omega_w^\circ:= B_{-} e_w= B_{-}wB$. The \emph{Schubert variety} and \emph{opposite Schubert variety} are defined as the closure of these cells
\[X_w := \overline{X_w^\circ},\quad\Omega_w := \overline{\Omega_w^\circ}.\]

There is also a natural matrix representative for all flags in a Schubert cell. By performing column reduction, every flag $F_\bullet\in X_{w,\ba}^\circ$ is represented by a unique matrix with pivot $1$'s at positions $(w_k,k)$ for all $k\in [n]$, and $0$'s below and to the right of each pivot $1$. We refer to this matrix as the \emph{canonical representative} or the \emph{canonical matrix} of $F_\bullet$. Similarly, every flag $F_\bullet\in \Omega_{w,\ba}^\circ$ is represented by a unique matrix with pivot $1$'s at positions $(w_k,k)$, and $0$'s above and to the right of each pivot $1$. An example of these canonical matrices is shown in \Cref{fig:Schub-cell}. It is immediate from this description that $X_{w}^\circ\cong \C^{\ell(w)}$ and $\Omega_w^\circ \cong \C^{\binom{n}{2}-\ell(w)}$.
\begin{figure}[ht]
    \centering
    \subcaptionbox{$X_{w}^\circ$\label{fig:Xw_classical}}[.4\textwidth]{
    \begin{tikzpicture}[scale = 0.7]
    \draw (0,0)--(4,0)--(4,4)--(0,4)--(0,0);
    \draw (1,0) -- (1,4);
    \draw (2,0) -- (2,4);
    \draw (3,0) -- (3,4);
    \draw (0,1) -- (4,1);
    \draw (0,2) -- (4,2);
    \draw (0,3) -- (4,3);
    \node at (-0.5,0.5) {$4$};
    \node at (-0.5,1.5) {$3$};
    \node at (-0.5,2.5) {$2$};
    \node at (-0.5,3.5) {$1$};
    \node at (0.5,-0.5) {$1$};
    \node at (1.5,-0.5) {$2$};
    \node at (2.5,-0.5) {$3$};
    \node at (3.5,-0.5) {$4$};
    \node at (0.5,0.5) {$0$};
    \node at (1.5,0.5) {$0$};
    \node at (2.5,0.5) {$0$};
    \node at (3.5,0.5) {$1$};
    \node at (0.5,1.5) {$0$};
    \node at (1.5,1.5) {$1$};
    \node at (2.5,1.5) {$0$};
    \node at (3.5,1.5) {$0$};
    \node at (0.5,2.5) {$1$};
    \node at (1.5,2.5) {$0$};
    \node at (2.5,2.5) {$0$};
    \node at (3.5,2.5) {$0$};
    \node at (0.5,3.5) {$\ast$};
    \node at (1.5,3.5) {$\ast$};
    \node at (2.5,3.5) {$1$};
    \node at (3.5,3.5) {$0$};
    \end{tikzpicture}}
    \subcaptionbox{$\Omega_{w}^\circ$\label{fig:Omegaw_classical}}[.4\textwidth]{
    \begin{tikzpicture}[scale = 0.7]
    \draw (0,0)--(4,0)--(4,4)--(0,4)--(0,0);
    \draw (1,0) -- (1,4);
    \draw (2,0) -- (2,4);
    \draw (3,0) -- (3,4);
    \draw (0,1) -- (4,1);
    \draw (0,2) -- (4,2);
    \draw (0,3) -- (4,3);
    \node at (-0.5,0.5) {$4$};
    \node at (-0.5,1.5) {$3$};
    \node at (-0.5,2.5) {$2$};
    \node at (-0.5,3.5) {$1$};
    \node at (0.5,-0.5) {$1$};
    \node at (1.5,-0.5) {$2$};
    \node at (2.5,-0.5) {$3$};
    \node at (3.5,-0.5) {$4$};
    \node at (0.5,0.5) {$\ast$};
    \node at (1.5,0.5) {$\ast$};
    \node at (2.5,0.5) {$\ast$};
    \node at (3.5,0.5) {$1$};
    \node at (0.5,1.5) {$\ast$};
    \node at (1.5,1.5) {$1$};
    \node at (2.5,1.5) {$0$};
    \node at (3.5,1.5) {$0$};
    \node at (0.5,2.5) {$1$};
    \node at (1.5,2.5) {$0$};
    \node at (2.5,2.5) {$0$};
    \node at (3.5,2.5) {$0$};
    \node at (0.5,3.5) {$0$};
    \node at (1.5,3.5) {$0$};
    \node at (2.5,3.5) {$1$};
    \node at (3.5,3.5) {$0$};
    \end{tikzpicture}}
    \caption{Canonical representatives of Schubert cells for $w=2314$}
    \label{fig:Schub-cell}   
\end{figure}

Schubert cells and Schubert varieties can also be described in terms of rank conditions on matrices. For any $i,k \in [n]$ and matrix $M \in G$, define the \emph{south-west rank} $\rank^{SW}_{i,k}(M)$ to be the rank of the $i \times k$ submatrix in the bottom-left corner of $M$. Similarly, define the \emph{north-west rank} $\rank^{NW}_{i,k}(M)$ to be the rank of the $i \times k$ submatrix in the top-left corner of $M$. Let $F_\bullet$ be a flag represented by a matrix $M_F \in G$. Then:
\[
\begin{aligned}
    F_\bullet \in X_w &\iff \rank^{SW}_{i,k}(M_F) \leq \rank^{SW}_{i,k}(w) &&\text{for all } i, k \in [n], \\
    F_\bullet \in \Omega_w &\iff \rank^{NW}_{i,k}(M_F) \leq \rank^{NW}_{i,k}(w) &&\text{for all } i, k \in [n].
\end{aligned}
\]
Moreover, $F_\bullet$ lies in the (opposite) Schubert cell if all $\leq$ above are replaced by $=$. We note that the values of these rank functions are independent of the choice of matrix representative $M_F$. For a detailed discussion of these rank conditions, along with their equivalence to the orbit and the canonical matrix definitions, we refer the readers to \cite[Chapter~10]{youngtableaux}.

The Schubert cells give rise to the \emph{Bruhat decomposition}:
\[
\fl_n = \bigsqcup_{w \in S_n} X_w^\circ = \bigsqcup_{w \in S_n} \Omega_w^\circ.
\]
Moreover, for each $w \in S_n$, the Schubert and opposite Schubert varieties admit decompositions that respect the Bruhat order:
\[
X_w = \bigsqcup_{u \leq w} X_u^\circ \quad \text{and} \quad \Omega_w = \bigsqcup_{w \leq v} \Omega_v^\circ.
\]

For permutations $u,v\in S_n$, the \emph{Richardson variety} is defined as $\cR_{u,v} = X_{v}\cap \Omega_u$, and the \emph{open Richardson variety} as $\cR^\circ_{u,v} = X^\circ_{v}\cap \Omega^\circ_{u}$. Richardson varieties share many of the nice geometric properties of Schubert varieties, although their structure was established much later and through very different techniques. Below, we summarize several key properties. See \cite{richardson} for details and proofs.
\begin{enumerate}
    \item The varieties $\cR^\circ_{u,v}$ and $\cR_{u,v}$ are nonempty if and only if $u \leq v$ in the Bruhat order. In that case, we have $\dim(\cR_{u,v}) = \dim(\cR_{u,v}^\circ) = \ell(v)-\ell(u)$.
    \item The Richardson varieties admit a stratification $\cR_{u,v} = \bigsqcup_{[x,y]\subseteq [u,v]}\cR_{x,y}^\circ$, where the disjoint union is over all Bruhat intervals $[x,y]$ contained in $[u,v]$.
    \item The open Richardson variety $\cR_{u,v}^{\circ}$ is the transverse intersection of $X_v^\circ$ and $\Omega_u^\circ$, and is smooth and dense in $\cR_{u,v}$.
\end{enumerate}

Finally, as subvarieties of $\fl_n$, Schubert varieties and Richardson varieties can also be described by the vanishing of multi-Pl\"ucker coordinates. Since there is limited literature on multi-Pl\"ucker coordinates, we include a brief proof of the following statement.

\begin{lemma}\label{lemma:SchubPlucker}
    Let $u, v \in S_n$ be permutations. Then
    \[\begin{aligned}
        X_v &= \{F_\bullet\in \fl_n:   \Delta_w(F_\bullet) = 0 \text{ for all }w \nleq v\}, & X_v^\circ &= X_v\cap \{\Delta_v \neq 0\},\\
        \Omega_u &= \{F_\bullet\in \fl_n:  \Delta_w(F_\bullet) = 0 \text{ for all }w \ngeq u\}, & \Omega_u^\circ &= \Omega_u \cap \{\Delta_u \neq 0\},\\
        \cR_{u,v} &= \{F_\bullet\in \fl_n:  \Delta_w(F_\bullet) = 0 \text{ for all } w \notin [u,v]\}, & \cR^\circ_{u,v} &= \cR_{u,v} \cap \{\Delta_u\Delta_v \neq 0\}.
    \end{aligned}
    \]
\end{lemma}
\begin{proof}
We prove the statement for $X_v$, and the remaining cases follow naturally. For the containment $\subseteq$, suppose $F_\bullet \in X_v$ and $\Delta_w(F_\bullet) \neq 0$ for some $w \in S_n$. Then, by analyzing the rank conditions defining $X_v$, we must have $w[k] \leq v[k]$ for all $k \in [n]$. By \Cref{prop:ehresmann}, this implies $w \leq v$. Therefore, if $w \nleq v$, we must have $\Delta_w(F_\bullet) = 0$, as claimed.

For the reverse containment, assume that $\Delta_w(F_\bullet) = 0$ for all $w \nleq v$, and we need $F_\bullet \in X_v$. By the Bruhat decomposition, $F_\bullet \in X_u^\circ$ for some $u \in S_n$. If $u \nleq v$, then $\Delta_u(F_\bullet) \neq 0$, contradicting the assumption. Hence $u \leq v$, and so $F_\bullet \in X_u^\circ \subseteq X_v$.
\end{proof}

\subsubsection{Related Schubert Geometry in the Grassmannian}\label{sec:prelim-2-4}

For $n \in \Z_{>0}$ and $k \in [n]$, the \emph{Grassmannian} $\gr_{k,n}$ is the space of all $k$-dimensional linear subspaces of $\C^n$:
\[\gr_{k,n}:=\{V\subseteq \C^n:\dim(V)=k\}.\]
The Grassmannian has a (simpler) Schubert geometry theory analogous to that of the flag variety. See \cite[Chapter~9]{youngtableaux} for details.

A \emph{matrix representative} of a subspace $V \in \gr_{k,n}$ is a $n \times k$ matrix $M_V$ whose column span is equal to $V$. For a subset $I \in \binom{[n]}{k}$, the \emph{Pl\"ucker coordinate} $\Delta_I(V)$ is defined as the determinant of the submatrix of $M_V$ with rows indexed by $I$. Similar to the case of the flag variety, these coordinates define the \emph{Pl\"ucker embedding} $\iota:\gr_{k,n}\hookrightarrow\P^{\binom{n}{k}-1}$ that sends a subspace $V$ to $\left(\Delta_{I}(F_\bullet)\right)_{I\in\binom{[n]}{k}}$. The Pl\"ucker coordinates satisfy the \emph{Pl\"ucker relations}:
\[\Delta_I\Delta_J(V)=\sum_{i\in I}\Delta_{I-i}\Delta_{J+i}(V),\quad\text{for any } I\in\binom{[n]}{k+1} \text{ and }J\in\binom{[n]}{k-1}.\]

Let $V \in \gr_{k,n}$ be represented by a matrix $M_V$. For any subset $K \subseteq [n]$, let $\rank_K(V)$ denote the rank of the submatrix of $M_V$ formed by the rows indexed by $K$.  Analogous to the flag variety, for $I \in \binom{[n]}{k}$ we define the following four \emph{Grassmannian Schubert varieties} and \emph{cells} in $\gr_{k,n}$ using rank conditions:
\[
\begin{aligned}
    X_I &:= \{V\in \gr_{k,n}: \rank_{[n-j+1,n]}(V) \leq \#(I\cap [n-j+1,n])\text{ for all }j\in[n]\},\\
    X_I^\circ &:= \{V\in \gr_{k,n}: \rank_{[n-j+1,n]}(V) = \#(I\cap [n-j+1,n])\text{ for all }j\in[n]\},\\
    \Omega_I &:= \{V\in \gr_{k,n}: \rank_{[j]}(V) \leq \#(I\cap [j])\text{ for all }j\in[n]\},\\
    \Omega_I^\circ &:= \{V\in \gr_{k,n}: \rank_{[j]}(V) = \#(I\cap [j])\text{ for all }j\in[n]\}.
\end{aligned}
\]

As in the flag variety, we have $X_I = \overline{X_I^\circ}$ and $\Omega_I = \overline{\Omega_I^\circ}$. Moreover, the Grassmannian Schubert varieties admit stratifications with respect to the Gale order on $\binom{[n]}{k}$ (\Cref{def:gale-order}):
\[X_{J} = \bigsqcup_{I\leq J}X_I^\circ\ \quad\text{and}\quad \Omega_J = \bigsqcup_{I\geq J}\Omega_I^\circ.\]

When $I \leq J$ in the Gale order, we define the \emph{open Grassmannian Richardson variety} as $\cR_{I,J}^\circ := X_J^\circ \cap \Omega_I^\circ$, and its closure, the \emph{Grassmannian Richardson variety}, as $\cR_{I,J} := X_J \cap \Omega_I$. Let $[I,J]$ denote the interval in the Gale order. Then, as in the flag variety, we have:
\[
\cR_{I,J} = \overline{\cR_{I,J}^{\circ}} \quad \text{and} \quad \cR_{I,J} = \bigsqcup_{[I',J'] \subseteq [I,J]} \cR_{I',J'}^\circ.
\]

Finally, these varieties can also be defined by vanishing of Pl\"ucker coordinates.

\begin{lemma}\label{lemma:grSchubPlucker}
Let $I,J\in\binom{[n]}{k}$. Then we have
\[
    \begin{aligned}
        X_J &= \{V\in \gr_{k,n}: \Delta_K(V) = 0\text{ for all }K \nleq J\} & X_J^\circ &= X_J\cap \{\Delta_J \neq 0\}\\
        \Omega_I &= \{V\in \gr_{k,n}:   \Delta_K(V) = 0\text{ for all }K \ngeq I\} & \Omega_I^\circ &= \Omega_I \cap \{\Delta_I \neq 0\}\\
        \cR_{I,J} &= \{V\in \gr_{k,n}:   \Delta_K(V) = 0\text{ for all }K \notin [I,J]\} & \cR^\circ_{I,J} &= \cR_{I,J} \cap \{\Delta_I\Delta_J \neq 0\}.
    \end{aligned}
\]
\end{lemma}


\subsection{Classical and Quantum Schubert Calculus}\label{sec:prelim-3}

\subsubsection{Schubert Calculus}\label{sec:prelim-3-1}

The central goal of Schubert calculus is to study the intersection theory of Schubert varieties, which is encoded in the cohomology ring $H^*(\fl_n)$. As a consequence of the Bruhat decomposition, this ring admits a $\Z$-basis consisting of \emph{Schubert classes} $\{\sigma_w:w\in S_n\}$, where each $\sigma_w$ represents the cohomology class of a Schubert variety:
\[\sigma_w:=[\Omega_w]=[X_{w_0w}]\in H^{2\ell(w)}(\fl_n).\]
In particular, $\sigma_\id=1\in H^0(\fl_n)$ is the fundamental class, and $\sigma_{w_0}=[\pt]\in H^{2\dim(\fl_n)}(\fl_n)$ is the class of a point. Let $\int_{\fl_n}:H^{2\dim(\fl_n)}(\fl_n)\to \Z$ denote the \emph{integration map}, which sends the cohomology class of a point $[\pt]$ to $1$. The following lemma shows that Schubert classes are self-dual under the cup product:

\begin{lemma}[Duality Lemma]\label{lemma:duality}
    Let $u,v\in S_n$ such that $\ell(u)+\ell(v)=\binom{n}{2}$. Then
    \[\int_{\fl_n}\sigma_u\cdot \sigma_v=
    \begin{cases}
        1, &\text{if }v=w_0u,\\
        0, &\text{otherwise}.
    \end{cases}\]
\end{lemma}
Given $u,v,w\in S_n$, the \emph{(generalized) Littlewood--Richardson coefficients} $c_{u,v}^w$ are the structure constants appearing in the expansion of the cup product of Schubert classes:
\[\sigma_u\cdot \sigma_v=\sum_{w\in S_n}c_{u,v}^w\sigma_w.\]
$c_{u,v}^w\in\mathbb{N}$ and is nonzero only if $\ell(w)=\ell(u)+\ell(v)$. By the duality lemma, they count the number of points in the transverse intersection of three Schubert varieties $X_{w_0u}$, $X_{w_0v}$, and $X_w$ in general position. Moreover, the coefficients exhibit an $S_3$-symmetry:
\[c_{u,v}^w=c_{v,u}^w=c_{u,w_0w}^{w_0v}=c_{w_0w,u}^{w_0v}=c_{v,w_0w}^{w_0u}=c_{w_0w,v}^{w_0u}.\]
A central open problem in Schubert calculus is to find a manifestly positive combinatorial rule for computing them.

These coefficients also appear in the cohomology classes of Richardson varieties. Since the open Richardson variety $\cR_{u,v}^\circ$ is a transverse intersection of $\Omega_u^\circ$ and $X_v^\circ$, we have
\[[\cR_{u,v}]=[\Omega_u]\cdot[X_v]=\sigma_u\cdot\sigma_{w_0v}=\sum_{w\in S_n}c_{u,w}^v\sigma_{w_0w}.\]

\subsubsection{Quantum Cohomology}\label{sec:prelim-3-3}
In this section, we provide an overview of the quantum cohomology of $\fl_n$. For a more detailed exposition, see \cite{quantum-intro}.

A \emph{degree} $d=(d_1,d_n,\dots,d_n)\in\Z_{\geq 0}^{n-1}$ is a tuple of nonnegative integers representing the homology class $\sum_{i=1}^{n-1}d_i[X_{s_i}]\in H_2(\fl_n)$, where $[X_{s_i}]$ denotes the homology class of the Schubert curve $X_{s_i}$. Let $C$ be a connected, reduced, nodal curve of genus $0$. A \emph{degree-$d$ stable map} is a morphism $f: C\to\fl_n$ such that $f_\ast[C]=d$, and the map satisfies certain stability conditions. The \emph{Kontsevich moduli space} $\overline{M}_{0,3}(\fl_n,d)$ parametrizes degree-$d$, genus-$0$ stable maps $f:C\to\fl_n$ with three marked points $p_1,p_2,p_3\in C$. The moduli space $\overline{M}_{0,3}(\fl_n,d)$ is an irreducible projective variety of pure dimension $\binom{n}{2} + 2|d|$, where $|d|:=\sum_{i=1}^{n-1}d_i$.

For $i\in\{1,2,3\}$, define the \emph{evaluation map} $\ev_i:\overline{M}_{0,3}(\fl_n,d)\rightarrow \fl_n$ which sends a point corresponding to a map $f:C\to\fl_n$ to the image $f(p_i)$ of the $i$-th marked point $p_i$. Given permutations $u,v,w\in S_n$ and degree $d$, the \emph{Gromov--Witten Invariant} $c_{u,v}^{w,d}$ is defined by:
\[c_{u,v}^{w,d}:=\int_{\overline{M}_{0,3}(\fl_n,d)}\ev_1^\ast(\sigma_u)\cdot\ev_2^\ast(\sigma_v)\cdot\ev_3^\ast(\sigma_{w_0w})\in \Z.\]
These invariants are nonnegative integers and are nonzero only if $\ell(u)+\ell(v)=\ell(w)+2|d|$. Geometrically, $c_{u,v}^{w,d}$ counts the number of rational curves of degree $d$ intersecting the Schubert varieties $X_{w_0u}$, $X_{w_0v}$, and $X_{w}$ in general position. These invariants satisfy a similar $S_3$-symmetry:
\[c_{u,v}^{w,d}=c_{v,u}^{w,d}=c_{u,w_0w}^{w_0v,d}=c_{w_0w,u}^{w_0v,d}=c_{v,w_0w}^{w_0u,d}=c_{w_0w,v}^{w_0u,d}.\]

Let $\Z[q]:=\Z[q_1,\dots,q_{n-1}]$ with $\deg(q_i)=2$. The \emph{(small) quantum cohomology ring} $QH^\ast(\fl_n)$ is defined as the free $\Z[q]$-module generated by the Schubert classes $\{\sigma_w:w\in S_n\}$, equipped with the \emph{quantum product} $\star$ defined by
\[\sigma_u\star\sigma_v=\sum_{\substack{w\in S_n\\d\in\Z_{\geq 0}^{n-1}}}c_{u,v}^{w,d}\; q^d\sigma_w,\]
where $q^d:=q_1^{d_1}\cdots q_{n-1}^{d_{n-1}}$. We use the symbol $\star$ to denote the quantum product on $QH^\ast(\fl_n)$, to distinguish it from the classical cup product $\sigma_u\cdot\sigma_v$ on $H^\ast(\fl_n)$.

Setting the quantum parameters $q_1=\cdots=q_{n-1}=0$, the quantum cohomology ring $QH^\ast(\fl_n)$ reduces to the classical $H^\ast(\fl_n)$. Moreover, when the degree $d=0$, the Gromov--Witten invariants $c_{u,v}^{w,d}$ recover the classical Littlewood--Richardson coefficients $c_{u,v}^w$.

\subsubsection{Two-Point Curve Neighborhoods}\label{sec:prelim-3-4}

For any permutations $u,v\in S_n$ and degree $d$, the \emph{Gromov--Witten variety} is defined as
\[GW_d(\Omega_u,X_v):= \ev_1^{-1}(\Omega_u)\cap \ev_2^{-1}(X_v)\subseteq \overline{M}_{0,3}(\fl_n,d).\]
These varieties exhibit many nice geometric properties. It was shown in \cite[Lemma~7.3]{FW} that if $GW_d(\Omega_u,X_v)$ is nonempty, then it is a reduced, closed subvariety of pure dimension $\ell(v)-\ell(u)+2|d|$. Later, \cite[Corollary~3.3]{curveneighborhood} established that it is irreducible.

The associated \emph{two-point curve neighborhood} is defined as the image
\[\Gamma_d(\Omega_u,X_v) :=\ev_3(GW_d(\Omega_u,X_v))\subseteq\fl_n.\]
Geometrically, this is the union of all degree-$d$ rational curves in $\fl_n$ that intersect both the Schubert variety $X_v$ and the opposite Schubert variety $\Omega_u$. These curve neighborhoods were introduced and studied in \cite{curveneighborhood, BCMP}. They also exhibit many nice geometric properties, and their cohomology classes encode information about Gromov--Witten invariants, as shown in the following proposition. Since we could not find an exact reference for the statement below, we include a proof.
\begin{prop}\label{prop:Gamma-nice}
    For any $u,v\in S_n$ and degree $d$, the two-point curve neighborhood $\Gamma_d(\Omega_u,X_v)$ is either empty or an irreducible, closed subvariety of $\fl_n$. If $q^d$ appears in the quantum product $\sigma_u\star \sigma_{w_0v}$, then $\Gamma_d(\Omega_u,X_v)$ is of pure dimension $\ell(v)-\ell(u)+2|d|$, and its cohomology class is given by:
    \[[\Gamma_d(\Omega_u,X_v)]=\frac{1}{c}\;[q^d]\sigma_u\star\sigma_{w_0v}=\frac{1}{c}\sum_{w\in S_n}c_{u,w}^{v,d}\sigma_{w_0w}\in H^\ast(\fl_n),\]
    where $[q^d]\sigma_u\star\sigma_{w_0v}$ denotes the component of quantum weight $q^d$ in the expansion, and $c\in\Z_{>0}$ is the degree of the projection map $\ev_3:GW_d(\Omega_u,X_v)\to\Gamma_d(\Omega_u,X_v)$.
\end{prop}
\begin{proof}
    The first statement follows immediately from the irreducibility of $GW_d(\Omega_u,X_v)$. By the projection formula, for any $w\in S_n$, we have
    \[\int_{\fl_n}(\ev_3)_*[GW_d(\Omega_u,X_v)]\cdot \sigma_{w} =\int_{\overline{M}_{0,3}(\fl_n,d)}\ev_1^\ast(\sigma_u)\cdot\ev_2^\ast(\sigma_{w_0v})\cdot\ev_3^\ast(\sigma_{w})=c_{u,w}^{v,d}.\]
    Applying \Cref{lemma:duality}, this implies
    \[(\ev_3)_\ast[GW_d(\Omega_u,X_v)]=\sum_{w\in S_n}c_{u,w}^{v,d}\sigma_{w_0w}=[q^d]\sigma_u\star\sigma_{w_0v}\in H^\ast(\fl_n).\]
    The dimension formula and the cohomology class then follow from the proper pushforward:
    \[(\ev_3)_*[GW_d(\Omega_u,X_v)]=\begin{cases}
        c\cdot [\Gamma_d(\Omega_u,X_v)] &\text{if }\Gamma_d(\Omega_u,X_v))\text{ has the expected dimension},\\
        0 &\text{otherwise}.
    \end{cases}.\]
\end{proof}

We remark that when $d=0$, the curve neighborhood recovers the Richardson variety $\Gamma_0(\Omega_u,X_v)=\cR_{u,v}$. Although the curve neighborhoods $\Gamma_d(\Omega_u,X_v)$ play an important role in the study of Gromov--Witten invariants, there is currently no explicit characterization of when a flag $F_\bullet\in\Gamma_d(\Omega_u,X_v)$, except in certain special cases. In \cite{BM15}, it was shown that the \emph{one-point curve neighborhoods} $\Gamma_d(X_v):= \Gamma_d(\Omega_\id,X_v)$ coincide with specific Schubert varieties, and a combinatorial description was provided to identify which Schubert variety arises. Separately, in \cite{LiMihalcea}, it was shown that when $|d|=1$, the two-point curve neighborhoods $\Gamma_d(\Omega_u,X_v)$ coincide with certain Richardson varieties. Despite these partial results, it would be desirable to have an explicit criterion for when a flag lies in a two-point curve neighborhood in the general case. We leave the following for further investigations. 
\begin{prob}
Give an explicit combinatorial criterion for when a flag $F_\bullet$ lies in the curve neighborhood $\Gamma_d(\Omega_u,X_v)$.
\end{prob}

\subsubsection{The Quantum Bruhat Graph}\label{sec:prelim-3-5}

In \cite{BFP-tilted-Bruhat}, Brenti, Fomin, and Postnikov introduced the \emph{quantum Bruhat graph} $\Gamma_n$ in their study of the quatum cohomology of $\fl_n$. The graph is a weighted directed graph on $S_n$ with the following two types of edges:
\[
\begin{cases}
w\rightarrow wt_{ij}\text{ of weight }1&\text{ if }\ell(wt_{ij})=\ell(w)+1,\\
w\rightarrow wt_{ij}\text{ of weight }q_{ij}:=q_iq_{i+1}\cdots q_{j-1}&\text{ if }\ell(wt_{ij})=\ell(w)+1-2(j-i),
\end{cases}
\]
where $1\leq i<j\leq n$. Edges of the first type are called \emph{strong Bruhat edges}, while those of the second type are called \emph{quantum edges}. An example for $n=3$ is shown in \Cref{fig:quantum-Bruhat-graph-n=3}.

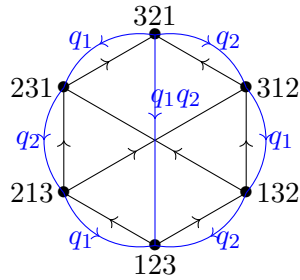
\begin{figure}[ht]
\centering
\begin{tikzpicture}[scale=0.7]
\coordinate (a) at (0,-2);
\coordinate (b) at (1.73205080757,-1);
\coordinate (c) at (1.73205080757,1);
\coordinate (d) at (0,2);
\coordinate (e) at (-1.73205080757,1);
\coordinate (f) at (-1.73205080757,-1);
\node at (a) {$\bullet$};
\node[below] at (a) {$123$};
\node at (b) {$\bullet$};
\node[right] at (b) {$132$};
\node at (c) {$\bullet$};
\node[right] at (c) {$312$};
\node at (d) {$\bullet$};
\node[above] at (d) {$321$};
\node at (e) {$\bullet$};
\node[left] at (e) {$231$};
\node at (f) {$\bullet$};
\node[left] at (f) {$213$};
\draw[postaction={decorate,decoration={markings,mark=at position 0.5 with {\arrow{>}}}}] (a) -- (b);
\draw[postaction={decorate,decoration={markings,mark=at position 0.5 with {\arrow{>}}}}] (a) -- (f);
\draw[postaction={decorate,decoration={markings,mark=at position 0.5 with {\arrow{>}}}}] (b) -- (c);
\draw[postaction={decorate,decoration={markings,mark=at position 0.4 with {\arrow{>}}}}] (b) -- (e);
\draw[postaction={decorate,decoration={markings,mark=at position 0.4 with {\arrow{>}}}}] (f) -- (c);
\draw[postaction={decorate,decoration={markings,mark=at position 0.5 with {\arrow{>}}}}] (f) -- (e);
\draw[postaction={decorate,decoration={markings,mark=at position 0.5 with {\arrow{>}}}}] (e) -- (d);
\draw[postaction={decorate,decoration={markings,mark=at position 0.5 with {\arrow{>}}}}] (c) -- (d);
\draw[blue, postaction={decorate,decoration={markings,mark=at position 0.5 with {\arrow{>}}}}] (b) to[bend left=40] (a);
\node[blue] at (1.4,-1.9) {$q_2$};
\draw[blue, postaction={decorate,decoration={markings,mark=at position 0.5 with {\arrow{>}}}}] (c) to[bend left=40] (b);
\node[blue] at (2.4,0) {$q_1$};
\draw[blue, postaction={decorate,decoration={markings,mark=at position 0.5 with {\arrow{>}}}}] (d) to[bend left=40] (c);
\node[blue] at (1.4,1.9) {$q_2$};
\draw[blue, postaction={decorate,decoration={markings,mark=at position 0.5 with {\arrow{>}}}}] (d) to[bend right=40] (e);
\node[blue] at (-1.4,1.9) {$q_1$};
\draw[blue, postaction={decorate,decoration={markings,mark=at position 0.5 with {\arrow{>}}}}] (e) to[bend right=40] (f);
\node[blue] at (-2.4,0) {$q_2$};
\draw[blue, postaction={decorate,decoration={markings,mark=at position 0.5 with {\arrow{>}}}}] (f) to[bend right=40] (a);
\node[blue] at (-1.4,-1.9) {$q_1$};
\draw[blue, postaction={decorate,decoration={markings,mark=at position 0.4 with {\arrow{>}}}}] (d) to (a);
\node[blue] at (0.4,0.8) {$q_1q_2$};
\end{tikzpicture}
\caption{The quantum Bruhat graph $\Gamma_3$ (unlabeled edges have weight $1$)}
\label{fig:quantum-Bruhat-graph-n=3}
\end{figure}
The quantum Bruhat graph encodes the \emph{quantum Chevalley--Monk formula} in the quantum cohomology ring $QH^\ast(\fl_n)$, first established in \cite[Theorem~1.3]{FGP-quantum-Schubert}:
\[\sigma_{s_k}\star \sigma_w=\sum_{w\to wt_{ij}}\wt(w\to wt_{ij})\sigma_{wt_{ij}},\]
where the sum is over all transpositions $t_{ij}$ with $i \leq k < j$ such that $w \to wt_{ij}$ is an edge in $\Gamma_n$, and $\wt(w\to w_{ij})$ denotes the weight of that edge.

An alternative criterion for the edge conditions in the quantum Bruhat graph can be uniformly expressed using the notion of \emph{cyclic intervals}:

\begin{defin}\label{def:cyclicinterval}
For $a, b \in [n]$, define the \emph{cyclic interval} $[a, b)_c$ as follows:
\[[a, b)_c :=\begin{cases}
\{a, a+1, \ldots, b-1\} & \text{if } a \leq b, \\
\{a, a+1, \ldots, n\} \cup \{1, \ldots, b-1\} & \text{if } a > b.\end{cases}\]
Cyclic intervals of the forms $[a,b]_c$, $(a,b]_c$, and $(a,b)_c$ are defined analogously.
\end{defin}

With this notation, the edge $w\rightarrow wt_{ij}$ is in the quantum Bruhat graph if and only if $w_k\not\in (w_i,w_j)_c$ for all $i<k<j$. Since this condition is invariant under cyclic shifts, we obtain the following immediate observation regarding the cyclic symmetry of the graph. See also \cite[Corollary~12]{cyclic-bruhat-graph}.
\begin{lemma}\label{lemma:cyclic-symmetry}
Let $\cyclic\in S_n$ be the long cycle $\cyclic=23\cdots n1$. Then the map $w\mapsto \cyclic w$ is an automorphism of the unweighted quantum Bruhat graph.
\end{lemma}

For a directed path $P$ in the quantum Bruhat graph $\Gamma_n$, we define its \emph{weight} $\wt(P) \in \Z[q] = \Z[q_1, \dots, q_{n-1}]$ as the product of the weights of the edges along the path. The following lemma, proved by Postnikov in \cite[Lemma~1]{Postnikov-quantum-Bruhat-graph}, asserts that all shortest paths correspond precisely to those with minimal weight. Although the final clause of the lemma is not stated explicitly in \cite[Lemma~1]{Postnikov-quantum-Bruhat-graph}, it follows directly from the proof given there.

\begin{lemma}\label{lemma:shortest-length-equals-weight}
For any $u, v \in S_n$, there exists a directed path from $u$ to $v$ in $\Gamma_n$. Moreover, there exists a degree $d_{u,v} \in \Z_{\geq 0}^{n-1}$ such that all shortest paths from $u$ to $v$ have the same weight $q^{d_{u,v}}$. In fact, the weight of any path from $u$ to $v$ is divisible by $q^{d_{u,v}}$, and any path with weight exactly $q^{d_{u,v}}$ must be a shortest path. \end{lemma}

We refer to such a degree $d_{u,v}$ as the \emph{minimal degree}, and the corresponding weight $q^{d_{u,v}}$ as the \emph{minimal weight}. The minimal degree $d_{u,v}$ is sometimes written as $d(u,v) = (d_1(u,v), \dots, d_{n-1}(u,v))$, or simply as $d_{\min}$ when the context is clear.

\begin{ex}
Let $u = 321$ and $v = 213$. As shown in \Cref{fig:quantum-Bruhat-graph-n=3}, there are two shortest paths from $u$ to $v$, each of length 2: $321 \rightarrow 231 \rightarrow 213$ and $321 \rightarrow 123 \rightarrow 213$.
Both paths have the same minimal weight $q_1 q_2$, so the minimal degree is $d_{321,213} = (1,1)$. It is straightforward to check that any other path from $u$ to $v$ has weight divisible by (and not equal to) $q_1q_2$.
\end{ex}

We remark that the minimal degrees $d_{u,v}$ are closely related to the unique minimal quantum weight $q^d$ that appears in a quantum product in $QH^\ast(\fl_n)$, as stated in the following proposition, proved in \cite[Corollary~3]{Postnikov-quantum-Bruhat-graph}.

\begin{prop}\label{prop:postnikov-minimal-degree}
For any $u,v\in S_n$, the minimal weight $q^{d_{u,v}}$ of a shortest path from $u$ to $v$ in the quantum Bruhat graph $\Gamma_n$ is the unique minimal quantum weight that appears in the quantum product of two Schubert classes $\sigma_u \star \sigma_{w_0 v}$ in $QH^\ast(\fl_n)$.
\end{prop}

Finally, we include a nice combinatorial property that is unique to quantum Bruhat graphs $\Gamma_n$. In $\Gamma_n$, label each directed edge $w\to wt_{ij}$ by the root $e_i-e_j\in\Phi^+$. Note that this label refers only to the root and does not reflect the weight of the edge. The following proposition is due to \cite[Theorem~6.6]{BFP-tilted-Bruhat}.
\begin{prop}\label{prop:shortest-path-BFP}
Fix a reflection ordering $\gamma=\gamma_1,\ldots,\gamma_{\binom{n}{2}}$ of the set of positive roots $\Phi^+$. For any $u,v\in S_n$, there exists a unique directed path from $u$ to $v$ in $\Gamma_n$ whose sequence of labels is strictly increasing with respect to $\gamma$. Moreover, this path is a shortest path from $u$ to $v$, of length $\ell(u,v)$.
\end{prop}

\subsubsection{Tilted Bruhat Intervals}\label{sec:prelim-3-6}

In \cite{BFP-tilted-Bruhat}, Brenti, Fomin, and Postnikov introduced \emph{tilted Bruhat intervals} as a tilted analogue of classical Bruhat intervals.

\begin{defin}\label{def:tilted-bruhat-interval} For $u, v \in S_n$, let $\ell(u,v)$ denote the length of the shortest path from $u$ to $v$ in the quantum Bruhat graph $\Gamma_n$. The \emph{tilted Bruhat interval} $[u,v]$ is defined as the set
    \[[u,v]:=\{w\in S_n:\ell(u,w)+\ell(w,v)=\ell(u,v)\},\]
    equipped with the partial order
    \[x\preceq y\iff \ell(u,x)+\ell(x,y)+\ell(y,v)=\ell(u,v).\]
\end{defin}
Intuitively, the tilted Bruhat interval consists of all elements $w \in S_n$ that lie on a shortest path from $u$ to $v$, and $x \preceq y$ if and only if there exists a shortest path from $u$ to $v$ that passes through $x$ before $y$. An example of a tilted Bruhat interval is shown in \Cref{fig:quantum-Bruhat-interval}.

\begin{figure}[ht]
\centering
\begin{tikzpicture}[scale=0.7]
    \node (a) at (0,0) {$\bullet$};
    \node[below] at (a) {$231$};
    \node (b) at (-0.6,1) {$\bullet$};
    \node[left] at (b) {$213$};
    \node (c) at (0,2) {$\bullet$};
    \node[above] at (c) {$123$};
    \node (d) at (0.6,1) {$\bullet$};
    \node[right] at (d) {$321$};
    \draw (0,0)--(-0.6,1)--(0,2)--(0.6,1)--(0,0);
    \end{tikzpicture}
    \caption{The tilted Bruhat interval $[231,123]$}
\label{fig:quantum-Bruhat-interval}
\end{figure}
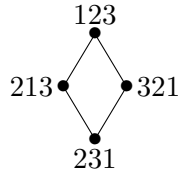
By definition, the subintervals of a tilted Bruhat interval $[u, v]$ are themselves tilted Bruhat intervals $[x, y]$. The interval $[u, v]$ naturally forms a ranked poset, with rank function given by $w \mapsto \ell(u, w)$. Moreover, these intervals are examples of \emph{thin} posets, a property that follows from \cite[Theorem~6.6]{BFP-tilted-Bruhat}.

\begin{lemma}\label{lemma:thin}
Every tilted Bruhat interval $[u, v]$ with $\ell(u, v) = 2$ is isomorphic to the diamond poset.
\end{lemma}

When $u \leq v$ in the classical Bruhat order, the tilted Bruhat interval coincides with the classical Bruhat interval $[u,v]$. 
We provide an efficient combinatorial rule (\Cref{thm:tilted-criterion}) to determine when a given permutation $w$ lies in $[u,v]$ in the tilted setting, analogous to the Ehresmann criterion, filling in a missing piece of the previous literature.


\subsubsection{Path Schubert Polynomials}\label{sec:prelim-3-7}

In \cite{Postnikov-quantum-Bruhat-graph}, Postnikov introduced the \emph{path Schubert polynomials} $\mathfrak{S}_{u,v}(x,q)$ for any permutations $u, v \in S_n$, which encode Gromov--Witten invariants. These are polynomials in $\Z[x_1, \dots, x_n, q_1, \dots, q_{n-1}]$ defined using the structure of the quantum Bruhat graph $\Gamma_n$, as described below. A directed path in $\Gamma_n$
\[w^{(0)}\xrightarrow{t_{a_1b_1}}w^{(1)}\xrightarrow{t_{a_2b_2}}\cdots \xrightarrow{t_{a_\ell b_\ell}}w^{(\ell)}\]
is said to be \emph{$x_k^\ell$-admissible} if it satisfies $a_1\leq a_2 \leq \dots \leq a_\ell\leq k < b_1,\dots,b_\ell$, where the values $b_1,\dots,b_\ell$ are all distinct. Given a sequence $\beta = (\beta_1,\dots,\beta_{n-1})$, a path $P$ in $\Gamma_n$ is called \emph{$x^\beta$-admissible} if it can be written as is a concatenation of paths $P_1,P_2,\dots,P_{n-1}$ such that each $P_k$ is $x_k^{\beta_k}$-admissible for all $k\in [n-1]$.

The \emph{path Schubert polynomial} $\mathfrak{S}_{u,v}$ is then defined by
\[\mathfrak{S}_{u,v}(x,q) := \sum_{\beta\in\Z_{\geq 0}^{n-1}}x^{\rho-\beta}\sum_{\substack{P:u\to v\\x^\beta\text{-admissible}}}\wt(P) ,\]
where $\rho = (n-1,n-2,\dots,1)$.

When expanded in the basis of classical Schubert polynomials, these path Schubert polynomials encode the Gromov–Witten invariants. See \cite[Theorem~4]{Postnikov-quantum-Bruhat-graph} for details.

\begin{prop}\label{prop:path-schub}
Let $u,v\in S_n$ and let $d$ be a degree. Then
    \[\mathfrak{S}_{u,v}(x,q) = \sum_{\substack{w\in S_n\\ d\in \Z_{\geq 0}^{n-1}}} c_{u,w}^{v,d}\; q^d \mathfrak{S}_{w_0w}(x).\]
\end{prop}


\subsection{Deodhar Decomposition and Total Positivity}\label{sec:prelim-4}

\subsubsection{Kazhdan--Lusztig R-Polynomials}\label{sec:prelim-4-1}

In \cite{KL79}, Kazhdan and Lusztig introduced a family of polynomials $\{R_{u,v}(q):u,v\in S_n\}$ in $\Z[q]$, known as the \emph{Kazhdan--Lusztig R-Polynomials} (or simply \emph{R-polynomials}), in their study of the Hecke algebra associated with $S_n$. These polynomials are used to give a recursive formula for the \emph{Kazhdan--Lusztig polynomials}, which play a prominent role
in several branches of mathematics including representation theory and algebraic geometry. We refer the readers to \cite{combinatorial-invariance} for a detailed exposition.

The R-polynomials are defined recursively using the following recurrence relations:
\begin{enumerate}
    \item $R_{u,v}(q)=0$ if $u\nleq v$,
    \item $R_{u,v}(q)=1$ if $u=v$,
    \item If $i\in \Des(v)$, then
    \[R_{u,v}(q)=\begin{cases}
        R_{us_i,vs_i}(q), &\text{if }i\in\Des(u),\\
        q\cdot R_{us_i,vs_i}(q)+(q-1)\cdot R_{u,vs_i}(q), &\text{if }i\not\in\Des(u).
    \end{cases}\]
\end{enumerate}

They can also be defined through the \emph{Hecke algebra} $\cH$ associated with $S_n$, which is a $\Z[q,q^{-1}]$-algebra generated by $\{T_i:i\in[n-1]\}$ with the following relations:
\begin{enumerate}
    \item $T_i^2=(q-1)T_i+q$,
    \item $T_iT_j=T_jT_i$ if $|i-j|>1$,
    \item $T_iT_{i+1}T_i=T_{i+1}T_iT_{i+1}$.
\end{enumerate}
The first relation is referred to as the \emph{Hecke relation}, and the last two are the \emph{braid relations}. The braid relations imply that $T_w :=T_{i_1}T_{i_2}\dots T_{i_{\ell}}$ is well-defined for any reduced word $w = s_{i_1}\cdots s_{i_{\ell}}$. It follows that the set $\{T_w:w\in S_n\}$ forms a $\Z[q,q^{-1}]$-basis for the Hecke algebra $\cH$. The \emph{trace map} $\epsilon:\cH \rightarrow \Z[q,q^{-1}]$ is the $\Z[q,q^{-1}]$-linear map defined by
\[\epsilon(T_w) = \begin{cases}
    1 & \text{if }w = \id,\\
    0 & \text{otherwise}.
\end{cases}\]
An important property of the trace is that it is invariant under conjugation. The following proposition gives an alternative formula for R-polynomials using the Hecke algebra and the trace map. See \cite[Theorem~2.3]{GL24}.
\begin{prop}\label{prop:hecke}
    For any $u,v\in S_n$, we have
    \[R_{u,v}(q) = q^{\ell(v)-\ell(u)}\cdot\epsilon(T_{v}^{-1} T_{u}).\]
\end{prop}

Finally, Lusztig, and independently Dyer \cite{Dyer-thesis}, proposed the following conjecture for R-polynomials, known as the \emph{combinatorial invariance conjecture}.
\begin{conj}
    Let $u,v,u',v'\in S_n$ be such that the Bruhat intervals $[u,v]\cong [u',v']$ are isomorphic as posets. Then $R_{u,v}(q)=R_{u',v'}(q)$.
\end{conj}

\subsubsection{The Deodhar Decomposition}\label{sec:prelim-4-2}

It is shown in \cite{Deodhar} that the R-polynomials $R_{u,v}(q)$ can also be interpreted as the $\F_q$-point count of open Richardson varieties:
\[R_{u,v}(q)=\#\cR_{u,v}^\circ(\F_q).\]
To provide a combinatorial formula for these polynomials, Deodhar \cite{Deodhar} introduced the \emph{Deodhar decomposition}, which expresses $\cR^\circ_{u,v}$ into a disjoint union simple pieces, each isomorphic to $\C^a\times (\C^\ast)^b$. We begin by introducing the combinatorial data that index these pieces. The following material is based on \cite{MarshRietsch}.

Given a word $\mathbf{v}=s_{i_1}s_{i_2}\cdots s_{i_\ell}$ for $v\in S_n$, and a subword $\mathbf{u}$ of $\mathbf{v}$ corresponding to $u\in S_n$, define $v^{(j)}$ and $u^{(j)}$ to be the product of the first $j$ factors in $\mathbf{v}$ and $\mathbf{u}$, respectively, for $j\in [\ell]$. The subword $\mathbf{u}$ is called \emph{positive} if $u^{(j-1)}\leq u^{(j)}$ for all $j\in [\ell]$, and \emph{distinguished} if $u^{(j)}\leq u^{(j-1)}s_{i_j}$ for all $j\in [\ell]$. In other words, if right multiplication by $s_{i_j}$ decreases the length of $u^{(j-1)}$, then in a positive subword we must have $u^{(j)}=u^{(j-1)}$, while in a distinguished subword we must have $u^{(j)}=u^{(j-1)}s_{i_j}$. A subword $\mathbf{u}$ is called a \emph{positive distinguished subword} of $\mathbf{v}$ if it is both positive and distinguished. We usually denote such a subword by $\mathbf{u}^+$. The following lemma appears in \cite{MarshRietsch}.
\begin{lemma}[{\cite[Lemma~3.5]{MarshRietsch}}]\label{lemma:distinguished}
    Given permutations $u\leq v$ and a reduced word $\mathbf{v}$ for $v$, there is a unique positive distinguished subword $\mathbf{u}^+$ of $\mathbf{v}$ corresponding to $u$.
\end{lemma}

We write $\mathbf{u}\prec\mathbf{v}$ if $\mathbf{u}$ is a distinguished subword of $\mathbf{v}$. In this case, define the index sets
\[
\begin{aligned}
    J^+_\mathbf{u}&:=\{j\in [\ell]:u^{(j-1)}<u^{(j)}\},\\
    J^\circ_\mathbf{u}&:=\{j\in [\ell]:u^{(j-1)}=u^{(j)}\},\\
    J^-_\mathbf{u}&:=\{j\in [\ell]:u^{(j-1)}>u^{(j)}\}.
\end{aligned}
\]

For each $i\in[n-1]$, let \(
\phi_i\left(\begin{smallmatrix}
    a & b\\
    c & d
\end{smallmatrix}\right)
\) denote the matrix obtained by replacing the $2 \times 2$ block in rows and columns $i$ and $i+1$ of the identity matrix with the given $2\times 2$ matrix \(
\left(\begin{smallmatrix}
    a & b\\
    c & d
\end{smallmatrix}\right)
\). Define the following matrices in $G$:
\begin{align*}
    y_i(p)&=\phi_{i}\begin{pmatrix}
             1 & 0\\ p & 1
        \end{pmatrix}, &
    \dot{s}_i(p)&=\phi_{i}\begin{pmatrix}
             0 & -1\\ 1 & 0
        \end{pmatrix}, &
    x_i(m)&=\phi_{i}\begin{pmatrix}
             1 & m\\ 0 & 1
        \end{pmatrix}.
\end{align*}
We are now ready to state the Deodhar decomposition, given by the following main theorem.

\begin{theorem}[{\cite[Proposition~5.2]{MarshRietsch}}]\label{thm:deodhar}
Let $u\leq v$ be permutations, and let $\mathbf{v}=s_{i_1}s_{i_2}\cdots s_{i_\ell}$ be a reduced word for $v$. The open Richardson variety $\cR^\circ_{u,v}$ admits a decomposition into Deodhar cells $\cD_{\mathbf{u},\mathbf{v}}$, indexed by distinguished subwords $\mathbf{u}\prec\mathbf{v}$ of $u$:
    \[\cR^\circ_{u,v}=\bigsqcup_{\mathbf{u}\prec \mathbf{v}}\cD_{\mathbf{u},\mathbf{v}},\quad\text{where }\cD_{\mathbf{u},\mathbf{v}}\cong (\C^\ast)^{|J_\mathbf{u}^\circ|}\times \C^{|J_\mathbf{u}^-|}.\]
    Each Deodhar cell $\cD_{\mathbf{u},\mathbf{v}}$ is parametrized as follows:
    \[\cD_{\mathbf{u},\mathbf{v}}:=\left\{gB=g_1g_2\cdots g_\ell B\;\middle\vert\;\begin{aligned}g_j&=\dot{s}_{i_j} &&\text{if }j\in J_\mathbf{u}^+,\\ g_j&=y_{i_j}(p_j) &&\text{if }j\in J_\mathbf{u}^\circ,\\ g_j&=x_{\alpha_i}(m_j)\dot{s}_{i_j}^{-1} &&\text{if }j\in J_\mathbf{u}^-.\end{aligned}\right\},\]
where $p_j\in \C^\ast$ and $m_j\in \C$ are parameters.
\end{theorem}

As a consequence a combinatorial formula for the R-polynomials.
\begin{cor}[{\cite[Theorem~1.3]{Deodhar}}]
    Let $u\leq v$ be permutations, and let $\mathbf{v}$ be a reduced word for $v$. Then
    \[R_{u,v}(q)=\sum_{\mathbf{u}\prec\mathbf{v}}(q-1)^{|J_{\mathbf{u}}^\circ|}q^{|J_{\mathbf{u}}^-|}.\]
\end{cor}

\subsubsection{The Totally Nonnegative Flag Variety}\label{sec:prelim-4-3}

A real matrix is said to be \emph{totally positive} if all of its minors are positive. The \emph{totally nonnegative flag variety} $\fl_n^{\geq 0}$, first defined by Lusztig \cite{Lusztig98}, is the closure of the set of flags in $\fl_n(\R)$ that can be represented by totally positive matrices. An alternative description of $\fl_n^{\geq 0}$, using Pl\"ucker coordinates and proven in \cite{Bloch-Karp} and \cite{Boretsky}, is given by:
\[\fl_n^{\geq 0}=\{F_\bullet\in \fl_n(\R):\Delta_I(F_\bullet)\geq 0\text{ for all }I\subseteq [n]\}.\]

The \emph{totally nonnegative parts} of Richardson varieties, denoted $\cR_{u,v}^{> 0}$ and $\cR_{u,v}^{\geq 0}$, are defined as the intersections
\[\cR_{u,v}^{> 0}:=\cR_{u,v}^\circ\cap\fl_n^{\geq0},\quad\cR_{u,v}^{\geq 0}:=\cR_{u,v}\cap\fl_n^{\geq0}.\]
The totally nonnegative parts $\cR_{u,v}^{> 0}$ are isomorphic to open cells, as first shown in \cite{rietsch3}. Moreover, these cells admit an explicit parametrization closely related to the Deodhar decomposition, as described in \cite[Theorem~11.3]{MarshRietsch}: 

\begin{theorem}\label{thm:deodhar-positive}
    Let $u\leq v$ be permutations, and let $\mathbf{v}=s_{i_1}s_{i_2}\cdots s_{i_\ell}$ be a reduced word for $v$. Let $\mathbf{u}^+$ be the unique positive distinguished subword of $\mathbf{v}$ corresponding to $u$. Then the totally nonnegative part of the Richardson variety $\cR_{u,v}^{> 0}$ is parametrized as follows:
    \[\cR_{u,v}^{>0}=\left\{gB=g_1g_2\cdots g_\ell B\;\middle\vert\;\begin{aligned}g_j&=\dot{s}_{i_j} &&\text{if }j\in J_{\mathbf{u}^+}^+,\\ g_j&=y_{i_j}(p_j),\text{ with }p_j\in\R_{>0} &&\text{if }j\in J_{\mathbf{u}^+}^\circ.\end{aligned}\right\}.\]
    
    Therefore, $\cR_{u,v}^{> 0}\cong (\R_{>0})^{\ell(v)-\ell(u)}$, which is homeomorphic to an open ball.
\end{theorem}

\subsubsection{History of Total Positivity}\label{sec:prelim-4-4}
Recall that a CW-complex is \emph{regular} if each attaching map is a homeomorphism. Given a CW-complex $X$, its \emph{face poset} $P(X)$ is the poset whose elements are the cells of $X$, ordered by containment of their closures, with a minimum element $\hat{0}$ adjoined. This notion was first introduced and studied by Bj\"orner in \cite{Bjorner-poset}.

It was shown in \cite{Bjorner-poset} that all Bruhat intervals are CW posets. However, the corresponding regular CW-complexes were constructed \textit{synthetically}, via abstract sequences of cell attachments. A question posed in \cite{Bjorner-poset} asked for a natural geometric realization of such a regular CW-complex whose face poset is a Bruhat interval. A conjectural answer was proposed by Fomin and Shapiro in \cite{fomin-shapiro}, known as the \emph{Fomin--Shapiro conjecture}, and was eventually proven by Hersh in \cite{hersh}.

Williams later generalized this conjecture in \cite{williams-shelling} to the interval posets of Bruhat intervals. The \emph{interval poset} of a poset $P$ is defined as the poset consisting of all intervals in $P$, ordered by containment, with a minimum element $\hat{0}$ adjoined at the bottom. Williams proved that the interval poset of a Bruhat interval $[u,v]$ is a CW poset, and conjectured that a natural geometric realization of the corresponding regular CW-complex is given by the totally nonnegative part of the Richardson variety, stratified as:
\[\cR_{u,v}^{\geq0}=\bigsqcup_{[x,y]\subseteq[u,v]}\cR_{x,y}^{>0},\]
whose face poset, as proven by Rietsch \cite{rietsch2}, is the interval poset of $[u,v]$.

This conjecture inspired a series of developments over the past decades and was ultimately resolved in \cite{GKL3}. Below, we outline the chronological progression of key results:
\begin{itemize}
    \item The totally nonnegative flag variety $\fl_n^{\geq 0}$ is contractible \cite{Lusztig98b}.
    \item Each open cell $\cR_{u,v}^{>0}$ is homeomorphic to an open ball \cite{rietsch3}.
    \item The closed cell $\cR_{u,v}^{\geq 0}$ is the closure of the open cell $\cR_{u,v}^{>0}$, identifying its face poset with the interval poset of the Bruhat interval $[u,v]$ \cite{rietsch2}.
    \item The interval poset of $[u,v]$ is a CW poset, conjecturally suggesting a regular CW-complex structure \cite{williams-shelling}.
    \item The stratification forms a CW-complex \cite{rietsch-williams}.
    \item Each closed cell $\cR_{u,v}^{\geq 0}$ is contractible, with boundary homotopy equivalent to a sphere \cite{lauren-morse}.
    \item The totally nonnegative flag variety $\fl_n^{\geq 0}$ is homeomorphic to a closed ball \cite{GKL1}.
    \item Each closed cell $\cR_{u,v}^{\geq 0}$ is homeomorphic to a closed ball, and the entire space forms a regular CW-complex \cite{GKL3}.
    \item The totally nonnegative J-Richardson varieties in the full flag variety of an arbitrary Kac-Moody group are regular CW-complexes homeomorphic to closed balls \cite{bao-he}.
\end{itemize}

Finally, we note the following result from \cite[Corollary~6.5]{BFP-tilted-Bruhat}:

\begin{prop}\label{prop:cw-complex}
    For each $u,v\in S_n$, the tilted Bruhat interval $[u,v]$ is a CW poset.
\end{prop}
It is then natural to ask for a CW-complex in $\fl_n$ whose face poset is the tilted Bruhat interval $[u,v]$. This is our focus in \Cref{sec:deodhar}.

%% file: tex/3-tilted-order.tex
\section{Combinatorics of Tilted Bruhat Orders}\label{sec:graph}

In this section, we present an explicit combinatorial formula for the minimal degrees $d_{u,v}$ in the quantum Bruhat graph. We also introduce two new partial orders on the symmetric group $S_n$, called the \emph{$\ba$-tilted Bruhat orders}, denoted by $\leq_\ba$ and $\lesssim_\ba$. These orders are indexed by an integer sequence $\ba = (a_1, a_2, \dots, a_n) \in [n]^n$, and specialize to the classical Bruhat order when $\ba = (1, 1, \dots, 1)$. These orders are closely related to the quantum Bruhat graph and tilted Bruhat intervals introduced by Brenti, Fomin, and Postnikov in \cite{BFP-tilted-Bruhat}. In particular, they provide an Ehresmann-like criterion for characterizing tilted Bruhat intervals. We further develop the theory of \emph{$\ba$-tilted reduced words}, and establish several properties analogous to those of classical reduced words.

\subsection{Minimal Degrees in the Quantum Bruhat Graph}\label{sec:minimal-degree}
We provide an explicit formula for the minimal degree $d(u,v)$ (or $d_{u,v}$) of a shortest path from $u$ to $v$ in the quantum Bruhat graph $\Gamma_n$ for $u,v\in S_n$ (\Cref{thm:weight-distance}). For background, see \Cref{sec:prelim-3-5}. 
\begin{defin}\label{def:up-down-path}
Let $A,B\subseteq[n]$ with $|A|=|B|$. The \emph{lattice path} $\p(A,B)$ is defined as the path starting at $(1,0)$ and ending at $(n+1,0)$ with $n$ steps, where the $i$-th step is:
\begin{itemize}
\item an upstep $(1,1)$ if $i\in A$ and $i\notin B$,
\item a downstep $(1,-1)$ if $i\notin A$ and $i\in B$,
\item a horizontal step $(1,0)$ if $i\in A\cap B$ or $i\notin A\cup B$. 
\end{itemize}
The \emph{depth} of the path is defined as the largest number $y \geq 0$ such that the path passes through a point $(x, -y)$ for some $x$. We denote this by $\depth(A, B)$.
\end{defin}
\begin{ex}\label{ex:lattice-path-PAB}
Let $n=7$, with $A=\{3,4,6,7\}$ and $B=\{1,2,3,5\}$. The lattice path $\p(A,B)$ is illustrated in \Cref{fig:lattice-path-PAB}, and we have $\depth(A,B)=2$.
\begin{figure}[ht]
\centering
\begin{tikzpicture}[scale=0.6]
\node at (0,0) {$\bullet$};
\node at (1,-1) {$\bullet$};
\node at (2,-2) {$\bullet$};
\node at (3,-2) {$\bullet$};
\node at (4,-1) {$\bullet$};
\node at (5,-2) {$\bullet$};
\node at (6,-1) {$\bullet$};
\node at (7,0) {$\bullet$};
\draw(0,0)--(2,-2)--(3,-2)--(4,-1)--(5,-2)--(7,0);
\draw[dashed](0,0)--(7,0);
\node (d) at (8.5,-1) {$\depth=2$};
\draw[->] (d)--(8.5,0);
\draw[->] (d)--(8.5,-2);
\node[above] at (0.6,-0.7) {$1$};
\node[above] at (1.6,-1.7) {$2$};
\node[above] at (2.5,-2.1) {$3$};
\node[above] at (3.4,-1.7) {$4$};
\node[above] at (4.6,-1.7) {$5$};
\node[above] at (5.4,-1.7) {$6$};
\node[above] at (6.4,-0.7) {$7$};
\end{tikzpicture}
\caption{The lattice path $\p(A,B)$ for $A=\{3,4,6,7\}$ and $B=\{1,2,3,5\}$}
\label{fig:lattice-path-PAB}
\end{figure}
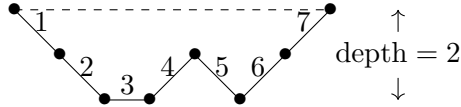
\end{ex}

\begin{theorem}\label{thm:weight-distance}
Let $u, v \in S_n$. The minimal weight of a shortest path from $u$ to $v$ in the quantum Bruhat graph is given by
\[q^{d(u,v)}=q_1^{d_1(u,v)}\cdots q_{n-1}^{d_{n-1}(u,v)},\]
where each exponent is defined by $d_k(u,v):=\depth(u[k],v[k])$, the depth of the lattice path $\p(u[k],v[k])$. Here, $w[k] := \{w_1, \dots, w_k\}$.
\end{theorem}
\begin{proof}
For $u, v \in S_n$, recall that the weight of a shortest path from $u$ to $v$ in $\Gamma_n$ is denoted by $q^{d(u,v)}$. Let $\depth(u,v)$ denote the vector whose $k$-th coordinate is $\depth(u[k], v[k])$, as defined in \Cref{def:up-down-path}. Our goal is to show that
\[q^{d(u,v)}=q^{\depth(u,v)}.\]
Our strategy is to explicitly construct a canonical shortest path $P_{u,v}$ using \Cref{prop:shortest-path-BFP}, and then show that its weight equals $q^{\depth(u,v)}$.

We begin by fixing a reflection ordering $\gamma=e_1-e_2,e_1-e_3,\ldots,e_1-e_n,e_2-e_3,\ldots e_2-e_{n},\ldots,e_{n-1}-e_n$ which lists all positive roots involving $e_1$ first, followed by those involving $e_2$, and so on. This is a valid reflection ordering since, for any $i<j<k$, the roots $e_i-e_j$, $e_i-e_k$ and $e_j-e_k$ appear in the correct order. One can also verify that this ordering corresponds to the reduced word $w_0=(s_1s_2\cdots s_{n-1})(s_1s_2\cdots s_{n-2})\cdots (s_1)$.

To construct $P_{u,v}$ explicitly, we first consider the initial segment of $\gamma$ consisting of the roots $e_1 - e_2, \ldots, e_1 - e_n$. Suppose we select a subset of these reflections, say $e_1 - e_{p_1}, e_1 - e_{p_2}, \dots, e_1 - e_{p_m}$, and apply them in sequence to build a directed subpath:
\[u=w^{(0)}\xrightarrow{t_{i_1j_1}}w^{(1)}\xrightarrow{t_{i_2j_2}}\cdots \xrightarrow{t_{i_\ell j_\ell}}w^{(\ell)}=u'\]
After this sequence, the value of $u'_1$ must equal $v_1$, since all remaining reflections in $\gamma$ involve indices greater than $1$. On the other hand, once $u'_1 = v_1$, the remaining steps of the path can be constructed inductively on $n$. Therefore, to construct the full path $P_{u,v}$, it suffices to find this initial segment, which we denote by $P^1_{u,v}$.

To construct the initial subpath $P^1_{u,v}$, we claim that the sequence of indices $p_1, p_2,\dots,p_m$ can be explicitly determined by the following algorithm. Starting with $p_0 = 1$, and assuming $p_0, p_1, \dots, p_{k-1}$ have been chosen, we define $p_k$ to be the smallest index greater than $p_{k-1}$ such that $u_{p_k}$ lies in the cyclic interval $(u_{p_{k-1}},v_1]_c$. This construction ensures that each edge in the subpath $u = w^{(0)} \xrightarrow{t_{1p_1}} w^{(1)} \xrightarrow{t_{1p_2}} \cdots \xrightarrow{t_{1p_m}} w^{(m)} = u'$ exists in the quantum Bruhat graph $\Gamma_n$, and the algorithm necessarily terminates when $p_m = u^{-1}(v_1)$. Here is an example with $v_1 = 4$:
\[\begin{tabular}{c|c|c|c}
& permutation & root & weight \\ \hline
$u$ & $\underline{\textbf{6}}5\underline{\textbf{7}}913428$ & $e_1-e_3$ & 1  \\ \hline
 & $\underline{\textbf{7}}56\underline{\textbf{9}}13428$ & $e_1-e_4$ & 1 \\ \hline
 & $\underline{\textbf{9}}567\underline{\textbf{1}}3428$ & $e_1-e_5$ & $q_1q_2q_3q_4$ \\\hline
 & $\underline{\textbf{1}}5679\underline{\textbf{3}}428$ & $e_1-e_6$ & 1 \\\hline
 & $\underline{\textbf{3}}56791\underline{\textbf{4}}28$ & $e_1-e_7$ & 1 \\\hline
 & $\textbf{4}56791328$ & & \\
 & $\cdots$ & & \\\hline
$v$ & $\textbf{4}\rule{1.46cm}{0.15mm}$ & & \\\hline
\end{tabular}.\]

Let $\wt(P_{u,v}^1)$ denote the weight of the initial subpath $P^1_{u,v}$. By \Cref{prop:shortest-path-BFP}, since the canonical path $P_{u,v}$ we constructed is a shortest path, we have $\wt(P_{u,v}^1) \cdot q^{d(u',v)} = q^{d(u,v)}$. By the induction hypothesis, since the remaining path does not involve the first index, we also have $q^{d(u',v)} = q^{\depth(u',v)}$. It remains to show that
\[
\wt(P_{u,v}^1)\cdot q^{\depth(u',v)}=q^{\depth(u,v)}.
\]
We now compare the exponents of $q^j$ on both sides of the equation above for all $j \in [n]$. In other words, we compare the depths of the two lattice paths $\p(u[j], v[j])$ and $\p(u'[j], v[j])$. 
\begin{enumerate}
\item \textbf{Case $u_1\leq v_1$.} In this case, we have $u_1 < u_{p_1} < \cdots < u_{p_m} = v_1$, and $\wt(P^1_{u,v}) = 1$. Our goal is to show that $\depth(u[j], v[j]) = \depth(u'[j], v[j])$ for all $j \in [n]$. If $j \geq p_m$, then $u[j] = u'[j]$, so the equality is immediate. If $j < p_m$, note that $v_1 \notin u[j]$, and $u'[j]$ is obtained from $u[j]$ by replacing the largest number less than $v_1$ (call it $b$) with $v_1$. As a result, the lattice paths $\p(u[j], v[j])$ and $\p(u'[j], v[j])$ differ only between steps $b$ and $v_1$. In this interval, we have:
\[\begin{aligned}
\p(u[j],v[j]):\ & (\nearrow\text{ or }\rightarrow)\ +\ (\text{a sequence of }\rightarrow\text{ and }\searrow\text{'s})\ +\ (\searrow), \\
\p(u'[j],v[j]):\ & (\rightarrow\text{ or }\searrow)\ +\ (\text{a sequence of }\rightarrow\text{ and }\searrow\text{'s})\ +\ (\rightarrow).
\end{aligned}\]
This local change does not affect the overall depth of the lattice path. Hence, we conclude that $\depth(u[j], v[j]) = \depth(u'[j], v[j])$.
\item \textbf{Case $u_1>v_1$.} In this case, there exists a unique index $k\in[m]$ such that
\[u_{p_k}<\cdots <u_{p_m}=v_1<u_1=u_{p_0}<\cdots <u_{p_{k-1}},\] 
and the weight of the initial subpath is $\wt(P^1_{u,v}) = q_1 q_2 \cdots q_{p_k - 1}$. For $j < p_k$, we aim to show that $\depth(u[j], v[j]) = \depth(u'[j], v[j]) + 1$. Note that $v_1 \notin u[j]$, and $u'[j]$ is obtained from $u[j]$ by deleting its largest element, $\max u[j]$, and replacing it with $v_1$, which is smaller than all elements in $u[j]$. As a result, the lattice paths $\p(u[j], v[j])$ and $\p(u'[j], v[j])$ differ as illustrated in \Cref{fig:path-depth-change-1}. Specifically, $\p(u'[j], v[j])$ is obtained from $\p(u[j], v[j])$ by shifting the lattice subpath between $(v_1{+}1,-)$ and $(\max u[j],-)$ (shown as the red curve in \Cref{fig:path-depth-change-1}) upward by one unit. This local adjustment reduces the depth by 1, yielding $\depth(u[j], v[j]) = \depth(u'[j], v[j]) + 1$.
\begin{figure}[ht]
\centering
\begin{tikzpicture}[scale=0.5]
\node at (0,0) {$\bullet$};
\node at (1,-1) {$\bullet$};
\node at (2,-1) {$\bullet$};
\node at (3,-2) {$\bullet$};
\node[right] at (3,-2) {$(v_1{+}1,-)$};
\node at (6,2) {$\bullet$};
\node[above] at (6,2) {$(\max u[j],-)$};
\node at (7,2) {$\bullet$};
\node at (8,1) {$\bullet$};
\node at (9,0) {$\bullet$};
\draw(0,0)--(1,-1)--(2,-1)--(3,-2);
\draw(6,2)--(7,2)--(9,0);
\draw[dashed](0,0)--(9,0);
\coordinate (A) at (3,-2);
\coordinate (B) at (6,2);
\draw[red,thick] (A) .. controls (4,2) and (2,1) .. (B);
\def\a{12}
\node at (\a,0) {$\bullet$};
\node at (\a+1,-1) {$\bullet$};
\node at (\a+2,-1) {$\bullet$};
\node at (\a+3,-1) {$\bullet$};
\node[right] at (\a+3,-1) {$(v_1{+}1,-)$};
\node at (\a+6,3) {$\bullet$};
\node[left] at (\a+6,3) {$(\max u[j],-)$};
\node at (\a+7,2) {$\bullet$};
\node at (\a+8,1) {$\bullet$};
\node at (\a+9,0) {$\bullet$};
\draw(\a+0,0)--(\a+1,-1)--(\a+2,-1)--(\a+3,-1);
\draw(\a+6,3)--(\a+7,2)--(\a+9,0);
\draw[dashed](\a+0,0)--(\a+9,0);
\coordinate (A2) at (\a+3,-1);
\coordinate (B2) at (\a+6,3);
\draw[red,thick] (A2) .. controls (\a+4,3) and (\a+2,2) .. (B2);
\end{tikzpicture}
\caption{A comparison of the lattice paths $\p(u[j], v[j])$ (left) and $\p(u'[j], v[j])$ (right) for $j < p_k$ in Case 2 of the proof of \Cref{thm:weight-distance}, where the red curve represents an arbitrary lattice subpath}
\label{fig:path-depth-change-1}
\end{figure}
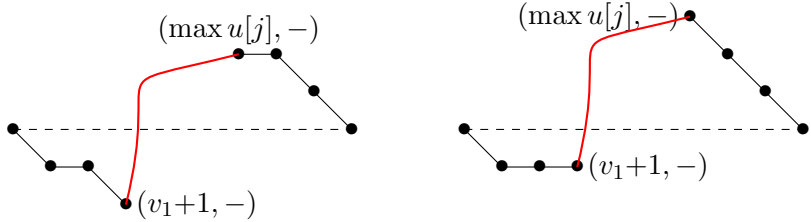

On the other hand, for $j \geq p_k$, we aim to show that $\depth(u[j], v[j]) = \depth(u'[j], v[j])$. This follows from the same reasoning used in Case 1 when $u_1 \leq v_1$, since in this range we have $u[j] = u'[j]$ or the replacement does not affect the depth.
\end{enumerate}
This completes the case analysis. The induction step now follows.
\end{proof}

\begin{remark}
We note that \Cref{thm:weight-distance} can also be obtained by combining the combinatorial formula for minimal degrees in quantum products in the Grassmannian \cite[Theorem~8.1]{Postnikov-quantum-Schur} with a geometric result of Buch, Chung, Li, and Mihalcea \cite[Lemma~4]{bclm2020-quantumK}. Nevertheless, our proof is independent and purely combinatorial.
\end{remark}


\begin{ex}
Consider $u=7364152$ and $v=2513746$ in $S_7$. We compute $d_k(u,v)$ for each $k \in [n-1]$. For $k = 1$, we have $u[1] = {7}$ and $v[1] = {2}$, so the lattice path has an upstep at position $7$ and a downstep at position $2$. This path has depth $1$, meaning that $d_1(u,v) = 1$.

We continue this procedure for each $k$. For example, when $k = 4$, the lattice path $\p(u[4], v[4])$ is discussed in \Cref{ex:lattice-path-PAB}, and has depth $2$, hence $d_4(u,v) = 2$. Carrying out this computation for all $k$, we obtain:
\[q^{d(u,v)}=q_1q_2q_3^2q_4^2q_5q_6.\]
\end{ex}

The following corollary follows immediately from \Cref{thm:weight-distance}.

\begin{cor}\label{cor:weight-distance}
    Let $w\in S_n$ and let $t_{ij}\in S_n$ be a transposition. Then the minimal weight of a shortest path from $w$ to $wt_{ij}$ in the quantum Bruhat graph is at most $q_{ij}:=q_iq_{i+1}\cdots q_{j-1}$.
\end{cor}
\begin{proof}
    This follows directly from \Cref{thm:weight-distance} and the observation that the lattice path $\p(w[k], w t_{ij}[k])$ has depth at most $1$ for $k \in [i, j)$ and depth $0$ for all other $k$.
\end{proof}

\subsection{Tilted Bruhat Orders}\label{sec:tilted-order}
In this section, we define the \emph{$\ba$-tilted Bruhat orders} $\leq_\ba$ and $\lesssim_\ba$ for any sequence $\ba \in [n]^n$ (\Cref{def:tilted-bruhat-order}), and establish a fundamental existence property of these orders (\Cref{thm:lesssim-exist}). We also establish an Ehresmann-like criterion for tilted Bruhat intervals (see \Cref{def:tilted-bruhat-interval}) using the $\ba$-tilted Bruhat orders (\Cref{thm:tilted-criterion}). We begin with the following definition.

\begin{defin}\label{def:shifted-gale-order}
For $r\in [n]$, the \emph{shifted linear order} $\leq_r$ is a total order on $[n]$ defined by
\[r<_r r+1 <_r \cdots <_r n <_r 1 <_r \cdots <_r r-1.\]
For $r,k\in [n]$, the \emph{shifted Gale order} $\leq_r$ is a partial order on the set $\binom{[n]}{k}$. Given two $k$-element subsets $A,B\in\binom{[n]}{k}$ with $A = \{a_1 <_r \cdots <_r a_k\}$ and $B = \{b_1<_r \cdots <_r b_k\}$, we define
$A\leq_r B$ in the shifted Gale order if $a_i \leq_r b_i$ for all $i\in [k]$.
\end{defin}

We are now ready to define the two $\ba$-tilted Bruhat orders.

\begin{defin}\label{def:tilted-bruhat-order}
    Let $\ba = (a_1, a_2, \dots, a_n) \in [n]^n$. The \emph{$\ba$-tilted Bruhat order} $\leq_\ba$ on $S_n$ is defined by
    \[u\leq_\ba v\quad\text{if}\quad u[k]\leq_{a_k}v[k]\text{ for all }k\in[n].\]
    Define an equivalence relation $\sim_\ba $ on $S_n$ by
    \[u\sim_\ba v\quad\text{if}\quad\size{u[k]\cap[a_k,a_{k+1})_c}=\size{v[k]\cap[a_k,a_{k+1}\}_c}\text{ for all }k\in[n],\]
    where $[a_k,a_{k+1})_c$ denotes the cyclic interval. Then the \emph{$\ba$-tilted Bruhat order modulo $\sim$}, denoted $\lesssim_\ba$, is defined by
    \[u\lesssim_\ba v\quad \text{if}\quad u\leq_\ba v \text{ and } u\sim_\ba v.\]
\end{defin}

It is immediate from the definition that $u \lesssim_\ba v$ implies $u \leq_\ba v$. When the sequence $\ba = (1, 1, \dots, 1)$, both orders recover the classical Bruhat order via the Ehresmann criterion. An example illustrating the orders $\leq_\ba$ and $\lesssim_\ba$ is given in \Cref{fig:a-tilted-bruhat-lessthan} and \Cref{fig:a-tilted-bruhat-lesssim}.

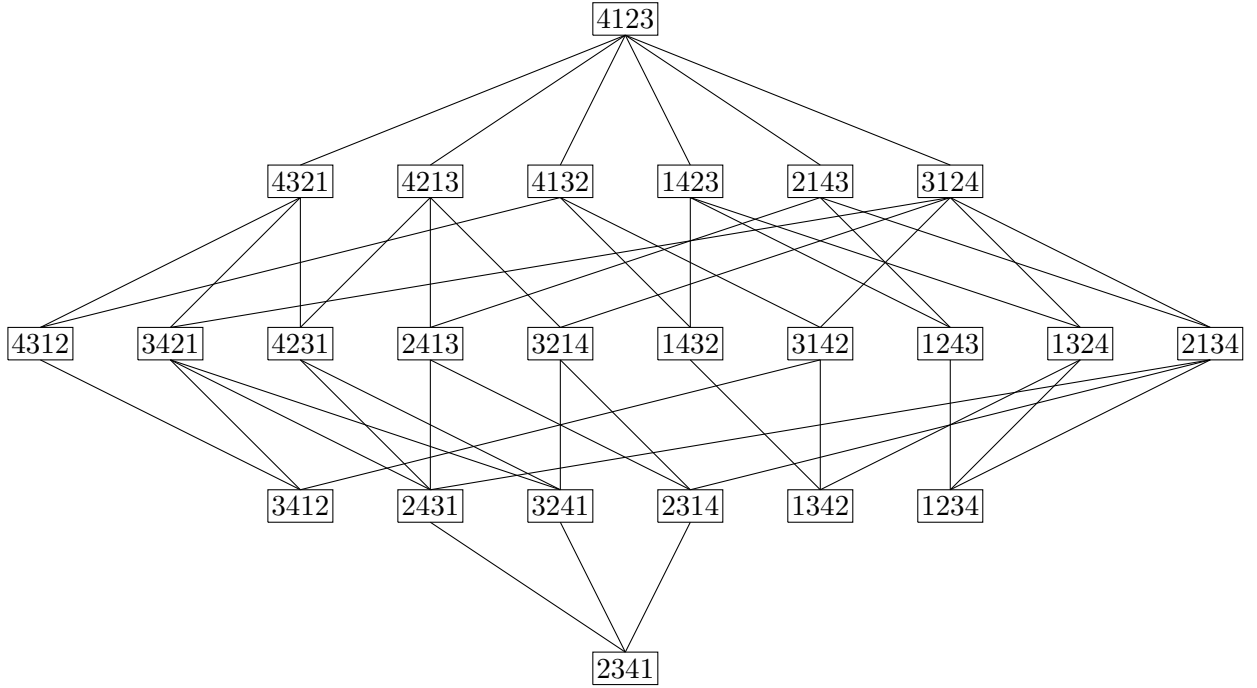
\begin{figure}[ht]
\centering
\begin{tikzpicture}[scale=0.54]
\def\a{0.8};
\def\b{0.4};
\def\h{4.0};
\newcommand\Rec[3]{
\node at (#1,#2) {#3};
\draw(#1-\a,#2-\b)--(#1-\a,#2+\b)--(#1+\a,#2+\b)--(#1+\a,#2-\b)--(#1-\a,#2-\b);
}
\Rec{0}{0}{$2341$}
\Rec{-10*\a}{\h}{$3412$}
\Rec{-6*\a}{\h}{$2431$}
\Rec{-2*\a}{\h}{$3241$}
\Rec{2*\a}{\h}{$2314$}
\Rec{6*\a}{\h}{$1342$}
\Rec{10*\a}{\h}{$1234$}
\Rec{-18*\a}{2*\h}{$4312$}
\Rec{-14*\a}{2*\h}{$3421$}
\Rec{-10*\a}{2*\h}{$4231$}
\Rec{-6*\a}{2*\h}{$2413$}
\Rec{-2*\a}{2*\h}{$3214$}
\Rec{2*\a}{2*\h}{$1432$}
\Rec{6*\a}{2*\h}{$3142$}
\Rec{10*\a}{2*\h}{$1243$}
\Rec{14*\a}{2*\h}{$1324$}
\Rec{18*\a}{2*\h}{$2134$}
\Rec{-10*\a}{3*\h}{$4321$}
\Rec{-6*\a}{3*\h}{$4213$}
\Rec{-2*\a}{3*\h}{$4132$}
\Rec{2*\a}{3*\h}{$1423$}
\Rec{6*\a}{3*\h}{$2143$}
\Rec{10*\a}{3*\h}{$3124$}
\Rec{0}{4*\h}{$4123$}

\draw(0,\b)--(-6*\a,\h-\b);
\draw(0,\b)--(-2*\a,\h-\b);
\draw(0,\b)--(2*\a,\h-\b);
\draw(-10*\a,\h+\b)--(-18*\a,2*\h-\b);
\draw(-10*\a,\h+\b)--(-14*\a,2*\h-\b);
\draw(-10*\a,\h+\b)--(6*\a,2*\h-\b);
\draw(-6*\a,\h+\b)--(-14*\a,2*\h-\b);
\draw(-6*\a,\h+\b)--(-10*\a,2*\h-\b);
\draw(-6*\a,\h+\b)--(-6*\a,2*\h-\b);
\draw(-6*\a,\h+\b)--(18*\a,2*\h-\b);
\draw(-2*\a,\h+\b)--(-14*\a,2*\h-\b);
\draw(-2*\a,\h+\b)--(-10*\a,2*\h-\b);
\draw(-2*\a,\h+\b)--(-2*\a,2*\h-\b);
\draw(2*\a,\h+\b)--(-6*\a,2*\h-\b);
\draw(2*\a,\h+\b)--(-2*\a,2*\h-\b);
\draw(2*\a,\h+\b)--(18*\a,2*\h-\b);
\draw(6*\a,\h+\b)--(2*\a,2*\h-\b);
\draw(6*\a,\h+\b)--(6*\a,2*\h-\b);
\draw(6*\a,\h+\b)--(14*\a,2*\h-\b);
\draw(10*\a,\h+\b)--(10*\a,2*\h-\b);
\draw(10*\a,\h+\b)--(14*\a,2*\h-\b);
\draw(10*\a,\h+\b)--(18*\a,2*\h-\b);
\draw(-18*\a,2*\h+\b)--(-10*\a,3*\h-\b);
\draw(-18*\a,2*\h+\b)--(-2*\a,3*\h-\b);
\draw(-14*\a,2*\h+\b)--(-10*\a,3*\h-\b);
\draw(-14*\a,2*\h+\b)--(10*\a,3*\h-\b);
\draw(-10*\a,2*\h+\b)--(-10*\a,3*\h-\b);
\draw(-10*\a,2*\h+\b)--(-6*\a,3*\h-\b);
\draw(-6*\a,2*\h+\b)--(-6*\a,3*\h-\b);
\draw(-6*\a,2*\h+\b)--(6*\a,3*\h-\b);
\draw(-2*\a,2*\h+\b)--(-6*\a,3*\h-\b);
\draw(-2*\a,2*\h+\b)--(10*\a,3*\h-\b);
\draw(2*\a,2*\h+\b)--(-2*\a,3*\h-\b);
\draw(2*\a,2*\h+\b)--(2*\a,3*\h-\b);
\draw(6*\a,2*\h+\b)--(-2*\a,3*\h-\b);
\draw(6*\a,2*\h+\b)--(10*\a,3*\h-\b);
\draw(10*\a,2*\h+\b)--(2*\a,3*\h-\b);
\draw(10*\a,2*\h+\b)--(6*\a,3*\h-\b);
\draw(14*\a,2*\h+\b)--(2*\a,3*\h-\b);
\draw(14*\a,2*\h+\b)--(10*\a,3*\h-\b);
\draw(18*\a,2*\h+\b)--(6*\a,3*\h-\b);
\draw(18*\a,2*\h+\b)--(10*\a,3*\h-\b);
\draw(-10*\a,3*\h+\b)--(0,4*\h-\b);
\draw(-6*\a,3*\h+\b)--(0,4*\h-\b);
\draw(-2*\a,3*\h+\b)--(0,4*\h-\b);
\draw(2*\a,3*\h+\b)--(0,4*\h-\b);
\draw(6*\a,3*\h+\b)--(0,4*\h-\b);
\draw(10*\a,3*\h+\b)--(0,4*\h-\b);
\end{tikzpicture}
\caption{The Hasse diagram of $\leq_{\ba}$ on $S_4$ for $\ba=(1,2,3,3)$}
\label{fig:a-tilted-bruhat-lessthan}
\end{figure}

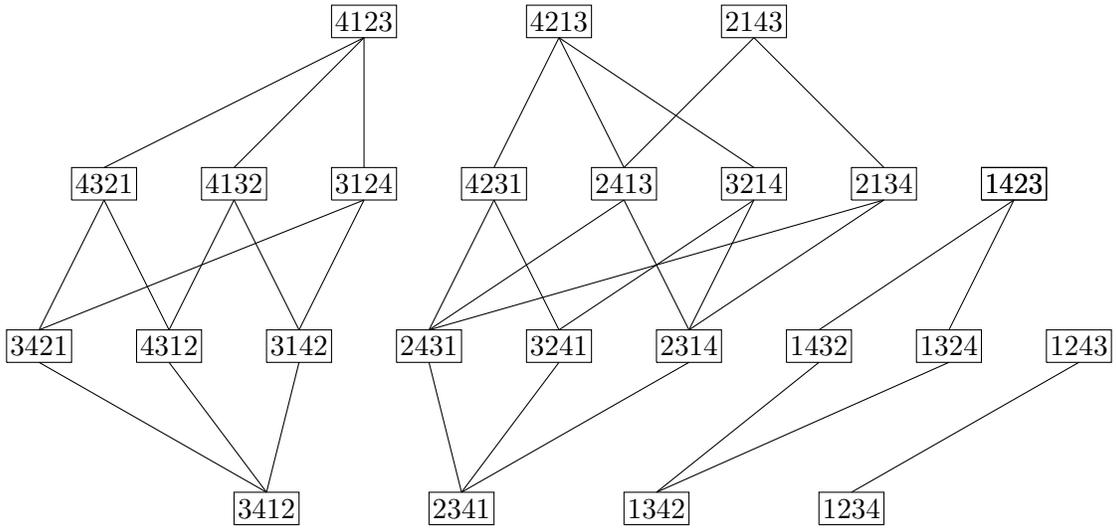
\begin{figure}[ht]
\centering
\begin{tikzpicture}[scale=0.54]
\def\a{0.8};
\def\b{0.4};
\def\h{4.0};
\newcommand\Rec[3]{
\node at (#1,#2) {#3};
\draw(#1-\a,#2-\b)--(#1-\a,#2+\b)--(#1+\a,#2+\b)--(#1+\a,#2-\b)--(#1-\a,#2-\b);
}
\Rec{-9*\a}{0}{$3412$}
\Rec{-3*\a}{0}{$2341$}
\Rec{3*\a}{0}{$1342$}
\Rec{9*\a}{0}{$1234$}
\Rec{-16*\a}{\h}{$3421$}
\Rec{-12*\a}{\h}{$4312$}
\Rec{-8*\a}{\h}{$3142$}
\Rec{-4*\a}{\h}{$2431$}
\Rec{0}{\h}{$3241$}
\Rec{4*\a}{\h}{$2314$}
\Rec{8*\a}{\h}{$1432$}
\Rec{12*\a}{\h}{$1324$}
\Rec{16*\a}{\h}{$1243$}
\Rec{-14*\a}{2*\h}{$4321$}
\Rec{-10*\a}{2*\h}{$4132$}
\Rec{-6*\a}{2*\h}{$3124$}
\Rec{-2*\a}{2*\h}{$4231$}
\Rec{2*\a}{2*\h}{$2413$}
\Rec{6*\a}{2*\h}{$3214$}
\Rec{10*\a}{2*\h}{$2134$}
\Rec{14*\a}{2*\h}{$1423$}
\Rec{14*\a}{2*\h}{$1423$}
\Rec{-6*\a}{3*\h}{$4123$}
\Rec{0*\a}{3*\h}{$4213$}
\Rec{6*\a}{3*\h}{$2143$}

\draw(-9*\a,0*\h+\b)--(-16*\a,\h-\b);
\draw(-9*\a,0*\h+\b)--(-12*\a,\h-\b);
\draw(-9*\a,0*\h+\b)--(-8*\a,\h-\b);
\draw(-3*\a,0*\h+\b)--(-4*\a,\h-\b);
\draw(-3*\a,0*\h+\b)--(0*\a,\h-\b);
\draw(-3*\a,0*\h+\b)--(4*\a,\h-\b);
\draw(3*\a,0*\h+\b)--(8*\a,\h-\b);
\draw(3*\a,0*\h+\b)--(12*\a,\h-\b);
\draw(9*\a,0*\h+\b)--(16*\a,\h-\b);
\draw(-16*\a,\h+\b)--(-14*\a,2*\h-\b);
\draw(-16*\a,\h+\b)--(-6*\a,2*\h-\b);
\draw(-12*\a,\h+\b)--(-14*\a,2*\h-\b);
\draw(-12*\a,\h+\b)--(-10*\a,2*\h-\b);
\draw(-8*\a,\h+\b)--(-10*\a,2*\h-\b);
\draw(-8*\a,\h+\b)--(-6*\a,2*\h-\b);
\draw(-4*\a,\h+\b)--(-2*\a,2*\h-\b);
\draw(-4*\a,\h+\b)--(2*\a,2*\h-\b);
\draw(-4*\a,\h+\b)--(10*\a,2*\h-\b);
\draw(0*\a,\h+\b)--(-2*\a,2*\h-\b);
\draw(0*\a,\h+\b)--(6*\a,2*\h-\b);
\draw(4*\a,\h+\b)--(2*\a,2*\h-\b);
\draw(4*\a,\h+\b)--(6*\a,2*\h-\b);
\draw(4*\a,\h+\b)--(10*\a,2*\h-\b);
\draw(8*\a,\h+\b)--(14*\a,2*\h-\b);
\draw(12*\a,\h+\b)--(14*\a,2*\h-\b);
\draw(-14*\a,2*\h+\b)--(-6*\a,3*\h-\b);
\draw(-10*\a,2*\h+\b)--(-6*\a,3*\h-\b);
\draw(-6*\a,2*\h+\b)--(-6*\a,3*\h-\b);
\draw(-2*\a,2*\h+\b)--(0*\a,3*\h-\b);
\draw(2*\a,2*\h+\b)--(0*\a,3*\h-\b);
\draw(2*\a,2*\h+\b)--(6*\a,3*\h-\b);
\draw(6*\a,2*\h+\b)--(0*\a,3*\h-\b);
\draw(10*\a,2*\h+\b)--(6*\a,3*\h-\b);
\end{tikzpicture}
\caption{The Hasse diagram of $\lesssim_{\ba}$ on $S_4$ for $\ba=(1,2,3,3)$}
\label{fig:a-tilted-bruhat-lesssim}
\end{figure}


We now aim to prove the following fundamental existence property of these orders: for any $u, v \in S_n$, there exists a sequence $\ba$ such that $u$ and $v$ are comparable under either $\leq_\ba$ or $\lesssim_\ba$. To establish this, we first present a lemma that relates the shifted Gale order $\leq_r$ to the lattice path construction in \Cref{def:up-down-path}.

\begin{lemma}\label{lemma:r-comparable}
For any $A, B \subseteq [n]$ with $|A| = |B|$ and any $r \in [n]$, we have $A \leq_r B$ in the shifted Gale order if and only if the lattice path $\p(A, B)$ passes through the point $(r, -\depth(A, B))$, where it reaches its minimum.
\end{lemma}
\begin{proof}
It is clear that $A\leq B$ in the Gale order if and only if the lattice path $\p(A,B)$ does not go below the $x$-axis, i.e. $\p(A,B)$ is a Dyck path. Similarly, $A\leq_r B$ if and only if the  \emph{cyclically shifted} lattice path $\p(A{-}r{+}1,B{-}r{+}1)$ does not go below the $x$-axis, where $A{-}r{+}1:=\{a-r+1\mod n: a\in A\}$.

Now, for general $A$ and $B$ with $|A| = |B|$, let the steps of $\p(A, B)$ be $g_1, g_2, \dots, g_n$, where each $g_i \in \{\rightarrow, \nearrow, \searrow\}$, as in \Cref{def:up-down-path}. Then $\p(A{-}r{+}1, B{-}r{+}1)$ corresponds to the path formed by the steps $g_r, g_{r+1}, \ldots, g_n, g_1, \ldots, g_{r-1}$, taken in this order. The lemma then follows from the simple observation that $\p(A{-}r{+}1, B{-}r{+}1)$ does not go below the $x$-axis if and only if $\p(A, B)$ reaches its minimum at the point $(r, -)$.
\end{proof}
We also present the following lemma, which explains the relationship between shifted Gale orders indexed by different values of $r$.

\begin{lemma}\label{lemma:both-r-comparable}
    Let $A, B \subseteq [n]$ with $|A| = |B|$, and let $r \in [n]$. If $A\leq_rB$, denote $[A,B]_r:=\{I\subseteq[n]:A\leq_rI\leq_rB\}$ as the interval in the shifted Gale order $\leq_r$. We have:
    \begin{enumerate}
        \item If $A \leq_r B$, then for any other $r' \in [n]$, we have
        \[A \leq_{r'} B\iff \size{A\cap[r,r')_c}=\size{B\cap[r,r')_c}.\]
        \item If both $A\leq_rB$ and $A\leq_{r'}B$ hold, then the intervals coincide: $[A,B]_r=[A,B]_{r'}$. Moreover, for any $I\in[A,B]_r$,
    \[
        \size{A\cap [r,r')_c} = \size{I\cap [r,r')_c}=\size{B\cap [r,r')_c}.
    \]
    \end{enumerate}
\end{lemma}
\begin{proof}
     We first prove the $\implies$ direction in (1). Suppose $A \leq_r B$ and $A \leq_{r'} B$. By \Cref{lemma:r-comparable}, the lattice path $\p(A, B)$ passes through both points $(r, -d)$ and $(r', -d)$, where $d = \depth(A, B)$. Since these points lie at the same height, the number of $\nearrow$ and $\searrow$ steps between them must be equal. Therefore,
     \[\size{A\cap [r,r')_c} = \size{B\cap [r,r')_c}.\]
     The $\impliedby$ direction follows from the same reasoning in reverse: if the counts over $[r, r')_c$ agree, then $\p(A, B)$ passes through both $(r,-d)$ and $(r',-d)$ at the same depth, implying $A \leq_{r'} B$.

     For (2), consider any $I \in [A, B]_r$. We have:
    \[\begin{aligned}
        I\geq_r A&\implies \size{I\cap [r,r')_c}\leq\size{A\cap [r,r')_c},\\
        I\leq_r B&\implies\size{I\cap [r,r')_c}\geq\size{B\cap [r,r')_c}.
    \end{aligned}\] 
    Since $\size{A \cap [r, r')_c} = \size{B \cap [r, r')_c}$ by (1), it follows that
    \[\size{I\cap [r,r')_c} = \size{A\cap [r,r')_c} = \size{B\cap [r,r')_c}.\]
    By the $\impliedby$ direction of (1), this implies $I \in [A, B]_{r'}$. Therefore, $[A,B]_{r}=[A,B]_{r'}$.
\end{proof}
An immediate corollary is another characterization of the $\ba$-tilted Bruhat order $\lesssim_\ba$.
\begin{cor}\label{cor:alter-a-lesssim}
$u \lesssim_\ba v$ if and only if $u[k] \leq_{a_k} v[k]$ and $u[k] \leq_{a_{k+1}} v[k]$ for all $k \in [n-1]$.
\end{cor}

We now prove the fundamental existence property of the $\ba$-tilted Bruhat orders.

\begin{theorem}\label{thm:lesssim-exist}
    For any $u,v\in S_n$, there exists a sequence $\ba\in[n]^n$ such that $u\leq_\ba v$. Furthermore, there exists a sequence $\ba\in[n]^n$ such that $u\lesssim_\ba v$.
\end{theorem}
\begin{proof}
    The existence of a sequence $\ba$ such that $u \leq_\ba v$ follows directly from \Cref{lemma:r-comparable}, which guarantees that for each $k \in [n]$, there exists $a_k$ such that $u[k] \leq_{a_k} v[k]$.

    To find a sequence $\ba$ where $u \lesssim_\ba v$, we apply \Cref{cor:alter-a-lesssim}, which reduces the condition to: for each $k \in [n-1]$, find $a_{k+1}$ such that both $u[k] \leq_{a_{k+1}} v[k]$ and $u[k+1] \leq_{a_{k+1}} v[k+1]$. By \Cref{lemma:r-comparable}, this is equivalent to showing that the lattice paths $\p_k := \p(u[k], v[k])$ and $\p_{k+1} := \p(u[k+1], v[k+1])$ share a lowest point $(r, -)$ for some $r \in [n]$.

    Assume $\p_k$ reaches its minimum at $(r, -)$. Cyclically shift both $\p_k$ and $\p_{k+1}$ left by $r-1$ steps (as in the proof of \Cref{lemma:r-comparable}) so that $\p_k$ becomes a Dyck path with minimum at $(1, 0)$. We now consider three cases:
    \begin{enumerate}
        \item \textbf{Case $u_{k+1}<v_{k+1}$:} In this case, $\p_{k+1}$ is obtained from $\p_k$ by shifting the lattice subpath between $(u_{k+1}{+}1, -)$ and $(v_{k+1}, -)$ upward by one unit. This local adjustment preserves the Dyck path structure, so $\p_{k+1}$ also has minimum at $(1, 0)$.
        \item \textbf{Case $u_{k+1}=v_{k+1}$:} Here, $\p_{k+1} = \p_k$, so the result is immediate.
        \item \textbf{Case $u_{k+1}>v_{k+1}$:} In this case, $\p_{k+1}$ is obtained from $\p_k$ by shifting the lattice subpath between $(v_{k+1}{+}1, -)$ and $(u_{k+1}, -)$ downward by one unit, increasing the depth by at most one. If $\depth(\p_{k+1}) = 0$, then both paths share the minimum at $(1, 0)$. If $\depth(\p_{k+1}) = 1$, let $(r, -1)$ be a lowest point of $\p_{k+1}$. Then $\p_k$ must pass through $(r, 0)$, so both paths share a minimum at $(r, -)$.
    \end{enumerate}
    In all cases, we can find $r$ such that both $u[k] \leq_r v[k]$ and $u[k+1] \leq_r v[k+1]$ hold. Thus, such a sequence $\ba$ exists.
\end{proof}

We now prove the following theorem, which establishes an Ehresmann-like criterion for tilted Bruhat intervals (see \Cref{def:tilted-bruhat-interval}) using the $\ba$-tilted Bruhat orders.

\begin{theorem}\label{thm:tilted-criterion}
For $u,v,w\in S_n$, the following statements are equivalent:
\begin{enumerate}
\item $w$ lies in the tilted Bruhat interval $[u,v]$;
\item for all sequences $\ba\in [n]^n$ such that $u\leq_\ba v$, we have $u\leq_\ba w\leq_\ba v$;
\item there exists a sequence $\ba\in[n]^n$ such that $u\leq_\ba w\leq_\ba v$;
\item for all sequences $\ba\in [n]^n$ such that $u\lesssim_\ba v$, we have $u\lesssim_\ba w\lesssim_\ba v$;
\item there exists a sequence $\ba\in[n]^n$ such that $u\lesssim_\ba w\lesssim_\ba v$.
\end{enumerate}
\end{theorem}
\begin{proof}
We prove the equivalence in the following order: $(2)\implies(4)\implies(5)\implies (3)\implies(1)\implies (2)$. $(2) \implies (4)$ reduces to showing that if $u \leq_\ba w \leq_\ba v$ and $u \lesssim_\ba v$, then $u \lesssim_\ba w \lesssim_\ba v$, which follows from \Cref{lemma:both-r-comparable}. $(4)\implies(5)$ follows from the existence result in \Cref{thm:lesssim-exist}. $(5)\implies (3)$ is immediate from the definitions.

For $(3) \implies (1)$, let $\ba$ be the sequence given by (3) such that $u \leq_\ba w \leq_\ba v$. For any $k \in [n]$, since $u[k] \leq_{a_k} w[k] \leq_{a_k} v[k]$, the three lattice paths $\p(u[k], v[k])$, $\p(u[k], w[k])$, and $\p(w[k], v[k])$ all reach their lowest point at $(a_k, -)$ by \Cref{lemma:r-comparable}. By the construction in \Cref{def:up-down-path}, we then have:
\[
\begin{aligned}
    \depth(u[k],v[k])&=\size{v[k]\cap[a_k]}-\size{u[k]\cap[a_k]},\\
    \depth(u[k],w[k])&=\size{w[k]\cap[a_k]}-\size{u[k]\cap[a_k]},\\
    \depth(w[k],v[k])&=\size{v[k]\cap[a_k]}-\size{w[k]\cap[a_k]}.\\
\end{aligned}
\]
It follows that
\[\depth(u[k],v[k])=\depth(u[k],w[k])+\depth(w[k],v[k]),\]
Therefore, by \Cref{thm:weight-distance}, the $k$-th coordinate of $d(u, v)$ equals the sum of the $k$-th coordinates of $d(u, w)$ and $d(w, v)$, where $d(u, v)$ is the degree of a shortest path from $u$ to $v$ in $\Gamma_n$. Since this holds for all $k$, we conclude that $d(u, v) = d(u, w) + d(w, v)$, and hence $w \in [u, v]$ by \Cref{lemma:shortest-length-equals-weight}.

Finally, for $(1)\implies (2)$, fix a sequence $\ba$ such that $u \leq_\ba v$. Since $w \in [u, v]$, it lies on a shortest path from $u$ to $v$, so by \Cref{lemma:shortest-length-equals-weight}, we have $d(u, v) = d(u, w) + d(w, v)$. Recall that the long cycle $\cyclic = 23\cdots n1 \in S_n$ defines an automorphism of the unweighted quantum Bruhat graph by \Cref{lemma:cyclic-symmetry}. Applying this automorphism, define $u' := \cyclic^{1 - a_k} u$, $v' := \cyclic^{1 - a_k}v$, and $w' := \cyclic^{1 - a_k} w$. Then $w'$ also lies on a shortest path from $u'$ to $v'$, so
\[d(u', v') = d(u', w') + d(w', v').\]
By \Cref{thm:weight-distance}, the $k$-th coordinate of $d(u', v')$ is equal to the depth of the lattice path $\p(u'[k], v'[k])$, which is $0$ since $u[k] \leq_{a_k} v[k]$. It follows that the $k$-th coordinates of both $d(u', w')$ and $d(w', v')$ must also be $0$, which translates to $u[k] \leq_{a_k} w[k]$ and $w[k] \leq_{a_k} v[k]$. Hence, we conclude that $u \leq_\ba w \leq_\ba v$.
\end{proof}
An important consequence of \Cref{thm:tilted-criterion} is that the intervals in all $\ba$-tilted Bruhat orders coincide with the tilted Bruhat intervals themselves. Therefore, there is no ambiguity in referring to the tilted Bruhat interval $[u, v]$ as an interval in any of the aforementioned partial orders. As a direct consequence of \Cref{thm:tilted-criterion}, we obtain the following:

\begin{cor}\label{cor:xy}
Let $[x, y] \subseteq [u, v]$ be any subinterval. For any sequence $\ba \in [n]^n$, we have: \begin{enumerate}
\item If $u \leq_\ba v$, then $x \leq_\ba y$;
\item If $u\lesssim_\ba v$, then $x\lesssim_\ba y$.
\end{enumerate}
\end{cor}

\subsection{Properties of Tilted Bruhat Orders}
Readers may skip forward to \Cref{sec:tilted-richardson-definition} for our main geometric objects of interest, as necessary notations have been provided.

We continue in this section with more combinatorial properties of tilted Bruhat orders that will be used later. In particular, we provide an explicit criterion for covering relations in the $\ba$-tilted Bruhat orders (\Cref{prop:covering}), and show that the partial order $\lesssim_\ba$ is a ranked poset by introducing the notion of \emph{$\ba$-tilted length} as a rank function (\Cref{thm:length-rank-function}). Furthermore, we prove an analogue of the classical lifting property for $\ba$-tilted Bruhat orders $\lesssim_\ba$ (\Cref{thm:lifting}).  We also discuss several related conjectures. 

Let $\lessdot_\ba$ and $\lesssimdot_\ba$ denote the covering relations of the $\ba$-tilted Bruhat orders $\leq_\ba$ and $\lesssim_\ba$, respectively. The following proposition characterizes these covering relations.

\begin{prop}\label{prop:covering}
    Let $w\in S_n$, let $t_{ij}\in S_n$ be a transposition, and let $\ba = (a_1,\dots, a_{n})\in[n]^n$. Then:
    \begin{enumerate}
        \item $w<_{\ba} wt_{ij}\iff a_k\notin (w_i,w_j]_c$ for all $k\in [i,j)$;
        \item $w\lesssim_{\ba} wt_{ij}\iff a_k\notin (w_i,w_j]_c$ for all $k\in [i,j]$;
        \item $w\lessdot_{\ba} wt_{ij}\iff a_k\notin (w_i,w_j]_c$ for all $k\in [i,j)$ and $w_k\notin (w_i,w_j)_c$ for all $k\in (i,j)$;
        \item $w\lesssimdot_{\ba} wt_{ij}\iff a_k\notin (w_i,w_j]_c$ for all $k\in [i,j]$ and $w_k\notin (w_i,w_j)_c$ for all $k\in (i,j)$.
    \end{enumerate}
    Furthermore, for any $w,w'\in S_n$ and sequence $\ba$, if $w\lessdot_\ba w'$ or $w\lesssimdot_\ba w'$, then there exists a transposition $t_{ij}\in S_n$ such that $w'=wt_{ij}$.
\end{prop}
\begin{proof}
    For (1), we analyze the lattice path $\p_k := \p(w[k], wt_{ij}[k])$ for each $k \in [n]$. By \Cref{lemma:r-comparable}, we have $w <_\ba wt_{ij}$ if and only if $\p_k$ reaches its minimum at $(a_k, -)$ for all $k \in [n]$. For $k \in [i, j)$, the path $\p_k$ has a single upstep at position $w_i$ and a single downstep at position $w_j$, while for all other $k$, the path $\p_k$ is a flat line. Therefore, the condition holds if and only if $a_k \notin (w_i, w_j]_c$ for all $k \in [i, j)$.

    For (2), the proof closely follows that of (1), except that we additionally require each path $\p_k$ to reach its minimum at both $(a_k, -)$ and $(a_{k+1}, -)$, which corresponds to the definition of $\lesssim_\ba$ in \Cref{cor:alter-a-lesssim}.

    Now assume $w <_\ba w'$ (or $w\lesssim_\ba w'$). Then by \Cref{thm:tilted-criterion}, we have $w \lessdot_\ba w'$ (or $w \lesssimdot_\ba w'$) if and only if the tilted Bruhat interval $[w, w']$ contains exactly two elements, which is equivalent to $w \to w'$ being an edge in the quantum Bruhat graph $\Gamma_n$. This implies that $w' = wt_{ij}$ for some transposition $t_{ij}$ and, by definition of the quantum Bruhat graph, we must have $w_k \notin (w_i, w_j)_c$ for all $k \in (i, j)$. Statements (3), (4), and the final claim follow immediately from this observation.
\end{proof}


\begin{cor}\label{cor:tilted-bruhat-graph}
    Let $u,v,w\in S_n$ and let $t_{ij}\in S_n$ be a transposition. If both $w$ and $wt_{ij}$ lie in the tilted Bruhat interval $[u,v]$, then they are comparable in the poset $[u,v]$.
\end{cor}
\begin{proof}
    Fix a sequence $\ba\in [n]^n$ such that $u\lesssim_\ba v$, whose existence is guaranteed by \Cref{thm:lesssim-exist}. By \Cref{thm:tilted-criterion}, it suffices to show that $w$ and $wt_{ij}$ are comparable under $\lesssim_\ba$. Since $w,wt_{ij}\in [u,v]$, we have $w\sim_\ba wt_{ij}$. This implies that the entries $a_i,a_{i+1},\dots a_j$ all lie either in the interval $(w_j,w_i]_c$ or in $(w_i,w_j]_c$. In the first case, by \Cref{prop:covering}, we have $w\lesssim_\ba wt_{ij}$; in the second case, again by \Cref{prop:covering}, we have $wt_{ij}\lesssim_\ba w$. This completes the proof.
\end{proof}

Now we aim to show that the partial order $\lesssim_\ba$ defines a ranked poset. We begin by introducing a tilted analogue of the classical permutation length.

\begin{defin}\label{def:a-tilted-length}
    For any sequence $\ba\in[n]^n$, the set of \emph{$\ba$-inversions} of a permutation $w\in S_n$ is defined as
    \[\inv_\ba(w):=\{(i,j)\in[n]^2:i<j,\ w_i>_{a_i}w_j\}.\]
    The \emph{$\ba$-tilted length} of $w$, denoted $\ell_\ba(w)$, is defined as the number of $\ba$-inversions
    \[\ell_\ba(w):=\left|\inv_\ba(w)\right|.\]
\end{defin}

We remark that when $\ba=(1,1,\dots,1)$, $\inv_\ba(w)$ recovers the classical inversion set of $w$, and $\ell_\ba(w)$ recovers its classical length. The following theorem shows that $\lesssim_\ba$ is a ranked poset, with rank function given by the $\ba$-tilted length.

\begin{theorem}\label{thm:length-rank-function}
    For any sequence $\ba\in[n]^n$ and any covering relation $w\lesssimdot_\ba w'$, we have $\ell_\ba(w')=\ell_\ba(w)+1$. Consequently, the poset on $S_n$ defined by $\lesssim_\ba$ is a ranked poset, with rank function $\ell_\ba(w)$.
\end{theorem}
\begin{proof}
    Since $w \lesssimdot_\ba w'$, it follows from \Cref{prop:covering} that $w = w t_{pq}$ for some transposition $t_{pq} \in S_n$, and the following conditions hold: $a_k\notin(w_p,w_q]_c$ for all $k\in [p,q]$ and $w_k\notin(w_p,w_q)_c$ for all $k\in (p,q)$. Define $\ell_\ba^k(w)$ as the number of $\ba$-inversions $(i,j)$ with first coordinate $i=k$, or equivalently,
    \[\ell_\ba^k(w):=\left|\{w_j:j>k,\ w_j<_{a_k}w_k\}\right|=\size{w[k,n]\cap[a_k,w_k)_c]},\]
    where $w[a,b]:=\{w_a,w_{a+1},\dots,w_b\}$. We now compare $\ell^k_\ba(w)$ and $\ell^k_\ba(w')$ for all $k\in [n]$ by considering the following cases:
    \begin{enumerate}
        \item \textbf{Case $k<p$ or $k>q$:} $w[k,n]=w'[k,n]$ and $w_k=w'_k$, so $\ell^k_\ba(w)=\ell^k_\ba(w')$.
        \item \textbf{Case $p<k<q$:} Here $w_k = w'_k$, and $w'[k,n]$ is obtained from $w[k,n]$ by replacing $w_q$ with $w_p$. Since $w_p$ and $w_q$ either both lie in $[a_k, w_k)_c$ or both lie outside it (by the covering condition), we have $\ell^k_\ba(w) = \ell^k_\ba(w')$.
        \item \textbf{Case $k=p$ or $k=q$:} If $k = p$, then $w[p,n] = w'[p,n]$, and since $a_p \notin (w_p, w_q]_c$,
        \[\begin{aligned}
            \ell_\ba^p(w')-\ell_\ba^p(w) &= \size{w[p,n]\cap[a_p,w_q)_c}-\size{w[p,n]\cap[a_p,w_p)_c}\\
            &=\size{w[p,n]\cap[w_p,w_q)_c}.
        \end{aligned}\]
         If $k = q$, then since $a_q \notin (w_p, w_q]_c$, we have
        \[\begin{aligned}
            \ell_\ba^q(w)-\ell_\ba^q(w') &= \size{w[q,n]\cap[a_q,w_q)_c}-\size{w[q,n]\cap[a_q,w_p)_c}\\
            &= \size{w[q+1,n]\cap[a_q,w_q)_c}-\size{w[q+1,n]\cap[a_q,w_p)_c}\\
            &=\size{w[q+1,n]\cap[w_p,w_q)_c}.
        \end{aligned}\]
        By taking the difference, we have
        \[\begin{aligned}
            \ell_\ba^p(w') + \ell_\ba^q(w') - \ell_\ba^p(w) - \ell_\ba^q(w)
            =&\size{w[p,n]\cap[w_p,w_q)_c}-\size{w[q+1,n]\cap[w_p,w_q)_c}\\
            =&\size{w[p,q]\cap[w_p,w_q)_c}=1,
        \end{aligned}\]
        where the last equality follows from $w_k\notin(w_p,w_q)_c$ for all $k\in (p,q)$.
    \end{enumerate}
    Summing over all $k \in [n]$, we conclude that $\ell_\ba(w') = \ell_\ba(w) + 1$, completing the proof.
\end{proof}

The following corollary follows immediately from \Cref{thm:length-rank-function} and \Cref{thm:tilted-criterion}.

\begin{cor}\label{cor:length-rank-function}
    Let $u,v\in S_n$ and $\ba\in [n]^n$ such that $u\lesssim_\ba v$. Then the length of a shortest path from $u$ to $v$ in the quantum Bruhat graph $\Gamma_n$ is given by
    \[\ell(u,v)=\ell_\ba(v)-\ell_\ba(u).\]
\end{cor}

Some basic properties of the $\ba$-tilted Bruhat orders are of further interests.
\begin{conj}
$\leq_\ba$ is a ranked poset.
\end{conj}
Note that the $\ba$-tilted length $\ell_\ba(w)$ does not serve as a rank function for $\leq_\ba$.
\begin{conj}
$\ell_\ba(w)$ is a connected poset and $\lesssim_\ba$ is connected when restricted to each equivalence class under $\sim_\ba$. Here, a poset is \emph{connected} if its Hasse diagram is connected.
\end{conj}

Now we state an analogue of the classical lifting property for $\ba$-tilted Bruhat orders $\lesssim_\ba$.

\begin{defin}\label{def:ascdes}
    Let $w\in S_n$ and $\ba\in[n]^n$. The \emph{$\ba$-descent set} and \emph{$\ba$-ascent set} of $w$ are
    \[\begin{aligned}
        \Des_\ba(w) &:= \{i\in [n-1]:a_i = a_{i+1},\ w_i >_{a_i}w_{i+1}\},\\
        \Asc_\ba(w) &:= \{i\in [n-1]:a_i = a_{i+1},\ w_i<_{a_i}w_{i+1}\}.
    \end{aligned}\]
\end{defin}

Equivalently, if $a_i=a_{i+1}$, then $i\in \Des_\ba(w)$ if and only if $ws_i\lesssim_\ba w$, and $i\in \Asc_\ba(w)$ if and only if $ws_i\gtrsim_\ba w$. We are now ready to state the lifting property.

\begin{theorem}\label{thm:lifting}
    Let $u,v\in S_n$, $\ba\in[n]^n$, and $i\in [n-1]$. If $u\lesssim_\ba v$ and $i\in \Des_\ba(v)\cap \Asc_\ba(u)$, then $us_i\lesssim_\ba v$ and $u\lesssim_\ba vs_i$.
\end{theorem}
\begin{proof}
    We prove only the first claim $us_i\lesssim_\ba v$, and the second follows by a similar argument. Since $u\lesssim_\ba v$, it suffices to show that
    \[us_i[i]\leq_{a_i}v[i].\]  
    Given that $a_i = a_{i+1}$ and $u \lesssim_\ba v$, we can apply \Cref{cor:alter-a-lesssim} to obtain:
    \[
        u[i-1]\leq_{a_i}v[i-1]\quad\text{and}\quad u[i+1]\leq_{a_i}v[i+1].
    \]
    We now apply \Cref{lemma:r-comparable}, reducing the problem to a combinatorial question about lattice paths: if the paths $\p_{i-1} := \p(u[i-1], v[i-1])$ and $\p_{i+1} := \p(u[i+1], v[i+1])$ both reach their minimum at $(a_i, 0)$, then the path $\p_i := \p(us_i[i], v[i])$ must also reach its minimum at $(a_i, 0)$. Without loss of generality, we may assume $a_i = 1$, since we can cyclically shift all paths to the left by $a_i - 1$ steps (as in the proof of \Cref{lemma:r-comparable}). In this case, the paths $\p_{i-1}$ and $\p_{i+1}$ are both Dyck paths with minimum at $(1, 0)$, and our goal is to show that the same holds for $\p_i$.

    We proceed by contradiction. Suppose $\p_i$ is not a Dyck path, then it must reach its minimum at some point $(r, -1)$. Comparing $\p_{i-1}$ and $\p_i$, we observe that $\p_{i-1}$ passes through $(r, 0)$, and $\p_i$ is obtained from $\p_{i-1}$ by shifting the subpath between $(v_i{+}1, -)$ and $(u_{i+1}, -)$ downward by one unit. This implies $v_i < r \leq u_{i+1}$.

    On the other hand, comparing $\p_i$ and $\p_{i+1}$, we observe that $\p_{i+1}$ is obtained from $\p_i$ by shifting the subpath between $(u_i{+}1, -)$ and $(v_{i+1}, -)$ upward by one unit. This implies $u_i < r \leq v_{i+1}$.
    
    Combining both, we obtain $v_i < r \leq v_{i+1}$, which contradicts the assumption that $i \in \Des_\ba(v)$. Therefore, $\p_i$ remains above the $x$-axis, and we conclude that $u s_i \lesssim_\ba v$.
\end{proof}

The following corollary follows immediately from \Cref{thm:lifting}.

\begin{cor}\label{cor:lifting}
    Let $w,w'\in S_n$ and $\ba\in [n]^n$ such that $w'\lesssimdot_\ba w$. If $i\in \Des_\ba(w)$, then either $i\in\Des_\ba(w')$ or $w'=ws_i$.
\end{cor}
\begin{proof}
    If $i\notin \Des_\ba(w')$, or equivalently $i\in \Asc_\ba(w')$, then by \Cref{thm:lifting} we have $ws_i\in [w',w]$. Since $w' \lesssimdot_\ba w$, the interval $[w', w]$ contains exactly two elements, $w'$ and $w$. Thus, it must be that $w' = w s_i$.
\end{proof}

In fact, we prove the following stronger version of the lifting property, which does not rely on the specific choice of $\ba$.

\begin{theorem}\label{thm:strong-lifting}
    Let $u,v\in S_n$ and $i\in[n-1]$. The following statements are equivalent:
    \begin{enumerate}
        \item There exists $\ba\in [n]^n$ such that $u\lesssim_\ba v$ and $i\in \Des_\ba(v)\cap \Asc_\ba(u)$;
        \item There exists $\ba\in [n]^n$ such that $u\lesssim_\ba v$, $a_i=a_{i+1}$, and $us_i,vs_i\in [u,v]$;
        \item The interval $[u,v]$ is invariant under right multiplication by $s_i$, i.e., $[u,v]=[u,v]\cdot s_i$, where $[u,v]\cdot s_i:=\{ws_i:w\in [u,v]\}$.
    \end{enumerate}
\end{theorem}
\begin{proof}
    The equivalence of (1) and (2) follows from \Cref{thm:lifting}. To prove $(1) \implies (3)$, let $w \in [u, v]$. Since $a_i = a_{i+1}$, we have either $i \in \Des_\ba(w)$ or $i \in \Asc_\ba(w)$. In the former case, \Cref{thm:lifting} implies that $w s_i \in [u, w] \subseteq [u, v]$; in the latter case, the same theorem gives $w s_i \in [w, v] \subseteq [u, v]$. In either case, we conclude that $w s_i \in [u, v]$, establishing (3).

    It remains to show $(3) \implies (2)$, that is, there exists $\ba \in [n]^n$ such that $u \lesssim_\ba v$ and $a_i = a_{i+1}$. By \Cref{thm:lesssim-exist}, there exists some $\ba \in [n]^n$ such that $u \lesssim_\ba v$, though possibly with $a_i \neq a_{i+1}$. We modify this sequence to $\ba' = (a_1, \dots, a_i, a_i, a_{i+2}, \dots)$ and claim that $u \lesssim_{\ba'} v$. Equivalently, it suffices to verify that $u[i+1] \leq_{a_i} v[i+1]$.

    By assumption (3), we have $w s_i \in [u, v]$ for every $w \in [u, v]$. This implies $w \sim_\ba w s_i$, so $w_i$ and $w_{i+1}$ must either both lie in $[a_i, a_{i+1})_c$ or both lie in $[a_{i+1}, a_i)_c$. We claim this choice is consistent across all $w \in [u, v]$. Indeed, if $w$ and $w'$ differ by a transposition $w' = w t_{pq}$, then such a transposition cannot simultaneously move $w_i$ and $w_{i+1}$ from one cyclic interval to the other. Since $[u, v]$ is connected by transpositions, the cyclic interval to which $w_i$ and $w_{i+1}$ belong is consistent throughout $[u, v]$.

    Now assume $u_i, u_{i+1}, v_i, v_{i+1}$ all lie in $[a_i, a_{i+1})_c$ (the case $[a_{i+1}, a_i)_c$ is analogous). Since $u \lesssim_\ba v$, by \Cref{cor:alter-a-lesssim} and \Cref{lemma:r-comparable}, the lattice path $\p_i := \p(u[i], v[i])$ reaches its minimum at both $(a_i, -)$ and $(a_{i+1}, -)$. Moreover, since $u[i+1] \leq_{a_{i+1}} v[i+1]$, the path $\p_{i+1} := \p(u[i+1], v[i+1])$ reaches its minimum at $(a_{i+1}, -)$. But under our assumption that all relevant entries lie in $[a_i, a_{i+1})_c$, the depth at $(a_i, -)$ equals the depth at $(a_{i+1}, -)$, so $\p_{i+1}$ also reaches its minimum at $(a_i, -)$. Hence, $u[i+1] \leq_{a_i} v[i+1]$. This verifies the condition needed for $u \lesssim_{\ba'} v$, and completes the proof.
\end{proof}

\begin{remark}
The lifting property does not hold for the alternative $\ba$-tilted Bruhat order $\leq_\ba$. For example, let $u = 4231$, $v = 3421$, $\ba = (4,2,2,2)$, and $i = 2$. Then $u \leq_\ba v$ and $i \in \Des_\ba(v) \cap \Asc_\ba(u)$, but $v s_i = 3241 \notin [u, v]$.
\end{remark}

\subsection{Tilted Reduced Words}\label{sec:tilted-reduced-word}
In this section, we define \emph{$\ba$-tilted reduced words} for a sequence $\ba \in [n]^n$ (\Cref{def:tiltedword}), which serve as the tilted analogues of classical reduced words under the $\ba$-tilted Bruhat order $\lesssim_\ba$. The material will be primarily used in \Cref{sec:deodhar}.

We start by introducing a technical notion called the \emph{flattening} of a sequence $\ba$, which is a process for producing a simpler sequence derived from $\ba$.

\begin{defin}\label{def:flatten}
Let $\ba = (a_1, a_2, \dots, a_n) \in [n]^n$. The set of \emph{jumps} of $\ba$ is defined by  
\[
\Jump_\ba := \{ j \in [n] : a_j \neq a_{j+1} \}, \quad \text{with the convention that } a_{n+1} := 1.
\]
If $\Jump_\ba \neq \emptyset$ (equivalently, $\ba \neq (1,1,\dots,1)$), the \emph{minimum jump} of $\ba$, denoted $\jmin(\ba)$ (or simply $\jmin$ when the context is clear), is defined as  
\[
\jmin(\ba) := \min \Jump_\ba,
\]  
that is, the smallest index $j$ such that $a_j \neq a_{j+1}$.

We define the \emph{flattening} of $\ba$, denoted $\flatten(\ba)$, as the sequence obtained by replacing the first $\jmin(\ba)$ entries of $\ba$ with the value $a_{\jmin+1}$. More formally,
\[
\flatten(\ba) := (a'_1, a'_2, \dots, a'_n), \quad \text{where } a'_j = \begin{cases}
a_{\jmin+1}, &\text{if } j \leq \jmin, \\
a_j, &\text{if } j > \jmin.
\end{cases}\]
\end{defin}

\begin{ex}
Let $\ba = (2,2,4,4,4,3)$. The jumps of $\ba$ are
\[
\Jump_\ba = \{2, 5, 6\}, \quad \text{since } a_2 \neq a_3,\ a_5 \neq a_6,\ \text{and } a_6 \neq a_7.
\]
Thus, $\jmin(\ba) = 2$, and $a_{\jmin+1} = a_3 = 4$. Therefore, $\flatten(\ba) = (4, 4, 4, 4, 4, 3)$.
\end{ex}

This example illustrates a key property of the flattening operation: it reduces the number of jumps by one:
\[\left|\Jump_{\flatten(\ba)}\right|=\jasize-1.\]
This property enables induction on the number of jumps in $\ba$. By iteratively applying the flattening process, we eventually reach the base case $\ba = (1,1,\dots,1)$ with no jumps, which is the classical setting. We now introduce a new concept that builds on this idea.

\begin{defin}\label{def:a-flattenable}
    Let $w \in S_n$ and $\ba=(a_1,\dots,a_n)\in[n]^n$. We say that $w$ is $\ba$-\emph{flattenable} if there exists an integer $0\leq p \leq \jmin(\ba)$ such that 
    \[w_i \in\begin{cases}
    [a_1,a_{\jmin+1})_c & \text{if }i\leq p,\\
    [a_{\jmin+1},a_1)_c & \text{if }p<i\leq \jmin.
    \end{cases}\]
\end{defin}
\begin{ex}
    We use \Cref{fig:vflattenable} to illustrate the definition of $\ba$-flattenable permutations. In the figure, each $\bullet$ represents an entry of the permutation $w$, and the red lines correspond to the values $a_1$ and $a_{\jmin+1}$ from the sequence $\ba$. The condition for $\ba$-flattenability requires that all $\bullet$’s in the first $\jmin$ columns lie within the shaded region.
\begin{figure}[ht]
    \centering
    \begin{tikzpicture}[scale = 0.7]
        \draw (0,0) -- (4,0) -- (4,4) -- (0,4) -- (0,0);
        \draw[pattern=north west lines, pattern color = blue!40] (0,4) -- (1.3,4) -- (1.3,3.2) -- (0,3.2) -- (0,4);
        \draw[pattern=north west lines, pattern color = blue!40] (2,1.3) -- (1.3,1.3) -- (1.3,3.2) -- (2,3.2) -- (2,1.3);
        \draw[pattern=north west lines, pattern color = blue!40] (0,1.3) -- (1.3,1.3) -- (1.3,0) -- (0,0) -- (0,1.3);
        \draw (0,4) -- (0,4.4);
        \draw (1.3,4) -- (1.3,4.4);
        \draw[->] (0,4.2) -- (0.45,4.2);
        \draw[->] (1.3,4.2) -- (0.85,4.2);
        \node at (0.65,4.2) {$p$};
        \draw (0,0) -- (0,-0.4);
        \draw (2,4) -- (2,-0.4);
        \draw[->] (0,-0.3) -- (0.6,-0.3);
        \draw[->] (2,-0.3) -- (1.4,-0.3);
        \node at (1,-0.3) {$\jmin$};
        \node at (-0.4,1.3) {$a_1$};
        \node at (-0.7, 3.2) {$a_{\jmin{+}1}$};
        \node at (1.5,2.5) {$\bullet$};
        \node at (1.8,1.7) {$\bullet$};
        \node at (0.3,0.4) {$\bullet$};
        \node at (0.65,3.6) {$\bullet$};
        \node at (1.0,0.9) {$\bullet$};
        \draw[red, very thick] (0,1.3) -- (2,1.3);
        \draw[red, very thick] (2,3.2) -- (3,3.2);
    \end{tikzpicture}
    \caption{Permutation matrix of an $\ba$-flattenable permutation. The $\bullet$’s represent the permutation, and the red lines represent the sequence $\ba$}
    \label{fig:vflattenable}
\end{figure}
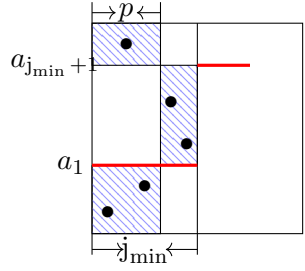
\end{ex}

The importance of $\ba$-flattenability is that it allows us to replace $\ba$ with its flattening $\flatten(\ba)$ while preserving comparability under $\lesssim_\ba$. This is formalized in the following lemma.
\begin{lemma}\label{lemma:u<flatten(a)v}
    Suppose $u\lesssim_{\ba}v$ and $v$ is $\ba$-flattenable, then $u\lesssim_{\flatten(\ba)}v$ and $u$ is $\ba$-flattenable. Conversely, if $u\lesssim_{\flatten(\ba)}v$, $u\sim_\ba v$, and both $u$ and $v$ are $\ba$-flattenable, then $u\lesssim_{\ba}v$.
\end{lemma}
\begin{proof}
    First, suppose $u\lesssim_\ba v$ and $v$ is $\ba$-flattenable. We show that $u$ is $\ba$-flattenable. Since $v$ is $\ba$-flattenable, let $p$ be the integer from \Cref{def:a-flattenable} such that
    \[v[p]\subseteq[a_1,a_{\jmin+1})_c,\quad v[p+1,\jmin]\subseteq[a_{\jmin+1},a_1)_c.\]
    Since $u\leq_\ba v$, we have $u[p]\leq_{a_1} v[p]$, and because $v[p]\subseteq [a_1,a_{\jmin+1})_c$, it follows that $u[p]\subseteq [a_1,a_{\jmin+1})_c$ as well. Moreover, the equivalence $u\sim_{\ba} v$ implies
    \[\size{u[\jmin] \cap [a_1,a_{\jmin+1})_c} = \size{v[\jmin] \cap [a_1,a_{\jmin+1})_c}=p.\]
    Therefore, the remaining values $u[p+1,\jmin]$ lie in the complementary interval $[a_1,a_{\jmin+1})_c$, and hence $u$ is $\ba$-flattenable.
    
    Now suppose $u \sim_\ba v$ and both $u$ and $v$ are $\ba$-flattenable. By \Cref{def:a-flattenable}, there exists an integer $p$ such that
    \[u[p],v[p]\subseteq[a_1,a_{\jmin+1})_c,\quad u[p+1,\jmin],v[p+1,\jmin]\subseteq[a_{\jmin+1},a_1)_c.\] It remains to show that $u\lesssim_{\ba}v\iff u\lesssim_{\flatten(\ba)}v$, or equivalently,
    \[u[k]\leq_{a_1} v[k]\iff u[k]\leq_{a_{\jmin+1}} v[k]\quad\text{for all } k\in [\jmin],\]
    which holds because the values of both $u$ and $v$ in $[\jmin]$ lie in the same cyclic intervals, so the relevant shifted orders coincide.
\end{proof}

We are now ready to define $\ba$-tilted reduced words.

\begin{defin}\label{def:tiltedword}
    An \emph{$\ba$-tilted word} for $w\in S_n$, denoted $\mathbf{w}=s_{i_1}s_{i_2}\cdots s_{i_\ell}$, is a sequence of simple transpositions whose product equals $w$, together with $\jasize$ bars ``$\mid$'' inserted between positions in the sequence, satisfying the following recursive conditions:
    \begin{enumerate}
        \item If $\Jump_\ba\neq \emptyset$, then no simple transpositions $s_i$ with $i\in \Jump_\ba$ appears after the last bar in $\mathbf{w}$;
        \item If $\Jump_\ba\neq \emptyset$, then the subsequence of $\mathbf{w}$ before the last bar is a $\flatten(\ba)$-tilted word for some $\ba$-flattenable permutation $w'\in S_n$.
    \end{enumerate}
    In the base case where $\Jump_\ba=\emptyset$, or equivalently $\ba =(1,1,\dots,1)$, an $\ba$-tilted word is simply a classical word for $w$, with no bars.
    
    We refer to both the simple transpositions and the bars in $\mathbf{w}$ as the \emph{factors} of the word. The \emph{length} of $\mathbf{w}$ is the total number of factors it contains. If the length of $\mathbf{w}$ is minimal among all $\ba$-tilted words for $w$, then $\mathbf{w}$ is called an \emph{$\ba$-tilted reduced word} for $w$. The length of such a reduced word is called the \emph{$\ba$-tilted word length} of $w$, and is denoted by $\ell^\word_\ba(w)$.
\end{defin}

\begin{ex}\label{ex:tiltedword}
    An example of an $\ba$-tilted word for $\ba=(4,4,2,2)$ and $w=3142$ is
    \[\mathbf{w}=s_1s_2s_3\mid s_1s_2s_3s_2s_1\mid s_1.\]
    In fact, this word is reduced, so $\ell^\word_\ba(w)=11$.
\end{ex}

The existence of $\ba$-tilted words for an arbitrary permutation is not obvious from the definition. The following lemma ensures that such words always exist.

\begin{lemma}\label{lemma:a-tilted-word-exist}
    For any $w\in S_n$ and $\ba\in[n]^n$, there exists an $\ba$-tilted word for $w$.
\end{lemma}
\begin{proof}
    We induct on $\jasize$. In the base case $\jasize=0$, or equivalently $\ba=(1,1,\dots,1)$, the claim follows from the fact that a classical word (with no bars) exist for $w$.
    
    For the inductive step, define $w'\in S_n$ to be the permutation obtained from $w$ by sorting its first $\jmin$ entries in increasing order with respect to the shifted order $<_{a_1}$. Let $u:=(w')^{-1}w$. By the inductive hypothesis, there exists a $\flatten(\ba)$-tilted word $\mathbf{w}'$ for $w'$, and by classical theory, there exists a reduced word $\mathbf{u}$ for $u$. We claim that the concatenated word $\mathbf{w} = \mathbf{w}' \mid \mathbf{u}$ is a $\ba$-tilted word for $w$. We verify the two conditions from \Cref{def:tiltedword}:
    \begin{enumerate}
        \item Since $w$ and $w'$ differ only in the first $\jmin$ entries, the permutation $u=(w')^{-1}w$ acts only on those positions. Hence, the reduced word $\mathbf{u}$ involves only $s_i$ with $i\geq \jmin$, and therefore none with $i\in \Jump_\ba$.
        \item By construction, the first $\jmin$ entries of $w'$ are sorted under $<_{a_1}$, which implies that $w'$ is $\ba$-flattenable. Therefore, condition (2) holds.
    \end{enumerate}
\end{proof}

\begin{construction}\label{construction:tiltedreducedword}
    The proof of \Cref{lemma:a-tilted-word-exist} suggests an algorithm for constructing a $\ba$-tilted word for any permutation $w\in S_n$. We illustrate this with the example $\ba=(3,3,1,1,1,6)$ and $w=136254$.

    Let the set of jumps of $\ba$ be denoted as $\Jump_\ba=\{\jump_1<\jump_2<\cdots<\jump_t\}$, where $t:=\jasize$. In our example, we have
    \[\Jump_\ba=\{\jump_1,\jump_2,\jump_3\}=\{2,5,6\}.\]
    We now construct a sequence of permutations
    \[\id=w^{(t+1)}\to w^{(t)}\to\cdots\to w^{(0)}=w,\]
    where for each $k\in [t]$, the permutation $w^{(k)}\in S_n$ is obtained from $w$ by sorting its first $\jump_k$ entries in increasing order with respect to the shifted order $<_{a_{\jump_k}}$. We set $w^{(0)}:=w$ and $w^{(t+1)}:=\id$ by convention. In our running example, the sequence becomes
    \[123456\longrightarrow 612345\longrightarrow123564\longrightarrow316254\longrightarrow136254.\]
    Next, for each adjacent pair $w^{(k+1)}\to w^{(k)}$, we compute a reduced word for the premutation $(w^{(k+1)})^{-1}w^{(k)}$, and then concatenate all of these reduced words, separated by bars ``$\mid$'', to form an $\ba$-tilted word for $w$. In our running example, we obtain:
    \[s_5s_4s_3s_2s_1\mid s_1s_2s_3s_5s_4\mid s_2s_1s_4s_3\mid s_1.\]
    In fact, all of the words constructed in this way are $\ba$-tilted reduced words, as can be deduced from \Cref{cor:length-property}.
\end{construction}


\begin{defin}\label{def:tilted-subword}
    Let $\mathbf{v}$ be an $\ba$-tilted word for $v\in S_n$. A \emph{subword} $\mathbf{u}$ of $\mathbf{v}$, corresponding to a permutation $u\in S_n$, is obtained by replacing some of the non-bar factors (i.e., simple transpositions $s_i$) in $\mathbf{v}$ with the identity element $1$, subject to the following conditions:
    \begin{enumerate}
        \item The sequence $\mathbf{u}$, after omitting the $1$'s, forms a valid $\ba$-tilted word (i.e., one satisfying the conditions in \Cref{def:tiltedword});
        \item The resulting product of $\mathbf{u}$ equals $u$, and $u\sim_\ba v$.
    \end{enumerate}
\end{defin}

\begin{ex}
    An example of a subword of the $\ba$-tilted word from \Cref{ex:tiltedword}
    \[\mathbf{w}=s_1s_2s_3\mid s_1s_2s_3s_2s_1\mid s_1,\]
    is the following subword corresponding to $u=4231$:
    \[\mathbf{u}=s_1s_2s_3\mid 1s_211s_1\mid 1.\]
\end{ex}

We also establish several fundamental properties of $\ba$-tilted reduced words and their subwords, which parallel well-known properties in the classical theory of reduced words. We summarize the results as follows:

\begin{enumerate}
    \item \textbf{Length Property:} The $\ba$-tilted word length $\ell^\word(w)$ serves as a rank function for the $\ba$-tilted Bruhat order $\lesssim_\ba$ (\Cref{thm:length-property}).
    \item \textbf{Subword Property:} Let $u,v\in S_n$ and let $\mathbf{v}$ be an $\ba$-tilted reduced word for $v$. Then, $u\lesssim_\ba v$ if and only if there exists a (reduced) subword $\mathbf{u}$ of $\mathbf{v}$ corresponding to $u$ (\Cref{thm:subword-property}). 
    \item \textbf{Word Property:} Any two $\ba$-tilted reduced words for the same permutation $w$ are connected by a sequence of moves involving \emph{braid relations} and \emph{bar relations} (\Cref{thm:word-property}).
\end{enumerate}

We begin with a lemma that will be useful later in the inductive arguments.

\begin{lemma}\label{lemma:tildef(u)}
    Let $\mathbf{w}$ be an $\ba$-tilted word for $w$ with $\ba\neq(1,1,\dots,1)$. Define $\flatten(\mathbf{w})$ to be the sequence obtained from $\mathbf{w}$ by removing the last bar. Then $\flatten(\mathbf{w})$ is a valid $\flatten(\ba)$-tilted word for $w$. Furthermore, if $w$ is $\ba$-flattenable, then $\mathbf{w}$ is reduced if and only if $\flatten(\mathbf{w})$ is reduced.
\end{lemma}
\begin{proof}
    It is clear from \Cref{def:tiltedword} that $\flatten(\mathbf{w})$ is a $\flatten(\ba)$-tilted word for $w$. Suppose $\mathbf{w}$ is reduced and $\flatten(\mathbf{w})$ is not, then there exists a shorter $\flatten(\ba)$-tilted word $\mathbf{w}'$ for $w$. Appending a bar to the end of $\mathbf{w}'$ yields an $\ba$-tilted word for $w$ shorter than $\mathbf{w}$, contradicting the reducedness of $\mathbf{w}$. The opposite direction follows by a similar argument.
\end{proof}

We can now prove the length property.

\begin{theorem}[Length Property]\label{thm:length-property}
    Let $\ba \in [n]^n$. The $\ba$-tilted word length $\ell^\word_\ba(w)$ defines a rank function for the $\ba$-tilted Bruhat order $\lesssim_\ba$. In other words, for any covering relation $w' \lesssimdot_\ba w$, we have $\ell^\word_\ba(w)=\ell^\word_\ba(w')+1$.
\end{theorem}
\begin{proof}
    Suppose the statement does not hold in general. Let $w' \lesssimdot_\ba w$ be a minimal counterexample, meaning that $\ell^\word_\ba(w) \neq \ell^\word_\ba(w') + 1$ and $\ell^\word_\ba(w)$ is as small as possible among such counterexamples. Let $\mathbf{w}$ be a $\ba$-tilted reduced word for $w$ of length $\ell^\word_\ba(w)$. We consider the following cases:
    \begin{enumerate}
        \item Suppose $\mathbf{w}$ ends with a bar. Then $w$ is $\ba$-flattenable, so by \Cref{lemma:tildef(u)}, we have $\ell^\word_\ba(w)=\ell^\word_{\flatten(\ba)}(w)+1$. Since $w' \lesssim_\ba w$, \Cref{lemma:u<flatten(a)v} implies that $w'$ is also $\ba$-flattenable and $\ell^\word_\ba(w')=\ell^\word_{\flatten(\ba)}(w')+1$. Therefore, $\ell^\word_{\flatten(\ba)}(w)\neq \ell^\word_{\flatten(\ba)}(w')+1$, which contradicts the minimality of the counterexample, since $w$ and $w'$ would then form a smaller counterexample under the flattened sequence $\flatten(\ba)$.
        \item Suppose $\mathbf{w}$ ends with $s_i$ and $i\in\Asc_\ba(w)$. Removing the last $s_i$ from $\mathbf{w}$ yields an $\ba$-tilted word for $ws_i$, so $\ell^\word_\ba(w)\geq \ell^\word_\ba(ws_i)+1$. Conversely, appending $s_i$ to an $\ba$-tilted reduced word for $w s_i$ gives an $\ba$-tilted word for $w$, so equality must hold. Since $w \lesssimdot_\ba w s_i$ by \Cref{prop:covering}, and $\ell_\ba^\word(ws_i)<\ell_\ba^\word(w)$, this contradicts the minimality of the counterexample.
        \item Suppose $\mathbf{w}$ ends with $s_i$ and $i\in\Des_\ba(w)$. 
        The same reasoning as in case (2) shows that $\ell^\word_\ba(w) = \ell^\word_\ba(w s_i) + 1$. This implies that $w' \neq ws_i$. By \Cref{cor:lifting}, $i \in \Des_\ba(w')$, and applying \Cref{thm:lifting} to the pair $w$ and $w' s_i$ gives $w' s_i \lesssimdot_\ba w s_i$. Since $\ell^\word_\ba(w s_i) < \ell^\word_\ba(w)$, the minimality of the counterexample implies $\ell^\word_\ba(w s_i) = \ell^\word_\ba(w' s_i) + 1$, so that $\ell^\word_\ba(w) = \ell^\word_\ba(w' s_i) + 2$. On the other hand, applying the same argument in case (2) to $w'$ gives $\ell^\word_\ba(w') \leq \ell^\word_\ba(w' s_i) + 1$, which implies $\ell^\word_\ba(w') < \ell^\word_\ba(w)$. By the minimality of the counterexample, we must have $\ell^\word_\ba(w')=\ell^\word_\ba(w's_i)+1=\ell^\word_\ba(w)-1$, which contradicts the assumption that $\ell^\word_\ba(w) \neq \ell^\word_\ba(w') + 1$.
    \end{enumerate}
    In all three cases, we reach a contradiction. Therefore, the statement must hold.
\end{proof}

The following corollary follows immediately from \Cref{thm:length-property}.

\begin{cor}\label{cor:length-property}
    Let $w\in S_n$, $\ba\in [n]^n$, and $i\in [n-1]$. Then
    \[
        \ell^\word_\ba(w)=\begin{cases}
        \ell^\word_\ba(ws_i) + 1 &\text{if }i\in \Des_\ba(w),\\
        \ell^\word_\ba(ws_i) - 1 &\text{if }i\in \Asc_\ba(w).
    \end{cases}\]
    Consequently, there exists an $\ba$-tilted reduced word for $w$ ending with $s_i$ if and only if $i\in\Des_\ba(w)$.
\end{cor}

We are now ready to prove the subword property.

\begin{theorem}[Subword Property]\label{thm:subword-property}
    Let $u,v\in S_n$. Let $\mathbf{v}$ be an $\ba$-tilted reduced word for $v$. Then $u\lesssim_\ba v$ if and only if there exists a (reduced) subword $\mathbf{u}$ of $\mathbf{v}$ corresponding to $u$. 
\end{theorem}
\begin{proof}
    We proceed by induction on $\ell^\word_\ba(v)$. In the base case $\ell^\word_\ba(v)=0$, which corresponds to $\ba=(1,1,\dots,1)$ and $v=\id$, the statement is immediate. For the inductive step, we first prove the $\impliedby$ direction. Let $\mathbf{v}'$ and $\mathbf{u}'$ denote the tilted words obtained by removing the last factor from $\mathbf{v}$ and $\mathbf{u}$, respectively. There are two cases:
    \begin{enumerate}
        \item Suppose $\mathbf{v}$ ends with a bar. Then so does $\mathbf{u}$, and hence $v$ and $u$ are $\ba$-flattenable. By the induction hypothesis, $u\lesssim_{\flatten(\ba)}v$, and since $u\sim_\ba v$, it follows from \Cref{lemma:u<flatten(a)v} that $u\lesssim_{\ba}v$.
        \item Suppose $\mathbf{v}$ ends with $s_i$. Then $i\in \Des_\ba(v)$, and $\mathbf{v}'$ is an $\ba$-reduced word for $vs_i$. Since $\mathbf{u}'$ is a subword of $\mathbf{v}'$ corresponding to either $u$ or $us_i$, the induction hypothesis gives either $u\lesssim_\ba vs_i$ or $u\lesssim_\ba vs_i$. In either case, applying \Cref{thm:lifting} gives $u\lesssim_\ba v$.
    \end{enumerate}
    We now prove the $\implies$ direction. Again, there are two cases:
    \begin{enumerate}
        \item Suppose $\mathbf{v}$ ends with a bar. Then $v$ is $\ba$-flattenable. Since $u\lesssim_\ba v$, by \Cref{lemma:u<flatten(a)v}, $u$ is also $\ba$-flattenable. By the induction hypothesis, there exists a (reduced) subword $\mathbf{u}'$ of the $\flatten(\ba)$-reduced word $\mathbf{v}'$, then $\mathbf{u}:=\mathbf{u}'\mid$ is the desired (reduced) subword of $\mathbf{v}$.
        \item Suppose $\mathbf{v}$ ends with $s_i$. Then $i\in \Des_\ba(v)$, and $\mathbf{v}'$ is an $\ba$-reduced word for $vs_i$. If $i\in \Des_\ba(u)$, then by \Cref{thm:lifting}, $us_i\lesssim_\ba vs_i$. By the induction hypothesis, there exists a (reduced) subword $\mathbf{u}'$ of $\mathbf{v}'$ corresponding to $u s_i$, and setting $\mathbf{u} := \mathbf{u}' s_i$ gives the desired (reduced) subword. If $i\in \Asc_\ba(u)$, then by \Cref{thm:lifting}, $u\lesssim_\ba vs_i$. By the induction hypothesis, there exists a (reduced) subword $\mathbf{u}'$ of $\mathbf{v}'$ corresponding to $u$, and setting $\mathbf{u} := \mathbf{u}' 1$ gives the desired (reduced) subword. 
    \end{enumerate}
\end{proof}

Finally, we are ready to prove the word property.

\begin{theorem}[Word Property]\label{thm:word-property}
    Any two $\ba$-tilted reduced words for the same permutation $w$ are connected by a sequence of moves involving the following relations:
    \begin{itemize}
        \item (Braid relation) $s_is_j=s_js_i$ if $|i-j|>1$,
        \item (Braid relation) $s_is_{i+1}s_i=s_{i+1}s_is_{i+1}$,
        \item (Bar relation) $s_i\mid\;=\;\mid s_i$. (Note: this move may result in an invalid $\ba$-tilted word. In such cases, the move is disallowed.)
    \end{itemize}
\end{theorem}
\begin{proof}
    We write $\mathbf{w}\sim\mathbf{w}'$ if the two $\ba$-tilted reduced words for $w$ are related by such moves. We proceed by induction on $\ell^\word_\ba(w)$. In the base case $\ell^\word_\ba(w)=0$, which corresponds to $\ba=(1,1,\dots,1)$ and $w=\id$, the statement is immediate. For the inductive step, we consider the following cases:
    \begin{enumerate}
        \item Suppose both words end in the same factor. Then the claim follows directly from the induction hypothesis.
        \item Suppose one word ends with $s_i$ and the other with $s_j$, where $|i-j|>1$. Denote the two words as $\mathbf{w} s_i$ and $\mathbf{w}' s_j$. Since $i, j \in \Des_\ba(w)$, by \Cref{cor:length-property} there exists an $\ba$-tilted reduced word of the form $\mathbf{w}'' s_i s_j$. By the induction hypothesis, we have:
        \[\mathbf{w}s_i\sim\mathbf{w}''s_js_i\sim\mathbf{w}''s_is_j\sim\mathbf{w}'s_j.\]
        \item Suppose one word ends with $s_i$ and the otherwith $s_{i+1}$. Denote the two words as $\mathbf{w} s_i$ and $\mathbf{w}' s_{i+1}$. Then $a_i=a_{i+1}=a_{i+2}$ and $w_{i}>_{a_i}w_{i+1}>_{a_i}w_{i+2}$. By \Cref{cor:length-property}, there exists an $\ba$-tilted reduced word of the form $\mathbf{w}'' s_i s_{i+1} s_i$, and the induction hypothesis gives:
        \[\mathbf{w}s_i\sim\mathbf{w}''s_is_{i+1}s_i\sim\mathbf{w}''s_{i+1}s_is_{i+1}\sim\mathbf{w}'s_{i+1}.\]
        \item Suppose one word ends with $s_i$ and the other with a bar. Denote the two words as $\mathbf{w} s_i$ and $\mathbf{w}' \mid\;$. Then $w$ and $w s_i$ are both $\ba$-flattenable by \Cref{lemma:u<flatten(a)v}. By \Cref{cor:length-property}, there exists a $\ba$-tilted reduced word of the form $\mathbf{w}'' \mid s_i$, and the induction hypothesis gives:
        \[\mathbf{w}s_i\sim\mathbf{w}''\mid s_i\sim\mathbf{w}''s_i\mid\;\sim\mathbf{w}'\mid\;.\qedhere\]
    \end{enumerate}
\end{proof}

%% file: tex/4-richardson-def.tex
\section{Definitions of Tilted Richardson Varieties}\label{sec:tilted-richardson-definition}

In this section, we define \emph{tilted Richardson varieties} $\cT_{u,v}$ and \emph{tilted Richardson varieties} $\cT_{u,v}^\circ$ for any pair of permutations $u,v\in S_n$. We provide four equivalent definitions, using:
\begin{enumerate}
\item rank conditions on certain submatrices (\Cref{sec:tilted-def-1});
\item intersections of cyclically rotated Grassmannian Richardson varieties (\Cref{sec:tilted-def-2});
\item vanishing loci of multi-Pl\"ucker coordinates (\Cref{sec:tilted-def-3});
\item intersections of two opposite tilted Schubert cells (\Cref{sec:tilted-def-4}).
\end{enumerate}
As an intermediate step, we first introduce $\cT^\circ_{u,v,\ba}$ and $\cT_{u,v,\ba}^\circ$, defined for permutations $u,v\in S_n$ and an integer sequence $\ba=(a_1,\dots,a_n)\in [n]^n$. We then show that these definitions are independent of the choice of the sequence $\ba$ as long as $u\leq_\ba v$, thereby justifying the omission of $\ba$ in our final definitions.

\subsection{Definition via Rank Conditions}\label{sec:tilted-def-1}

We recall the notion of \emph{cyclic intervals} $[a, b)_c$, defined earlier in \Cref{def:cyclicinterval}.

\begin{defin}\label{def:main}
Let $F_\bullet\in \fl_n$ be a flag represented by a matrix $M_F\in G$. For any subset $S\subseteq [n]$ and integer $k\in [n]$, define $\rank_S(F_k)$ as the rank of the submatrix of $M_F$ consisting of the rows indexed by $S$ and the first $k$ columns.  Given permutations $u,v\in S_n$ and a sequence $\ba\in [n]^n$, we define the \emph{tilted Richardson variety} $\cT_{u,v,\ba}$ as follows:
\begin{align*}
    \cT_{u,v,\ba} := \left\{F_{\bullet} \in \fl_n:
    \begin{array}{c}
         \rank_{[a_k,i)_c}(F_k)\leq \size{u[k]\cap [a_k,i)_c}   \\
         \rank_{[i,a_k)_c}(F_k)\leq \size{v[k]\cap [i,a_k)_c} 
    \end{array}, 
    \forall i,k\in [n]
     \right\},
\end{align*}
and define the \emph{open tilted Richardson variety} $\cT_{u,v,\ba}^\circ$ as follows:
\begin{align*}
    \cT_{u,v,\ba}^\circ := \left\{F_{\bullet} \in \fl_n:
    \begin{array}{c}
         \rank_{[a_k,i)_c}(F_k)= \size{u[k]\cap [a_k,i)_c}   \\
         \rank_{[i,a_k)_c}(F_k)= \size{v[k]\cap [i,a_k)_c} 
    \end{array}, 
    \forall i,k\in [n]
     \right\}.
\end{align*}
\end{defin}

We note that these rank functions are well-defined, as they are independent of the choice of the matrix representative $M_F$. Moreover, the right-hand sides of the inequalities correspond exactly to the rank values computed from the permutation matrices $u$ and $v$.
\begin{ex}\label{ex:4321-3142} 
Let $u = 4231, v = 3142$ and $\ba = (4,2,2,3)$. In \Cref{fig:rank-ex}, the $\star$ and $\bullet$ represent $u$ and $v$, respectively. For each $k\in [4]$, the red horizontal line in column $k$ indicates the cutoff of $[n]$ under the order $\leq_{a_k}$ .

For a flag $F_\bullet \in \cT_{u,v,\ba}$, there are $8$ rank conditions imposed on $F_2$ according to \Cref{def:main}. Each of these conditions corresponds to rank conditions on submatrices of $M_{F}$ in the first two columns with (cyclically) consecutive rows that start or end at the red line. For example, the condition $\rank_{\{1,2,3\}}(F_2)\leq 2 = \size{v[2]\cap\{1,2,3\}}$ corresponds to the shaded submatrix in \Cref{fig:4321-3142} having rank at most $2$, the number of $\bullet$ in that region. Similarly, the condition $\rank_{\{2,3\}}(F_2)\leq 1 = \size{u[2]\cap \{2,3\}}$ corresponds to the shaded submatrix in \Cref{fig:4321-3142.2} having rank at most $1$, the number of $\star$ in that region.
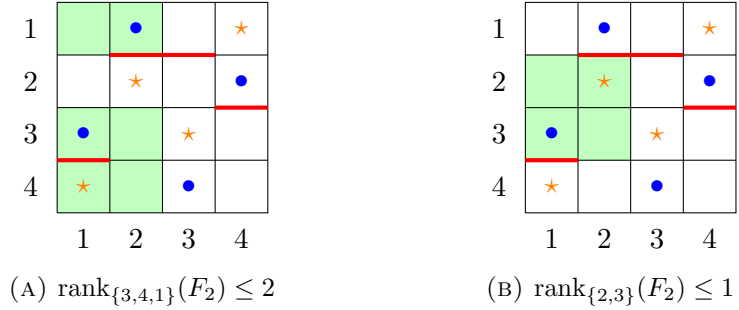
\begin{figure}[ht]
\centering
\subcaptionbox{$\rank_{\{3,4,1\}}(F_2)\leq 2$\label{fig:4321-3142}}[.4\textwidth]{
    \begin{tikzpicture}[scale=0.7]
    \fill [green, opacity  = 0.25] (0,3) rectangle (2,4);
    \fill [green, opacity  = 0.25] (0,0) rectangle (2,2);
    \draw (0,0)--(4,0)--(4,4)--(0,4)--(0,0);
    \draw (1,0) -- (1,4);
    \draw (2,0) -- (2,4);
    \draw (3,0) -- (3,4);
    \draw (0,1) -- (4,1);
    \draw (0,2) -- (4,2);
    \draw (0,3) -- (4,3);
    \node at (-0.5,0.5) {$4$};
    \node at (-0.5,1.5) {$3$};
    \node at (-0.5,2.5) {$2$};
    \node at (-0.5,3.5) {$1$};
    \node at (0.5,-0.5) {$1$};
    \node at (1.5,-0.5) {$2$};
    \node at (2.5,-0.5) {$3$};
    \node at (3.5,-0.5) {$4$};
    \node[orange] at (0.5,0.5) {$\star$};
    \node[orange] at (1.5,2.5) {$\star$};
    \node[orange] at (2.5,1.5) {$\star$};
    \node[orange] at (3.5,3.5) {$\star$};
    \node[blue] at (0.5,1.5) {$\bullet$};
    \node[blue] at (1.5,3.5) {$\bullet$};
    \node[blue] at (2.5,0.5) {$\bullet$};
    \node[blue] at (3.5,2.5) {$\bullet$};
    \draw[line width=0.55mm, red] (0,1) -- (1,1);
    \draw[line width=0.55mm, red] (1,3) -- (2,3);
    \draw[line width=0.55mm, red] (2,3) -- (3,3);
    \draw[line width=0.55mm, red] (3,2) -- (4,2);
    \end{tikzpicture}
    }
\subcaptionbox{$\rank_{\{2,3\}}(F_2)\leq 1$\label{fig:4321-3142.2}}[.4\textwidth]{
    \begin{tikzpicture}[scale=0.7]
    \fill [green, opacity  = 0.25] (0,1) rectangle (2,3);
    \draw (0,0)--(4,0)--(4,4)--(0,4)--(0,0);
    \draw (1,0) -- (1,4);
    \draw (2,0) -- (2,4);
    \draw (3,0) -- (3,4);
    \draw (0,1) -- (4,1);
    \draw (0,2) -- (4,2);
    \draw (0,3) -- (4,3);
    \node at (-0.5,0.5) {$4$};
    \node at (-0.5,1.5) {$3$};
    \node at (-0.5,2.5) {$2$};
    \node at (-0.5,3.5) {$1$};
    \node at (0.5,-0.5) {$1$};
    \node at (1.5,-0.5) {$2$};
    \node at (2.5,-0.5) {$3$};
    \node at (3.5,-0.5) {$4$};
    \node[orange] at (0.5,0.5) {$\star$};
    \node[orange] at (1.5,2.5) {$\star$};
    \node[orange] at (2.5,1.5) {$\star$};
    \node[orange] at (3.5,3.5) {$\star$};
    \node[blue] at (0.5,1.5) {$\bullet$};
    \node[blue] at (1.5,3.5) {$\bullet$};
    \node[blue] at (2.5,0.5) {$\bullet$};
    \node[blue] at (3.5,2.5) {$\bullet$};
    \draw[line width=0.55mm, red] (0,1) -- (1,1);
    \draw[line width=0.55mm, red] (1,3) -- (2,3);
    \draw[line width=0.55mm, red] (2,3) -- (3,3);
    \draw[line width=0.55mm, red] (3,2) -- (4,2);
    \end{tikzpicture}
}
\caption{Rank conditions on $\cT_{u,v,\ba}$ for $u = 4231$, $v = 3142$, and $\ba=(4,2,2,3)$}
\label{fig:rank-ex}    
\end{figure}
\end{ex}

\begin{remark}
    If $u\leq v$ in the Bruhat order, which corresponds to the special case $\ba = (1,\dots, 1)$, then the tilted Richardson varieties $\cT_{u,v,\ba}^\circ$ and $\cT_{u,v,\ba}$ recover the classical Richardson varieties $\cR_{u,v}^\circ$ and $\cR_{u,v}$, respectively.
\end{remark}

The following theorem gives a necessary condition for the nonemptiness of tilted Richardson varieties in terms of the $\ba$-tilted Bruhat order $\leq_\ba$ (see \Cref{def:tilted-bruhat-order}) and also characterizes their $T$-fixed points.

\begin{theorem}\label{thm:tilted-T-fixed-point}
    For $u,v\in S_n$, if $\cT_{u,v,\ba} \neq\emptyset$ or $\cT_{u,v,\ba}^\circ \neq\emptyset$, then $u\leq_\ba v$. Furthermore, if $u\leq_\ba v$, the set of $T$-fixed points in $\cT_{u,v,\ba}$ corresponds exactly to the elements of the tilted Bruhat interval $[u,v]$.
\end{theorem}
\begin{proof}
    We prove the first statement only for $\cT_{u,v}$. The case for $\cT_{u,v}^\circ$ is similar. Assume, for contradiction, that $u\not\leq_\ba v$. By definition, there exist indices $k,i\in [n]$ such that $\size{u[k]\cap [a_k,i)_c} < \size{v[k]\cap [a_k,i)_c}$, or equivalently,
    \[\size{u[k]\cap [a_k,i)_c} + \size{v[k]\cap [i,a_k)_c}<k.\]
    However, if $F_\bullet\in \cT_{u,v}$, then the inequality above implies
    $\rank_{[a_k,i)_c}(F_k) + \rank_{[i,a_k)_c}(F_k)<k$ by \Cref{def:main}. This contradicts with the fact that
    \[\rank_{[a_k,i)_c}(F_k) + \rank_{[i,a_k)_c}(F_k)\geq \dim(F_k)=k.\]
    Thus, we must have $u\leq_\ba v$. The second statement follows directly from \Cref{thm:tilted-criterion}.
\end{proof}

In particular, the set of $T$-fixed points in $\cT_{u,v,\ba}$ is independent of the choice of the sequence $\ba$ as long as $u\leq_\ba v$. We will later show (\Cref{cor:independent-a}) that the varieties $\cT_{u,v,\ba}$ and $\cT_{u,v,\ba}^\circ$ themselves are independent of the choice of the sequence $\ba$.

\subsection{Definition via Cyclically Rotated Grassmannian Richardson Varieties}\label{sec:tilted-def-2}
We present an alternative definition of tilted Richardson varieties using intersections of cyclically rotated Richardson varieties in the Grassmannian (\Cref{thm:altdefT}). Later, we use this alternative definition to show that $\cT_{u,v,\ba}$ and $\cT_{u,v,\ba}$ are independent of the sequence $\ba$ as long as $u\leq_\ba v$ (\Cref{cor:independent-a}). We begin by defining these cyclically rotated varieties.

For a subspace $V\in \gr_{k,n}$, let $M_V$ be a $n\times k$ matrix that represents $V$. Define the \emph{cyclic rotation map} $\cyclic:\gr_{k,n} \rightarrow \gr_{k,n}$ by setting
\[
\cyclic(M_V) :=
\begin{bmatrix}
    \text{---}  & \hspace{-0.2cm}\vec{v}_n & \hspace{-0.2cm}\text{---} \\
    \text{---}  &\hspace{-0.2cm} \vec{v}_1 & \hspace{-0.2cm}\text{---} \\
    & \vdots & \\
    \text{---}  & \hspace{-0.2cm}\vec{v}_{n-1} & \hspace{-0.2cm}\text{---} \\
\end{bmatrix}, 
\quad\text{where}\quad
M_V = 
\begin{bmatrix}
    \text{---}  & \hspace{-0.2cm}\vec{v}_1 & \hspace{-0.2cm}\text{---} \\
    \text{---}  &\hspace{-0.2cm} \vec{v}_2 & \hspace{-0.2cm}\text{---} \\
    & \vdots & \\
    \text{---}  & \hspace{-0.2cm}\vec{v}_n & \hspace{-0.2cm}\text{---} \\
\end{bmatrix}.
\]
For an index set $I = \{i_1,\dots,i_k\}\subseteq [n]$, define its cyclic rotation as $
\cyclic(I):= \{i_1+1,\dots,i_k+1\}$, where the index $n+1$ is identified with $1$.

\begin{defin}
    For subsets $I,J\subseteq [n]$ with $|I|=|J|$ and $r\in[n]$ such that $I\leq_r J$, define the \emph{cyclically rotated (open) Grassmannian Richardson varieties} by
    \[\cR^{\circ}_{I,J,r}:=\cyclic^{r-1}(\cR^{\circ}_{\cyclic^{1-r}(I),\cyclic^{1-r}(J)})\quad\text{and}\quad\cR_{I,J,r}:=\cyclic^{r-1}(\cR_{\cyclic^{1-r}(I),\cyclic^{1-r}(J)}).\]
\end{defin}

Cyclically rotated Grassmannian Richardson varieties are special cases of \emph{positroid varieties} (see \cite[Section~6]{KLSjuggling}).
Like Grassmannian Richardson varieties, they can also be characterized via vanishing conditions on Pl\"ucker coordinates. 

\begin{prop}\label{prop:RIJadef}
Let $I,J\subseteq [n]$, $r\in[n]$ with $|I|=|J|$ and $I\leq_r J$ under the shifted Gale order. Denote by $[I,J]_r:=\{K\subseteq [n]:I\leq_r K\leq _r J\}$ the interval in the shifted Gale order. Then
\[\cR_{I,J,r} = \{V\in \gr_{k,n}: \Delta_K(V) = 0 \text{ for all }K\notin[I,J]_r\}.\]
The corresponding open cell is $\cR^\circ_{I,J,r} = \cR_{I,J,r}\cap\{\Delta_I\Delta_J\neq 0\}$.
\end{prop}
\begin{proof}
    By \Cref{lemma:grSchubPlucker} and the identity $\cyclic_\ast (\Delta_K)=\Delta_{\cyclic(K)}$, $\cR_{I,J,r} = \{V\in \gr_{k,n}:  \Delta_{\cyclic^{r-1}(K)}(V) = 0 \text{ for all }K\notin[\cyclic^{1-r}(I),\cyclic^{1-r}(J)]\}$. The desired statement follows from the fact that $A\leq_r B\iff \cyclic^{1-r}(A)\leq \cyclic^{1-r}(B)$. The open case follows from similar arguments.
\end{proof}

We now present the alternative definition of tilted Richardson varieties as intersections of preimages of cyclically rotated Richardson varieties under natural projection maps.

\begin{theorem}\label{thm:altdefT}
    Let $\pi_k:\fl_n \rightarrow \gr_{k,n}$ be the projection map defined by $\pi_k(F_\bullet)=F_k$. For $u,v\in S_n$ and a sequence $\ba\in[n]^n$, we have
    \[\cT_{u,v,\ba} = \bigcap_{k = 1}^{n-1}\pi_k^{-1}(\cR_{u[k],v[k],a_k})\quad\text{ and }\quad\cT^\circ_{u,v,\ba} = \bigcap_{k = 1}^{n-1}\pi_k^{-1}(\cR^\circ_{u[k],v[k],a_k}).\]
\end{theorem}
\begin{proof}
This follows immediately by comparing the rank conditions in \Cref{def:main} with the rank conditions on Grassmannian Richardson varieties in \Cref{sec:prelim-2-4}, after applying a cyclic rotation.
\end{proof}

Since each $\pi_k^{-1}(\cR_{u[k],v[k],a_k})$ is a closed subvariety of $\fl_n$ and each $\cR^\circ_{u[k],v[k],a_k}\subseteq \cR_{u[k],v[k],a_k}$ is open, we obtain the following immediate corollary.

\begin{cor}
    $\cT_{u,v,\ba}$ is a closed subvariety of $\fl_n$ and $\cT_{u,v,\ba}^\circ \subseteq \cT_{u,v,\ba}$ is an open subvariety. 
\end{cor}

Our next goal is to establish that the tilted Richardson varieties are independent of the choice of the sequence $\ba$ as long as $u\leq_\ba v$. To do so, we first prove the following lemma.

\begin{lemma}\label{lemma:rotateRich}
   Let $I,J \subseteq [n]$ with $|I| = |J|=k$. If there exist distinct indices $r\neq r'\in [n]$ such that $I\leq_r J$ and $I\leq_{r'}J$, then $\cR_{I,J,r} = \cR_{I,J,r'}$ and $\cR^\circ_{I,J,r} = \cR^\circ_{I,J,r'}$.
\end{lemma}
\begin{proof}
    Since $\cR^\circ_{I,J,r}=\cR_{I,J,r}\cap\{\Delta_I\Delta_J\neq 0\}$, it suffices to prove the closed case. By \Cref{prop:RIJadef}, we have
    \begin{align*}
        \cR_{I,J,r} &= \{V\in \gr_{k,n}:  \Delta_K(V) = 0 \text{ for all }K\notin[I,J]_r\},\\
        \cR_{I,J,r'} &= \{V\in \gr_{k,n}:  \Delta_K(V) = 0\text{ for all }K\notin[I,J]_{r'}\}.
    \end{align*}
    Since \Cref{lemma:both-r-comparable} implies $[I,J]_r=[I,J]_{r'}$, we have $\cR_{I,J,r}=\cR_{I,J,r'}$.
\end{proof}

Our main result now follows from \Cref{thm:altdefT} and \Cref{lemma:rotateRich}.

\begin{cor}\label{cor:independent-a}
The varieties $\cT_{u,v,\ba}$ and $\cT_{u,v,\ba}^\circ$ are both independent of $\ba$ as long as $u\leq_{\ba}v$.
\end{cor}

As a consequence of \Cref{cor:independent-a}, we omit the sequence $\ba$ and denote the the (open) tilted Richardson varieties simply by $\cT_{u,v}$  and $\cT_{u,v}^\circ$. In particular, when $u\leq v$ in Bruhat order, these varieties recover the classical Richardson varieties $\cT_{u,v}^\circ = \cR_{u,v}^\circ$ and $\cT_{u,v}=\cR_{u,v}$.

\subsection{Definition via Pl\"ucker Coordinates}\label{sec:tilted-def-3}

In this section, we give yet another characterization of tilted Richardson varieties using the vanishing of certain multi-Plücker coordinates (\Cref{thm:altdefTplucker}). We begin with the following lemma.

\begin{lemma}\label{lemma:plucker-set-to-perm}
    Let $F_\bullet\in \fl_n$ and $I\in\binom{[n]}{k}$, if $\Delta_I(F_\bullet)\neq 0$, then there exists a permutation $w\in S_n$ such that $w[k]=I$ and $\Delta_w(F_\bullet)\neq 0$.
\end{lemma}
\begin{proof}
Without loss of generality, assume $I = [k]$. We prove the contrapositive statement. Suppose $\Delta_w(F_\bullet) = 0$ for every permutation $w$ with $w[k] = [k]$. Equivalently, $\Delta_w(F_\bullet) = 0$ for all $w$ where $w\ngeq s_k$ in the Bruhat order. It follows from \Cref{lemma:SchubPlucker} that $F_\bullet\in \Omega_{s_k}$. Hence $\Delta_{[k]}(F_\bullet) = 0$, which is a contradiction.
\end{proof}

We now establish our third characterization of tilted Richardson varieties.

\begin{theorem}\label{thm:altdefTplucker}
For permutations $u,v\in S_n$, we have
\[\cT_{u,v} = \{F_\bullet\in \fl_n: \Delta_w(F_\bullet) = 0 \text{ for all }w\notin[u,v]\}.\]
The corresponding open cell is given by $\cT^\circ_{u,v} = \cT_{u,v}\cap\{\Delta_u\Delta_v\neq 0\}$.
\end{theorem}

\begin{proof}
    Let $\ba$ be a sequence such that $u\leq_{\ba}v$. Suppose $F_\bullet\in \cT_{u,v}$ and $w\notin[u,v]$. By \Cref{thm:tilted-criterion}, there exists some $k\in[n]$ such that $w[k]\notin[u[k],v[k]]_{a_k}$. Since $F_k\in\cR_{u[k],v[k],a_k}$ by \Cref{thm:altdefT}, it follows from \Cref{prop:RIJadef} that $\Delta_{w[k]}(F_k)=0$. Consequently, we obtain $\Delta_w(F_\bullet)=0$, which proves the inclusion $\subseteq$.

    Conversely, suppose $\Delta_{w}(F_\bullet)=0$ for all $w\notin [u,v]$. If $F_\bullet\notin\cT_{u,v}$, by \Cref{prop:RIJadef}, there exists some $k\in[n]$ and $I\in{\binom{[n]}{k}}$ such that $\Delta_I(F_k)\neq 0$ but $I\notin [u[k],v[k]]_{a_k}$. By \Cref{lemma:plucker-set-to-perm}, we find a permutation $w$ such that $\Delta_w(F_\bullet)\neq0$ and $w[k]=I$, hence $w\notin[u,v]$. This contradicts our assumption.

    Finally, for the open variety $\cT_{u,v}^\circ$, note that $\cR^\circ_{u[k],v[k],a_k}= \cR_{u[k],v[k],a_k}\cap\{\Delta_{u[k]}\Delta_{v[k]}\neq 0\}$. Therefore, by \Cref{thm:altdefT}, $\cT^\circ_{u,v}=\cT_{u,v}\cap\{\Delta_u\Delta_v=\prod_{k=1}^{n-1}\Delta_{u[k]}\Delta_{v[k]}\neq 0\}$ as desired.
\end{proof}

\begin{remark}
    One may wonder whether every tilted Richardson variety is trivially isomorphic to some classical Richardson variety. Specifically, given $u,v\in S_n$, does there always exist $w\in S_n$ such that $\cT_{u,v}=w\cR_{u',v'}$ for some classical Richardson variety $\cR_{u',v'}$?  Computational checks on the level of $T$-fixed points show that this is not always true. For example, when $u=512346$ and $v=246513$, the tilted Bruhat interval $[u,v]$ is not equal to $w\cdot[u',v']$ for any classical Bruhat interval $[u',v']$ and $w\in S_n$. Thus, tilted Richardson varieties $\cT_{u,v}$ are a genuinely new class of varieties.
\end{remark}

\subsection{Definition via Tilted Schubert Cells}\label{sec:tilted-def-4}
In this section, we introduce \emph{tilted Schubert cells} $X_{w,\ba}^\circ$ and \emph{opposite tilted Schubert cells} $\Omega_{w,\ba}^\circ$ for a permutation $w\in S_n$ and a sequence $\ba\in [n]^n$, and show that their intersections $X_{v,\ba}^\circ\cap \Omega_{u,\ba}^\circ$ define the open tilted Richardson varieties $\cT_{u,v}^\circ$ whenever $u\lesssim_\ba v$ (\Cref{thm:tilted-Schub-cell-intersect}). We start with some definitions.

\begin{defin}\label{def:tiltedrothe}
    Given $w\in S_n$ and a sequence $\ba\in [n]^n$, the \emph{tilted Rothe diagram} is defined as the following subset of an $n\times n$ grid:
    \[D_{\ba}(w) := \{(w_i,k)\in [n]^2: i>k,\ w_i<_{a_k}w_k\}.\]
    Similarly, the \emph{opposite tilted Rothe diagram} is defined as:
    \[D_{\ba}^\op(w) = \{(w_i,k)\in [n]^2: i>k,\ w_i>_{a_k}w_k\}.\]
\end{defin}

It is obvious that the numbers of cells in these diagrams are directly related to the $\ba$-tilted length $\ell_\ba(w)$ as defined in \Cref{def:a-tilted-length}. Specifically, $|D_{\ba}(w)|=\ell_\ba(w)$ and $|D_{\ba}^\op(w)|=\binom{n}{2}-\ell_\ba(w)$. We now illustrate \Cref{def:tiltedrothe} with a concrete example below.

\begin{ex}\label{ex:tilteddiagram}
Consider $w = 4321$ and $\ba = (4,4,2,2)$. We draw the diagrams in \Cref{fig:tilted-rothe} as follows. Place a $\bullet$ at position $(w_k,k)$ for each $k\in[4]$. Draw a red horizontal line at each column $k$ at the cutoff for $\leq_{a_k}$. For $D_{\ba}(w)$, draw rays to the right and down from each $\bullet$ until they hit the red line or the right boundary. Similarly, for $D_{\ba}^\op(w)$, draw rays to the right and up. The green boxes remaining represent the tilted Rothe daigrams $D_{\ba}(w)$ and $D_{\ba}^\op(w)$, respectively. 
\begin{figure}[ht]
    \centering
    \subcaptionbox{$D_{\ba}(w) = \{(1,2),(2,2)\}$\label{fig:Du}}[.4\textwidth]{
    \begin{tikzpicture}[scale = 0.7]
    \fill [green, opacity  = 0.25] (1,4) rectangle (2,2);
    \draw (0,0)--(4,0)--(4,4)--(0,4)--(0,0);
    \draw (1,0) -- (1,4);
    \draw (2,0) -- (2,4);
    \draw (3,0) -- (3,4);
    \draw (0,1) -- (4,1);
    \draw (0,2) -- (4,2);
    \draw (0,3) -- (4,3);
    \node at (-0.5,0.5) {$4$};
    \node at (-0.5,1.5) {$3$};
    \node at (-0.5,2.5) {$2$};
    \node at (-0.5,3.5) {$1$};
    \node at (0.5,-0.5) {$1$};
    \node at (1.5,-0.5) {$2$};
    \node at (2.5,-0.5) {$3$};
    \node at (3.5,-0.5) {$4$};
    \node at (0.5,0.5) {$\bullet$};
    \node at (1.5,1.5) {$\bullet$};
    \node at (2.5,2.5) {$\bullet$};
    \node at (3.5,3.5) {$\bullet$};
    \draw[line width = 0.35mm] (0.5,0) -- (0.5,0.5) -- (4,0.5);
    \draw[line width = 0.35mm] (0.5,4) -- (0.5,1);
    \draw[line width = 0.35mm] (1.5,1) -- (1.5,1.5) -- (4,1.5);
    \draw[line width = 0.35mm] (2.5,0) -- (2.5,2.5) -- (4,2.5);
    \draw[line width = 0.35mm] (2.5,4) -- (2.5,3);
    \draw[line width = 0.35mm] (3.5,3) -- (3.5,3.5) -- (4,3.5);
    \draw[line width=0.55mm, red] (0,1) -- (2,1);
    \draw[line width=0.55mm, red] (2,3) -- (4,3);
    \end{tikzpicture}}
    \subcaptionbox{$D_{\ba}^\op(w) = \{(1,1),(2,1),(3,1),(1,3)\}$\label{fig:Dv}}[.42\textwidth]{
    \begin{tikzpicture}[scale = 0.7]
    \fill [green, opacity  = 0.25] (0,4) rectangle (1,1);
    \fill [green, opacity  = 0.25] (2,4) rectangle (3,3);
    \draw (0,0)--(4,0)--(4,4)--(0,4)--(0,0);
    \draw (1,0) -- (1,4);
    \draw (2,0) -- (2,4);
    \draw (3,0) -- (3,4);
    \draw (0,1) -- (4,1);
    \draw (0,2) -- (4,2);
    \draw (0,3) -- (4,3);
    \node at (-0.5,0.5) {$4$};
    \node at (-0.5,1.5) {$3$};
    \node at (-0.5,2.5) {$2$};
    \node at (-0.5,3.5) {$1$};
    \node at (0.5,-0.5) {$1$};
    \node at (1.5,-0.5) {$2$};
    \node at (2.5,-0.5) {$3$};
    \node at (3.5,-0.5) {$4$};
    \node at (0.5,0.5) {$\bullet$};
    \node at (1.5,1.5) {$\bullet$};
    \node at (2.5,2.5) {$\bullet$};
    \node at (3.5,3.5) {$\bullet$};
    \draw[line width = 0.35mm] (0.5,1) -- (0.5,0.5) -- (4,0.5);
    \draw[line width = 0.35mm] (1.5,4) -- (1.5,1.5) -- (4,1.5);
    \draw[line width = 0.35mm] (1.5,0) -- (1.5,1);
    \draw[line width = 0.35mm] (2.5,3) -- (2.5,2.5) -- (4,2.5);
    \draw[line width = 0.35mm] (3.5,4) -- (3.5,3.5) -- (4,3.5);
    \draw[line width = 0.35mm] (3.5,0) -- (3.5,3);
    \draw[line width=0.55mm, red] (0,1) -- (2,1);
    \draw[line width=0.55mm, red] (2,3) -- (4,3);
    \end{tikzpicture}}
    \caption{Tilted Rothe diagrams for $w=4321$ and $\ba=(4,4,2,2)$}
    \label{fig:tilted-rothe}
\end{figure}
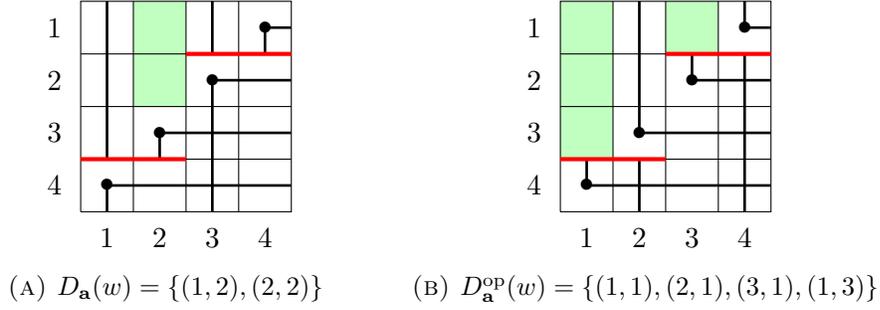
\end{ex}

One may notice similarities between \Cref{fig:Du} and the classical Rothe diagram of a permutation introduced in \cite{rothe}. Indeed, when $\ba = (1,\dots,1)$, the tilted Rothe diagram $D_{\ba}(w)$ recovers the classical Rothe diagram of the permutation $w^{-1}$. 

We now define the tilted analogues of Schubert and opposite Schubert cells.

\begin{defin}\label{def:tilted-Schub-cell}
    Given $w\in S_n$ and a sequence $\ba\in [n]^n$, define the \emph{tilted Schubert cell} as
    \[
        X_{w,\ba}^\circ := \{F_\bullet\in \fl_n: \Delta_w(F_\bullet)\neq 0 \text{ and }  \Delta_{w[k-1]\cup\{i\}}(F_\bullet) = 0\text{ for all }(i,k)\in D_{\ba}^\op(w)\}.
    \]
    Similarly, define the \emph{opposite tilted Schubert cell} as
    \[
        \Omega^\circ_{w,\ba} := \{F_\bullet\in \fl_n: \Delta_w(F_\bullet)\neq 0 \text{ and }  \Delta_{w[k-1]\cup\{i\}}(F_\bullet) = 0\text{ for all }(i,k)\in D_{\ba}(w)\}.
    \]
\end{defin}

The next proposition provides an explicit matrix representative for flags in these tilted Schubert cells, which we call the \emph{canonical representative} or the \emph{canonical matrix}.

\begin{prop}\label{prop:tilted-Schub-cell}
    A flag $F_\bullet\in X^\circ_{w,\ba}$ (resp. $\Omega^\circ_{w,\ba}$) if and only if $F_\bullet$ is represented by a matrix $M_F=(m_{i,k})_{i,k\in[n]}$ such that
    \[
        m_{i,k}=\begin{cases}
            1,&\text{ if }i=w_k,\\
            \ast,&\text{ if }(i,k)\in D_{\ba}(w) \text{ (resp. $D_{\ba}^\op(w)$)},\\
            0,&\text{ otherwise}.
        \end{cases}
    \]
    Here $\ast$ represents an arbitrary number in $\C$. Consequently, we have
    \[
        X_{w,\ba}^\circ\cong \C^{\ell_{\ba}(w)}\quad\text{and}\quad\Omega_{w,\ba}^\circ\cong \C^{\binom{n}{2}-\ell_{\ba}(w)}.
    \]
\end{prop}

\begin{proof}
    We prove the result for the case $X_{w,\ba}^\circ$; the case for $\Omega_{w,\ba}^\circ$ is similar. By performing column reduction, the condition $\Delta_w(F_\bullet)\neq 0$ implies that $F_\bullet$ can be represented by a unique matrix $M_F=(m_{i,k})_{i,k\in[n]}$ with $1$'s at positions $(w_k,k)$ for all $k\in [n]$ and $0$'s to the right of these $1$'s. The conditions on the remaining entries follow from the fact that $m_{i,k}=0\iff \Delta_{w[k-1]\cup\{i\}}=0$ for each $(i,k)\in D_\ba(w)$.
\end{proof}

\begin{ex}\label{ex:tilted-Schub-cell}
    We demonstrate \Cref{prop:tilted-Schub-cell} with $w = 4321$ and $\ba = (4,4,2,2)$ using diagrams in \Cref{fig:tilted-Schub-cell}. We suggest the readers to compare them with the tilted Rothe diagrams from \Cref{fig:tilted-rothe}.
    \begin{figure}[ht]
    \centering
    \subcaptionbox{$X_{w,\ba}^\circ$\label{fig:Xw}}[.4\textwidth]{
    \begin{tikzpicture}[scale = 0.7]
    \draw (0,0)--(4,0)--(4,4)--(0,4)--(0,0);
    \draw (1,0) -- (1,4);
    \draw (2,0) -- (2,4);
    \draw (3,0) -- (3,4);
    \draw (0,1) -- (4,1);
    \draw (0,2) -- (4,2);
    \draw (0,3) -- (4,3);
    \node at (-0.5,0.5) {$4$};
    \node at (-0.5,1.5) {$3$};
    \node at (-0.5,2.5) {$2$};
    \node at (-0.5,3.5) {$1$};
    \node at (0.5,-0.5) {$1$};
    \node at (1.5,-0.5) {$2$};
    \node at (2.5,-0.5) {$3$};
    \node at (3.5,-0.5) {$4$};
    \node at (0.5,0.5) {$1$};
    \node at (1.5,1.5) {$1$};
    \node at (2.5,2.5) {$1$};
    \node at (3.5,3.5) {$1$};
    \node at (1.5,0.5) {$0$};
    \node at (2.5,0.5) {$0$};
    \node at (3.5,0.5) {$0$};
    \node at (2.5,1.5) {$0$};
    \node at (3.5,1.5) {$0$};
    \node at (3.5,2.5) {$0$};
    \node at (0.5,1.5) {$0$};
    \node at (0.5,2.5) {$0$};
    \node at (0.5,3.5) {$0$};
    \node at (1.5,2.5) {$\ast$};
    \node at (1.5,3.5) {$\ast$};
    \node at (2.5,3.5) {$0$};
    \draw[line width=0.55mm, red] (0,1) -- (2,1);
    \draw[line width=0.55mm, red] (2,3) -- (4,3);
    \end{tikzpicture}}
    \subcaptionbox{$\Omega_{w,\ba}^\circ$\label{fig:Omegaw}}[.4\textwidth]{
    \begin{tikzpicture}[scale = 0.7]
    \draw (0,0)--(4,0)--(4,4)--(0,4)--(0,0);
    \draw (1,0) -- (1,4);
    \draw (2,0) -- (2,4);
    \draw (3,0) -- (3,4);
    \draw (0,1) -- (4,1);
    \draw (0,2) -- (4,2);
    \draw (0,3) -- (4,3);
    \node at (-0.5,0.5) {$4$};
    \node at (-0.5,1.5) {$3$};
    \node at (-0.5,2.5) {$2$};
    \node at (-0.5,3.5) {$1$};
    \node at (0.5,-0.5) {$1$};
    \node at (1.5,-0.5) {$2$};
    \node at (2.5,-0.5) {$3$};
    \node at (3.5,-0.5) {$4$};
    \node at (0.5,0.5) {$1$};
    \node at (1.5,1.5) {$1$};
    \node at (2.5,2.5) {$1$};
    \node at (3.5,3.5) {$1$};
    \node at (1.5,0.5) {$0$};
    \node at (2.5,0.5) {$0$};
    \node at (3.5,0.5) {$0$};
    \node at (2.5,1.5) {$0$};
    \node at (3.5,1.5) {$0$};
    \node at (3.5,2.5) {$0$};
    \node at (1.5,2.5) {$0$};
    \node at (1.5,3.5) {$0$};
    \node at (0.5,2.5) {$\ast$};
    \node at (0.5,3.5) {$\ast$};
    \node at (0.5,1.5) {$\ast$};
    \node at (2.5,3.5) {$\ast$};
    \draw[line width=0.55mm, red] (0,1) -- (2,1);
    \draw[line width=0.55mm, red] (2,3) -- (4,3);
    \end{tikzpicture}}
    \caption{Canonical representatives of tilted Schubert cells for $w=4321$ and $\ba=(4,4,2,2)$}
    \label{fig:tilted-Schub-cell}   
    \end{figure}
\end{ex}

\begin{remark}
    When $\ba = (1,1,\dots,1)$, we recover the classical Schubert cells $X_{w,\ba}^\circ = X_{w}^\circ$ and $\Omega_{w,\ba}^\circ = \Omega_w^\circ$. In this case, the canonical matrices in \Cref{prop:tilted-Schub-cell} coincide with the canonical representatives of Schubert cells in \Cref{sec:prelim-2-3}.
\end{remark}

As a consequence of \Cref{prop:tilted-Schub-cell}, we obtain a \emph{tilted Bruhat decomposition} of the flag variety into tilted Schubert cells, in analogy with the classical Bruhat decomposition:
\begin{lemma}\label{lemma:tilted-bruhat-decomp}
    Let $\ba\in[n]^n$. The tilted Schubert cells define the following decomposition of the flag variety:
    \[\fl_n=\bigsqcup_{w\in S_n}X^\circ_{w,\ba }=\bigsqcup_{w\in S_n}\Omega^\circ_{w,\ba }.\]
\end{lemma}
\begin{proof}
    For any point $F_\bullet \in \fl_n$, there is a unique way to reduce $F_\bullet$ to its canonical representative as described in \Cref{prop:tilted-Schub-cell}, using column operations. From this representative, one can uniquely determine a permutation $w$ such that $F_\bullet \in X^\circ_{w, \ba}$. The argument for the decomposition using $\Omega^\circ_{w, \ba}$ is similar.
\end{proof}

Finally, we state the main result of this section, which establishes the relationship between tilted Richardson varieties and the intersections of tilted Schubert and opposite tilted Schubert cells.

\begin{theorem}\label{thm:tilted-Schub-cell-intersect}
    For permutations $u,v\in S_n$ and a sequence $\ba\in [n]^n$,
\begin{enumerate}
    \item if $u\leq_\ba v$, then $\cT^\circ_{u,v}\subseteq X_{v,\ba}^\circ\cap \Omega_{u,\ba}^\circ$;
    \item if $u\lesssim_\ba v$, then $\cT_{u,v}^\circ= X_{v,\ba}^\circ\cap \Omega_{u,\ba}^\circ$.
\end{enumerate}
\end{theorem}

\begin{proof}
    In fact, we prove the following stronger version of the theorem: for any $u,v\in S_n$ and $\ba\in [n]^n$, we have (1) $\cT^\circ_{u,v,\ba}\subseteq X_{v,\ba}^\circ\cap \Omega_{u,\ba}^\circ$; (2) if $u\sim_\ba v$, then $\cT_{u,v,\ba}^\circ= X_{v,\ba}^\circ\cap \Omega_{u,\ba}^\circ$. The original theorem then follows by restricting to the case $u\leq_\ba v$.

    To prove (1), we focus on proving $\cT^\circ_{u,v,\ba}\subseteq \Omega_{u,\ba}^\circ$, as the proof for $\cT^\circ_{u,v,\ba}\subseteq X_{v,\ba}^\circ$ is similar. Our goal is to show that any $F_\bullet\in \cT^\circ_{u,v,\ba}$ satisfies the Pl\"ucker conditions in \Cref{def:tilted-Schub-cell}. By \Cref{thm:altdefTplucker},  we already have $\Delta_u(F_\bullet)\neq 0$. It remains to verify that $\Delta_{u[k-1]\cup\{i\}}(F_\bullet) = 0$ for all $(i,k)\in D_{\ba}(u)$. This follows from the fact that $i<_{a_k}u_k$ implies $u[k-1]\cup\{i\}<_{a_k}u[k]$. Since $F_\bullet\in\pi_k^{-1}(\cR^\circ_{u[k],v[k],a_k})$, by \Cref{prop:RIJadef} we have $\Delta_{u[k-1]\cup\{i\}}(F_\bullet)=0$, as required.
    
    To prove (2), it suffices to show the $\supseteq$ direction. Equivalently, we must verify that
    \[X^\circ_{v,\ba}\cap \Omega^\circ_{u,\ba}\subseteq \pi_k^{-1}(\cR^\circ_{u[k],v[k],a_k})\text{ for all }k\in [n-1].\]
    We proceed by induction on $k$. For the base case $k = 1$, we must show that every $F_\bullet\in X^\circ_{v,\ba}\cap \Omega^\circ_{u,\ba}$ satisfies $F_\bullet\in \pi_1^{-1}(\cR^\circ_{u_1,v_1,a_1})$. Since $F_\bullet\in\Omega^\circ_{u,\ba}$ implies $\Delta_{u_1}(F_\bullet)\neq 0$, and $F_\bullet\in X^\circ_{v,\ba}$ implies $\Delta_{v_1}(F_\bullet)\neq 0$, it remains to verify that $\Delta_i(F_\bullet) = 0$ for all $i\in [a_1,u_1)_c\cup(v_1,a_1-1]_c$. The case $i\in [a_1,u_1)_c$ follows directly from the Pl\"ucker conditions on $\Omega^\circ_{u,\ba}$, while the case $i\in (v_1,a_1-1]_c$ follows from those on $X^\circ_{v,\ba}$.
    
    For the induction step, we need to show the following implication:
    \[
         F_\bullet\in\pi_{k-1}^{-1}(\cR^\circ_{u[k-1],v[k-1],a_{k-1}})\cap X^\circ_{v,\ba}\cap \Omega^\circ_{u,\ba}\implies F_\bullet\in\pi_k^{-1}(\cR^\circ_{u[k],v[k],a_k}).
    \]
    Since $u\sim_\ba v$, it follows from \Cref{lemma:rotateRich} that $\cR^\circ_{u[k-1],v[k-1],a_{k-1}} = \cR^\circ_{u[k-1],v[k-1],a_{k}}$. Our goal is to show that $F_\bullet$ satisfies the Pl\"ucker conditions given in \Cref{prop:RIJadef}: that is, if $I\not\geq_{a_k}u[k]$ or $I\not\leq_{a_k}v[k]$, then $\Delta_I(F_\bullet)=0$. We prove the case $I\not\geq_{a_k}u[k]$, as the case $I\nleq_{a_k} v[k]$ follows by a similar argument. 
    
    To show that $\Delta_I(F_\bullet)=0$, we consider the incidence Pl\"ucker relation \eqref{eqn:incidenceplucker1} associated with $u[k-1]$ and $I$:
    \[
        \Delta_{u[k-1]}\Delta_I(F_\bullet)=\sum_{i\in I\setminus u[k-1]}\Delta_{I-i}\Delta_{u[k-1]+i}(F_\bullet).
    \]
    Since $\Delta_{u[k-1]}(F_\bullet)\neq 0$, it suffices to show that each term on the right-hand side vanishes. There are two cases for $i\in I\setminus u[k-1]$:
    \begin{enumerate}
        \item If $i\geq_{a_k}u_k$: then the assumption $I\not\geq_{a_k}u[k]$ implies $I\setminus \{i\}\not\geq_{a_k}u[k-1]$. Therefore, $\Delta_{I-i}(F_\bullet)=0$, since $F_\bullet\in\pi_{k-1}^{-1}(\cR^\circ_{u[k-1],v[k-1],a_{k}})$;
        \item If $i<_{a_k}u_k$: then $(i,k)\in D_\ba(u)$. Therefore, $\Delta_{u[k-1]+i}(F_\bullet)=0$, since $F_\bullet\in \Omega^\circ_{u,\ba}$.
    \end{enumerate}
    In either case, each summand on the right-hand side vanishes, and we conclude that $\Delta_I(F_\bullet) = 0$, as desired.
\end{proof}

\begin{remark}
    In the classical setting, Schubert cells $X_w^\circ$ and $\Omega_w^\circ$ along with their closures, Schubert varieties $X_w$ and $\Omega_w$, have many nice geometric properties. Motivated by this, we define \emph{tilted Schubert varieties} $\overline{X_{v}^\circ}$ and $\overline{\Omega_u^\circ}$ as the closures of the corresponding tilted Schubert cells in the flag variety. However, many aspects of these tilted Schubert cells and varieties remain mysterious, as many properties in the classical setting fail or remain unknown in the tilted setting. We outline several key differences:
    \begin{itemize}
        \item In the classical case, $\cR_{u,v}=X_v\cap \Omega_u$. However, in the tilted case, we do not always have $\cT_{u,v}=\overline{X^\circ_{v,\ba}}\cap\overline{\Omega^\circ_{u,\ba}}$. For example, when $u=14325$, $v=15342$, $\ba=(3,3,1,1,1)$, there exists $w=15243$, with $e_w\in \overline{X^\circ_{v,\ba}}\cap\overline{\Omega^\circ_{u,\ba}}$ but $e_w\notin \cT_{u,v}$.
        \item In the classical setting, every Schubert variety $X_w$ or $\Omega_w$ can be realized as a Richardson variety. However, there exist tilted Schubert varieties that do not coincide with any tilted Richardson varieties. For example, when $w=2314$ and $\ba=(3,3,1,1)$, $\overline{X^\circ_{w,\ba}}$ does not equal any tilted Richardson variety. This is evident by examining the set of $T$-fixed points in $\overline{X^\circ_{w,\ba}}$, which does not equal any tilted Bruhat interval.
        \item The inclusion relations of classical Schubert varieties respect the Bruhat order, satisfying $X_u\subseteq X_v$ if and only if $u\leq v$. However, it is not true in general that $u\lesssim_\ba v$ implies $\overline{X^\circ_{u,\ba}}\subseteq \overline{X^\circ_{v,\ba}}$. For example, when $u=13245$, $v=23145$, and $\ba=(3,3,1,1,1)$, we find that for $w=35241$, $e_w\in \overline{X^\circ_{u,\ba}}$ but $e_w\notin \overline{X^\circ_{v,\ba}}$.
        \item Given a sequence $\ba$, the tilted Bruhat decomposition $\fl_n=\bigsqcup_{w\in S_n}X^\circ_{w,\ba}$ is not necessarily  a stratification. For example, when $w=3124$ and $\ba=(3,3,1,1)$, the tilted Schubert variety $\overline{X_{w,\ba}^\circ}$ is not a union of tilted Schubert cells. Specifically, we have $X_{w,\ba}^\circ\cong \C$ and $\overline{X_{w,\ba}^\circ}\cong \P^1=X_{w,\ba}^\circ\cup\{e_v\}$ for $v=3421$, but $\{e_v\}$ is not a tilted Schubert cell for $\ba$ since $\{e_v\}\subsetneq X_{v,\ba}^\circ$.
    \end{itemize}
\end{remark}
\begin{conj}\label{conj:intersect_transversely}
If $u\lesssim_\ba v$, then $X^\circ_{v,\ba}$ and $\Omega^\circ_{u,\ba}$ intersect transversely.
\end{conj}
\Cref{conj:intersect_transversely} would imply that the tilted Richardson variety $\cT^\circ_{u,v}$ is smooth. 
\begin{prob}
Describe all $T$-fixed points in the tilted Schubert variety $\overline{X_{w,\ba}^\circ}$ explicitly. 
\end{prob}
We know $u\leq_\ba v$ implies $e_u\in \overline{X^\circ_{v,\ba}}$, but it is not enough to determine all $T$-fixed points.

%% file: tex/5-geometric.tex
\section{Geometric Properties of Tilted Richardson Varieties}\label{sec:tilted-richardson-geometric}

In this section, we establish some fundamental geometric properties of the tilted Richardson varieties, analogous to those of the Richardson varieties. In particular, we show that:

\begin{enumerate}
\item There is a stratification $\cT_{u,v}=\bigsqcup_{[x,y]\subseteq[u,v]}\cT_{x,y}^{\circ}$ indexed by subintervals of the tilted Bruhat interval $[u,v]$ (\Cref{sec:geometry-1});
\item The dimensions of $\cT_{u,v}$ and $\cT_{u,v}^{\circ}$ are given by the length of the shortest path from $u$ to $v$ in the quantum Bruhat graph (\Cref{sec:geometry-2});
\item The closure relation $\overline{\cT_{u,v}^{\circ}}=\cT_{u,v}$ holds (\Cref{sec:geometry-3}).
\end{enumerate}

Our proofs differ substantially from Richardson’s original proof of these results for classical Richardson varieties in \cite{richardson}, which heavily relies on the fact that the Richardson variety $\cR_{u,v}$ arises as the transverse intersection of two opposite Schubert varieties $X_v\cap\Omega_u$, or equivalently, as the intersection of two opposite Borel orbits $BvB\cap B_-uB$ in the flag variety. In contrast, tilted Richardson varieties do not inherit this orbit intersection structure, requiring us to develop new methods to establish their geometric properties.

\subsection{A Stratification of Tilted Richardson Varieties}\label{sec:geometry-1}

\begin{theorem}\label{thm:Tunion}
    There is a stratification $\cT_{u,v} = \bigsqcup_{[x,y]\subseteq [u,v]}\cT_{x,y}^\circ$ indexed by subintervals $[x,y]$ of the tilted Bruhat interval $[u,v]$.
\end{theorem}
\begin{proof}
    It follows directly from \Cref{def:main} that the strata $\cT^\circ_{x,y}$ on the right-hand side are disjoint and each stratum $\cT_{x,y}^\circ$ is contained in $\cT_{u,v}$ for all subintervals $[x,y]\subseteq [u,v]$. It remains to show that for any $F_\bullet \in \cT_{u,v}$, there exists $[x,y]\subseteq [u,v]$ such that $F_\bullet\in \cT_{x,y}^\circ$. 
    
     We fix a sequence $\ba\in [n]^n$ such that $u\lesssim_{\ba} v$, whose existence is guaranteed by \Cref{thm:lesssim-exist}. For each $k\in [n]$, since $F_k\in \cR_{u[k],v[k],a_k}$, and since $\cR_{u[k],v[k],a_k}$ admits a stratification $\cR_{u[k],v[k],a_k}=\bigsqcup_{I,J}\cR^\circ_{I,J,a_k}$, the subspace $F_k$ must lie in one of these strata, say $\cR^\circ_{I_k,J_k,a_k}$ for some subsets $I_k,J_k\in \binom{[n]}{k}$ satisfying $u[k]\leq_{a_k}I_k\leq_{a_k}J_k\leq_{a_k}v[k]$.

    We now show that $I_{k-1}\subseteq I_k$ for every $k\in [n]$. Consider the incidence Pl\"ucker relation \eqref{eqn:incidenceplucker1} associated with $I_{k-1}$ and $I_k$:
    \[
        \Delta_{I_{k-1}}\Delta_{I_k}(F_\bullet)=\sum_{i\in I_k\setminus I_{k-1}}\Delta_{I_k-i}\Delta_{I_{k-1}+i}(F_\bullet).
    \]
    Since the left-hand side is nonzero, at least one term $\Delta_{I_k-i}\Delta_{I_{k-1}+i}(F_\bullet)$ on the right-hand side must also be nonzero for some $i\in I_k\setminus I_{k-1}$. Because $F_k\in \cR^\circ_{I_k,J_k,a_k}$, \Cref{prop:RIJadef} and $\Delta_{I_{k-1}+i}(F_\bullet)\neq 0$ implies that $I_{k-1}\cup\{i\}\geq_{a_k} I_k$. On the other hand, using the condition $u\lesssim_\ba v$ and \Cref{lemma:rotateRich}, we have $F_{k-1}\in \cR^\circ_{I_k,J_k,a_{k-1}}=\cR^\circ_{I_{k-1},J_{k-1},a_{k}}$, implying $I_k\setminus\{i\}\geq_{a_k} I_{k-1}$, or equivalently $I_k\geq_{a_k} I_{k-1}\cup\{i\}$. These inequalities can only hold if $I_k=I_{k-1}\cup\{i\}$, confirming $I_{k-1}\subseteq I_k$. Similarly, we obtain $J_{k-1}\subseteq J_k$ for every $k\in [n]$.

     Finally, we define permutations $x,y\in S_n$ by setting $x_k = I_k\setminus I_{k-1}$ and $y_k = J_k\setminus J_{k-1}$ for all $k\in [n]$. It follows from \Cref{thm:altdefT} that 
    $F_\bullet \in \cT_{x,y}^\circ$. Since $u[k]\leq_{a_k}I_k\leq_{a_k}J_k\leq_{a_k}v[k]$ for all $k\in[n]$, we conclude that $u\leq_\ba x\leq_\ba y\leq_\ba v$, and thus $F_\bullet\in \cT_{u,v}$ lies in $\cT^\circ_{x,y}$ for some subinterval $[x,y]\subseteq [u,v]$.
\end{proof}

\subsection{Dimensions of Tilted Richardson Varieties}\label{sec:geometry-2}

In this section, we show that the dimension of $\cT_{u,v}^\circ$ is given by $\ell(u,v)$, the length of the shortest path from $u$ to $v$ in the quantum Bruhat graph (\Cref{thm:dimension}). The key observation is that since $\cT_{u,v}^\circ$ is the intersection $X_{v,\ba}^\circ\cap \Omega_{u,\ba}^\circ$ by \Cref{thm:tilted-Schub-cell-intersect}, its codimension is at most the sum of the codimensions of the tilted Schubert cells $X_{v,\ba}^\circ$ and $\Omega_{u,\ba}^\circ$.

For any permutation $x\in S_n$, the \emph{affine chart} $x\Omega_{\id}^\circ$ is the permuted opposite Schubert cell, consisting of flags satisfying $\Delta_x \neq 0$. These flags can also be represented by matrices with $1$'s at positions $(x_k,k)$ for all $k\in [n]$ and $0$'s to the right of $1$'s.

\begin{ex}\label{ex:chart}
    Let $x = 3142$. The affine chart $x\Omega_{\id}^\circ$ can be identified as the following affine space, where each $*$ represents an arbitrary number in $\C$:
    \begin{center}
    \begin{tikzpicture}[scale = 0.7]
    \draw (0,0)--(4,0)--(4,4)--(0,4)--(0,0);
    \draw (1,0) -- (1,4);
    \draw (2,0) -- (2,4);
    \draw (3,0) -- (3,4);
    \draw (0,1) -- (4,1);
    \draw (0,2) -- (4,2);
    \draw (0,3) -- (4,3);
    \node at (-0.5,0.5) {$4$};
    \node at (-0.5,1.5) {$3$};
    \node at (-0.5,2.5) {$2$};
    \node at (-0.5,3.5) {$1$};
    \node at (0.5,-0.5) {$1$};
    \node at (1.5,-0.5) {$2$};
    \node at (2.5,-0.5) {$3$};
    \node at (3.5,-0.5) {$4$};
    \node at (0.5,1.5) {$1$};
    \node at (1.5,3.5) {$1$};
    \node at (2.5,0.5) {$1$};
    \node at (3.5,2.5) {$1$};
    \node at (1.5,1.5) {$0$};
    \node at (2.5,3.5) {$0$};
    \node at (2.5,1.5) {$0$};
    \node at (3.5,0.5) {$0$};
    \node at (3.5,3.5) {$0$};
    \node at (3.5,1.5) {$0$};
    \node at (2.5,2.5) {$*$};
    \node at (1.5,0.5) {$*$};
    \node at (1.5,2.5) {$*$};
    \node at (0.5,0.5) {$*$};
    \node at (0.5,2.5) {$*$};
    \node at (0.5,3.5) {$*$};
    \end{tikzpicture}
    \end{center}
\end{ex}

The following lemma describes the relationship between tilted Richardson varieties, open tilted Richardson varieties, and affine charts. In particular, it shows that the open tilted Richardson varieties $\cT_{u,v}^\circ$ are quasi-affine varieties.

\begin{lemma}\label{lemma:uOmega}
    For any permutations $u,v,x\in S_n$,
    \begin{enumerate}
        \item if $x\notin [u,v]$, then $\cT_{u,v} \cap x\Omega_{\id}^\circ = \cT_{u,v}^\circ \cap x\Omega_{\id}^\circ = \emptyset$,
        \item $\cT_{u,v}^\circ = \cT_{u,v} \cap u\Omega_{\id}^\circ \cap v\Omega_{\id}^\circ$.
    \end{enumerate}
\end{lemma}
\begin{proof}
    This result follows directly from \Cref{thm:altdefTplucker}. 
\end{proof}

We are now ready to state the dimension formula.

\begin{theorem}\label{thm:dimension}
    The open tilted Richardson variety $\cT^\circ_{u,v}$ is equidimensional of dimension $\ell(u,v)$.
\end{theorem}
\begin{proof}
    We proceed by induction on $\ell(u,v)$. In the base case $\ell(u,v)=0$, or $u=v$, the statement is immediate since $\cT_{u,u}^\circ=\{e_u\}$ is a single point.
    
    Fix a sequence $\ba\in [n]^n$ such that $u\lesssim_{\ba} v$, whose existence is guaranteed by \Cref{thm:lesssim-exist}. Since $\cT_{u,v}^\circ=X_{v,\ba}^\circ\cap \Omega_{u,\ba}^\circ$ by \Cref{thm:tilted-Schub-cell-intersect}, 
    $\codim(\cT^\circ_{u,v})$ is at most the sum of $\codim(X_{v,\ba}^\circ)$ and $\codim(\Omega_{u,\ba}^\circ)$. These values are given by $\binom{n}{2}-\ell_\ba(v)$ and $\ell_\ba(u)$, respectively. Thus, 
    \[\codim(\cT_{u,v}^\circ) \leq \codim(X^\circ_{v,\ba})+ \codim(\Omega^\circ_{u,\ba}) = \binom{n}{2}-\ell_\ba(v)+\ell_\ba(u).\]
    By \Cref{cor:length-rank-function}, $\ell_\ba(v)-\ell_\ba(u)=\ell(u,v)$, and thus every irreducible component of $\cT^\circ_{u,v}$ has dimension at least $\ell(u,v)$.

    It remains to show $\dim(\cT^\circ_{u,v})\leq\ell(u,v)$. Suppose, for contradiction, that there exists an irreducible component $Z\subseteq \cT^\circ_{u,v}$ with $\dim(Z)\geq \ell(u,v)+1$. Let $\overline{Z}$ be the closure of $Z$ in $\fl_n$ and define its boundary as $\partial Z:=\overline{Z}\setminus Z$. By \Cref{lemma:uOmega}, $Z=\overline{Z}\cap u\Omega_{\id}^\circ \cap v\Omega_{\id}^\circ$. Therefore,
    \[
        \partial Z=\overline{Z}\cap (\fl_n\setminus(u\Omega_{\id}^\circ \cap v\Omega_{\id}^\circ))=\overline{Z}\cap \{\Delta_u\Delta_v=0\}.
    \]
    We consider two cases:
    \begin{enumerate}
        \item If $\partial Z\neq \emptyset$, then since $\partial Z$ is the intersection of $\overline{Z}$ with a union of hypersurfaces $\{\Delta_u\Delta_v=0\}$, we have $\dim(\partial Z)\geq\dim(\overline{Z})-1\geq\ell(u,v)$. However, from \Cref{thm:Tunion}, $\partial Z\subseteq\cT_{u,v}\setminus\cT^\circ_{u,v}=\bigsqcup_{[x,y]\subsetneq[u,v]}\cT^\circ_{x,y}$. By the induction hypothesis, each $\cT^\circ_{x,y}$ has dimension at most $\ell(x,y) \leq \ell(u,v) - 1$, hence $\dim(\partial Z) \leq \ell(u,v) - 1$, a contradiction.
        \item If $\partial Z= \emptyset$, then $Z=\overline{Z}$ is a projective variety. However, $Z\subseteq u\Omega_{\id}^\circ \cap v\Omega_{\id}^\circ$ is quasi-affine. This is only possible if $Z$ is a point, contradicting $\dim(Z)\geq \ell(u,v)+1\geq 1$.
    \end{enumerate}
    Since both cases lead to a contradiction, we conclude that $\dim(Z)=\ell(u,v)$, and thus every irreducible component of $\cT_{u,v}^\circ$ has dimension $\ell(u,v)$.
\end{proof}

\subsection{Closure Relations of Tilted Richardson Varieties}\label{sec:geometry-3}

In this section, we prove the closure relation $\overline{\cT_{u,v}^\circ}=\cT_{u,v}$ (\Cref{thm:closure}). While one direction $\overline{\cT_{u,v}^\circ}\subseteq\cT_{u,v}$ is immediate from the definition, the other direction requires a more delicate analysis of the geometry of $\cT_{u,v}$ on different affine charts.

We now outline the strategy of the proof. The argument proceeds by induction on $\ell(u,v)$. By \Cref{thm:Tunion}, it is enough to show that $\cT_{u,x}^{\circ}\subseteq\overline{\cT_{u,v}^\circ}$ for all $x\in[u,v]$ such that $\ell(u,x)=\ell(u,v)-1$, and dually, $\cT_{y,v}^{\circ}\subseteq\overline{\cT_{u,v}^\circ}$ for all $y\in[u,v]$ such that $\ell(y,v)=\ell(u,v)-1$. We focus on the first case.

To prove this, we introduce a new subvariety of $\fl_n$, called the \emph{skew tilted Schubert cell} $X_{v,x,\ba}^\circ$ (\Cref{def:skew-tilted-Schub-cell}). It is closely related to the tilted Schubert cell from \Cref{def:tilted-Schub-cell}, and is isomorphic to an affine space $\C^{\ell_\ba(v)}$. A key step in the argument is to show that $\cT_{u,v}\cap x\Omega_\id^\circ\cap u\Omega_\id^\circ = X_{v,x,\ba}^\circ\cap \Omega_{u,\ba}^\circ$ (\Cref{prop:uxvanishing}). This identity allows us to lower-bound the dimension of $\cT_{u,v}\cap x\Omega_\id^\circ\cap u\Omega_\id^\circ$, which ultimately leads the desired closure relation via a straightforward geometric argument.

We begin with a lemma that focuses on the intersection $\cT_{u,v}\cap x\Omega_\id^\circ\cap u\Omega_\id^\circ$, which effectively ``isolates" the two relevant strata $\cT_{u,v}^\circ$ and $\cT_{u,x}^\circ$.

\begin{lemma}\label{lemma:uxOmega}
    For any $u,v,x\in S_n$ such that $x\in[u,v]$ and $\ell(u,x)=\ell(u,v)-1$, we have
    \[\cT_{u,v}\cap x\Omega_{\id}^\circ \cap u\Omega_{\id}^\circ = \cT_{u,x}^\circ \sqcup (\cT_{u,v}^\circ \cap x\Omega_{\id}^\circ).\]
\end{lemma}

\begin{proof}
    By \Cref{thm:Tunion}, there is a decomposition $\cT_{u,v} = \bigsqcup_{[x,y]\subseteq [u,v]}\cT_{x,y}^\circ$. According to \Cref{lemma:uOmega}, the only strata that intersect both $u\Omega_{\id}^\circ$ and $x\Omega_{\id}^\circ$ are $\cT^\circ_{u,x}$ and $\cT^\circ_{u,v}$, which implies the claim.
\end{proof}

We now define the skew tilted Schubert cells.

\begin{defin}\label{def:skew-tilted-Schub-cell}
    Fix $x,v\in S_n$ and a sequence $\ba\in [n]^n$ such that $x\lesssimdot_\ba v$ and $x=vt_{pq}$ for $p<q$. Define the following subsets of an $n\times n$ grid (the second is the tilted Rothe diagram from \Cref{def:tiltedrothe}):
    \begin{align*}
        D^\op_\ba(x,v) &:= \{(x_i,k)\in [n]^2: i>k,\ x_i<_{a_k}v_k\},\\
        D^\op_\ba(x) &:= \{(x_i,k)\in [n]^2: i>k,\ x_i<_{a_k}x_k\}.
    \end{align*}
    
    The \emph{skew tilted Schubert cell} $X_{v,x,\ba}^\circ$ is defined as the set of flags $F_{\bullet} \in \fl_n$ satisfying the following Pl\"ucker conditions:
    \begin{enumerate}
        \item $\Delta_x(F_\bullet)\neq 0$;
        \item $\Delta_{x[k-1]\cup\{i\}}(F_\bullet)=0$ for all $(i,k)\in D^\op_\ba(x,v)\cap D^\op_\ba(x)$;
        \item For all $(i,k)\in D^\op_\ba(x,v)\setminus D^\op_\ba(x)$ (necessarily with $k=q$), we require:
        \[\Delta_{x[p-1]+i}\Delta_{x[q]}(F_\bullet)=\Delta_{x[q-1]+i}\Delta_{x[p-1]+x_q}(F_\bullet).\]
    \end{enumerate}
\end{defin}

We claim that $|D^\op_{\ba}(x,v)| = \binom{n}{2}-\ell_{\ba}(v)$ and that the skew tilted Schubert cell satisfies $X_{v,x,\ba}^\circ \cong \C^{\ell_{\ba}(v)}$. Instead of giving a formal proof for these claims, they will become clear once we present the following example for \Cref{def:skew-tilted-Schub-cell}.

\begin{ex}
    Consider $v=465123$, $x=265143$ and $\ba = (2,2,2,6,6,6)$. We draw the diagrams $D^\op_\ba(x)$ and $D^\op_\ba(x,v)$ in \Cref{fig:tilted-rothe-skew}. The diagram $D^\op_\ba(x)$ is defined as in \Cref{def:tiltedrothe}. To construct $D^\op_\ba(x,v)$, place a $\bullet$ at position $(x_k,k)$ for each $k\in[6]$, and place an $\star$ at position $(v_k,k)$ for the two indices $k=1,5$ where $v$ differs from $x$. Draw a red horizontal line at each column $k$ at the cutoff for $\leq_{a_k}$. Draw rays to the right from each $\bullet$ until they hit the the right boundary, and draw rays upward from the $\bullet$, or from the $\star$ instead if present in that column, until they hit the red line.
    
    The remaining boxes, which represent $D^\op_{\ba}(x,v)$ and $D^\op_{\ba}(x)$, are colored either green or blue. The green boxes are those shared by both diagrams, while the blue boxes indicate where they differ. Notably, it is generally true that the blue boxes in each diagram lie entirely within one of two the columns where $x$ and $v$ differ, and that the number of blue boxes differs by exactly one. As a result, we have
    \[|D^\op(x,v)|=|D^\op(x)|-1=\binom{n}{2}-\ell_\ba(v).\]
    \begin{figure}[ht]
    \centering
    \subcaptionbox{$D^\op_\ba(x,v)$\label{fig:D_xv}}[.4\textwidth]{
    \begin{tikzpicture}[scale = 0.6]
    \fill [green, opacity  = 0.25] (0,2) rectangle (1,0);
    \fill [green, opacity  = 0.25] (0,6) rectangle (3,5);
    \fill [green, opacity  = 0.25] (3,4) rectangle (4,2);
    \fill [blue, opacity  = 0.25] (4,3) rectangle (5,4);
    \draw (0,0)--(6,0)--(6,6)--(0,6)--(0,0);
    \draw (1,0) -- (1,6);
    \draw (2,0) -- (2,6);
    \draw (3,0) -- (3,6);
    \draw (4,0) -- (4,6);
    \draw (5,0) -- (5,6);
    \draw (0,1) -- (6,1);
    \draw (0,2) -- (6,2);
    \draw (0,3) -- (6,3);
    \draw (0,4) -- (6,4);
    \draw (0,5) -- (6,5);
    \node at (-0.5,0.5) {$6$};
    \node at (-0.5,1.5) {$5$};
    \node at (-0.5,2.5) {$4$};
    \node at (-0.5,3.5) {$3$};
    \node at (-0.5,4.5) {$2$};
    \node at (-0.5,5.5) {$1$};
    \node at (0.5,-0.5) {$1$};
    \node at (1.5,-0.5) {$2$};
    \node at (2.5,-0.5) {$3$};
    \node at (3.5,-0.5) {$4$};
    \node at (4.5,-0.5) {$5$};
    \node at (5.5,-0.5) {$6$};
    \draw[line width = 0.35mm] (0.5,4.5) -- (6,4.5);
    \draw[line width = 0.35mm, orange] (0.5,2.5) -- (0.5,5);
    \draw[line width = 0.35mm] (1.5,5) -- (1.5,0.5) -- (6,0.5);
    \draw[line width = 0.35mm] (2.5,5) -- (2.5,1.5) -- (6,1.5);
    \draw[line width = 0.35mm] (3.5,6) -- (3.5,5.5) -- (6,5.5);
    \draw[line width = 0.35mm] (3.5,0) -- (3.5,1);
    \draw[line width = 0.35mm] (4.5,2.5) -- (6,2.5);
    \draw[line width = 0.35mm, orange] (4.5,0) -- (4.5,1);
    \draw[line width = 0.35mm, orange] (4.5,4.5) -- (4.5,6);
    \draw[line width = 0.35mm] (5.5,6) --(5.5,3.5) -- (6,3.5);
    \draw[line width = 0.35mm] (5.5,0) -- (5.5,1);
    \draw[line width=0.55mm, red] (0,5) -- (3,5);
    \draw[line width=0.55mm, red] (3,1) -- (6,1);
    \node at (0.5,4.5) {$\bullet$};
    \node at (1.5,0.5) {$\bullet$};
    \node at (2.5,1.5) {$\bullet$};
    \node at (3.5,5.5) {$\bullet$};
    \node at (4.5,2.5) {$\bullet$};
    \node at (5.5,3.5) {$\bullet$};
    \node[orange, font=\scriptsize] at (0.5,2.5) {$\bigstar$};
    \node[orange, font=\scriptsize] at (4.5,4.5) {$\bigstar$};
    \end{tikzpicture}}
    \subcaptionbox{$D^\op_\ba(x)$\label{fig:D_x}}[.4\textwidth]{
    \begin{tikzpicture}[scale = 0.6]
    \fill [green, opacity  = 0.25] (0,2) rectangle (1,0);
    \fill [blue, opacity  = 0.25] (0,4) rectangle (1,2);
    \fill [green, opacity  = 0.25] (0,6) rectangle (3,5);
    \fill [green, opacity  = 0.25] (3,4) rectangle (4,2);
    \draw (0,0)--(6,0)--(6,6)--(0,6)--(0,0);
    \draw (1,0) -- (1,6);
    \draw (2,0) -- (2,6);
    \draw (3,0) -- (3,6);
    \draw (4,0) -- (4,6);
    \draw (5,0) -- (5,6);
    \draw (0,1) -- (6,1);
    \draw (0,2) -- (6,2);
    \draw (0,3) -- (6,3);
    \draw (0,4) -- (6,4);
    \draw (0,5) -- (6,5);
    \node at (-0.5,0.5) {$6$};
    \node at (-0.5,1.5) {$5$};
    \node at (-0.5,2.5) {$4$};
    \node at (-0.5,3.5) {$3$};
    \node at (-0.5,4.5) {$2$};
    \node at (-0.5,5.5) {$1$};
    \node at (0.5,-0.5) {$1$};
    \node at (1.5,-0.5) {$2$};
    \node at (2.5,-0.5) {$3$};
    \node at (3.5,-0.5) {$4$};
    \node at (4.5,-0.5) {$5$};
    \node at (5.5,-0.5) {$6$};
    \node at (0.5,4.5) {$\bullet$};
    \node at (1.5,0.5) {$\bullet$};
    \node at (2.5,1.5) {$\bullet$};
    \node at (3.5,5.5) {$\bullet$};
    \node at (4.5,2.5) {$\bullet$};
    \node at (5.5,3.5) {$\bullet$};
    \draw[line width = 0.35mm] (0.5,5) -- (0.5,4.5) -- (6,4.5);
    \draw[line width = 0.35mm] (1.5,5) -- (1.5,0.5) -- (6,0.5);
    \draw[line width = 0.35mm] (2.5,5) -- (2.5,1.5) -- (6,1.5);
    \draw[line width = 0.35mm] (3.5,6) -- (3.5,5.5) -- (6,5.5);
    \draw[line width = 0.35mm] (3.5,0) -- (3.5,1);
    \draw[line width = 0.35mm] (4.5,6) -- (4.5,2.5) -- (6,2.5);
    \draw[line width = 0.35mm] (4.5,0) -- (4.5,1);
    \draw[line width = 0.35mm] (5.5,6) --(5.5,3.5) -- (6,3.5);
    \draw[line width = 0.35mm] (5.5,0) -- (5.5,1);
    \draw[line width=0.55mm, red] (0,5) -- (3,5);
    \draw[line width=0.55mm, red] (3,1) -- (6,1);
    \end{tikzpicture}}
    \caption{Tilted Rothe Diagrams in \Cref{def:skew-tilted-Schub-cell} for $v=465123$, $x=265143$ and $\ba = (2,2,2,6,6,6)$}
    \label{fig:tilted-rothe-skew}
    \end{figure}
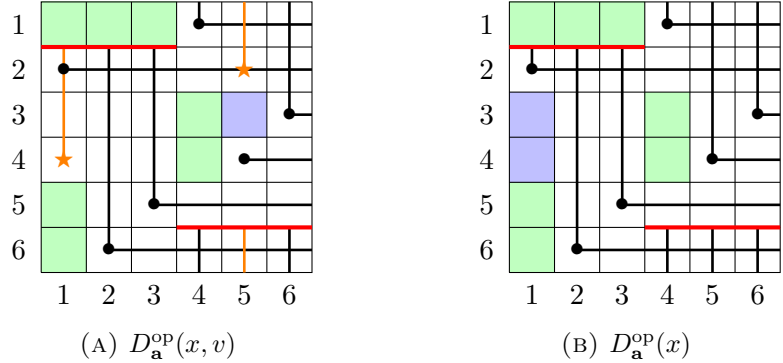

    We now illustrate the Plücker conditions in \Cref{def:skew-tilted-Schub-cell}. Each flag $F_\bullet \in X_{v,x,\ba}^\circ$ can be represented by a matrix as shown in \Cref{fig:X_vxa}. To justify this, note that the condition $\Delta_x(F_\bullet) \neq 0$ ensures that $F_\bullet$ lies in the affine chart $x\Omega^\circ$, meaning it can be represented by a matrix with $1$'s at positions $(x_k, k)$ for all $k \in [6]$ and $0$'s to the right of $1$'s. Condition (2) in \Cref{def:skew-tilted-Schub-cell} then corresponds to placing $0$'s in the green boxes of \Cref{fig:D_xv}, while condition (3), the quadratic Plücker relation, corresponds to the vanishing of a $2 \times 2$ minor indicated by the blue brackets in \Cref{fig:X_vxa}. The remaining entries of the matrix are filled with arbitrary values in $\C$ (denoted by $\ast$). This matrix representative makes it clear that the number of free variables equals $\binom{n}{2}-|D^\op(x,v)|=\ell_{\ba}(v)$, hence $X_{v,x,\ba}^\circ \cong \C^{\ell_{\ba}(v)}$.
    \begin{figure}[ht]
    \centering
    \begin{tikzpicture}[scale = 0.6]
    \fill [green, opacity  = 0.25] (0,2) rectangle (1,0);
    \fill [green, opacity  = 0.25] (0,6) rectangle (3,5);
    \fill [green, opacity  = 0.25] (3,4) rectangle (4,2);
    \draw (0,0)--(6,0)--(6,6)--(0,6)--(0,0);
    \draw (1,0) -- (1,6);
    \draw (2,0) -- (2,6);
    \draw (3,0) -- (3,6);
    \draw (4,0) -- (4,6);
    \draw (5,0) -- (5,6);
    \draw (0,1) -- (6,1);
    \draw (0,2) -- (6,2);
    \draw (0,3) -- (6,3);
    \draw (0,4) -- (6,4);
    \draw (0,5) -- (6,5);
    \node at (-0.5,0.5) {$6$};
    \node at (-0.5,1.5) {$5$};
    \node at (-0.5,2.5) {$4$};
    \node at (-0.5,3.5) {$3$};
    \node at (-0.5,4.5) {$2$};
    \node at (-0.5,5.5) {$1$};
    \node at (0.5,-0.5) {$1$};
    \node at (1.5,-0.5) {$2$};
    \node at (2.5,-0.5) {$3$};
    \node at (3.5,-0.5) {$4$};
    \node at (4.5,-0.5) {$5$};
    \node at (5.5,-0.5) {$6$};
    \node at (0.5,4.5) {$1$};
    \node at (1.5,0.5) {$1$};
    \node at (2.5,1.5) {$1$};
    \node at (3.5,5.5) {$1$};
    \node at (4.5,2.5) {$1$};
    \node at (5.5,3.5) {$1$};
    \node at (1.5,4.5) {$0$};
    \node at (2.5,4.5) {$0$};
    \node at (3.5,4.5) {$0$};
    \node at (4.5,4.5) {$0$};
    \node at (5.5,4.5) {$0$};
    \node at (2.5,0.5) {$0$};
    \node at (3.5,0.5) {$0$};
    \node at (4.5,0.5) {$0$};
    \node at (5.5,0.5) {$0$};
    \node at (3.5,1.5) {$0$};
    \node at (4.5,1.5) {$0$};
    \node at (4.5,5.5) {$0$};
    \node at (5.5,1.5) {$0$};
    \node at (5.5,2.5) {$0$};
    \node at (5.5,5.5) {$0$};
    \node at (0.5,0.5) {$0$};
    \node at (0.5,1.5) {$0$};
    \node at (0.5,2.5) {$a$};
    \node at (0.5,3.5) {$ab$};
    \node at (0.5,5.5) {$0$};
    \node at (1.5,1.5) {$*$};
    \node at (1.5,2.5) {$*$};
    \node at (1.5,3.5) {$*$};
    \node at (1.5,5.5) {$0$};
    \node at (2.5,2.5) {$*$};
    \node at (2.5,3.5) {$*$};
    \node at (2.5,5.5) {$0$};
    \node at (3.5,2.5) {$0$};
    \node at (3.5,3.5) {$0$};
    \node at (4.5,3.5) {$b$};
    \draw[line width=0.55mm, blue] (0,2) -- (1,2) -- (1,4) -- (0,4) -- (0,2);
    \draw[line width=0.55mm, blue] (4,2) -- (5,2) -- (5,4) -- (4,4) -- (4,2);
    \end{tikzpicture}
    \caption{The matrix representative of a skew tilted Schubert cell $X_{v,x,\ba}^\circ$ for $v=465123$, $x=265143$ and $\ba = (2,2,2,6,6,6)$}
    \label{fig:X_vxa}
    \end{figure}
\end{ex}

We are ready to prove \Cref{prop:uxvanishing}, which relies on the following two lemmas.

\begin{lemma}\label{lemma:indplucker}
    Let $I\in \binom{[n]}{a}, J\in \binom{[n]}{b}$ with $1\leq b < a < n$. Suppose there exists an integer $k\in (b,a)$ and subsets $K_1,K_2\in \binom{[n]}{k}$ such that for some $r,r'\in [n]$, we have $F_\bullet \in \cR_{K_1,K_2,r} = \cR_{K_1,K_2,r'}$. Then the following identity holds:
    \[
        \sum_{i\in I \cap [r,r')_c} \Delta_{I- i}\Delta_{J+i}(F_\bullet) = 0.
    \]
\end{lemma}
\begin{proof}
    Fix any $K\in \binom{[n]}{k}$ such that $\Delta_K(F_\bullet) \neq 0$. Multiply the desired expression by $\Delta_K(F_\bullet)$, and expand $\Delta_{J+i}\Delta_K(F_\bullet)$ using \eqref{eqn:incidenceplucker3}, the left hand side becomes 
    \[\sum_{i\in I \cap [r,r')_c} \Delta_{I- i}\Delta_{J+ i}\Delta_K(F_\bullet)=\sum_{\substack{i\in I\cap [r,r')_c\\j\in K}}\Delta_{I-i}\Delta_{J+j}\Delta_{K-j+i}(F_\bullet).\]
    By \Cref{prop:RIJadef}, $\Delta_{K-j+i}(F_\bullet)$ is nonzero only if $K\setminus\{j\}\cup\{i\}\in[K_1,K_2]_r=[K_1,K_2]_{r'}$, which by \Cref{lemma:both-r-comparable} implies $i,j\in [r,r')_c$. So the sum above can be rewritten as
    \[\sum_{\substack{i\in I\cap[r,r')_c\\j\in K\cap[r,r')_c}}\Delta_{I-i}\Delta_{J+j}\Delta_{K-j+i}(F_\bullet)=\sum_{\substack{i\in I\\j\in K\cap[r,r')_c}}\Delta_{I-i}\Delta_{J+j}\Delta_{K-j+i}(F_\bullet).\]
    For each fixed $j \in K \cap [r, r')_c$, $\sum_{i\in I}\Delta_{I-i}\Delta_{K-j+i}(F_\bullet)=0$ by \eqref{eqn:incidenceplucker2}, completing the proof.
\end{proof}

\begin{lemma}\label{lemma:induxplucker}
    Let $u,v\in S_n$ and $\ba\in [n]^n$ such that $u\lesssim_\ba v$. For any subsets $I\in \binom{[n]}{q}$, $J\in \binom{[n]}{p-1}$ with $1\leq p < q < n$, suppose a flag $F_\bullet$ satisfies $F_\bullet \in\pi_{k}^{-1}(\cR_{u[k],v[k],a_k})$ for every $p\leq k<q$. Then the following identity holds:
    \[\sum_{\substack{i\in I\cap [a_p,a_q)_c}}\Delta_{I- i}\Delta_{J+ i}(F_\bullet) = 0.\]
\end{lemma}
\begin{proof}
    For each integer $k$ with $p\leq k<q$, \Cref{lemma:indplucker} gives the relation:
    \[\sum_{\substack{i\in I\cap [a_k,a_{k+1})_c}}\Delta_{I- i}\Delta_{J+ i}(F_\bullet) = 0.\]
    Summing these equations over $k=p,\dots,q-1$, and subtracting the full-sum relation $\sum_{i\in I}\Delta_{I-i}\Delta_{J+i}(F_\bullet) = 0$ (as given by \eqref{eqn:incidenceplucker2}) the appropriate number of times, we obtain the desired result.
\end{proof}

We are now ready to prove the key intersection result of this section.

\begin{prop}\label{prop:uxvanishing}
    Let $u,v,x\in S_n$ and $\ba\in[n]^n$ such that $u\lesssim_\ba v$, $x\in[u,v]$ and $\ell(u,x)=\ell(u,v)-1$. Then, we have
    \[\cT_{u,v}\cap x\Omega_{\id}^\circ \cap u\Omega_{\id}^\circ = X_{v,x,\ba}^\circ\cap\Omega_{u,\ba}^\circ.\]
    As a consequence, the intersection $\cT_{u,v}\cap x\Omega_{\id}^\circ \cap u\Omega_{\id}^\circ$ is equidimensional of dimension $\ell(u,v)$.
\end{prop}
\begin{proof}
Throughout this proof, we assume that $x=vt_{pq}$ for some integers $p<q$. We begin by proving the containment $\subseteq$. Since the left-hand side lies in $\Omega_{u,\ba}^\circ$ by \Cref{thm:tilted-Schub-cell-intersect}, it suffices to show that
\[\cT_{u,v}\cap x\Omega_{\id}^\circ \subseteq X_{v,x,\ba}^\circ.\]
That is, we need to verify that any flag $F_\bullet\in \cT_{u,v}\cap x\Omega_{\id}^\circ$ satisfies the Pl\"ucker conditions listed in \Cref{def:skew-tilted-Schub-cell}. Since $F_\bullet\in x\Omega_{\id}^\circ$, we already have $\Delta_x(F_\bullet)\neq 0$. It remains to verify the two additional types of Pl\"ucker conditions.

We first verify condition (2) of \Cref{def:skew-tilted-Schub-cell}. For all $(i,k)\in D^\op_\ba(x,v)\cap D^\op_\ba(x)$, we claim that $x[k-1]\cup\{i\}\not\leq_{a_k}v[k]$. This will imply $\Delta_{x[k-1]\cup\{i\}}(F_\bullet)=0$ by \Cref{prop:RIJadef}, since $F_\bullet\in \pi_k^{-1}(\cR_{u[k],v[k],a_k})$. The claim follows by case analysis:
\begin{enumerate}
    \item If $k<p$ or $k>q$: then $(i,k)\in D^\op(x,v)$ implies $i>_{a_k}v_k$, so $x[k-1]\cup\{i\}>_{a_k}v[k]$;
    \item If $k=q$: then $(i,q)\in D^\op(x)$ implies $i>_{a_q}x_q$, so $x[k-1]\cup\{i\}>_{a_k}v[k]$;
    \item If $p<k<q$: suppose for contradiction that $x[k-1]\cup\{i\}\leq_{a_k}v[k]$. Then we must have $\{i,v_q\}\leq_{a_k}\{v_k,v_p\}$. But $(i,k)\in D^\op(x,v)$ implies $i>_{a_k}v_k$, so we must have $v_q<_{a_k}v_k<_{a_k}i\leq_{a_k}v_p$, which contradicts the condition $x=vt_{pq}\lessdot_\ba v$ by \Cref{prop:covering}.
\end{enumerate}
    
We now verify condition (3) of \Cref{def:skew-tilted-Schub-cell}. For all $(i,q)\in D^\op_\ba(x,v)\setminus D^\op_\ba(x)$, we apply the incidence Pl\"ucker relation from \eqref{eqn:incidenceplucker3}:
\[\Delta_{x[p-1]+i}\Delta_{x[q]}(F_\bullet)=\sum_{k\in[p,q]}\Delta_{x[p-1]+x_k}\Delta_{x[q]-x_k+i}(F_\bullet).\]
Moving the $k=q$ term to the left-hand side, we obtain:
\[\Delta_{x[p-1]+i}\Delta_{x[q]}(F_\bullet)-\Delta_{x[p-1]+x_q}\Delta_{x[q-1]+i}(F_\bullet)=\sum_{k\in[p,q)}\Delta_{x[p-1]+x_k}\Delta_{x[q]-x_k+i}(F_\bullet).\]
Since the left-hand side matches condition (3), it remains to show that the right-hand side vanishes. We analyze the terms in the right-hand side through the following cases:
\begin{enumerate}
    \item If $k=p$: then $(i,k)\in D^\op(x,v)$ implies $i>_{a_q}v_q$, so $x[q]\setminus\{x_p\}\cup\{i\}>_{a_q}v[q]$. Hence, $\Delta_{x[q]-x_p+i}(F_\bullet)=0$ by \Cref{prop:RIJadef}, since $F_\bullet\in \pi_q^{-1}(\cR_{u[q],v[q],a_q})$;
    \item If $p<k<q$ and $x_k>_{a_p}x_q$: then $x[p-1]\cup\{x_k\}>_{a_p}v[p]$, so $\Delta_{x[p-1]+x_k}(F_\bullet)=0$, since $F_\bullet\in \pi_p^{-1}(\cR_{u[p],v[p],a_p})$;
    \item If $p<k<q$ and $i>_{a_q}x_k$: then $x[q]\setminus\{x_k\}\cup\{i\}>_{a_q}v[q]$, so $\Delta_{x[q]-x_k+i}(F_\bullet)=0$, since $F_\bullet\in \pi_q^{-1}(\cR_{u[q],v[q],a_q})$;
    \item In the remaining case where $p<k<q$ and none of the above apply, the following conditions must hold:
    \begin{enumerate}
        \item $x_k\leq_{a_p} x_q$, and by \Cref{prop:covering}, this can be strengthened to $x_k<_{a_p} x_p<_{a_p}x_q$;
        \item $i\leq_{a_q}x_k$;
        \item Additionally, since $(i,q)\in D^\op_\ba(x,v)\setminus D^\op_\ba(x)$, we have $x_p<_{a_q}i<_{a_q}x_q$ .
    \end{enumerate}
    Placing these six integers on a circle in an increasing order (clockwise), we observe the following cyclic configuration (up to rotation):
    \[a_p\leq x_k<a_q\leq x_p<i<x_q<a_p.\]
    In this setting, the right-hand side simplifies to
    \[\sum_{x_k\in [a_p,a_q)_c} \Delta_{x[p-1]+x_k}\Delta_{x[q]-x_k+i}(F_\bullet),\]
    which vanishes by \Cref{lemma:induxplucker}.
\end{enumerate}

We now prove the $\supseteq$ direction. Since, by definition, the right-hand side $X_{v,x,\ba}^\circ\cap\Omega_{u,\ba}^\circ$ is contained in both $x\Omega_\id^\circ$ and $u\Omega_\id^\circ$, it suffices to show that $X_{v,x,\ba}^\circ\cap\Omega_{u,\ba}^\circ\subseteq\cT_{u,v}$, or equivalently,
\[ X_{v,x,\ba}^\circ\cap\Omega_{u,\ba}^\circ\subseteq \pi_k^{-1}(\cR_{u[k],v[k],a_k}) \text{ for all }k\in [n].\]
We proceed by induction on $k$. The overall structure, including the base case, is the same as in the proof of \Cref{thm:tilted-Schub-cell-intersect}. For the induction step, we must show the implication:
\[F_\bullet\in\pi_{k-1}^{-1}(\cR_{u[k-1],v[k-1],a_{k-1}})\cap X^\circ_{v,x,\ba}\cap \Omega^\circ_{u,\ba}\implies F_\bullet\in\pi_k^{-1}(\cR_{u[k],v[k],a_k}).\]
Since $u\lesssim_\ba v$, it follows from \Cref{lemma:rotateRich} that $\cR_{u[k-1],v[k-1],a_{k-1}} = \cR_{u[k-1],v[k-1],a_{k}}$. Our goal is to show that $F_\bullet$ satisfies the Pl\"ucker conditions in \Cref{prop:RIJadef}: that is, if $K\not\geq_{a_k}u[k]$ or $K\not\leq_{a_k}v[k]$, then $\Delta_K(F_\bullet)=0$. The case $K\not\geq_{a_k}u[k]$ is handled exactly as in the proof of \Cref{thm:tilted-Schub-cell-intersect}, so we focus on the case $K\not\leq_{a_k}v[k]$.

To show that $\Delta_K(F_\bullet)=0$, we consider the incidence Pl\"ucker relation \eqref{eqn:incidenceplucker1} associated with $x[k-1]$ and $K$:
\[\Delta_{x[k-1]}\Delta_{K}(F_\bullet) = \sum_{i\in K\setminus x[k-1]}\Delta_{x[k-1]+i}\Delta_{K-i}(F_\bullet).\]
Since $\Delta_{x[k-1]}(F_\bullet)\neq 0$, it suffices to show that right-hand side vanishes. We analyze each term on the right-hand side by considering the following three cases for $i\in K\setminus x[k-1]$:
\begin{enumerate}
    \item If $i\leq_{a_k}v_k$: then the assumption $K\not\leq_{a_k}v[k]$ implies $K\setminus \{i\}\not\leq_{a_k}v[k-1]$. Therefore, $\Delta_{K-i}(F_\bullet)=0$, since $F_\bullet\in\pi_{k-1}^{-1}(\cR_{u[k-1],v[k-1],a_{k}})$;
    \item If $i>_{a_k}v_k$ and $i>_{a_k}x_k$: then $(i,k)\in D^\op_\ba(x,v)\cap D^\op_\ba(x)$, Therefore, $\Delta_{x[k-1]+i}(F_\bullet)=0$, since $F_\bullet\in X^\circ_{v,x,\ba}$.
    \item If $i>_{a_k}v_k$ but $i\leq_{a_k}x_k$: this scenario occurs only when $k=q$, $i\in[x_p,x_q)_c$, and $(i,k)\in D^\op(x,v)\setminus D^\op(x)$. In this case, we must show that the sum
    \[\sum_{i\in K\cap(x_p,x_q]_c}\Delta_{x[q-1]+i}\Delta_{K-i}(F_\bullet)\]
    vanishes. To prove this, we wish to multiply both sides by $\Delta_{x[p-1]+x_q}(F_\bullet)$ and apply condition (3) from \Cref{def:skew-tilted-Schub-cell}. Although it is not guaranteed that $\Delta_{x[p-1] + x_q}(F_\bullet)$ is nonzero, in the degenerate case where it vanishes (e.g., when $a = 0$ in \Cref{fig:X_vxa}), the Pl\"ucker conditions in \Cref{def:skew-tilted-Schub-cell} reduce precisely to those defining $X_{x,\ba}^\circ$ in \Cref{def:tilted-Schub-cell}. This implies that $F_\bullet\in X_{x,\ba}^\circ$ in such case. Therefore, by \Cref{thm:tilted-Schub-cell-intersect}, we conclude that
    \[F_\bullet\in X_{x,\ba}^\circ\cap\Omega_{u,\ba}^\circ=\cT_{u,x}^\circ\subseteq \cT_{u,v}.\]
    Hence, we may safely assume that $\Delta_{x[p-1] + x_q}(F_\bullet) \neq 0$. Multiplying both sides by this nonzero Pl\"ucker coordinate and applying condition (3) from \Cref{def:skew-tilted-Schub-cell}, we obtain:
    \[\sum_{i\in K\cap(x_p,x_q]_c} \Delta_{x[p-1]+x_q}\Delta_{x[q-1]+i}\Delta_{K-i}(F_\bullet) = \sum_{i\in K\cap(x_p,x_q]_c} \Delta_{x[q]}\Delta_{x[p-1]+i}\Delta_{K-i}(F_\bullet).\]
    To prove the vanishing of the right-hand side, we divide the sum into three parts and show that each of them vanishes:
    \begin{enumerate}
        \item For $i\in K\cap (x_q,a_p)_c$, we have
        \[\sum_{i\in K\cap(x_q,a_p)_c}\Delta_{x[p-1]+i}\Delta_{K-i}(F_\bullet)=0.\]
        This is because each term vanishes individually: since $i>_{a_p}x_q$, we have $(i,p)\in D^\op_\ba(x,v)\cap D^\op_\ba(x)$. Therefore, $\Delta_{x[p-1]+i}(F_\bullet)=0$, since $F_\bullet\in X^\circ_{v,x,\ba}$.
        \item For $i\in K\cap[a_q,x_p]_c$, we have
        \[\sum_{i\in K\cap[a_q,x_p]_c}\Delta_{x[p-1]+i}\Delta_{K-i}(F_\bullet)=0.\]
        This is because each term vanishes individually: since $i\leq_{a_q}x_p$, the assumption $K\not\leq_{a_q}v[q]$ implies $K\setminus \{i\}\not\leq_{a_q}v[q-1]$. Therefore, $\Delta_{K-i}(F_\bullet)=0$, since $F_\bullet\in\pi_{q-1}^{-1}(\cR_{u[q-1],v[q-1],a_{q}})$.
        \item For $i\in K\cap[a_p,a_q)_c$, we have
        \[\sum_{i\in K\cap[a_p,a_q)_c}\Delta_{x[p-1]+i}\Delta_{K-i}(F_\bullet)=0.\]
        This follows directly from \Cref{lemma:induxplucker} and the induction hypothesis.
    \end{enumerate}
    Combining these three identities and subtracting their sum from the full-sum relation
    \[\sum_{i\in K}\Delta_{x[p-1]+i}\Delta_{K-i}(F_\bullet) = 0\]
    (as given by \eqref{eqn:incidenceplucker2}), we conclude that the original right-hand side vanishes.
\end{enumerate}
In either case, we have shown that $\Delta_I(F_\bullet) = 0$, as required. This completes the induction and establishes the intersection claim in the theorem.

Finally, for the dimension statement, since
\[\cT_{u,v}\cap x\Omega_{\id}^\circ \cap u\Omega_{\id}^\circ=X_{v,x,\ba}^\circ\cap\Omega_{u,\ba}^\circ\]
is cut out by $\binom{n}{2}-\ell(u,v)$ equations, every irreducible component has dimension at least $\ell(u,v)$. On the other hand, by \Cref{thm:dimension}, every irreducible component has dimension at most $\ell(u,v)$. Hence, the intersection is equidimensional of dimension $\ell(u,v)$.
\end{proof}

Finally, we arrive at the desired closure result.

\begin{theorem}\label{thm:closure}
    For all $u,v\in S_n$, we have $\cT_{u,v} = \overline{\cT_{u,v}^\circ}$.
\end{theorem}
\begin{proof}
    We proceed by induction on $\ell(u,v)$. the base case $\ell(u,v)=0$ corresponds to $u=v$, in which case $\cT_{u,u} = \cT_{u,u}^\circ = \{e_u\}$, so the claim holds.
    
    Now assume $\ell(u,v)>0$. By \Cref{thm:Tunion}, it suffices to show that $\cT^\circ_{u,x} \subseteq \overline{\cT_{u,v}^\circ}$ for all $x\in [u,v]$ with $\ell(u,x) = \ell(u,v)-1$. From \Cref{lemma:uxOmega}, we have the decomposition
    \[\cT_{u,v}\cap x\Omega_\id^\circ\cap u\Omega_\id^\circ = \cT_{u,x}^\circ \sqcup (\cT^\circ_{u,v}\cap x\Omega_\id^\circ).\]
    Taking closures on both sides yields
    \[\overline{\cT_{u,v}\cap x\Omega_\id^\circ\cap u\Omega_{\id}^\circ} = \cT_{u,x} \cup \overline{\cT^\circ_{u,v}\cap x\Omega_\id^\circ}.\]
    Let $Z$ be an irreducible component of the left-hand side. By \Cref{prop:uxvanishing}, we know that $\dim(Z) = \ell(u,v)$. On the other hand, by \Cref{thm:dimension}, $\dim(\cT_{u,x}) = \ell(u,x) = \ell(u,v) - 1$, so $Z$ cannot be contained in $\cT_{u,x}$. Therefore, $Z$ must be contained in $\overline{\cT^\circ_{u,v} \cap x\Omega_{\id}^\circ}$. Since this holds for every irreducible components $Z$, we conclude that
    \[\overline{\cT_{u,v}\cap x\Omega_{\id}^\circ\cap u\Omega_{\id}^\circ} \subseteq \overline{\cT^\circ_{u,v}\cap x\Omega_{\id}^\circ}.\]
    Now by applying \Cref{lemma:uxOmega} again, we obtain
    \[\cT^\circ_{u,x}\subseteq\overline{\cT_{u,v}\cap x\Omega_{\id}^\circ\cap u\Omega_{\id}^\circ}\subseteq\overline{\cT^\circ_{u,v}\cap x\Omega_{\id}^\circ}\subseteq\overline{\cT_{u,v}^\circ},\]
    which completes the proof. 
\end{proof}

%% file: tex/6-deodhar.tex
\section{Tilted Deodhar Decomposition and Its Applications}\label{sec:deodhar}

In this section, we introduce the \emph{tilted Deodhar decomposition} of tilted Richardson varieties $\cT_{u,v}^\circ$, in analogy with the classical Deodhar decomposition of Richardson varieties $\cR^\circ_{u,v}$. This decomposition plays a central role in proving the irreducibility of $\cT_{u,v}^\circ$ and has further applications in the development of a theory of \emph{tilted Kazhdan--Lusztig $R$-polynomials} $R^\tilt_{u,v}(q)$ and \emph{total positivity} for tilted Richardson varieties.

\subsection{Tilted Distinguished Subwords and Regular Tilted Reduced Words}\label{sec:tilted-distinguished}
We introduce \emph{tilted distinguished subwords} (\Cref{def:distinguished}) and \emph{regular tilted reduced words} (\Cref{def:regular-word}), which serve as the combinatorial data labeling each cell in the tilted Deodhar decomposition. For background on tilted reduced words, see \Cref{sec:tilted-reduced-word}. We begin by establishing some notation that will be used consistently throughout the section.

\begin{notation}\label{notation:tiltedword}
    Let $\ba = (a_1, \dots, a_n) \in [n]^n$ be a sequence, and let the set of jumps of $\ba$ be denoted by $\Jump_\ba = { \jump_1 < \jump_2 < \cdots < \jump_t }$, where $t := \jasize$.

    Let $\mathbf{v} = s_{i_1}s_{i_2} \cdots s_{i_\ell}$ be a $\ba$-tilted reduced word for $v \in S_n$ of length $\ell$. Define the index sets of bar and non-bar factors in $\mathbf{v}$ as
    \[\begin{aligned}
        J_\mathbf{v}^\mid&:=\{j\in[\ell]:\text{the $j$-th factor of $\mathbf{v}$ is a bar}\},\\
        J_\mathbf{v}^s&:=\{j\in[\ell]:\text{the $j$-th factor of $\mathbf{v}$ is a non-bar factor $s_{i_j}$}\}.
    \end{aligned}\]
    Let $\mathbf{u}$ be a subword of $\mathbf{v}$ corresponding to $u \in S_n$. For $0 \leq j \leq \ell$, let $\mathbf{v}^{(j)}$ and $\mathbf{u}^{(j)}$ denote the subsequences consisting of the first $j$ factors of $\mathbf{v}$ and $\mathbf{u}$, respectively. Define $v^{(j)} \in S_n$ and $u^{(j)} \in S_n$ to be the products of the first $j$ factors of $\mathbf{v}$ and $\mathbf{u}$, respectively, with the convention that $v^{(0)} = u^{(0)} = \id$.

    Now define a sequence of sequences $\ba^{(\ell)}, \ba^{(\ell - 1)}, \dots, \ba^{(0)}$ in $[n]^n$ recursively as follows: start with $\ba^{(\ell)} := \ba$, and for each $j = \ell, \dots, 1$, set
    \[\ba^{(j-1)}:=\begin{cases}
        \ba^{(j)}, &\text{if }j\in J_\mathbf{v}^s,\\
        \flatten(\ba^{(j)}), &\text{if }j\in J_\mathbf{v}^\mid.
    \end{cases}\]
    With this setup, for each $0 \leq j \leq \ell$, the subword $\mathbf{v}^{(j)}$ is a $\ba^{(j)}$-tilted reduced word for $v^{(j)}$, and $\mathbf{u}^{(j)}$ is a subword of $\mathbf{v}^{(j)}$ corresponding to $u^{(j)}$.
\end{notation}
\begin{ex}
    Let $n = 4$ and $\ba = (3, 2, 2, 1)$. Then the set of jumps is
    \[\Jump_\ba=\{\jump_1,\jump_2\}=\{1,3\},\quad \text{so }t=2.\]
    Let the $\ba$-tilted reduced word for $v=1423$ be
    \[\mathbf{v}=s_1s_3s_2\mid s_2s_1\mid s_2,\]
    with total length $\ell=8$, including $2$ bars. The index sets are:
    \[
        J_\mathbf{v}^\mid=\{4,7\},\quad J_\mathbf{v}^s=\{1,2,3,5,6,8\}.
    \]
    We track the intermediate sequences $\mathbf{v}^{(j)}$, permutations $v^{(j)}$, and sequences $\ba^{(j)}$:
    \begin{align*}
        \mathbf{v}^{(0)}&=\emptyset & v^{(0)}&=1234 & \ba^{(0)}&=(1,1,1,1)\\
        \mathbf{v}^{(1)}&=s_1 & v^{(1)}&=2134 & \ba^{(1)}&=(1,1,1,1)\\
        \mathbf{v}^{(2)}&=s_1s_3 & v^{(2)}&=2143 & \ba^{(2)}&=(1,1,1,1)\\
        \mathbf{v}^{(3)}&=s_1s_3s_2 & v^{(3)}&=2413 & \ba^{(3)}&=(1,1,1,1)\\
        \mathbf{v}^{(4)}&=s_1s_3s_2\mid & v^{(4)}&=2413 & \ba^{(4)}&=(2,2,2,1)\\
        \mathbf{v}^{(5)}&=s_1s_3s_2\mid s_2 & v^{(5)}&=2143 & \ba^{(5)}&=(2,2,2,1)\\
        \mathbf{v}^{(6)}&=s_1s_3s_2\mid s_2s_1 & v^{(6)}&=1243 & \ba^{(6)}&=(2,2,2,1)\\
        \mathbf{v}^{(7)}&=s_1s_3s_2\mid s_2s_1\mid & v^{(7)}&=1243 & \ba^{(7)}&=(3,2,2,1)\\
        \mathbf{v}^{(8)}&=s_1s_3s_2\mid s_2s_1\mid s_2 & v^{(8)}&=1423 & \ba^{(8)}&=(3,2,2,1)
    \end{align*}
    Notice that for each $j$, the prefix $\mathbf{v}^{(j)}$ is an $\ba^{(j)}$-tilted reduced word for $v^{(j)}$.
\end{ex}


\begin{defin}\label{def:distinguished}
    Let $\mathbf{v}=s_{i_1}s_{i_2}\cdots s_{i_\ell}$ be an $\ba$-tilted reduced word for $v\in S_n$. A subword $\mathbf{u}$ of $\mathbf{v}$ is called a \emph{(tilted) distinguished subword} if
    \[u^{(j)}\leq_{\ba^{(j)}} u^{(j-1)}s_{i_j}\text{ for all }j\in J_\mathbf{v}^s.\]
    Equivalently, if right multiplication by $s_{i_j}$ decreases $u^{(j-1)}$ in the $\leq_{\ba^{(j)}}$ order, then we must have $u^{(j)} = u^{(j-1)} s_{i_j}$. We write $\mathbf{u}\prec\mathbf{v}$ if $\mathbf{u}$ is a distinguished subword of $\mathbf{v}$. In this case, define the following index sets:
    \[
    \begin{aligned}
        J^+_\mathbf{u}&:=\{j\in J_\mathbf{v}^s: u^{(j-1)}<_{\ba^{(j)}}u^{(j)}\},\\
        J^\circ_\mathbf{u}&:=\{j\in J_\mathbf{v}^s:u^{(j-1)}=u^{(j)}\},\\
        J^-_\mathbf{u}&:=\{j\in J_\mathbf{v}^s:\;u^{(j-1)}>_{\ba^{(j)}}u^{(j)}\}.
    \end{aligned}
    \]
\end{defin}
It follows directly from the definition that if $\mathbf{u} \prec \mathbf{v}$, then for every $0 \leq j \leq \ell$, the prefix $\mathbf{u}^{(j)}$ is a distinguished subword of $\mathbf{v}^{(j)}$.

\begin{ex}\label{ex:distinguished}
    Let $v=246513$ and $\ba =(5,5,5,1,1,1)$. Consider the $\ba$-tilted reduced word for $v$ given by
    \[\mathbf{v}=s_3s_4s_5s_1s_2s_3s_4s_3s_2s_1\mid s_1s_2.\]
    There are four distinguished subwords $\mathbf{u} \prec \mathbf{v}$ corresponding to $u = 512346$ listed below. In each subword, we mark factors $s_{i_j}$ with $\mathcolor{red}{\widetilde{s_{i_j}}}$ when $j \in J^-_\mathbf{u}$.
    \begin{align*}
        \mathbf{u}&= 111111s_4s_3s_2s_1\mid 11 & \mathbf{u}&= 111s_111s_4s_3s_2s_1\mid 1\mathcolor{red}{\widetilde{s_2}}\\
        \mathbf{u}&= s_31111\mathcolor{red}{\widetilde{s_3}}s_4s_3s_2s_1\mid 11 & 
        \mathbf{u}&= s_311s_11\mathcolor{red}{\widetilde{s_3}}s_4s_3s_2s_1\mid 1\mathcolor{red}{\widetilde{s_2}}
    \end{align*}
\end{ex}

We proceed with the notion of \emph{regular} $\ba$-tilted words, which will play an important role in the main theorem of the tilted Deodhar decomposition. 
\begin{defin}\label{def:bigrass}
    A permutation $w\in S_n$ is called a \emph{Grassmannian permutation} if it has at most one descent. It is a \emph{bi-Grassmannian permutation} if both $w$ and $w^{-1}$ are Grassmannian permutations. For $a,b\in [n]$, define the bi-Grassmannian permutation $s_{a,b}\in S_n$ by
    \[s_{a,b}(i):=\begin{cases}
        i+a &\text{if }i\leq b,\\
        i-b &\text{if }b<i\leq a+b,\\
        i &\text{if }i>a+b.
    \end{cases}\]
    Equivalently, in one-line notation,
    \[s_{a,b}=(a+1)\;\cdots\;(a+b)\;1\;\cdots\;a\;(a+b+1)\;\cdots\;n\in S_n.\]
\end{defin}

\begin{defin}\label{def:regular-word}
    An $\ba$-tilted word $\mathbf{v}$ for $v \in S_n$ is called \emph{regular} if, for each bar in $\mathbf{v}$ at position $j \in J_\mathbf{v}^\mid$, there exists a subsequence of simple transpositions immediately preceding that bar in $\mathbf{v}$, specifically a subsequence of the form $s_{i_{k+1}} \cdots s_{i_{j-1}}$ (often marked in blue as $\mathcolor{blue}{\overline{s_{i_{k+1}} \cdots s_{i_{j-1}}}}$), which forms a reduced word for the bi-Grassmannian permutation $s_{q-p,p}\in S_n$, where $q:=\jmin(\ba^{(j)})$ and $p:=\size{v[q]\cap[\ba^{(j)}_1,\ba^{(j)}_q)_c}$. In particular, $p$ is the integer appearing in \Cref{def:a-flattenable} for the $\ba^{(j)}$-flattenable permutation $v^{(j)}$.
\end{defin}
\begin{ex}
    The $\ba$-tilted reduced word from \Cref{ex:distinguished} is regular. The subsequence marked in blue, immediately preceding the bar, has product $s_{2,1}:=312456$, which satisfies the condition in \Cref{def:regular-word}:
    \[\mathbf{v}=s_3s_4s_5s_1s_2s_3s_4s_3\mathcolor{blue}{\overline{s_2s_1}}\mid s_1s_2.\]
\end{ex}

\begin{construction}\label{construction:regular-tiltedreducedword}
    We describe an algorithm that constructs a regular $\ba$-tilted reduced word for any permutation $w \in S_n$. We illustrate the algorithm using the example $\ba = (3,3,1,1,1,6)$ and $w = 136254$. The set of jumps of $\ba$ in this example is
    \[\Jump_\ba=\{\jump_1,\jump_2,\jump_3\}=\{2,5,6\},\quad\text{where }t=3.\]
    First, we construct a sequence of permutations:
    \[\id\to \mathcolor{blue}{\tilde{w}^{(t)}\to} w^{(t)}\to\cdots\to\mathcolor{blue}{\tilde{w}^{(1)}\to} w^{(1)}\to w,\]
    where for each $k\in [t]$, the permutation $w^{(k)}\in S_n$ is obtained from $w$ by sorting its first $\jump_k$ entries in increasing order with respect to the shifted order $<_{a_{\jump_k}}$, and the permutation $\mathcolor{blue}{\tilde{w}^{(k)}}\in S_n$ is obtained from $w$ by sorting its first $\jump_k$ entries in increasing order with respect to the shifted order $<_{a_{\jump_k+1}}$ (with the convention that $a_{n+1}=1$). In our running example, the resulting sequence is:
    \[123456\to \mathcolor{blue}{123456 \to}  612345\to\mathcolor{blue}{612354\to}123564\to\mathcolor{blue}{136254\to}316254\to136254.\]
    Next, for each adjacent pair $w'\to w''$ in the sequence above, we compute a reduced word for $(w')^{-1}w''$, concatenate all these reduced words, and insert a bar after every two reduced word segments to form a regular $\ba$-tilted word for $w$. In our running example, we obtain:
    \[\mathcolor{blue}{\overline{s_5s_4s_3s_2s_1}}\mid s_5\mathcolor{blue}{\overline{s_1s_2s_3s_4}}\mid s_2s_4s_3\mathcolor{blue}{\overline{s_1}}\mid s_1.\]
    All words constructed in this way are regular $\ba$-tilted reduced words.
\end{construction}

Finally, we state the following key lemma.
\begin{lemma}\label{lemma:regular-subword}
    Any subword of a regular $\ba$-tilted reduced word is also regular.
\end{lemma}
\begin{proof}
    Let $\mathbf{v}$ be a regular $\ba$-tilted reduced word, and let $\mathbf{u}$ be a subword of $\mathbf{v}$. Since $\mathbf{v}$ is regular, by \Cref{def:regular-word}, for every bar in $\mathbf{v}$ at position $j \in J^\mid_\mathbf{v}$, there exists a subsequence $\mathcolor{blue}{\overline{s_{i_{k+1}} \cdots s_{i_{j-1}}}}$ immediately preceding the bar that forms a reduced word for the bi-Grassmannian permutation $s_{q-p,p}$. It remains to show that $\mathbf{u}$ also contains this subsequence immediately before position $j$, i.e., that $u^{(j)} = u^{(k)} s_{q-p,p}$. Since $v^{(j)}$ is $\ba^{(j)}$-flattenable and $v^{(k)} = v^{(j)} s_{q-p,p}$, we have:
    \[v^{(k)}[q-p]\subseteq[a^{(j)}_{q+1},a^{(j)}_1)_c,\quad v^{(k)}[q-p+1,q]\subseteq[a^{(j)}_1,a^{(j)}_{q+1})_c.\]
    Because $\mathbf{v}^{(k)}$ is an $\ba^{(k)}$-tilted reduced word and $\mathbf{u}^{(k)}$ is a subword, \Cref{thm:subword-property} implies that $u^{(k)} \lesssim_{\ba^{(k)}} v^{(k)}$. Therefore,
    \[u^{(k)}[q-p]\subseteq[a^{(j)}_{q+1},a^{(j)}_1)_c,\quad u^{(k)}[q-p+1,q]\subseteq[a^{(j)}_1,a^{(j)}_{q+1})_c.\]
    On the other hand, since $u^{(j)}$ is $\ba^{(j)}$-flattenable, we have:
    \[u^{(j)}[p]\subseteq[a^{(j)}_{1},a^{(j)}_{q+1})_c,\quad u^{(j)}[p+1,q]\subseteq[a^{(j)}_{q+1},a^{(j)}_{1})_c.\]
    This implies that $u^{(j)}$ and $u^{(k)}$ have completely swapped the values in the first $q$ positions:
    \[u^{(j)}[p]=u^{(k)}[q-p+1,q],\quad u^{(j)}[p+1,q]=u^{(k)}[q-p].\]
    Hence, the permutation $(u^{(k)})^{-1} u^{(j)}$ sends the values in positions $[q]$ according to $s_{q-p,p} \cdot x$, where $x \in S_p \times S_{q-p}$ is a block permutation that preserves the two segments. Since $s_{q-p,p} \cdot x$ is a subword of $s_{q-p,p}$, we have $x = \id$, and $(u^{(k)})^{-1} u^{(j)} = s_{q-p,p}$, completing the proof.
\end{proof}

\subsection{Tilted Deodhar Decomposition}\label{sec:tilted-deodhar-main}

In this section, we prove the main theorem of the tilted Deodhar decomposition. We begin by introducing the notion of \emph{tilted relative position} (\Cref{def:tilted-relative-position}), which is analogous to the classical notion of relative position. This leads naturally to the definition of \emph{tilted Deodhar cells} $\cD_{\mathbf{u},\mathbf{v}}$ (\Cref{prop:tilted-deodhar-cell}), indexed by an $\ba$-tilted reduced word $\mathbf{v}$ and a distinguished subword $\mathbf{u} \prec \mathbf{v}$. Finally, we state and prove the main theorem of the tilted Deodhar decomposition (\Cref{thm:tilted-deodhar}).

We begin with some definitions. For each $i \in [n - 1]$, let $
\phi_i\left(\begin{smallmatrix}
    a & b\\
    c & d
\end{smallmatrix}\right)$ denote the matrix obtained by replacing the $2 \times 2$ block in rows and columns $i$ and $i+1$ of the identity matrix with the given $2\times 2$ matrix $\left(\begin{smallmatrix}
    a & b\\
    c & d
\end{smallmatrix}\right)$. Using this notation, define the following elements of $G$:
\begin{align*}
    y_i(p)&=\phi_{i}\begin{pmatrix}
             1 & 0\\ p & 1
        \end{pmatrix}, &
    \dot{s}_i&=\phi_{i}\begin{pmatrix}
             0 & -1\\ 1 & 0
        \end{pmatrix}, &
    x_i(m)&=\phi_{i}\begin{pmatrix}
             1 & m\\ 0 & 1
        \end{pmatrix}.
\end{align*}

We now recall the classical notion of \emph{relative position} (see \cite[Section~4.2]{MarshRietsch}), and then introduce the notion of \emph{tilted relative position}, defined using tilted Schubert cells and the tilted Bruhat decomposition (see \Cref{sec:tilted-def-4}).

\begin{defin}\label{def:tilted-relative-position}
    Let $y_1 = g_1 B$, $y_2 = g_2 B \in G/B$. The double coset $y_1^{-1} y_2 = B g_1^{-1} g_2 B$ is well-defined and equals $BwB$ for some permutation $w \in S_n$ by the Bruhat decomposition. We call this permutation the \emph{relative position} of $y_1$ and $y_2$, and denote it by $y_1 \xrightarrow{w} y_2$.

    For any $w \in S_n$ and $\ba\in[n]^n$, we define the following \emph{tilted relative positions}:
    \begin{align*}
    e \xrightarrow{w, \ba} gB &\quad \text{if } gB \in X^\circ_{w, \ba},\\
    w_0 \xrightarrow{w, \ba} gB &\quad \text{if } gB \in \Omega^\circ_{w, \ba}.
    \end{align*}
\end{defin}

We now state several properties of classical and tilted relative positions.

\begin{prop}\label{prop:relative-position}
    Let $w \in S_n$ and $\ba \in [n]^n$. Let $y_1, y_2, y_3 \in \fl_n$ be flags. Then:
    \begin{enumerate}
        \item If $y_1 \xrightarrow{w} y_2$, then $y_2 \xrightarrow{w^{-1}} y_1$.
        \item If $y_1 \xrightarrow{w} y_2 \xrightarrow{s_i} y_3$, then
        \[\begin{cases}
            y_1 \xrightarrow{w s_i} y_3 & \text{if } \ell(w s_i) > \ell(w), \\
            y_1 \xrightarrow{w s_i} y_3 \text{ or } y_1 \xrightarrow{w} y_3 & \text{if } \ell(w s_i) < \ell(w).
        \end{cases}\]
        \item If $e \xrightarrow{w, \ba} y_1 \xrightarrow{s_i} y_2$, then
        \[
        \begin{cases}
            e \xrightarrow{w s_i, \ba} y_2 & \text{if } i \in \Asc_\ba(w), \\
            e \xrightarrow{w s_i, \ba} y_2 \text{ or } e \xrightarrow{w, \ba} y_2 & \text{if } i \in \Des_\ba(w).
        \end{cases}
        \]
        \item If $w_0 \xrightarrow{w, \ba} y_1 \xrightarrow{s_i} y_2$, then
        \[
        \begin{cases}
            w_0 \xrightarrow{w s_i, \ba} y_2 & \text{if } i \in \Des_\ba(w), \\
            w_0 \xrightarrow{w s_i, \ba} y_2 \text{ or } w_0 \xrightarrow{w, \ba} y_2 & \text{if } i \in \Asc_\ba(w).
        \end{cases}
        \]
        \item Suppose $e \xrightarrow{w, \ba} y$. If $i \in \Des_\ba(w)$, then there exists a unique flag $y' \in \fl_n$ such that
        \[
        e \xrightarrow{w s_i,\ba} y' \xrightarrow{s_i} y.
        \]
        The map $\pi_{w s_i}^{w, \ba} : X^\circ_{w, \ba} \to X^\circ_{w s_i, \ba}$ defined by $y \mapsto y'$ is algebraic and given explicitly by
        \[
        gB \mapsto g \dot{s}_i B,
        \]
        where $g \in G$ is the canonical representative of $gB$, as described in \Cref{prop:tilted-Schub-cell}.
    \end{enumerate}
\end{prop}
\begin{proof}
    Parts (1) and (2) are stated in \cite[Section~4.2]{MarshRietsch}. For part (3), let $y_1 = gB$, where $g$ is the canonical representative of $gB$ as described in \Cref{prop:tilted-Schub-cell}. Depending on whether $i \in \Des_\ba(w)$ or $i \in \Asc_\ba(w)$, the set of $\ast$ entries in the $i$-th and $(i+1)$-th columns of $g$ are nested. See \Cref{fig:i-th-column-ast} for an illustration of the structure in these columns.
    \begin{figure}[ht]
\centering
\subcaptionbox{Case $i\in \Des_\ba(w)$. \label{fig:i-des}}[.4\textwidth]{
    \begin{tikzpicture}[scale = 0.7]
        \draw (0,0) -- (4,0) -- (4,4) -- (0,4) -- (0,0);
        \draw (1,0) -- (1,4);
        \draw (1.5,0) -- (1.5,4);
        \draw (2,0) -- (2,4);
        \draw[red, very thick] (1,0.7) -- (2,0.7);
        \node at (1.25,4.3) {$i$};
        \node at (1.25,1.7) {$1$};
        \node[blue] at (1.25,2.4) {$\ast$};
        \node[blue] at (1.25,3.1) {$c$};
        \node at (1.75,3.1) {$1$};
        \node[blue] at (1.25,3.7) {$\ast$};
        \node[blue] at (1.75,3.7) {$\ast$};
        \node[blue] at (1.25,0.3) {$\ast$};
        \node[blue] at (1.75,0.3) {$\ast$};
    \end{tikzpicture}}
\subcaptionbox{Case $i\in \Asc_\ba(w)$. \label{fig:i-asc}}[.4\textwidth]{
    \begin{tikzpicture}[scale = 0.7]
        \draw (0,0) -- (4,0) -- (4,4) -- (0,4) -- (0,0);
        \draw (1,0) -- (1,4);
        \draw (1.5,0) -- (1.5,4);
        \draw (2,0) -- (2,4);
        \draw[red, very thick] (1,0.7) -- (2,0.7);
        \node at (1.25,4.3) {$i$};
        \node at (1.75,1.7) {$1$};
        \node[blue] at (1.75,2.4) {$\ast$};
        \node at (1.25,3.1) {$1$};
        \node[blue] at (1.25,3.7) {$\ast$};
        \node[blue] at (1.75,3.7) {$\ast$};
        \node[blue] at (1.25,0.3) {$\ast$};
        \node[blue] at (1.75,0.3) {$\ast$};
    \end{tikzpicture}
}
\caption{The $i$-th and $(i+1)$-th row in the canonical representative described in \Cref{prop:tilted-Schub-cell}. The red line indicates $a_i=a_{i+1}$. Zero entries are omitted, and only nonzero entries are shown.}
\label{fig:i-th-column-ast}    
\end{figure}
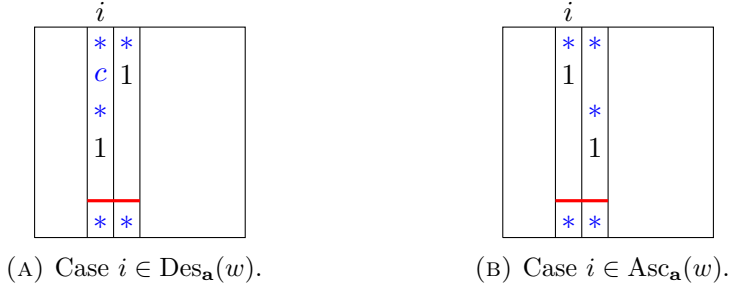
    
    Since $y_1\xrightarrow{s_i}y_2$, we have $y_2=g'B$, where $g'=g\cdot\phi_i\left(\begin{smallmatrix}p&1\\1&0\end{smallmatrix}\right)$ for some $p\in \C$. We analyze the following cases:
    \begin{enumerate}
        \item If $i\in\Des_\ba(w)$ (as in \Cref{fig:i-des}) and $p=0$, then $e\xrightarrow{ws_i,\ba}y_2$ because $y_2=g'B=g'\cdot\phi_i\left(\begin{smallmatrix}1&-c\\0&1\end{smallmatrix}\right)B$, and $g'\cdot\phi_i \left(\begin{smallmatrix}1&-c\\0&1\end{smallmatrix}\right)=g\cdot\phi_i\left(\begin{smallmatrix}0&1\\1&-c\end{smallmatrix}\right)$ is the canonical representative for $y_2\in X_{ws_i,\ba}^\circ$.
        \item If $i\in\Des_\ba(w)$ and $p\neq 0$, then $e\xrightarrow{w,\ba}y_2$ because $y_2=g'B=g'\cdot\phi_i\left(\begin{smallmatrix}1/p&1\\0&-p\end{smallmatrix}\right)B$, and $g'\cdot\phi_i\left(\begin{smallmatrix}1/p&1\\0&-p\end{smallmatrix}\right)=g\cdot\phi_i\left(\begin{smallmatrix}1&0\\1/p&1\end{smallmatrix}\right)$ is the canonical representative for $y_2\in X_{w,\ba}^\circ$.
        \item If $i\in\Asc_\ba(w)$ (as in \Cref{fig:i-asc}), then $e\xrightarrow{ws_i,\ba}y_2$ because $y_2=g'B=g\cdot\phi_i\left(\begin{smallmatrix}p&1\\1&0\end{smallmatrix}\right)B$ is already the canonical representative for $y_2\in X_{ws_i,\ba}^\circ$.
    \end{enumerate}
    This completes the proof of (3). The argument for (4) is analogous. For part (5), existence follows from case (1): when $i \in \Des_\ba(w)$ and $p = 0$, we have $g\dot{s}_iB\in X_{ws_i,\ba}^\circ$. For uniqueness, suppose there exist two such flags $y'$ and $y''$. Then, by part (1), we have $y' \xrightarrow{s_i} y \xrightarrow{s_i} y''$, so by part (2), either $y' = y''$ or $y' \xrightarrow{s_i} y''$. In the latter case, since $e \xrightarrow{w s_i, \ba} y' \xrightarrow{s_i} y''$, applying part (3) yields $e \xrightarrow{w, \ba} y''$, which contradicts the assumption that $e \xrightarrow{w s_i, \ba} y''$.
\end{proof}

We are now ready to define the tilted Deodhar cells. See \Cref{notation:tiltedword} for the notation used in the following proposition.

\begin{prop}\label{prop:tilted-deodhar-cell}
    Suppose $u\lesssim_\ba v$, and let $\mathbf{v}=s_{i_1}s_{i_2}\cdots s_{i_\ell}$ be an $\ba$-tilted reduced word for $v$. For any flag $gB \in \cT_{u,v}^\circ$, there exists a unique distinguished subword $\mathbf{u} \prec \mathbf{v}$ corresponding to $u$, and a unique sequence of flags $g_0 B, g_1 B, \dots, g_\ell B \in \fl_n$ such that the following diagram of tilted relative positions holds:
\[\begin{tikzcd}[row sep=large]
    e\dar[swap]{v^{(0)},\ba^{(0)}} & e\dar[swap]{v^{(1)},\ba^{(1)}} & & e\dar[swap]{v^{(\ell)},\ba^{(\ell)}}\\
    B=g_0B\rar{s_{i_1}} & g_1B\rar{s_{i_2}} & \cdots\rar{s_{i_\ell}} &g_\ell B=gB\\
    w_0\uar{u^{(0)},\ba^{(0)}} & w_0\uar{u^{(1)},\ba^{(1)}} & & w_0\uar{u^{(\ell)},\ba^{(\ell)}}
\end{tikzcd}.\]
    If $j \in J^\mid_\mathbf{v}$, we require that $g_{j-1} B = g_j B$. Define the \emph{tilted Deodhar cell} $\cD_{\mathbf{u}, \mathbf{v}}$ to be the set of $gB \in \cT_{u,v}^\circ$ for which the above diagram is satisfied.
\end{prop}
\begin{proof}
    Our goal is to start from $g_{(\ell)}B=gB$, and iteratively construct a unique sequence of flags $g_\ell B,g_{\ell-1}B,\dots,g_0B$ such that the diagram in the proposition holds. We also verify along the way that $u^{(j)}\lesssim_{\ba^{(j)}}v^{(j)}$ at each step. There are two types of induction steps depending on whether $j\in J_\mathbf{v}^s$ or $j\in J_\mathbf{v}^\mid$, corresponding to the following two diagrams:
    \[
    \begin{tikzcd}[column sep=small]
        & e\dlar[swap]{v^{(j-1)},\ba^{(j)}}\drar{v^{(j)},\ba^{(j)}} & \\
        g_{j-1}B\arrow[rr, "s_{i_j}"] & & g_jB\\
        & w_0\ular{u^{(j-1)},\ba^{(j)}}\urar[swap]{u^{(j)},\ba^{(j)}} &
    \end{tikzcd}\hspace{6ex}
    \begin{tikzcd}[column sep=small]
        & e\dlar[swap]{v^{(j-1)},\flatten(\ba^{(j)})}\drar{v^{(j)},\ba^{(j)}} & \\
        g_{j}B\arrow[rr, equals] & & g_{j}B\\
        & w_0\ular{u^{(j-1)},\flatten(\ba^{(j)})}\urar[swap]{u^{(j)},\ba^{(j)}} &
    \end{tikzcd}.
    \]
    \begin{enumerate}
        \item In the first diagram (where $j\in J_\mathbf{v}^s$), suppose $g_jB$ is given and $u^{(j)}\lesssim_{\ba^{(j)}}v^{(j)}$. Since $\mathbf{v}^{(j)}$ is an $\ba^{(j)}$-tilted reduced word, we have $i_j\in\Des_{\ba^{(j)}}(v^{(j)})$. Then, by \Cref{prop:relative-position} (5), there exists a unique flag $g_{j-1}B\in\fl_n$ making the top triangle commute. By \Cref{prop:relative-position} (4), we have $u^{(j-1)}\in\{u^{(j)}s_{i_j},u^{(j)}\}$. In either case, we have $u^{(j-1)}\sim_{\ba_{(j)}}v^{(j-1)}$. Furthermore, since $g_{j-1}B\in X_{v^{(j-1)},\ba^{(j)}}^\circ\cap \Omega_{u^{(j-1)},\ba^{(j)}}^\circ$, which is nonempty, we conclude $u^{(j-1)}\lesssim_{\ba_{(j)}}v^{(j-1)}$ by \Cref{thm:tilted-Schub-cell-intersect}.
        \item In the second diagram (where $j\in J_\mathbf{v}^\mid$), suppose $g_{j}B$ is given and $u^{(j)}\lesssim_{\ba^{(j)}}v^{(j)}$. Since $\mathbf{v}^{(j)}$ is an $\ba^{(j)}$-tilted reduced word, $v^{(j)}$ is $\ba^{(j)}$-flattenable. Then, by \Cref{lemma:u<flatten(a)v}, $u^{(j)}$ is also $\ba^{(j)}$-flattenable, and $u^{(j)}\lesssim_{\flatten(\ba^{(j)})}v^{(j)}$. The commutative diagram holds because, by \Cref{thm:tilted-Schub-cell-intersect}, we have:
        \[g_{j}B\in X_{v^{(j)},\ba^{(j)}}^\circ\cap \Omega_{u^{(j)},\ba^{(j)}}^\circ=\cT_{u^{(j)},v^{(j)}}^\circ=X_{v^{(j)},\flatten(\ba^{(j)})}^\circ\cap \Omega_{u^{(j)},\flatten(\ba^{(j)})}^\circ.\]
    \end{enumerate}
    Finally, for the last step, we have $v^{(0)}=\id$ and $\ba^{(0)}=(1,1,\dots,1)$, an thus $g_0B=B$ and $u^{(0)}=\id$. The conditions for $\mathbf{u}$ to be a distinguished subword of $\mathbf{v}$ are verified at each step of the induction, and uniqueness follows from \Cref{prop:relative-position} (5).
\end{proof}

Now we are ready to state the main theorem of the tilted Deodhar decomposition.

\begin{theorem}\label{thm:tilted-deodhar}
Suppose $u\lesssim_\ba v$, and let $\mathbf{v}=s_{i_1}s_{i_2}\cdots s_{i_\ell}$ be a regular $\ba$-tilted reduced word for $v$. Then the open tilted Richardson variety $\cT^\circ_{u,v}$ admits a decomposition into tilted Deodhar cells $\cD_{\mathbf{u},\mathbf{v}}$, indexed by tilted distinguished subwords $\mathbf{u}\prec\mathbf{v}$ corresponding to $u$:
    \[\cT^\circ_{u,v}=\bigsqcup_{\mathbf{u}\prec \mathbf{v}}\cD_{\mathbf{u},\mathbf{v}},\quad\text{where }\cD_{\mathbf{u},\mathbf{v}}\cong (\C^\ast)^{|J_\mathbf{u}^\circ|}\times \C^{|J_\mathbf{u}^-|}.\]
    Each Deodhar cell $\cD_{\mathbf{u},\mathbf{v}}$ admits the following explicit parametrization:
    \[\cD_{\mathbf{u},\mathbf{v}}\cong\left\{gB=g_1g_2\cdots g_\ell B\;\middle\vert\;\begin{aligned}g_j&=\dot{s}_{i_j} &&\text{if }j\in J_\mathbf{u}^+,\\ g_j&=y_{i_j}(p_j) &&\text{if }j\in J_\mathbf{u}^\circ,\\ g_j&=x_{\alpha_i}(m_j)\dot{s}_{i_j}^{-1} &&\text{if }j\in J_\mathbf{u}^-.\end{aligned}\right\},\]
    where $p_j\in \C^\ast$ and $m_j\in \C$ are parameters.
\end{theorem}
\begin{proof}
    It follows immediately from \Cref{prop:tilted-deodhar-cell} that $\cT_{u,v}^\circ$ is the disjoint union of tilted Deodhar cells $\cD_{\mathbf{u},\mathbf{v}}$ where $\mathbf{u}\prec\mathbf{v}$. It remains to prove the explicit parametrization.

    We begin by introducing a definition needed for the argument. Define the \emph{standard chart} $G_{u,\ba}\subseteq G$ to be the set of matrices $M=(m_{i,k})_{i,k\in [n]}$ such that
    \[m_{i,k}=\begin{cases}
        \pm1 &\text{if }i=u_k,\\
        0 &\text{if }i<_{a_k}u_k.
    \end{cases}\]
    
    Let $G_{\mathbf{u},\mathbf{v}}\subseteq G$ denote the candidate parametrization for the cell $\cD_{\mathbf{u},\mathbf{v}}$ in the theorem. We proceed by induction on $\ell^\word_\ba(v)$ to show that the map $G_{\mathbf{u},\mathbf{v}}\to \cD_{\mathbf{u},\mathbf{v}}$ is an isomorphism and that $G_{\mathbf{u},\mathbf{v}}\subseteq G_{u,\ba}$. Let $\mathbf{v}'$ and $\mathbf{u}'$ be the words obtained from $\mathbf{v}$ and $\mathbf{u}$ by removing the last factor. Let $u'\in S_n$ and $v'\in S_n$ be the corresponding permutations. By the induction hypothesis, we have $G_{\mathbf{u}',\mathbf{v}'}\cong \cD_{\mathbf{u}',\mathbf{v}'}$ and $G_{\mathbf{u}',\mathbf{v}'}\subseteq G_{u',\ba'}$. We consider three cases:
    \begin{enumerate}
        \item Suppose $\mathbf{v}$ ends with $s_i$ and $i\in \Des_\ba(u')$, so $\mathbf{u}=\mathbf{u}'s_i$. By \Cref{prop:relative-position}, any flags $F_\bullet$ such that $F_\bullet\xrightarrow{s_i}F'_\bullet$ for some $F'_\bullet\in \cD_{\mathbf{u}', \mathbf{v}'}$ must lie in $\cD_{\mathbf{u},\mathbf{v}}$. We thus obtain the following cartesian square:
        \[\begin{tikzcd}
            G_{\mathbf{u}',\mathbf{v}'}\times\{x_i(\C)\cdot \dot{s}_i\}\rar{\sim}\dar{pr_1} & \cD_{\mathbf{u},\mathbf{v}}\dar{\pi^{v,\ba}_{vs_i}}\\
            G_{\mathbf{u}',\mathbf{v}'}\rar{\sim} & \cD_{\mathbf{u}', \mathbf{v}'}.
        \end{tikzcd}.\]
        It is straightforward to check that $G_{\mathbf{u},\mathbf{v}}=G_{\mathbf{u}',\mathbf{v}'}\times\{x_i(\C)\cdot \dot{s}_i\}\subseteq G_{us_i,\ba}$.
        \item Suppose $\mathbf{v}$ ends with $s_i$ and $i\in \Asc_\ba(u')$, so $\mathbf{u}\in\{\mathbf{u}'s_i,\mathbf{u}'1\}$. By \Cref{prop:relative-position}, among all flags $F_\bullet$ such that $F_\bullet\xrightarrow{s_i}F'_\bullet$ for some  $F'_\bullet\in \cD_{\mathbf{u}', \mathbf{v}'}$, exactly one lies in $\cD_{\mathbf{u}'s_i,\mathbf{v}}\subseteq \cT_{us_i,v}^\circ$, and the rest lie in $\cD_{\mathbf{u}'1,\mathbf{v}}\subseteq \cT_{u,v}^\circ$. To show the following two cartesian squares:
        \[\begin{tikzcd}
            G_{\mathbf{u}',\mathbf{v}'}\times\{y_i(\C^\ast)\}\rar{\sim}\dar{pr_1} & \cD_{\mathbf{u}'1,\mathbf{v}}\dar{\pi^{v,\ba}_{vs_i}}\\
            G_{\mathbf{u}',\mathbf{v}'}\rar{\sim} & \cD_{\mathbf{u}',\mathbf{v}'}
        \end{tikzcd}\hspace{3ex}
        \begin{tikzcd}
            G_{\mathbf{u}',\mathbf{v}'}\times\{\dot{s}_i\}\rar{\sim}\dar{pr_1} & \cD_{\mathbf{u}'s_i,\mathbf{v}}\dar{\pi^{v,\ba}_{vs_i}}\\
            G_{\mathbf{u}',\mathbf{v}'}\rar{\sim} & \cD_{\mathbf{u}',\mathbf{v}'}
        \end{tikzcd},\]
        it suffices to show that all flags $F_\bullet$ represented by matrices $M_F\in G_{\mathbf{u}',\mathbf{v}'}\times\{\dot{s}_i\}$ belong to $\cT_{us_i,v}^\circ$ but not to $\cT_{u,v}^\circ$. Since $G_{\mathbf{u}',\mathbf{v}'}\subseteq G_{u',\ba}$, the entries of $M_F$ in the $i$-th column and in rows indexed by $[a_i,u_{i+1})_c$ are all zero. Hence,
        \[\rank_{[a_i,u_{i+1})_c}(F_i)=\rank_{[a_i,u_{i+1})_c}(F_{i-1}).\]
        By the rank condition in \Cref{def:main}, this implies that $F_\bullet$ belongs to $\cT_{us_i,v}^\circ$ but not to $\cT_{u,v}^\circ$. Finally, it is clear that $G_{\mathbf{u}'1,\mathbf{v}} =G_{\mathbf{u}',\mathbf{v}'}\times\{y_i(\C^\ast)\}\subseteq G_{u,\ba}$ and $G_{\mathbf{u}'s_i,\mathbf{v}}=G_{\mathbf{u}',\mathbf{v}'}\times \{\dot{s}_i\}\subseteq G_{us_i,\ba}$.
        \item Suppose $\mathbf{v}$ ends with a bar, so $\mathbf{u}=\mathbf{u}'\mid$. We need to show that for all flags $F_\bullet\in \cD_{\mathbf{u},\mathbf{v}}$ represented by $M_F=g_1\cdots g_{\ell-1}\in G_{u,\flatten(\ba)}$, we in fact have $M_F\in G_{u,\ba}$. Since $v$ is $\ba$-flattenable, let $p$ be the integer from \Cref{def:a-flattenable} such that $v[p]\subseteq[a_1,a_{\jmin+1})_c$ and $v[p+1,\jmin]\subseteq[a_{\jmin+1},a_1)_c$. Comparing the two standard charts $G_{u,\flatten(\ba)}$ and $G_{u,\ba}$, it suffices to show that all entries in the submatrix of $M_F$ consisting of rows indexed by $[a_1,a_{\jmin+1})_c$ and columns indexed by $[p+1,\jmin]$ are zero.
        
        Since $\mathbf{v}$ is regular, it follows from \Cref{lemma:regular-subword} that $\mathbf{u}$ is also regular. Then, by \Cref{def:regular-word} we have $M_F=M_{F'}\cdot g_{k+1}\cdots g_{\ell-1}$, where $g_{k+1}\cdots g_{\ell-1}$ is the permutation matrix corresponding to the bi-Grassmannian permutation $s_{\jmin-p,p}$, and $M_{F'}$ represents the flag $F'_\bullet\in \cD_{u^{(k)},v^{(k)}}\subseteq \cT^\circ_{u^{(k)},v^{(k)}}$. Therefore, the desired statement on $M_F$ reduces to showing that all entries in the submatrix of $M_{F'}$ consisting of rows indexed by $[a_1,a_{\jmin+1})_c$ and columns indexed by $[\jmin-p]$ are zero. This follows from the rank condition $\rank_{[a_1,a_{\jmin+1})_c}(F'_{\jmin-p})=0$, as required by the definition of $\cT_{u^{(k)},v^{(k)}}^\circ$ in \Cref{def:main}.
    \end{enumerate}
    This completes the induction step and the proof.
\end{proof}

\begin{remark}
    We note that the regularity condition on the $\ba$-tilted reduced word $\mathbf{v}$ is necessary in the current proof of \Cref{thm:tilted-deodhar}. However, we conjecture that \Cref{thm:tilted-deodhar} holds for any $\ba$-tilted reduced word $\mathbf{v}$, even without assuming that $\mathbf{v}$ is regular.
\end{remark}

\subsection{Irreducibility of Tilted Richardson Varieties}\label{sec:irreducible}

In this section, we use the tilted Deodhar decomposition to prove the irreducibility of the tilted Richardson varieties $\cT_{u,v}^\circ$ and $\cT_{u,v}$ (\Cref{thm:closure}). We begin with the following definition.

\begin{defin}
    Let $\mathbf{v}=s_{i_1}s_{i_2}\cdots s_{i_\ell}$ be an $\ba$-tilted reduced word for $v\in S_n$. A distinguished subword $\mathbf{u}\prec\mathbf{v}$ is called a \emph{positive distinguished subword} if $J^-_\mathbf{u}=\emptyset$. We denote such a subword by $\mathbf{u}^+$.
\end{defin}

The following lemma identifies the unique positive distinguished subword.

\begin{lemma}\label{lemma:tilted-pds}
    Suppose $u\lesssim_\ba v$, and let $\mathbf{v}=s_{i_1}s_{i_2}\cdots s_{i_\ell}$ be an $\ba$-tilted reduced word for $v$. Then there exists a unique positive distinguished subword $\mathbf{u}^+\prec\mathbf{v}$ corresponding to $u$.
\end{lemma}
\begin{proof}
    By definition, if $\mathbf{u}$ is a positive distinguished subword, its last factor $\mathbf{u}_\ell$ must satisfy
    \[\mathbf{u}_\ell=\begin{cases}
        s_{i_\ell} &\text{if } i_\ell\in\Des_\ba(u),\\
        1 &\text{if }i_\ell\in \Asc_\ba(u).
    \end{cases}\]
    Therefore, we can construct such a subword $\mathbf{u}^+$ inductively from right to left. This process is deterministic and produces a unique result.
\end{proof}

We now state the main result of this section.

\begin{theorem}\label{thm:irreducible}
    For $u,v\in S_n$, the tilted Richardson varieties $\cT_{u,v}^\circ$ and $\cT_{u,v}$ are irreducible.
\end{theorem}
\begin{proof}
    Let $\ba\in[n]^n$ be a sequence such that $u\lesssim_\ba v$, whose existence is guaranteed by \Cref{thm:lesssim-exist}. Let $\mathbf{v}$ be a regular $\ba$-tilted reduced subword for $v$. By \Cref{thm:tilted-deodhar}, 
    \[\cT^\circ_{u,v}=\bigsqcup_{\mathbf{u}\prec \mathbf{v}}\cD_{\mathbf{u},\mathbf{v}} \implies \overline{\cT^\circ_{u,v}} =\bigcup_{\mathbf{u}\prec \mathbf{v}}\overline{\cD_{\mathbf{u},\mathbf{v}}}.\]
    Let $Z\subseteq \overline{\cT^\circ_{u,v}}$ be any irreducible component. Then $Z$ must be contained in one of the closures $\overline{\cD_{\mathbf{u},\mathbf{v}}}$. However, by \Cref{thm:dimension}, we know that $\dim(Z)=\ell(u,v)$, and the dimension of each tilted Deodhar cell is given by
    \[\dim(\cD_{\mathbf{u},\mathbf{v}})=|J_\mathbf{u}^\circ|+|J_\mathbf{u}^-|=\ell^\word_\ba(v)-\ell^\word_\ba(u)-|J_\mathbf{u}^-|=\ell(u,v)-|J_\mathbf{u}^-|\leq \ell(u,v).\]
    By \Cref{lemma:tilted-pds}, there exists a unique Deodhar cell $\cD_{\mathbf{u}^+,\mathbf{v}}$ of maximal dimension $\ell(u,v)$. Therefore, $Z\subseteq \overline{\cD_{\mathbf{u}^+,\mathbf{v}}}$. Since this holds for every irreducible component $Z$, we conclude that $\overline{\cT_{u,v}^\circ}= \overline{\cD_{\mathbf{u}^+,\mathbf{v}}}$. As $\cD_{\mathbf{u}^+,\mathbf{v}} \cong(\C^\ast)^{\ell(u,v)}$ is irreducible, it follows that $\cT_{u,v}^\circ$ is irreducible. Finally, by \Cref{thm:closure}, the closure $\cT_{u,v}=\overline{\cT_{u,v}^\circ}$ is also irreducible.
\end{proof}

\subsection{Tilted Kazhdan--Lusztig R-Polynomials}\label{sec:tilted-r-poly}

In this section, we define and study the \emph{tilted Kazhdan--Lusztig R-polynomials} $R^\tilt_{u,v}(q)$ (or simply \emph{tilted R-polynomials}), which count the number of $\F_q$-points in the open tilted Richardson variety $\cT_{u,v}^\circ$. These polynomials generalize the classical Kazhdan--Lusztig R-polynomials defined in \cite{KL79} when $u\leq v$ in the Bruhat order. See \Cref{sec:prelim-4-1} for a discussion of the classical case. We present three equivalent formulas for this new family of polynomials (\Cref{prop:r-poly-1}, \Cref{prop:r-poly-2}, and \Cref{thm:hecke}) and conclude with several open questions. We begin with a direct combinatorial formula, which follows immediately as a corollary of \Cref{prop:tilted-deodhar-cell}.

\begin{prop}\label{prop:r-poly-1}
    Let $u,v\in S_n$. The tilted R-polynomial $R^\tilt_{u,v}(q)$ is defined as
    \[R^\tilt_{u,v}(q):=\#\cT_{u,v}^\circ(\F_q).\]
    Then, for any $\ba\in[n]^n$ such that $u\lesssim_\ba v$, and any $\ba$-tilted reduced word $\mathbf{v}$ for $v$, we have
    \[R^\tilt_{u,v}(q)=\sum_{\mathbf{u}\prec\mathbf{v}}(q-1)^{|J_\mathbf{u}^\circ|}q^{|J_{\mathbf{u}}^-|}.\]
\end{prop}

We present a recursive formula, analogous to the original definition of classical R-polynomials.
\begin{prop}\label{prop:r-poly-2}
    Let $\left\{R_{u,v,\ba}(q):u,v\in S_n,\;\ba\in[n]^n\right\}$ be the family of polynomial defined recursively as follows:
    \begin{enumerate}
    \item $R_{u,v,\ba}(q)=1$ if $u=v$;
    \item $R_{u,v,\ba}(q)=0$ if $u\not\lesssim_\ba v$;
    \item If $u\lesssim_\ba v$ and $i\in \Des_\ba(v)$, then
    \[R_{u,v,\ba}(q)=\begin{cases}
        R_{us_i,vs_i,\ba}(q), &\text{if }i\in\Des_\ba(u),\\
        q\cdot R_{us_i,vs_i,\ba}(q)+(q-1)\cdot R_{u,vs_i,\ba}(q), &\text{if }i\in\Asc_\ba(u).
    \end{cases}\]
    \item If $u\lesssim_\ba v$ and $\Des_\ba(v)=\emptyset$, then $R_{u,v,\ba}(q)=R_{u,v,\flatten(\ba)}(q)$.
    \end{enumerate}
    Then, $R^\tilt_{u,v}(q)=R_{u,v,\ba}(q)$ for any $\ba$ such that $u\lesssim_\ba v$.
\end{prop}
\begin{proof}
    This follows immediately from \Cref{prop:r-poly-1} and the recursive structure of tilted distinguished subwords.
\end{proof}

Finally, we provide an alternative formula for $R^\tilt_{u,v}(q)$ using the Hecke algebra $\cH$. See \Cref{sec:prelim-4-1} for the definition and properties of the Hecke algebra $\cH$. For a tilted word $\mathbf{w}=s_{i_1}s_{i_2}\cdots s_{i_\ell}$, define the corresponding element of $\cH$ by
\[T_\mathbf{w}:=T_{i_1}T_{i_2}\cdots T_{i_\ell}.\]

\begin{theorem}\label{thm:hecke}
    Suppose $u\lesssim_\ba v$. Let $\mathbf{u}$ and $\mathbf{v}$ be $\ba$-tilted reduced words for $u$ and $v$, respectively. Then
    \[R^\tilt_{u,v}(q) = q^{\ell(u,v)}\cdot \epsilon(T_{\mathbf{v}}^{-1} T_{\mathbf{u}}),\]
    where $\epsilon$ is the trace map.
\end{theorem}

We begin with the following technical lemma.
\begin{lemma}\label{lemma:heckeconjugate}
    Let $x=s_{a,b}\in S_{a+b}$ be the bi-Grassmannian permutation defined in \Cref{def:bigrass}. For any permutation $w\notin S_b\times S_a$, we have
    \[T_xT_wT_x^{-1} = \sum_{w'\notin S_a\times S_b} c_{w'} T_{w'}\quad\text{where each }c_{w'}\in \Z[q,q^{-1}].\]
\end{lemma}
\begin{proof}
    Suppose there exists $y\in S_a\times S_b$ such that $T_y$ appears in the expansion of $T_xT_wT_x^{-1}$. Without loss of generality, assume such a $y$ is miminal in the Bruhat order. By \Cref{prop:hecke}, it follows that $\epsilon(T_xT_wT_x^{-1}T_y^{-1}) \neq 0$. Therefore, to prove the original statement, it suffices to show that
    \[\epsilon(T_xT_wT_x^{-1}T_y^{-1}) = 0 \quad\text{for all } y\in S_a\times S_b.\]
    Since $x=s_{a,b}$ permutes the block $S_a$ and $S_b$, we have $T_yT_x = T_{y'}T_x$ for some $y'\in S_b\times S_a$. Since the trace is invariant under conjugation, we get
    \[\epsilon(T_xT_wT_x^{-1}T_y^{-1})= \epsilon(T_xT_wT_{y'}^{-1}T_x^{-1})=\epsilon(T_wT_{y'}^{-1}).\]
    However, since $w\notin S_b\times S_a$ and $y'\in S_b\times S_a$, it is impossible for $w\leq y'$, so $\epsilon(T_wT_{y'}^{-1})=0$, which completes the proof.
\end{proof}

\begin{proof}[Proof of \Cref{thm:hecke}]
    By the word property in \Cref{thm:word-property}, the elements $T_\mathbf{v}$ and $T_\mathbf{u}$ are independent of the choice of the $\ba$-tilted reduced words, so we denote them by $T_{v,\ba}$ and $T_{u,\ba}$, respectively. We will prove that the right-hand side of the theorem satisfies the same recursive relations as in \Cref{prop:r-poly-2}. More precisely, we show that for any $u\sim_\ba v$,
    \[R_{u,v,\ba}(q)=q^{\ell^\word_\ba(v)-\ell^\word_\ba(u)}\cdot\epsilon(T_{\mathbf{v},\ba}^{-1}T_{\mathbf{u},\ba}),\]
    by verifying all recurrence relations from \Cref{prop:r-poly-2}. Recurrence (1) is immediate.
    
    For recurrence (3), since $i\in \Des_\ba(v)$, there is an $\ba$-tilted reduced word $\mathbf{v}$ for $v$ ending with $s_i$, so $T_{v,\ba} = T_{vs_i,\ba}T_i$. There are two cases to consider:
    \begin{itemize}
        \item If $i\in \Des_\ba(u)$, then there is an $\ba$-tilted reduced word $\mathbf{u}$ of $u$ ending with $s_i$, so $T_{u,\ba} = T_{us_i,\ba}T_i$. Since conjugating by $T_i$ preserves the trace, we have
        \[\epsilon(T_{v,\ba}^{-1}T_{u,\ba}) = \epsilon(T_i^{-1}T_{vs_i,\ba}^{-1}T_{us_i,\ba}T_i)=\epsilon(T_{vs_i,\ba}^{-1}T_{us_i,\ba}),\]
        which proves the recurrence.
        \item If $i\in \Asc_\ba(u)$, then using the Hecke relation $T_i^{-1}=\frac{q-1}{q}+\frac{1}{q}T_i$, we obtain
        \[\epsilon(T_{v,\ba}^{-1}T_{u,\ba}) = \epsilon(T_{vs_i,\ba}^{-1}T_{u,\ba}T_i^{-1})= \frac{q-1}{q}\epsilon(T_{vs_i,\ba}^{-1}T_{u,\ba}) + \frac{1}{q}\epsilon(T_{vs_i,\ba}^{-1}T_{us_i,\ba}).\]
        The recurrence holds after multiplying both sides by the appropriate power of $q$.
    \end{itemize}
    
    For recurrence (4), if $\Des_\ba(v) = \emptyset$, then $v$ is $\ba$-flattenable. Since $u\lesssim_\ba v$, it follows from \Cref{lemma:u<flatten(a)v} that $u$ is also $\ba$-flattenable. Then by \Cref{lemma:tildef(u)}, we have $T_{v,\flatten(\ba)}=T_{v,\ba}$ and $T_{u,\flatten(\ba)}=T_{u,\ba}$, so the recurrence holds.

    It remains to prove recurrence (2), which states that if $u\sim_\ba v$ but $u\not\lesssim_\ba v$, then $\epsilon(T_{v,\ba}^{-1}T_{u,\ba})=0$. We proceed by induction on $\ell_{\ba}^\word(v)$. In the base case $\ell^\word_\ba(v)=0$, which corresponds to $\ba=(1,1,\dots,1)$ and $v=\id$, we must have $u\neq \id$, and the statement is immediate. We now apply two reductions:

    \noindent \textit{Reduction to $v$ is $\ba$-flattenable:} If $v$ is not $\ba$-flattenable, then $\Des_\ba(v)\neq\emptyset$. Take any $i\in \Des_\ba(v)$. By recurrence (4), we have:
    \[\epsilon(T_{v,\ba}^{-1}T_{u,\ba}) = \begin{cases}
        \epsilon(T_{vs_i,\ba}^{-1}T_{us_i,\ba}) & \text{if }i\in\Des_\ba(u),\\
        \frac{q-1}{q}\epsilon(T_{vs_i,\ba}^{-1}T_{u,\ba}) + \frac{1}{q}\epsilon(T_{vs_i,\ba}^{-1}T_{us_i,\ba})&\text{if }i\in\Asc_\ba(u).
    \end{cases}.
    \]
    In either case, the right-hand side vanishes by the induction hypothesis, so $\epsilon(T_{v,\ba}^{-1}T_{u,\ba})=0$.

    \noindent \textit{Reduction to $u$ is not $\ba$-flattenable:} Suppose both $u$ and $v$ are $\ba$-flattenable. Then recurrence (3) gives:
    \[\epsilon(T_{v,\ba}^{-1}T_{u,\ba}) =\epsilon(T_{v,\flatten(\ba)}^{-1}T_{u,\flatten(\ba)}).\]
    Since $u\not\lesssim_\ba v$, we have $u\not\lesssim_{\flatten(\ba)}v$ by \Cref{lemma:u<flatten(a)v}. The induction hypothesis then implies that the right-hand side is zero, so $\epsilon(T_{v,\ba}^{-1}T_{u,\ba})=0$.

    Using the two reductions above, we may assume that $v$ is $\ba$-flattenable but $u$ is not. Let $\mathbf{v} = \mathbf{v}'\mathbf{x}\mid$ be a regular $\ba$-tilted reduced word for $v$, where $x=s_{p,q}$ is the bi-Grassmannian permutation as in \Cref{def:regular-word}. Let $\mathbf{u} = \mathbf{u}'\mathbf{x}\mid\mathbf{u}^\circ$ be the regular $\ba$-tilted reduced word for $u$ constructed in \Cref{construction:regular-tiltedreducedword}, such that the first $\jmin$ entries of $u'$ are increasing under the order $<_{a_{\jmin+1}}$. Since $u$ is not $\ba$-flattenable, we have $u^\circ\notin S_q\times S_p$. By \Cref{lemma:heckeconjugate},
    \[\epsilon(T_{v,\ba}^{-1} T_{u,\ba}) =\epsilon(T_x^{-1}T_{\mathbf{v}'}^{-1} T_{\mathbf{u}'} T_x T_{u^\circ}) = \epsilon(T_{\mathbf{v}'}^{-1} T_{\mathbf{u}'} T_x T_{u^\circ}T_x^{-1}) = \sum_{w\in S_{p+q}\setminus S_p\times S_q} c_{w} \; \epsilon(T_{\mathbf{v}'}^{-1}T_{\mathbf{u}'}T_{w}).\]
    It remains to show that every term in the above sum vanishes. By construction, $\mathbf{v}'$ is a $\flatten(\ba)$-reduced word for $v'$, and $\mathbf{u}'\mathbf{w}$ is a $\flatten(\ba)$-reduced word for $u'w$. Since both $v'x$ and $u'x$ are $\ba$-flattenable, we have:
    \[
        v'[p],u'[p] \subseteq [a_{\jmin+1}, a_1)_c,\quad
        v'[p+1,p+q],u'[p+1,p+q] \subseteq [a_1, a_{\jmin+1})_c.
    \]
    However, since $w\notin S_p\times S_q$, we have $u'w[p]\not\subseteq[a_{\jmin+1}, a_1)_c$, which implies $u'w[p]\not\leq_{a_{\jmin+1}}v'[p]$, and thus $u'w\not\lesssim_{\flatten(\ba)}v'$. By the induction hypothesis,
    \[\epsilon(T_{\mathbf{v}'}^{-1}T_{\mathbf{u}'}T_{w})=\epsilon(T_{v',\flatten(\ba)}^{-1}T_{u'w,\flatten(\ba)})=0,\]
    as desired.
\end{proof}

There are several natural questions related to the tilted R-polynomials. 
\begin{conj}\label{conj:tiltedKLconj}
Let $u,v,u'v'\in S_n$ be such that the tilted Bruhat intervals $[u,v]\cong [u',v']$ are isomorphic as posets, then $R^\tilt_{u,v}(q)=R^\tilt_{u',v'}(q)$.
\end{conj}
In particular, in the case where $u\leq v$ and $u'\leq v'$ in strong Bruhat order, \Cref{conj:tiltedKLconj} reduces to the combinatorial invariance conjecture (see e.g., \cite{Dyer-hecke, combinatorial-invariance}) in type $A$. 

The next conjecture is an analogue of \cite[Corollary~3.4]{Dyer-hecke}.
\begin{conj}
For a tilted Bruhat interval $[u, v]$, define the \emph{tilted Bruhat interval graph} as the directed graph with vertex set $[u, v]$ and an edge $w \xrightarrow{e_i - e_j} wt_{ij}$ whenever $w \prec wt_{ij}$ in the poset $[u, v]$. Here, recall that $w$ and $wt_{ij}$ are always comparable by \Cref{cor:tilted-bruhat-graph}. For any reflection ordering $\gamma$ of the set of positive roots $\Phi^+$, we have
\[R^\tilt_{u,v}(q)=q^{\frac{\ell(u,v)}{2}}\sum_{P:u\to v}\left(q^{\frac{1}{2}}-q^{-\frac{1}{2}}\right)^{\ell(P)},\]
where the sum is over all directed paths $P$ from $u$ to $v$ in the tilted Bruhat interval graph whose sequence of edge labels is strictly increasing with respect to $\gamma$ and $\ell(P)$ denotes the number of edges in the path $P$.
\end{conj}

\subsection{Total Positivity of Tilted Richardson Varieties}\label{sec:total-positive}

In this section, we define the \emph{totally nonnegative parts} of tilted Richardson varieties (\Cref{def:TNNtilted}), provide an explicit parametrization (\Cref{thm:tilted-deodhar-positive}), and show that these spaces form CW-complexes (\Cref{thm:CW-tilted}). We begin with the following definition.

\begin{defin}\label{def:TNNflag}
    Let $\ba\in[n]^{n}$ be a sequence. The \emph{$\ba$-tilted totally nonnegative flag variety} is the subset $\fl_n^{\ba,\geq0}\subseteq \fl_n(\R)$ consisting of all real flags $F_\bullet$ such that the Pl\"ucker coordinates
    \[\{\Delta_{i_1,\dots,i_k}(F_\bullet):i_1<_{a_k}<i_2<_{a_k}\cdots <_{a_k}i_k\}\]
    are either all nonnegative or all nonpositive, for every $k\in [n]$.
\end{defin}
\begin{ex}
    If $n=4$ and $\ba=(4,3,2,2)$. Then a flag $F_\bullet\in\fl_n(\R)$ lies in $\fl_n^{\ba,\geq0}$ if and only if the following groups of Pl\"ucker coordinates of $F_\bullet$ satisfy:
    \begin{align*}
        \Delta_4,\Delta_1,\Delta_2,\Delta_3&\quad \text{have the same sign},\\ \Delta_{34},\Delta_{31},\Delta_{32},\Delta_{41},\Delta_{42},\Delta_{12}&\quad \text{have the same sign},\\ \Delta_{234},\Delta_{231},\Delta_{241},\Delta_{341}&\quad \text{have the same sign}.
    \end{align*}
\end{ex}

We now define the totally nonnegative parts of tilted Richardson varieties.
\begin{defin}\label{def:TNNtilted}
Suppose $u\leq_\ba v$. The \emph{totally nonnegative parts} of the tilted Richardson varieties are defined as
\[\cT_{u,v}^{> 0}:=\cT_{u,v}^\circ\cap\fl_n^{\ba,\geq0},\quad\cT_{u,v}^{\geq 0}:=\cT_{u,v}\cap\fl_n^{\ba,\geq0}.\]
This definition is independent of the choice of $\ba$.
\end{defin}
\begin{prop}
\Cref{def:TNNtilted} is independent of the choice of $\ba$ so that the totally nonnegative parts are well-defined as above.
\end{prop}
\begin{proof}
    We must show that for any two sequences $\ba$ and $\ba'$ such that $u\leq_\ba v$ and $u\leq_{\ba'} v$, the following conditions are equivalent:
    \begin{align*}
    \{\Delta_{i_1,\dots,i_k}(F_\bullet):i_1<_{a_k}<i_2<_{a_k}\cdots <_{a_k}i_k\}&\quad\text{have the same sign},\\
    \{\Delta_{j_1,\dots,j_k}(F_\bullet):j_1<_{a'_k}<j_2<_{a'_k}\cdots <_{a'_k}j_k\}&\quad\text{have the same sign}.
    \end{align*}
    Since $u[k]\leq_{a_k}v[k]$ and $u[k]\leq_{a'_k}v[k]$, by \Cref{lemma:both-r-comparable}, a Pl\"ucker coordinate $\Delta_I(F_\bullet)\neq 0$ for a $k$-element subset $I\subseteq [n]$ only if
    \[\size{I\cap[a_k,a'_k)_c}=\size{u[k]\cap[a_k,a'_k)_c}=\size{v[k]\cap[a_k,a'_k)_c}=d.\]
    If we sort the elements of $I$ under both the shifted orders $<_{a_k}$ and $<_{a'_k}$, we obtain
    \[I=\{i_1<_{a_k}<i_2<_{a_k}\cdots <_{a_k}i_k\}=\{j_1<_{a'_k}<j_2<_{a'_k}\cdots <_{a'_k}j_k\}.\]
    Then, the signs of the Plücker coordinates satisfy:
    \[\Delta_{i_1,\dots,i_k}(F_\bullet)=(-1)^{d(k-d)}\Delta_{j_1,\dots,j_k}(F_\bullet)\]
    so the two sign conditions are equivalent. This completes the proof.
\end{proof}

The totally nonnegative parts $\cT_{u,v}^{> 0}$ are isomorphic to open cells, and they admit explicit parametrizations closely related to the tilted Deodhar decomposition. We state this parametrization in the following theorem.
\begin{theorem}\label{thm:tilted-deodhar-positive}
    Suppose $u\lesssim_\ba v$, and let $\mathbf{v}=s_{i_1}s_{i_2}\cdots s_{i_\ell}$ be a regular $\ba$-tilted reduced word for $v$. Let $\mathbf{u}^+$ be the unique positive distinguished subword of $\mathbf{v}$ corresponding to $u$. Then the totally nonnegative part of the tilted Richardson variety $\cT_{u,v}^{> 0}$ admits the following parametrization:
    \[\cT_{u,v}^{>0}\cong\left\{gB=g_1g_2\cdots g_\ell B\;\middle\vert\;\begin{aligned}g_j&=\dot{s}_{i_j}&&\text{if }j\in J_{\mathbf{u}^+}^+,\\ g_j&=y_{i_j}(\pm p_j),\text{ with }p_j\in\R_{>0} &&\text{if }j\in J_{\mathbf{u}^+}^\circ.\end{aligned}\right\},\]
    where the sign $\pm$ in each term $y_j(\pm p_j)$ is fixed in advance and depends only on $\mathbf{v}$. Consequently, $\cT_{u,v}^{> 0}\cong (\R_{>0})^{\ell(u,v)}$, and in particular, it is homeomorphic to an open ball.
\end{theorem}

\begin{ex}\label{ex:tilted-positive}
    Let $u=4231$, $v=3142$, and let $\ba=(4,4,2,2)$ be a sequence such that $u\lesssim_\ba v$. A pair consisting of a regular $\ba$-tilted reduced word for $v$ and the corresponding positive distinguished subword for $u$ is given by:
    \begin{align*}
        \mathbf{v}&=s_1s_2s_3\mid s_1s_2s_3s_2s_1\mid s_1,\\
        \mathbf{u}&=s_1s_2s_3\mid 111s_2s_1\mid 1.
    \end{align*}
    Then, the parametrization of $\cT_{u,v}^{>0}$ as described in \Cref{thm:tilted-deodhar-positive} is given by:
    \[\dot{s}_1\dot{s}_2\dot{s}_3y_1(a)y_2(b)y_3(-c)\dot{s}_2\dot{s}_1y_1(-d)=\begin{pmatrix}
        c & 0 & 0 & -1\\
        d & -1 & 0 & 0\\
        ad & -a & -1 & 0\\
        1 & 0 & -b & 0
    \end{pmatrix},\]
    where $a,b,c,d\in\R_{>0}$.
\end{ex}

To prove \Cref{thm:tilted-deodhar-positive}, we present the following lemma, which reveals the local structure of the totally nonnegative flag variety. For any $k,r\in [n]$, define the \emph{$r$-tilted totally nonnegative Grassmannian} as the subset $\gr_{k,n}^{r,\geq 0}\subseteq \gr_{k,n}(\R)$ consisting of all $k$-dimensional subspaces $V\subseteq \R^n$ such that the Pl\"ucker coordinates
\[\{\Delta_{i_1,\dots,i_k}(V):i_1<_{r}<i_2<_{r}\cdots <_{r}i_k\}\]
are all nonnegative or nonpositive.

\begin{lemma}\label{lemma:local-positive}
    Let $U\in \gr_{k-1,n}(\R)$ and $W\in \gr_{k+1,n}(\R)$ be two linear subspaces such that $U\subseteq W$. Define the set of all $k$-dimensional subspaces between $U$ and $W$ as
    \[\P(U,W):=\{V\in \gr_{k,n}(\R): U\subseteq V \subseteq W\}\subseteq \gr_{k,n}(\R).\]
    It is clear that $\P(U,W)\cong\P^1(\R)\cong S^1$. Fix $r\in [n]$. Then:
    \begin{enumerate}
        \item There exists a unique subspace $V_1\in \P(U,W)$ such that the data $\left\{\rank_{[i,r)_c}(V_1):i\in[n]\right\}$ is distinct from that of any other $V\in \P(U,W)$. Similarly, there exists a unique subspace $V_2\in \P(U,W)$ such that the data $\left\{\rank_{[r,i)_c}(V_2):i\in[n]\right\}$ is distinct from that of any other $V\in \P(U,W)$.
        \item Suppose $U\in\gr_{k-1,n}^{r,\geq0}$ and $W\in\gr_{k+1,n}^{r,\geq0}$. Then among all $V\in \P(U,W)$, the subset of points lying in $\gr_{k,n}^{r,\geq 0}$ falls into one of the four configurations illustrated below, where red and solid markers indicate membership in $\gr_{k,n}^{r,\geq 0}$:
        \[\begin{tikzpicture}[scale = 0.75]
        \draw[color=red] (0,1) arc (90:270:1);
        \draw[style=dashed] (0,-1) arc (-90:90:1);
        \node[circle, fill=red, inner sep=2pt, label=above:$V_1$] at (0,1) {};
        \node[circle, fill=red, inner sep=2pt, label=above:$V_2$] at (0,-1) {};
        \draw[style=dashed] (3,0) circle (1);
        \node[circle, draw=black, inner sep=2pt, label=above:$V_1$] at (3,1) {};
        \node[circle, draw=black, inner sep=2pt, label=above:$V_2$] at (3,-1) {};
        \draw[style=dashed] (6,0) circle (1);
        \node[circle, fill=red, inner sep=2pt, label=above:{$V_1=V_2$}] at (6,1) {};
        \draw[style=dashed] (9,0) circle (1);
        \node[circle, draw=black, inner sep=2pt, label=above:{$V_1=V_2$}] at (9,1) {};
        \end{tikzpicture}\]
    \end{enumerate}
\end{lemma}
\begin{proof}
    We prove the case for $r=1$. The general statement follows from this case by applying the cyclic rotation map $\cyclic^{r-1}$, defined in \Cref{sec:tilted-def-2}.

    We begin with part (1). Consider any flag $F_\bullet\in \fl_n(\R)$ such that the three subspaces $F_{k-1},F_{k},F_{k+1}$ are equal to $U,V,W$, respectively. Fix all other subspaces of $F_\bullet$, and allow $F_k=V$ to vary over $\P(U,W)$. By \Cref{prop:relative-position} (2), the Schubert stratification of $\fl_n$ cuts this family of flags into two strata: a single point $F_\bullet'$ and an affine line isomorphic to $\R$. By examining the rank conditions defining the Schubert cells, the subspace $V\in\P(U,W)$ corresponding to the unique flag $F_\bullet'$ is the special subspace $V_1$. The subspace $V_2$ can be obtained in the same way by considering the opposite Schubert stratification.

    For part (2), assume $U\in\gr_{k-1,n}^{\geq0}$ and $W\in\gr_{k+1,n}^{\geq0}$. Then there exists a flag $F_\bullet\in \fl_n(\R)$ such that $F_{k-1}=U$, $F_{k+1}=W$, and $F_i\in \gr_{i,n}^{\geq 0}$ for all $i< k$ and $i>k$. This follows from the known fact that $\gr_{k,n}^{\geq0}=\pi_k(\fl_n^{\geq 0})$, as stated in \cite[Corollary~1.2]{network-williams}. One may construct such a flag by taking a preimage of $U$ under $\pi_{k-1}:\fl_n^{\geq 0}\longrightarrow\gr_{k-1,n}^{\geq0}$, and similarly for $W$.
    
    In this setting, whether a subspace $V$ lies in $\gr_{k,n}^{\geq 0}$ reduces to whether the corresponding flag lies in $\fl_n^{\geq 0}$. Let $F^1_\bullet$ be the flag obtained by taking $F_k=V_1$. By the classical Richardson stratification, $F^1_\bullet\in \cR_{u,v}^\circ$ for some $u\leq v$. Since $V_1$ is the special point in $\P(U,W)$, it follows that $i\in\Asc(v)$. Let $\mathbf{v}$ be any reduced word for $v$. Under the classical Deodhar decomposition, the family of flags takes the form of one of the following two configurations:
    \[\begin{tikzpicture}[scale = 0.75]
        \draw (0,0) circle (1);
        \node[circle, fill=black, inner sep=2pt, label=above:{$F_\bullet^1\in\cD_{\mathbf{u},\mathbf{v}}$}] at (0,1) {};
        \node[circle, label=right:{$\cD_{\mathbf{u}1,\mathbf{v}s_i}$}] at (1,0) {};
        \node[circle, fill=black, inner sep=2pt, label=below:{$\cD_{\mathbf{u}s_i,\mathbf{v}s_i}$}] at (0,-1) {};
        \draw (5.5,0) circle (1);
        \node[circle, fill=black, inner sep=2pt, label=above:{$F_\bullet^1\in\cD_{\mathbf{u},\mathbf{v}}$}] at (5.5,1) {};
        \node[circle, label=right:{$\cD_{\mathbf{u}s_i,\mathbf{v}s_i}$}] at (6.5,0) {};
    \end{tikzpicture}\]
    The conclusion now follows from the parametrization of the Deodhar cells in \Cref{thm:deodhar} and their totally nonnegative parts in \Cref{thm:deodhar-positive}, based on the following observations:
    \begin{enumerate}
        \item If $F_\bullet\in \fl_n^{\geq 0}$, then $\mathbf{u}=\mathbf{u}^+$, and $F^1_\bullet\in \fl_n^{\geq 0}$. The first statement holds because only the maximal Deodhar cell $\cD_{\mathbf{u}^+,\mathbf{v}}$ intersects $\fl_n^{\geq 0}$. The second follows from the fact that the parametrization of $F^1_\bullet$ is obtained by removing the last factor from the positive parametrization of $F_\bullet$ in \Cref{thm:deodhar-positive}, which preserves total nonnegativity.
        \item Conversely, if $\mathbf{u}=\mathbf{u}^+$ and $F^1_\bullet\in \fl_n^{\geq 0}$, then the set of points in $\fl_n^{\geq 0}$ must follow the desired configuration. This is because, in \Cref{thm:deodhar-positive}, the only allowed final factors in the parametrization are $y_i(\R_{>0})$ or $\dot{s}_i$.
    \end{enumerate}\end{proof}

Now we are ready to prove the main theorem.

\begin{proof}[Proof of \Cref{thm:tilted-deodhar-positive}]
    Let $F_\bullet\in \cT_{u,v}^\circ(\R)$. By the tilted Deodhar decomposition in \Cref{thm:tilted-deodhar}, $F_\bullet$ lies in a unique tilted Deodhar cell $F_\bullet\in \cD_{\mathbf{u},\mathbf{v}}(\R)$. We proceed by induction on $\ell^\word_\ba(v)$ to prove the following two statements: 
    \begin{enumerate}
        \item If $F_\bullet\in \cT_{u,v}^{>0}$, then the associated subword $\mathbf{u}\prec\mathbf{v}$ must be the unique positive distinguished subword $\mathbf{u}^+$ corresponding to $u$;
        \item Moreover, if $F_\bullet\in \cD_{\mathbf{u}^+,\mathbf{v}}(\R)$, then in the parametrization given by \Cref{thm:tilted-deodhar}, $F_\bullet\in \cT_{u,v}^\circ$ if and only if each Deodhar parameter $p_j$ has a fixed predetermined sign.
    \end{enumerate}
    In the base case $\ell^\word_\ba(v)=0$, which corresponds to $\ba=(1,1,\dots,1)$ and $u=v=\id$, the statement is immediate. If $\mathbf{v}$ ends with a bar, then we may remove the final bar and work with $\flatten(\ba)$ in place of $\ba$, since $\cT_{u,v}^\circ$ does not depend on the specific choice of $\ba$. The result then follows by the induction hypothesis.
    
    The only nontrivial case occurs when $\mathbf{v}$ ends with a simple transposition $s_i$. Let $\mathbf{v}'$ and $\mathbf{u}'$ be the words obtained by removing the last factor from $\mathbf{v}$ and $\mathbf{u}$, respectively. Let $u'\in S_n$ and $v'=vs_i\in S_n$ be the associated permutations. Now consider the set of flags obtained from $F_\bullet$ by replacing the $i$-th subspace $F_i$ with another subspace in $\P(F_{i-1},F_{i+1})$. By the definition of tilted Deodhar cells in \Cref{prop:tilted-deodhar-cell}, we are in one of the local configurations depicted in \Cref{fig:local-Duv}.
    
    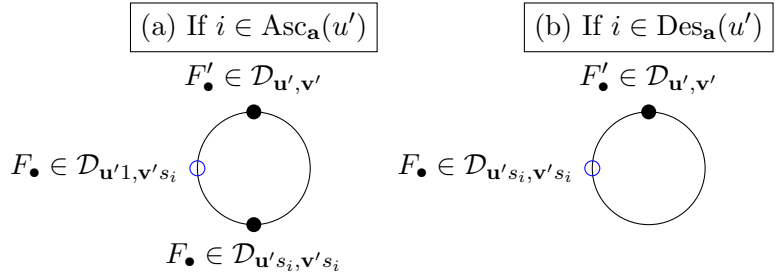
\begin{figure}[ht]
    \centering
        \begin{tikzpicture}[scale = 0.75]
        \node[rectangle, draw=black] at (0,2.5) {(a) If $i\in\Asc_\ba(u')$};
        \draw (0,0) circle (1);
        \node[circle, fill=black, inner sep=2pt, label=above:{$F'_\bullet\in\cD_{\mathbf{u}',\mathbf{v}'}$}] at (0,1) {};
        \node[circle, draw=blue, inner sep=2pt, label=left:{$F_\bullet\in \cD_{\mathbf{u}'1,\mathbf{v}'s_i}$}] at (-1,0) {};
        \node[circle, fill=black, inner sep=2pt, label=below:{$F_\bullet\in\cD_{\mathbf{u}'s_i,\mathbf{v}'s_i}$}] at (0,-1) {};
        \node[rectangle, draw=black] at (7,2.5) {(b) If $i\in\Des_\ba(u')$};
        \draw (7,0) circle (1);
        \node[circle, fill=black, inner sep=2pt, label=above:{$F'_\bullet\in\cD_{\mathbf{u}',\mathbf{v}'}$}] at (7,1) {};
        \node[circle, draw=blue, inner sep=2pt, label=left:{$F_\bullet\in \cD_{\mathbf{u}'s_i,\mathbf{v}'s_i}$}] at (6,0) {};
        \end{tikzpicture}
    \caption{Local configurations of flags in the tilted Deodhar decomposition obtained by replacing $F_i$ in $F_\bullet$ with another subspace in $\P(F_{i-1},F_{i+1})$}
    \label{fig:local-Duv}
    \end{figure}
    We compare these configurations to those described in \Cref{lemma:local-positive}. The black bullet points in \Cref{fig:local-Duv} correspond to the special points $V_1$ and $V_2$ from \Cref{lemma:local-positive}, as determined by the rank conditions in \Cref{def:main}. Case (a) corresponds to the case $V_1\neq V_2$, while case (b) corresponds to $V_1=V_2$.
    
    If $F_\bullet\in \cT_{u,v}^{>0}$, then by \Cref{lemma:local-positive}, we must be in case (a), and $F'_\bullet\in \cT_{u',v'}^{>0}$. By the induction hypothesis, $\mathbf{u}'\prec\mathbf{v}'$ must be the positive distinguished subword, so $\mathbf{u}\prec\mathbf{v}$ must also be the positive distinguished subword. This proves part (1).

    For part (2), let $g'\in G$ be the matrix representative of $F'_\bullet$ from the parametrization in \Cref{thm:tilted-deodhar}. Then the parametrization of the flags in case (a) of \Cref{fig:local-Duv} is given by:
    \[
        \begin{tikzpicture}[scale = 0.75]
        \draw (0,0) circle (1);
        \node[circle, fill=black, inner sep=2pt, label=above:$g'$] at (0,1) {};
        \node[circle, label=left:{$g'\cdot y(\pm\R_{>0})$}] at (-1,0) {};
        \node[circle, label=right:$g'\cdot y(\mp\R_{>0})$] at (1,0) {};
        \node[circle, fill=black, inner sep=2pt, label=below:$g'\cdot \dot{s}_i$] at (0,-1) {};
        \end{tikzpicture}\]
    If $F'_\bullet\in \cT_{u',v'}^\circ$, then by \Cref{lemma:local-positive}, we have $F_\bullet \in \cT_{u',v'}^\circ$ if and only if it is represented by $g'\cdot y(\pm\R_{>0})$ (with a predetermined sign for the parameter) or by $g'\cdot \dot{s}_i$. Part (2) then follows from the induction hypothesis.
\end{proof}

\begin{construction}
    We present an algorithm to determine the sign of each Deodhar parameter $p_j$ in the parametrization given in \Cref{thm:tilted-deodhar-positive}. To do this, we iteratively construct a sequence of sign vectors $\sign^{(j)}\in\{+,-\}^{n}$ as follows. For each $j$, the $k$-th entry of $\sign^{(j)}$ records whether the following two collections of Pl\"ucker coordinates in the prefix $g^{(j)}:=g_1\cdots g_j$ have the same sign or opposite signs:
    \[
        \left\{\Delta_{i_1,\dots,i_{k-1}}(g^{(j)}):i_1<_{a^{(j)}_k}\cdots <_{a^{(j)}_k} i_{k-1}\right\}\quad\text{and}\quad \left\{\Delta_{i_1,\dots,i_{k}}(g^{(j)}):i_1<_{a^{(j)}_k}\cdots <_{a^{(j)}_k} i_{k}\right\}.
    \]
    Start with $\sign^{(0)}=(+,\cdots,+)$ and update $\sign^{(j)}$ for $j=1,\ldots,\ell$ according to the cases:
     \begin{enumerate}
         \item If $j\in J_{\mathbf{u}^+}^+$, set $\sign^{(j)}=\sign^{(j-1)}$. The sign of the Deodhar parameter $p_j$ in $g_j=y_{i_j}(\pm p_j)$ is determined by the product of the $i_j$-th and $(i_j+1)$-th entry of $\sign^{(j)}$.
         \item If $j\in J_{\mathbf{u}^+}^\circ$, obtain $\sign^{(j)}$ from $\sign^{(j-1)}$ by swapping the $i_j$-th and $(i_j+1)$-th entries.
         \item If $j\in J_{\mathbf{u}^+}^\mid$, let $q:=\jmin(\ba^{(j)})$ and let $p:=\size{v[q]\cap[\ba^{(j)}_1,\ba^{(j)}_q)_c}$ be the integer appearing in \Cref{def:a-flattenable} for the $\ba^{(j)}$-flattenable permutation $v^{(j)}$. Then obtain $\sign^{(j)}$ from $\sign^{(j-1)}$ by multiplying all entries in positions $[p+1,q]$ by $(-1)^p$.
     \end{enumerate}
    We illustrate this algorithm using the running example from \Cref{ex:tilted-positive}. In this case, we obtain the following sequence of sign vectors. The circled entries indicate the positions used to determine the signs of the Deodhar parameters $a,b,c,d$.
    \begin{align*}
        \sign^{(0)}&=(+,+,+,+) & \sign^{(4)}&=(+,+,+,-) & \sign^{(8)}&=(+,+,+,-)\\
        \sign^{(1)}&=(+,+,+,+) & \sign^{(5)}&=(\tikz[baseline]{\node[draw=blue,anchor=base,ellipse,inner xsep=-1pt,inner ysep=0pt]{$+,+$}},+,-) & \sign^{(9)}&=(+,+,+,-)\\
        \sign^{(2)}&=(+,+,+,+) & \sign^{(6)}&=(+,\tikz[baseline]{\node[draw=blue,anchor=base,ellipse,inner xsep=-1pt,inner ysep=0pt]{$+,+$}},-) & \sign^{(10)}&=(+,-,+,-)\\
        \sign^{(3)}&=(+,+,+,+) & \sign^{(7)}&=(+,+,\tikz[baseline]{\node[draw=blue,anchor=base,ellipse,inner xsep=-1pt,inner ysep=0pt]{$+,-$}}) & \sign^{(11)}&=(\tikz[baseline]{\node[draw=blue,anchor=base,ellipse,inner xsep=-1pt,inner ysep=0pt]{$+,-$}},+,-)
    \end{align*}
\end{construction}

Finally, we present the following lemma, which plays a crucial role in showing that the collection of cells $\cT_{u,v}^{>0}$ forms a CW-complex.

\begin{lemma}\label{lemma:positive-coeff}
    Let $\rho_{\mathbf{u}^+,\mathbf{v}}:(\R_{>0})^{\ell(u,v)}\xrightarrow{\sim}\cT_{u,v}^{> 0}$ be the parametrization from \Cref{thm:tilted-deodhar-positive}, defined in terms of the Deodhar parameters $p_j\in\R^{>0}$. Then any Pl\"ucker coordinate $\Delta_I$ of a matrix in the image of $\rho_{\mathbf{u}^+,\mathbf{v}}$ is a polynomial in the parameters $p_j$, with either all nonnegative or all nonpositive coefficients.
\end{lemma}
\begin{proof}
    By the Cauchy--Binet formula, the Pl\"ucker coordinate $\Delta_I$ is a polynomial in the Deodhar parameters $p_j$, with each $p_j$ appearing with degree at most one. Since $\rho_{\mathbf{u}^+, \mathbf{v}}$ maps $(\R_{>0})^{\ell(u,v)}$ into $\cT_{u,v}^{>0}$ by \Cref{thm:tilted-deodhar-positive}, it follows that $\Delta_I$ evaluates to either a nonnegative or a nonpositive real number for all positive inputs. Without loss of generality, we assume the sign is nonnegative.
    
    Suppose, for contradiction, that some monomial in $\Delta_I$ has a negative coefficient. By letting the variables appearing in that monomial tend to infinity, and the others tend to zero, the value of $\Delta_I$ becomes negative, contradicting its nonnegativity. 
\end{proof}

As a direct consequence of \Cref{lemma:positive-coeff}, we obtain the following theorem.

\begin{theorem}\label{thm:CW-tilted}
    For any $u,v\in S_n$, the totally nonnegative part of the tilted Richardson variety $\cT_{u,v}^{\geq 0}$, stratified as
    \[\cT_{u,v}^{\geq0}=\bigsqcup_{[x,y]\subseteq[u,v]}\cT_{x,y}^{>0},\]
    forms a CW-complex.
\end{theorem}
\begin{proof}
    The proof is identical to the proof given in \cite[Section~5]{rietsch-williams}, with \Cref{lemma:positive-coeff} serving as a key ingredient in the tilted setting.
\end{proof}

There are several interesting directions in the line of research of total positivity. 
\begin{conj}\label{conj:positive_part_closure}
$\overline{\cT_{u,v}^{>0}}=\cT_{u,v}^{\geq 0}$.
\end{conj}
\Cref{conj:positive_part_closure} would imply that the face poset of the CW-complex in \Cref{thm:CW-tilted} is given by the interval poset of the tilted Bruhat interval $[u,v]$. We note, however, that the strategy used in \cite{rietsch2} for the classical case does not apply here, due to the absence of a unique maximal element $w_0$ in the tilted Bruhat orders.
\begin{conj}
The CW-complex in \Cref{thm:CW-tilted} is regular.
\end{conj}

%% file: tex/7-curve-neighborhood.tex
\section{Connections to Quantum Schubert Calculus}\label{sec:curve-neighborhood-sec}

In this section, we connect tilted Richardson varieties with quantum Schubert calculus. Specifically, we show that tilted Richardson varieties $\cT_{u,v}$ are equivalent to the two-point curve neighborhoods $\Gamma_{d_{u,v}}(\Omega_u, X_v)$ in minimal degree (\Cref{thm:Gamma=T}), providing an explicit geometric description of a new class of curve neighborhoods (\Cref{thm:cohomology}). As a consequence, we derive an explicit formula for the cohomology classes $[\cT_{u,v}]$ in terms of Gromov--Witten invariants, and use this to prove a quantum analogue of the descent-cycling formula (\Cref{thm:descent-cycling}).

\subsection{Equivalence to Minimal-Degree Curve Neighborhoods}\label{sec:curve-neighborhood}
For $i\in \{1,2,3\}$, let $\ev_i: \overline{M}_{0,3}(\fl_n,d)\rightarrow \fl_n$ be the evaluation maps as defined in \Cref{sec:prelim-3-3}. Recall from \Cref{sec:prelim-3-4} that the two-point curve neighborhood of degree $d$ is defined as 
\[\Gamma_d(\Omega_u,X_v) = \ev_3(\ev_1^{-1}(\Omega_u)\cap \ev_2^{-1}(X_v)).\]
The curve neighborhoods are the ``quantum‘’ generalizations of Richardson varieties in the context of Schubert calculus. Understanding the expansions of the cohomology classes $[\Gamma_d(\Omega_u,X_v)]$ into Schubert classes would lead to a solution to the quantum Schubert calculus problem (\Cref{prop:Gamma-nice}). However, unlike Richardson varieties, very little in the way of explicit descriptions of $\Gamma_d(\Omega_u,X_v)$ was known \cite{LiMihalcea,BM15}. Here we provide a complete answer for the minimal-degree case $d = d_{u,v}$.



We begin by introducing some definition used in the proof. The \emph{Bruhat graph} is an undirected, weighted graph on the vertex set $S_n$, where an edge connects $w$ and $w t_{ij}$ for each transposition $t_{ij} \in S_n$. The edge is assigned weight $q_{ij} := q_i q_{i+1} \cdots q_{j-1}$.

The moduli space $\overline{M}_{0,3}(\fl_n, d)$ admits a natural $T$-action. The $T$-fixed points of this moduli space can be described in terms of the Bruhat graph by the following lemma (see \cite[Lemma~4.2]{FW}):

\begin{lemma}\label{lemma:t-fixed-tree}
For any distinct $w, w' \in S_n$, there exists a $T$-invariant curve $C \subseteq \fl_n$ containing both $e_w$ and $e_{w'}$ if and only if $w$ and $w'$ are adjacent in the Bruhat graph. Moreover, in this case, the curve $C$ is unique and has degree equal to the weight of the edge $(w, w')$.

Consequently, the $T$-fixed points of $\overline{M}_{0,3}(\fl_n, d)$ correspond to collections of edges (with multiplicities) in the Bruhat graph whose total weight equals $q^d$, and which together form a connected tree.
\end{lemma}

We are now ready to prove the following proposition, which shows that the $T$-fixed points in the minimal-degree curve neighborhood $\Gamma_{d_{u,v}}(\Omega_u, X_v)$ correspond precisely to the elements in the tilted Bruhat interval $[u, v]$.

\begin{prop}\label{prop:T-fixed-pts}
    For $u,v\in S_n$, the set of $T$-fixed points in $\Gamma_{d_{u,v}}(\Omega_u,X_v)$ is $\{e_w:w\in [u,v]\}$. 
\end{prop}
\begin{proof}
    We begin with one direction. Suppose $e_w \in \Gamma_{d_{u,v}}(\Omega_u, X_v)$. Then there exists a stable curve $C$ of degree $d_{u,v}$ passing through $e_w$ and intersecting both $\Omega_u$ and $X_v$. Since $\ev_3^{-1}(e_w)$ is $T$-invariant, and every $T$-invariant closed subvariety contains a $T$-fixed point, we may assume that $C$ is $T$-invariant. By \Cref{lemma:t-fixed-tree}, such a $T$-invariant stable curve corresponds to a connected tree $C^T$ in the Bruhat graph. This tree contains $w$, along with some $u'\geq u$ and some $v'\leq v$, and has total weight $q^{d_{u,v}}$.

    Our goal is to construct a shortest path in the quantum Bruhat graph from $u$ to $v$ that passes through $w$. If such a path exists, then $w \in [u, v]$ by definition of the tilted Bruhat interval. The construction proceeds in three steps:
    \begin{enumerate}
        \item Since $C^T$ is a connected tree, there is a unique path in $C^T$ from $u'$ to $w$ and another from $w$ to $v'$. Concatenating them gives a path $P$ (possibly with repeated edges) from $u'$ to $v'$ in $C^T$ that passes through $w$.
        \item Replace each edge in $P$ with a shortest path in the quantum Bruhat graph between its endpoints. This yields a directed path $P'$ from $u'$ to $v'$ in the quantum Bruhat graph that still passes through $w$.
        \item Since $u' \geq u$ and $v' \leq v$ in the strong Bruhat order, we can prepend a strong Bruhat path from $u$ to $u'$ and append one from $v'$ to $v$, obtaining a path $P''$ from $u$ to $v$ in the quantum Bruhat graph that passes through $w$.
    \end{enumerate}
    We now claim that the weight of the path $P''$ is at most $q^{d_{u,v}}$. The strong Bruhat edges used to connect $u$ to $u'$ and $v'$ to $v$ have weight $1$, so they do not contribute to the total weight. Hence, $\wt(P'') = \wt(P')$. Moreover, by \Cref{cor:weight-distance}, each subpath in $P'$ has weight less than or equal to the weight of the corresponding edge in $P$, and the total weight of $P$ is at most the total weight of $C^T$, which is $q^{d_{u,v}}$.

    We remark that some edges $(w, w')$ in the path $P$ may be used in both directions. However, this does not affect the total weight: in constructing $P'$, each such edge is replaced by a directed shortest path from $w$ to $w'$ in the quantum Bruhat graph. Since one of the two directions, either $w \to w'$ or $w' \to w$, has weight $1$, each edge contributes at most once to the total quantum weight.

    We now prove the other direction. Recall that in the quantum Bruhat graph, a \emph{strong Bruhat edge} (denoted $\xrightarrow{1}$) is an edge with weight $1$, and a \emph{quantum edge} (denoted $\mathcolor{blue}{\xrightarrow{q}}$) is an edge with nontrivial quantum weight. For any $w \in [u, v]$, we claim that there exist two shortest paths, one from $u$ to $w$ and one from $w$ to $v$, of the following form:
    \[
    \begin{aligned}
        u\xrightarrow{1}\cdots\xrightarrow{1} u'\mathcolor{blue}{\xrightarrow{q}\cdots\xrightarrow{q}}w,\\
        w\mathcolor{blue}{\xrightarrow{q}\cdots\xrightarrow{q}}v'\xrightarrow{1}\cdots\xrightarrow{1}v.
    \end{aligned}
    \]
    We prove the claim for the first path, and the second follows by a similar argument. Begin with a shortest path from $u$ to $w$ in the quantum Bruhat graph. If the path contains a subpath of the form $x \mathcolor{blue}{\xrightarrow{q}} y \xrightarrow{1} z$, then by \Cref{lemma:thin}, we may replace this subpath with an alternative path $x \to y' \to z$ in the quantum Bruhat graph. Since both paths have the same total weight by \Cref{lemma:shortest-length-equals-weight}, the only two possibilities for the replacement path are:
    \[x\xrightarrow{1}y'\mathcolor{blue}{\xrightarrow{q}}z\quad\text{or}\quad x\mathcolor{blue}{\xrightarrow{q}}y'\mathcolor{blue}{\xrightarrow{q}}z.\]
    In either case, the number of quantum edges either increases or shifts further to the right in the path. This process must eventually terminates, yielding a path of the desired form.

    Finally, we concatenate the two paths to obtain a shortest path from $u$ to $v$ that passes through $w$, of the form:
    \[
        u\xrightarrow{1}\cdots\xrightarrow{1} u'\mathcolor{blue}{\xrightarrow{q}\cdots\xrightarrow{q}}w\mathcolor{blue}{\xrightarrow{q}\cdots\xrightarrow{q}}v'\xrightarrow{1}\cdots\xrightarrow{1}v.
    \]
    By \Cref{lemma:t-fixed-tree}, the subpath from $u'$ to $v'$ corresponds to a $T$-invariant stable curve $C$ of degree $d_{u,v}$ passing through $e_w$, intersecting $\Omega_u$ at $e_{u'}$, and intersecting $X_v$ at $e_{v'}$. Therefore, by definition, we conclude that $e_w \in \Gamma_{d_{u,v}}(\Omega_u, X_v)$.
\end{proof}

We are now ready to prove the main theorem.

\begin{theorem}\label{thm:Gamma=T}
    Let $u,v\in S_n$, then $\Gamma_{d_{u,v}}(\Omega_u,X_v)= \cT_{u,v}$.
\end{theorem}
\begin{proof}
    By \Cref{prop:T-fixed-pts} and \Cref{thm:tilted-T-fixed-point}, $\Gamma_{d_{u,v}}(\Omega_u, X_v)$ and $\cT_{u,v}$ share the same set of $T$-fixed points.
    Since $\Gamma_{d_{u,v}}(\Omega_u, X_v)$ is $T$-invariant and closed in $\fl_n$, we have
    \[\Gamma_{d_{u,v}}(\Omega_u,X_v)\subseteq \{F_\bullet\in\fl_n:\Delta_w(F_\bullet)=0\text{ for all }w\notin [u,v]\}.\]
    By \Cref{thm:altdefTplucker}, the right-hand side is precisely $\cT_{u,v}$, so we have $\Gamma_{d_{u,v}}(\Omega_u, X_v) \subseteq \cT_{u,v}$.
    
    Finally, by \Cref{thm:dimension} and \Cref{prop:Gamma-nice}, both varieties have the same dimension:
    \[\dim(\Gamma_{d_{u,v}}(\Omega_u,X_v))=\dim(\cT_{u,v})=\ell(v)-\ell(u)+2|d_{u,v}|=\ell(u,v),\]
    and $\cT_{u,v}$ is irreducible by \Cref{thm:irreducible}. Therefore, the two varieties must be equal.
\end{proof}

\subsection{Cohomology Classes of Tilted Richardson Varieties}\label{sec:cohomology}

In this section, we compute the cohomology class of the tilted Richardson variety $\cT_{u,v}$ (\Cref{thm:cohomology}), which, by \Cref{thm:Gamma=T}, coincides with the cohomology class of the minimal-degree curve neighborhood $\Gamma_{d_{u,v}}(\Omega_u, X_v)$ in the cohomology ring $H^*(\fl_n)$. By \Cref{prop:Gamma-nice}, it remains to compute the constant $c$ in \Cref{prop:Gamma-nice}.


Our strategy is to use path Schubert polynomials $\mathfrak{S}_{u,v}(x, q)$ introduced by Postnikov \cite{Postnikov-quantum-Bruhat-graph} (see \Cref{sec:prelim-3-7} for more details). We show that there exists a monomial of the form $q^{d_{u,v}}x^{\rho - \beta}$ that appears in $\mathfrak{S}_{u,v}(x, q)$ with coefficient $1$. This implies $c = 1$. 

We say that a directed edge $w \to w'$ lies in the tilted Bruhat interval $[u, v]$ if $w'$ covers $w$ in the poset $[u, v]$. A directed path lies in $[u, v]$ if all of its edges lie in $[u, v]$. We begin with the following lemma, which plays a key role in the construction of the special monomial in the path Schubert polynomial.

\begin{lemma}\label{lemma:cover-max}
    Let $u,v\in S_n$ and $i\in[n-1]$ be such that $u_j=v_j$ for all $j\in [i-1]$. Among all directed paths in the tilted Bruhat interval $[u, v]$ of the form
    \[u=w^{(0)}\xrightarrow{t_{ip_1}}w^{(1)}\xrightarrow{t_{ip_2}}\cdots\xrightarrow{t_{ip_\ell}}w^{(\ell)}=u'\]
    with $p_j>i$ for all $j$, there exists a unique path of maximal length $\ell$. Moreover, this path is $x_i^\ell$-admissible, and the $i$-th entry of the final permutation satisfies $u'_i = v_i$.
\end{lemma}
\begin{proof}
    We construct the desired path via the following greedy algorithm. Starting with $w^{(0)} := u$, and assuming $w^{(0)}, \dots, w^{(k-1)}$ have been chosen, we define $p_k>i$ to be the largest index such that the directed edge $w^{(k-1)}\xrightarrow{t_{ip_k}} w^{(k)}$ lies in $[u, v]$. We terminate when no such $p_k$ exists.

    Fix a sequence $\ba = (a_1, \dots, a_n) \in [n]^n$ such that $u \leq_\ba v$. The path is $x_i^\ell$-admissible: the construction ensures that the transposed entries $u_{p_k}$ are strictly increasing with respect to the shifted order $<_{a_i}$, so the indices $p_1, \dots, p_k$ are distinct. Additionally, if $u'_i \neq v_i$, then the path could be extended by taking the next step from $u'$ along the canonical shortest path from $u'$ to $v$ constructed in the proof of \Cref{thm:weight-distance}, which would yield a further transposition $u' \xrightarrow{t_{ip_{\ell+1}}} u''$. Thus, we must terminate at $u'_i = v_i$.
    
    We now prove that this is the unique maximal-length path of the given form. The proof is by induction on $\ell(u, v)$. Suppose there exists another path of length $\ell' \geq \ell$:
    \[u\xrightarrow{t_{iq_1}}x^{(1)}\xrightarrow{t_{iq_2}}\cdots\xrightarrow{t_{iq_{\ell'}}}x^{(\ell')},\]
    We aim to show that $\ell' = \ell$ and this path coincides with the greedy path. If $\ell = 0$, then the only possible path is the singleton $u$, so the claim is immediate. If $p_1 = q_1$, then $w^{(1)} = x^{(1)}$, and the induction hypothesis applied to $[w^{(1)}, v]$ completes the argument. So we may assume $p_1 \neq q_1$. By construction, $p_1 > q_1$. Now consider the path
    \[u\xrightarrow{t_{ip_1}}y^{(0)}\xrightarrow{t_{iq_1}}y^{(1)}\xrightarrow{t_{iq_2}}\cdots\xrightarrow{t_{iq_{\ell'}}}y^{(\ell')}.\]
    We claim that this path lies entirely in $[u, v]$ and has length $\ell' + 1$, contradicting the maximality of $\ell'$. Let $\ba$ be as before, with $u \leq_\ba v$. Using \Cref{thm:tilted-criterion}, we work in the partial order $\leq_\ba$ and verify the claim step by step:
    \begin{enumerate}
        \item For each $k \in [\ell']$, we have $y^{(k)} = x^{(k)} t_{q_1 p_1}$, and $y^{(k)} \gtrdot_\ba x^{(k)}$. The first identity is immediate. The second follows from the assumption $u \lessdot_\ba u t_{1p_1}$ and \Cref{prop:covering}.
        \item For each $k \in [\ell']$, we have $y^{(k-1)} \lessdot_\ba y^{(k)}$. This follows from $x^{(k-1)} \lessdot_\ba x^{(k)}$ and again by applying \Cref{prop:covering}.
        \item For each $k \in [\ell']$, we have $y^{(k)} \leq_\ba v$. This is proven inductively. Assume $y^{(k-1)} \leq_\ba v$ and $x^{(k)} \leq_\ba v$. It then suffices to show that
        \[
        y^{(k)}[j] = \max_{\leq_{a_j}} \left( x^{(k)}[j], y^{(k-1)}[j] \right) \quad \text{for all } j \in [n],
        \]
        where $\max_{\leq_{a_j}}$ denotes the least upper bound in the lattice defined by $\leq_{a_j}$. If $j < q_1$, then $y^{(k)}[j] = x^{(k)}[j]$. If $j \geq q_k$, then $y^{(k)}[j] = y^{(k-1)}[j]$. For $q_1 \leq j < q_k$, since the only values that differ in $x^{(k)}, y^{(k)}, y^{(k-1)}$ are among $\{ u_i, u_{p_1}, u_{q_{k-1}}, u_{q_k} \}$, the statement reduces to
        \[
        \{ u_{p_1}, u_{q_k} \} = \max_{\leq_{a_j}} \left( \{ u_i, u_{p_1} \}, \{ u_{q_{k-1}}, u_{q_k} \} \right),
        \]
        which holds since $u_i <_{a_j} u_{p_1}$ and $u_{q_{k-1}} <_{a_j} u_{q_k}$, and the intervals $[u_i, u_{p_1})_c$ and $[u_{q_{k-1}}, u_{q_k})_c$ do not intersect.
    \end{enumerate}
    Therefore, the path from $u$ to $y^{(\ell')}$ lies in $[u, v]$ and has length $\ell' + 1$, contradicting the assumption that $\ell'$ is maximal. This proves that the greedy path is the unique maximal-length path of the specified form.
\end{proof}

Denote the \emph{reverse lexicographic order} (or \emph{revlex} for short) as the monomial order on the polynomial ring $\Z[q_1,\dots,q_{n-1},x_1,\dots,x_n]$ that compares exponents from left to right, where the first differing exponent determines the order, and the monomial with the larger exponent in that position is considered smaller. The following lemma shows that the leading monomial of $\mathfrak{S}_{u,v}(x, q)$ under this order appears with coefficient $1$.

\begin{lemma}\label{lemma:coeff1}
Let $u, v \in S_n$. The leading monomial of $\mathfrak{S}_{u,v}(x, q)$ under the revlex order is of the form $q^{d_{u,v}} x^{\rho - \beta}$ with coefficient $1$.
\end{lemma}
\begin{proof}
     We construct the path $P_{u,v}$ as a concatenation of $n-1$ subpaths $P_1, P_2, \dots, P_{n-1}$, where each $P_i$ is the unique maximal-length $x_i^{\beta_i}$-admissible path in $[u, v]$, as constructed in \Cref{lemma:cover-max}. By construction, $P_{u,v}$ is an $x^\beta$-admissible path from $u$ to $v$.

     Since $P_{u,v}$ is a shortest path from $u$ to $v$, its quantum weight is $q^{d_{u,v}}$ by \Cref{lemma:shortest-length-equals-weight}. The corresponding monomial in $\mathfrak{S}_{u,v}(x, q)$ is $q^{d_{u,v}} x^{\rho - \beta}$, and it is the leading monomial with respect to the revlex order due to the length-maximality of each $P_i$. Moreover, $P_{u,v}$ is the unique $x^\beta$-admissible path, again by the uniqueness of each maximal-length $P_i$. Therefore, this leading monomial appears with coefficient $1$.
\end{proof}

\begin{remark}
    In the special case where $v = w_0$, the path Schubert polynomial $\mathfrak{S}_{u,v}(x, q)$ coincides with the classical Schubert polynomial $\mathfrak{S}_u(x)$. The leading monomial constructed in \Cref{lemma:coeff1} corresponds to the \emph{bottom pipedream} for $u$, introduced in \cite{rc-graph}.
\end{remark}

We are now ready to compute the cohomology class $[\cT_{u,v}]$.

\begin{theorem}\label{thm:cohomology}
    Let $u,v\in S_n$. The cohomology classes $[\cT_{u,v}]$ and $[\Gamma_{d_{u,v}}(\Omega_u,X_v)]$ in $H^\ast(\fl_n)$ both equal the minimal quantum degree component of the quantum product $\sigma_{u}\star \sigma_{w_0v}$:
    \[[\cT_{u,v}]=[\Gamma_{d_{u,v}}(\Omega_u,X_v)]=[q^{d_{u,v}}]\sigma_u\star\sigma_{w_0v}=\sum_{w\in S_n}c_{u,w}^{v,d_{u,v}}\sigma_{w_0w}\in H^\ast(\fl_n).\]
\end{theorem}

\begin{proof}
    By \Cref{thm:Gamma=T} and \Cref{prop:Gamma-nice}, we have
    \[[\cT_{u,v}]=[\Gamma_{d_{u,v}}(\Omega_u,X_v)]=\frac{1}{c}\;[q^{d_{u,v}}]\sigma_u\star\sigma_{w_0v}=\frac{1}{c}\sum_{w\in S_n}c_{u,w}^{v,d_{u,v}}\sigma_{w_0w}\in H^\ast(\fl_n),\]
    where $c\in\Z_{>0}$ is the degree of the projection map $\ev_3:GW_{d_{u,v}}(\Omega_u,X_v)\to\Gamma_{d_{u,v}}(\Omega_u,X_v)$. It remains to show that $c=1$. Since $[\cT_{u,v}]$ expands integrally in the Schubert basis, each Gromov--Witten invariant $c_{u, w}^{v, d_{u,v}}$ must be divisible by $c$. By \Cref{prop:path-schub}, this implies that every coefficient in the monomial expansion of $[q^{d_{u,v}}] \mathfrak{S}_{u,v}(x, q)$ is divisible by $c$. However, by \Cref{lemma:coeff1}, there exists a monomial whose coefficient is equal to $1$, so $c=1$. This concludes the proof.
\end{proof}

\subsection{A Descent-Cycling Formula for Gromov--Witten Invariants}\label{sec:descent-cycling}

In this section, we present a descent-cycling formula for Gromov--Witten invariants (\Cref{thm:descent-cycling}) as an application of the theory of tilted Richardson varieties.

Our proof differs from Knutson’s original approach to the descent-cycling formula for Littlewood--Richardson coefficients in \cite{knutson-cycling}, which is combinatorial and based on divided difference operators on Schubert polynomials. In contrast, our argument is geometric, relying heavily on the tilted Richardson varieties. It remains an open question whether there exists a combinatorial proof of the same identity for Gromov--Witten invariants using divided difference operators.

We also remark that a similar relation among Gromov--Witten invariants was established in \cite{LeungLi}. However, their results and ours are independent and do not recover one another.

We begin by introducing some terminology. For $i \in [n]$, let $\fl_{[n] \setminus {i}}$ denote the \emph{partial flag variety}, consisting of nested linear subspaces omitting the $i$-th subspace:
\[\fl_{[n]\setminus i}:=\{F_1\subseteq\cdots \subseteq F_{i-1}\subseteq F_{i+1}\subseteq\cdots\subseteq F_n=\C^n:\dim (F_j)=j\text{ for all }j\in[n]\setminus\{i\}\}.\]
Let $\hpi:\fl_n\to\fl_{[n]\setminus\{i\}}$ be the natural projection map that forgets the $i$-th subspace $F_i$ of a complete flag $F_\bullet \in \fl_n$. This projection $\hpi$ is a proper, flat morphism whose fibers are isomorphic to $\P^1$. Define the \emph{divided difference operator} $\partial_i : H^\ast(\fl_n) \to H^{\ast}(\fl_n)$ as the composition
\[\partial_i:=(\hpi)^\ast\circ(\hpi)_\ast,\] 
where $(\hpi)\ast$ is the proper pushforward and $(\hpi)^*$ is the flat pullback. We now state the following proposition, which describes the action of $\partial_i$ on cohomology classes of tilted Richardson varieties $\cT_{u,v}$.

\begin{prop}\label{prop:push-pull}
Let $u,v\in S_n$ and $i\in [n-1]$ such that $[u,v]\cdot s_i=[u,v]$. Then:
\begin{enumerate}
    \item $\partial_i[\cT_{u,v}]=0$;
    \item $\partial_i[\cT_{us_i,v}]= \partial_i[\cT_{u,vs_i}]=[\cT_{u,v}]$.
\end{enumerate}
\end{prop}
\begin{proof}
    By \Cref{thm:strong-lifting}, fix a sequence $\ba\in[n]^n$ such that $u\lesssim_\ba v$ and $i\in\Asc_\ba(u)\cap\Des_\ba(v)$. We first establish the following geometric statements:
    \begin{enumerate}
        \item[(a)] $\hpi^{-1}(\hpi(\cT_{u,v}))=\cT_{u,v}$;
        \item[(b)] $\hpi^{-1}(\hpi(\cT_{us_i,v}))=\hpi^{-1}(\hpi(\cT_{u,vs_i}))=\cT_{u,v}$;
        \item[(c)] The projection $\hpi$ restricts to an isomorphism over $\cT_{us_i, v}^\circ$ and $\cT_{u,vs_i}^\circ$.
    \end{enumerate}
    For (a), it suffices to show that $\hpi^{-1}(\hpi(\cT_{u,v}))\subseteq \cT_{u,v}$. Any $F_\bullet \in \hpi^{-1}(\hpi(\cT_{u,v}))$ is either equal to some $F'_\bullet \in \cT_{u,v}$ or satisfies $F_\bullet \xrightarrow{s_i} F'_\bullet$. If $F_\bullet = F'_\bullet$, the result is clear. Otherwise, by \Cref{thm:Tunion}, we have $F'_\bullet \in \cT_{x,y}^\circ$ for some $[x, y] \subseteq [u, v]$, or equivalently, $e \xrightarrow{y, \ba} F'_\bullet \xleftarrow{x, \ba} w_0$. Since $F_\bullet \xrightarrow{s_i} F'_\bullet$, \Cref{prop:relative-position} implies that $e \xrightarrow{y', \ba} F_\bullet \xleftarrow{x', \ba} w_0$ for some $x' \in \{x, x s_i\}$ and $y' \in \{y, y s_i\}$. In either case, $x', y' \in [u, v]$, so $F_\bullet \in \cT_{x', y'}^\circ \subseteq \cT_{u,v}$.

    For (b), we prove only that $\cT_{u,v} \subseteq \hpi^{-1}(\hpi(\cT_{u s_i, v}))$,, as the other case is analogous. Let $F_\bullet \in \cT_{u,v}$. We aim to show that there exists $F'_\bullet \in \cT_{u s_i, v}$ such that either $F_\bullet = F'_\bullet$ or $F_\bullet \xrightarrow{s_i} F'_\bullet$. Since $F_\bullet \in \cT_{x,y}^\circ$ for some $[x, y] \subseteq [u, v]$, we consider two cases. If $i \in \Des_\ba(x)$, then by \Cref{thm:lifting}, we have $[x, y] \subseteq [u s_i, v]$, which implies $F_\bullet \in \cT_{x, y}^\circ \subseteq \cT_{u s_i, v}$. If $i \in \Asc_\ba(x)$, then by \Cref{prop:relative-position} (5), there exists $F'_\bullet$ such that $w_0 \xrightarrow{x s_i, \ba} F'_\bullet \xrightarrow{s_i} F_\bullet$. Moreover, we have $e \xrightarrow{y', \ba} F'_\bullet$ for some $y' \in \{y, y s_i\}$, and therefore $F'_\bullet \in \cT_{x s_i, y'}^\circ \subseteq \cT_{u s_i, v}$.
    
    For (c), suppose for contradiction that two distinct flags $F_\bullet, F'_\bullet \in \cT_{u s_i, v}^\circ$ have the same image under $\hpi$. Then necessarily $F_\bullet \xrightarrow{s_i} F'_\bullet$. However, since both $w_0\xrightarrow{us_i,\ba}F_\bullet$ and $w_0\xrightarrow{us_i,\ba}F'_\bullet$, this contradicts \Cref{prop:relative-position} (4).
    
    We now return to the main statements. For part (1), by statement (a), the projection $\cT_{u,v} \to \hpi(\cT_{u,v})$ has fibers isomorphic to $\P^1$, and is not generically finite. Therefore, the pushforward vanishes: $(\hpi)_\ast[\cT_{u,v}]=0$, and hence $\partial_i[\cT_{u,v}]=0$.
    
    For part (2), we prove only $\partial_i[\cT_{u s_i, v}] = [\cT_{u,v}]$, and the other identity is similar. By statement (c), the projection $\cT_{u s_i, v} \to \hpi(\cT_{u s_i, v})$ restricts to an isomorphism over the dense open subset $\cT_{u s_i, v}^\circ \subseteq \cT_{u s_i, v}$, so the projection has degree one:
    \[(\hpi)_\ast[\cT_{us_i,v}]=[\hpi(\cT_{us_i,v})].\]
    Then, using statement (b), we compute:
    \[\partial_i[\cT_{us_i,v}]=(\hpi)^\ast[\hpi(\cT_{us_i,v})]=[\hpi^{-1}(\hpi(\cT_{us_i,v}))]=[\cT_{u,v}]. \qedhere\]
\end{proof}

The following well-known lemma describes the action of the operator $\partial_i$ on Schubert classes. It is a direct corollary of \Cref{prop:push-pull}.
\begin{lemma}\label{lemma:push-pull}
    Let $w\in S_n$ and $i\in [n-1]$. Then
    \[\partial_i\sigma_w=\begin{cases}
        \sigma_{ws_i} &\text{if }w>ws_i,\\
        0 &\text{if }w<ws_i.
    \end{cases}\]
\end{lemma}
\begin{proof}
    This follows directly from \Cref{prop:push-pull} in the case $v = w_0$, since $\cT_{u, w_0} = \Omega_u$.
\end{proof}

Finally, we state the descent-cycling formula for Gromov--Witten invariants.

\begin{theorem}[Descent-Cycling Formula]\label{thm:descent-cycling}
    Let $u,v,w\in S_n$ and $i\in [n-1]$ such that $[u,v]\cdot s_i=[u,v]$, and $ws_i>w$. Then:
    \begin{enumerate}
        \item $c_{u,w}^{v,d_{u,v}}=0$;
        \item $c_{u,ws_i}^{v,d_{u,v}} =c_{us_i,w}^{v,d_{us_i,v}}=c_{u,w}^{vs_i,d_{u,vs_i}}$.
    \end{enumerate}
\end{theorem}
\begin{proof}
    Part (1) follows from the Schubert expansion of the first equation $\partial_i[\cT_{u, v}] = 0$ in \Cref{prop:push-pull}, combined with \Cref{thm:cohomology} and \Cref{lemma:push-pull}.

    Part (2) follows by comparing all sides of the second equation in \Cref{prop:push-pull}:
    \[[\cT_{u,v}]=\partial_i[\cT_{us_i,v}]= \partial_i[\cT_{u,vs_i}],\]
    combined with \Cref{thm:cohomology} and \Cref{lemma:push-pull}.
\end{proof}

%% file: tex/8-projections.tex
\section{Projections of Tilted Richardson Varieties}\label{sec:projection}

Let $\pi_k:\fl_n\rightarrow \gr(k,n)$ be the natural projection onto the $k$-th flag. In \cite{KLSjuggling, KLSprojection}, Knutson--Lam--Speyer studied the images of Richardson varieties $\cR_{u,v}$ under $\pi_k$. Here we extend their results to tilted Richardson varieties. In particular, we show that
\begin{itemize}
    \item the image of any tilted Richardson variety under $\pi_k$ is a positroid variety (\Cref{thm:proj=positroid});
    \item the restriction $\pi_k\colon \cT_{u,v}\to \pi_k(\cT_{u,v})$ is birational if and only if $u\le_{\ba}^k v$
    (\Cref{def:ktilted}, \Cref{thm:birational}).
\end{itemize}

We begin with a recursion for open tilted Richardson varieties.

\begin{prop}\label{prop:recur}
    For $u\lesssim_\ba v$ and $i\in \Des_\ba(v)$,
    \[\cT_{u,v}^\circ \cong 
    \begin{cases}
        \bigl(\cT_{u,vs_i}^\circ \times \C^\ast\bigr)\;\sqcup\;\bigl(\cT_{us_i,vs_i,\ba}^\circ \times \C\bigr)
        & \text{if } us_i>_\ba u,\\
        \cT_{us_i,vs_i}^\circ
        & \text{if } us_i<_\ba u.
    \end{cases}\]
    Here we write $\cT_{us_i,vs_i,\ba}^\circ$ to emphasize that it is empty if $us_i\nleq_\ba vs_i$.
\end{prop}

\begin{proof}
    Since $i\in \Des_\ba(v)$, set $\mathbf{v}$ to be an $\ba$-tilted reduced word for $v$ that ends with $s_i$. By \Cref{thm:tilted-deodhar}, 
    \[\cT_{u,v}^\circ = \bigsqcup_{\mathbf{u}\prec \mathbf{v}}D_{\mathbf{u},\mathbf{v}}.\]
    Let $\mathbf{v^-},\mathbf{u^-}$ be obtained from $\mathbf{v},\mathbf{u}$ by removing the last letter (note that for $\mathbf{u}^-$, the removed letter could be either $1$ or $s_i$). Then $\mathbf{v^-}$ is an $\ba$-tilted reduced word for $vs_i$ and $\mathbf{u}^-\prec \mathbf{v^-}$.
    If $us_i <_\ba u$ then every distinguished subword $\mathbf{u}$ of $\mathbf{v}$ ends with $s_i$ instead of $1$. Therefore by the parametrization in \Cref{thm:tilted-deodhar},
    \begin{align}\label{eqn:Jan18aaa}
    \begin{split}
        D_{\mathbf{u^-},\mathbf{v^-}}&\cong D_{\mathbf{u},\mathbf{v}}\\
        gB &\mapsto g\dot{s}_iB.
    \end{split}
    \end{align}
     Since $\mathbf{u}\prec \mathbf{v}\iff \mathbf{u^-}\prec \mathbf{v^-}$, we have $\cT_{u,v}^\circ \cong \cT_{us_i,vs_i}^\circ$ under the map in \eqref{eqn:Jan18aaa}.
     
    If $us_i>_\ba u$, then by \Cref{thm:tilted-deodhar},
    \[\cT_{u,v}^\circ = \bigsqcup_{\substack{\mathbf{u}\prec \mathbf{v}\\\mathbf{u}\text{ ends with }1}}D_{\mathbf{u},\mathbf{v}}\sqcup \bigsqcup_{\substack{\mathbf{u}\prec \mathbf{v}\\\mathbf{u}\text{ ends with }s_i}}D_{\mathbf{u},\mathbf{v}}.\]
    If $\mathbf{u}\prec \mathbf{v}$ ends with $1$, then 
    \begin{align}\label{eqn:Jan18bbb}
    \begin{split}
        D_{\mathbf{u^-},\mathbf{v^-}}\times \C^*&\cong D_{\mathbf{u},\mathbf{v}}\\
        (gB,p) &\mapsto gy_i(p)B.
    \end{split}
    \end{align}
    Otherwise if $\mathbf{u}\prec \mathbf{v}$ ends with $s_i$, then 
    \begin{align}\label{eqn:Jan18ccc}
    \begin{split}
        D_{\mathbf{u^-},\mathbf{v^-}}\times \C&\cong D_{\mathbf{u},\mathbf{v}}\\
        (gB,m) &\mapsto gx_i(m)\dot{s}_i^{-1}B.
    \end{split}
    \end{align}
    Since $\{\mathbf{u}^-:\mathbf{u}\prec\mathbf{v}\text{ ends with }1\}$ and $\{\mathbf{u}^-:\mathbf{u}\prec\mathbf{v}\text{ ends with }s_i\}$ are the sets of distinguished subwords of $\mathbf{v}^-$ for $u$ and $us_i$ respectively, using the maps \eqref{eqn:Jan18bbb} and \eqref{eqn:Jan18ccc}, we have
    $$\bigsqcup_{\substack{\mathbf{u}\prec \mathbf{v}\\\mathbf{u}\text{ ends with }1}}D_{\mathbf{u},\mathbf{v}} \cong \cT^\circ_{u,vs_i}\times \C^* \text{ and }\bigsqcup_{\substack{\mathbf{u}\prec \mathbf{v}\\\mathbf{u}\text{ ends with }s_i}}D_{\mathbf{u},\mathbf{v}} \cong \cT^\circ_{us_i,vs_i}\times \C.$$
    Combining these pieces gives the stated decomposition of $\cT_{u,v}^\circ$.
\end{proof}

\begin{prop}\label{prop:projrecur}
    For $u\lesssim_\ba v, i\in \Des_\ba(v)$ and $k\neq i$,
    \[\pi_k(\cT_{u,v}) = \begin{cases}
        \pi_k(\cT_{u,vs_i}) &\text{ if }us_i>_\ba u\\
        \pi_k(\cT_{us_i,vs_i}) & \text{ if }us_i<_\ba u
    \end{cases}.\]
\end{prop}
\begin{proof}

    We first claim that $\pi_k(\cT_{u,v})=\overline{\pi_k(\cT^\circ_{u,v})}$. Since $\pi_k$ is proper, the image of the closed set $\cT_{u,v}=\overline{\cT^\circ_{u,v}}$ is closed, hence $\pi_k(\cT_{u,v})=\pi_k(\overline{\cT^\circ_{u,v}})=\overline{\pi_k(\cT^\circ_{u,v})}$.

    If $us_i<_\ba u$, then by \Cref{prop:recur} we have $\cT_{u,v}^\circ \cong \cT_{us_i,vs_i}^\circ$ via the map $gB\mapsto g\dot{s}_iB$. Since $k\neq i$, this identification does not change the $k$-th flag, and hence $\pi_k(\cT^\circ_{u,v})=\pi_k(\cT^\circ_{us_i,vs_i})$, which implies $\pi_k(\cT_{u,v})=\pi_k(\cT_{us_i,vs_i})$.

    If $us_i>_\ba u$, then by \Cref{prop:recur} we have
    \[\cT_{u,v}^\circ\cong\bigl(\cT_{u,vs_i}^\circ\times \C^\ast\bigr)\;\sqcup\;\bigl(\cT_{us_i,vs_i,\ba}^\circ\times \C\bigr).\]
    Since this isomorphism does not affect the $k$-th flag, it follows that
    \[
    \pi_k(\cT_{u,v}^\circ)=\pi_k(\cT_{u,vs_i}^\circ)\,\cup\,\pi_k(\cT_{us_i,vs_i,\ba}^\circ).\]
    Moreover, $\cT_{us_i,vs_i,\ba}\subseteq \cT_{u,vs_i}$, so taking closures yields $\pi_k(\cT_{u,v})=\pi_k(\cT_{u,vs_i})$ as claimed.
\end{proof}

To prove \Cref{thm:proj=positroid}, we also need a version of \Cref{lemma:u<flatten(a)v} where we flatten $\ba$ from the back instead of the front:
\begin{defin}
    For $\ba \in [n]^n$, let $\jmax$ be the smallest index such that $a_{\jmax+1} = a_{\jmax+2} = \dots = a_n$. Define $\bflat(\ba)\in [n]^n$ to be the sequence obtained from $\ba$ by changing the last $n-\jmax$ numbers to $a_{\jmax}$. 
\end{defin}
\begin{lemma}\label{lemma:backflatten}
    If $u\lesssim_\ba v$ and $\Des_\ba(v)\subseteq [\jmax]$, then  $u\lesssim_{\bflat(\ba)} v$. 
\end{lemma}
\begin{proof}
    Since $u\sim_\ba v$, we have $u\sim_{\bflat(\ba)} v$. It is then enough to check that $u[k] \leq_{a_{\jmax}} v[k]$ for all $k > \jmax$. Equivalently, this is $u[k+1,n]\geq_{a_{\jmax}} v[k+1,n]$ for all $k > \jmax$. Since $u\leq_\ba v$, we have
    \begin{equation}\label{eqn:Nov19aaa}
        u[k+1,n] \geq_{a_{n}} v[k+1,n] \text{ for all }k > \jmax.
    \end{equation}
    Since $\Des_\ba(v)\subseteq [\jmax]$, 
    $v_{\jmax+1} <_{a_n} \dots <_{a_n} v_n$. Therefore by \eqref{eqn:Nov19aaa}, $u_k \geq_{a_n} v_k$ for all $k > \jmax$. 

    Since $u\sim_\ba v$, we have
    \[|u[\jmax+1, n] \cap [ a_{n}, a_{\jmax})_c| = |v[\jmax+1, n] \cap [a_n, a_{\jmax})_c|.\]
    Therefore for all $k > \jmax$ either $u_k,v_k \in [ a_{n}, a_{\jmax})_c$ or $u_k,v_k \in [a_{\jmax}, a_n)_c$. Since $u_k \geq_{a_n} v_k$, we can conclude that $u_k \geq_{a_{\jmax}} v_k$, and thus $u[k+1,n]\geq_{a_{\jmax}} v[k+1,n]$ for all $k > \jmax$ as we desired.
\end{proof}

\begin{theorem}\label{thm:proj=positroid}
    The image $\Pi_{u,v}:=\pi_k(\cT_{u,v})$ is a positroid variety.
\end{theorem}
\begin{proof}
     We will proceed by induction on $t = |J_\ba|$. Consider the two base cases: $J_\ba = \emptyset$ and $J_\ba = \{n\}$. In these two base cases, $\ba = (a_1 = a_2 = \dots = a_n)$ and $\cT_{u,v}$ is a (rotated) Richardson variety. By \cite[Theorem~5.9]{KLSjuggling}, $\pi_k(\cT_{u,v})$ is a positroid variety. Now suppose $t\geq 1$ and $J_\ba \neq \{n\}$. We claim that there is a chain of tilted Richardson varieties:
     \begin{equation}\label{eqn:chainofTuv}
         \cT_{u,v} = \cT_{u^{(0)},v^{(0)},\ba^{(0)}} \rightarrow \dots \rightarrow\cT_{u^{(m)},v^{(m)},\ba^{(m)}}\rightarrow \dots \rightarrow \cT_{u^{(N)},v^{(N)},\ba^{(N)}}  
     \end{equation}
     such that $\ba^{(N)} = (a_1 = a_2 = \dots = a_n)$,  $u^{(p)}\lesssim_{\ba^{(p)}} v^{(p)}$ for all $0\leq p\leq N$ and one of the following two conditions holds:
    \begin{enumerate}
        \item $v^{(p+1)} = v^{(p)}s_i$, $\ba^{(p+1)} = \ba^{(p)}$ and $u^{(p+1)}$ is either $u^{(p)}$ or $u^{(p)}s_i$ as in \Cref{prop:projrecur}  where 
        \[i\in\begin{cases}
            \Des_{\ba^{(p)}}(v^{(p)})\cap [\jmin^{(p)}] & \text{ if }p<m\\
            \Des_{\ba^{(p)}}(v^{(p)})\cap [\jmax^{(p)}+1,n] & \text{ if }p \geq m.
        \end{cases} \]
         and $\jmin^{(p)} = \min\{i<n: \ba^{(p)}_i \neq \ba^{(p)}_{i+1}\}$ and $\jmax^{(p)} = \max\{i<n: \ba^{(p)}_i \neq \ba^{(p)}_{i+1}\}$. 
        \item $u^{(p+1)} = u^{(p)}, v^{(p+1)} = v^{(p)}$ and 
        \[\ba^{(p+1)} = \begin{cases}
            \flatten(\ba^{(p)}) & \text{ if }p<m\\
            \bflat(\ba^{(p)}) & \text{ if }p\geq m.
        \end{cases}\]
    \end{enumerate}
     

     In particular, by \Cref{prop:projrecur}, \Cref{lemma:backflatten} and \Cref{lemma:u<flatten(a)v}, $$\pi_k(\cT_{u^{(p-1)},v^{(p-1)}}) = \pi_k(\cT_{u^{(p)},v^{(p)}}) \text{ for all }p\in [N],$$
     hence $\pi_k(\cT_{u,v})$ is a positroid variety by the base cases. 

     Let us explain how to obtain such a chain of tilted Richardson varieties. If $\jump_{t_p}^{(p)}\leq k$, we can keep applying \Cref{prop:projrecur} as long as $\Des_\ba(v)\cap[\jmin^{(p)}] \neq \emptyset$. Each of these steps is of the form (1). Now if $\Des_\ba(v)\cap[\jump_{t_p}^{(p)}] = \emptyset$, then by \Cref{lemma:u<flatten(a)v}, $u\lesssim_{\flatten(\ba)}v$ and $\cT_{u^{(p)},v^{(p)},\ba^{(p)}} = \cT_{u^{(p)},v^{(p)},\flatten(\ba^{(p)})}$ and we set $u^{(p+1)} = u^{(p)}, v^{(p+1)} = v^{(p)}$ and $\ba^{(p+1)} = \flatten(\ba^{(p)})$ as in case (2). Since $\jmin$ increases every time we flatten $\ba^{(p)}$, there will come a time where $\jmin^{(p)} > k$. We set $m$ to be the index of the first time this happens.
     
     Now for $p\geq m$, we similarly apply \Cref{prop:projrecur} until $\Des_\ba(v)\cap[\jmax^{(p)}+1,n] = \emptyset$. Then by \Cref{lemma:backflatten}, we flatten $\ba$ from the back and reduce $|J_\ba|$ without changing $\pi_k(\cT_{u,v})$. This process will end once $J_{\ba^{(p)}} \subseteq \{n\}$ and we obtain the chain in \eqref{eqn:chainofTuv} as desired. 
\end{proof}

\begin{defin}\label{def:ktilted}
    For $u\leq_\ba v\in S_n$, we say $u$ is less than or equals $v$ in \emph{$k$-tilted Bruhat order} (denoted as $u\leq_\ba^{k} v$) if there exists a shortest path in 
    $\Gamma_n$
    \begin{equation}\label{eqn:kchain}
        u = w^{(0)} \xrightarrow{t_{c_1,d_1}} w^{(1)} \xrightarrow{t_{c_2,d_2}} \dots \xrightarrow{t_{c_\ell,d_\ell}} w^{(\ell)} = v
    \end{equation}
    such that $c_j\leq k<d_j$ for all $j\in [\ell]$.
\end{defin}

\begin{remark}
    In the case where $u\leq v$ in strong Bruhat order, the $k$-tilted Bruhat order is equivalent to the $k$-Bruhat order first appeared in \cite{LS82} and studied extensively in \cite{BS98}. 
\end{remark}

\begin{lemma}\label{lemma:ktiltedBruhat}
        For $u\lesssim_\ba v, i\in \Des_\ba(v)$ and $i\neq k$, we have
        \begin{enumerate}
        \item if $u\leq_\ba^k v$ then $i\in \Des_\ba(u)$. 
            \item if $i\in \Des_\ba(u)$, then $u\leq_\ba^k v$ if and only if $us_i\leq_\ba^k vs_i$.
        \end{enumerate}
        
    \end{lemma}
    \begin{proof}
        We first prove (1). For $u\leq_\ba^k v$, let $u = w^{(0)} \xrightarrow{t_{c_1,d_1}} w^{(1)} \xrightarrow{t_{c_2,d_2}} \dots \xrightarrow{t_{c_\ell,d_\ell}} w^{(\ell)} = v$ be as in \eqref{eqn:kchain}. We claim that $i\in \Des_\ba(w^{(j)})$ for all $j\in [\ell]$. Suppose $i\in \Des_\ba(w^{(j)})$ for some $j$, by the lifting property (\Cref{thm:lifting}), either $i\in \Des_\ba(w^{(j-1)})$ or $w^{(j)}s_i,w^{(j-1)}s_i\in [w^{(j-1)},w^{(j)}]$. Since $[w^{(j-1)},w^{(j)}]$ is a rank $1$ interval, the latter implies that $w^{(j)} = w^{(j-1)}s_i$, impossible if $i\neq k$. Therefore $i\in \Des_\ba(w^{(j-1)})$ and we are done by induction.

        For ($\implies$) direction of (2), if $u\leq_\ba^k v$, then by part (1), $i\in \Des_\ba(w^{(j)})$ for all $0\leq j\leq \ell$. We claim that 
        \begin{equation}\label{eqn:Feb9aaa}
            us_i = w^{(0)}s_i \xrightarrow{t_{c'_1,d'_1}} w^{(1)}s_i \xrightarrow{t_{c'_2,d'_2}} \dots \xrightarrow{t_{c'_\ell,d'_\ell}} w^{(\ell)}s_i = vs_i
        \end{equation}
        is a shortest path on $\Gamma_n$ such that $c'_j\leq k <d'_j$ for all $j\in [\ell]$. 
        By the lifting property (\Cref{thm:lifting}), $w^{(j)}s_i\in [w^{(j-1)}s_i,w^{(j)}]$.
        Consider the rank $2$ interval $[w^{(j-1)}s_i,w^{(j)}]$ as in \Cref{fig:ktiltedrank2}, we have $t_{c_j',d_j'} = s_i t_{c_j,d_j} s_i$. Since $i\neq k$, it is straightforward to see that $c'_j \leq k < d'_j$. Therefore \eqref{eqn:Feb9aaa} produces a chain from $us_i$ to $vs_i$ as in \Cref{def:ktilted} and we can conclude that $us_i \leq_\ba^{k} vs_i$.
    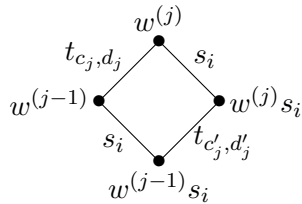
\begin{figure}[h]
    \centering
    \begin{tikzpicture}[scale = 0.8]
        \node at (0,0) {$\bullet$};
        \node[left] at (-0.4,0.3) {$s_i$};
        \node[below] at (0,0) {$w^{(j-1)}s_i$};
        \node at (-1,1) {$\bullet$};
\node[left] at (-1,1) {$w^{(j-1)}$};
\node[left] at (-0.4,1.75) {$t_{c_j,d_j}$};
\node at (0,2) {$\bullet$};
\node[above] at (0,2) {$w^{(j)}$};
\node at (1,1) {$\bullet$};
\node[right] at (0.4,1.7) {$s_i$};
\node[right] at (1,1) {$w^{(j)}s_i$};
\node[right] at (0.4,0.3) {$t_{c_j',d_j'}$};
\draw(0,0)--(-1,1)--(0,2)--(1,1)--(0,0);
    \end{tikzpicture}
    \caption{The tilted Bruhat interval $[us_i, ut_{a,b}]$}
    \label{fig:ktiltedrank2}
\end{figure}

The ($\impliedby$) direction of (2) follows by similar reasoning, replacing $\Des_\ba(w^{(j)})$ with $\Asc_\ba(w^{(j)})$ whenever necessary. 
\end{proof}

Motivated by the connection between $k$-Bruhat orders and projections of Richardson varieties \cite{KLSjuggling,KLSprojection}, we prove a generalization to all tilted Richardson varieties. 

\begin{theorem}\label{thm:birational}
    The map $\pi_k:\cT_{u,v}\rightarrow \Pi_{u,v}$ is birational if and only if $u\leq_\ba^k v$. 
\end{theorem}

\begin{proof}
    If $\cT_{u,v}$ is a Richardson variety, this is \cite[Corollary~3.4]{KLSprojection}. Similarly, if $\cT_{u,v}$ is a rotated Richardson variety, then $u\leq_\ba v$ where $\ba = (a,\dots,a)$ for some $a\in [n]$. Equivalently, $\cyclic^{-a+1}(\cT_{u,v}) = \cT_{\cyclic^{-a+1}u,\cyclic^{-a+1}v}$ is a Richardson variety. Since $\Gamma_n$ is invariant under left multiplication by $\cyclic$, $u\leq_\ba^k v \iff \cyclic^{-a+1}u\leq_{(1,\dots,1)}^k \cyclic^{-a+1}v$. Since $\cyclic$ commutes with $\pi_k$ and $\cyclic$ sends a positroid variety to a positroid variety, \Cref{thm:birational} holds in this case.

    Now when $\cT_{u,v}$ is not a (rotated) Richardson variety, as in the proof of \Cref{thm:proj=positroid}, there exists a chain of tilted Richardson varieties
    \begin{equation}\label{eqn:Feb25aaa}
        \cT_{u,v} = \cT_{u^{(0)},v^{(0)}} \rightarrow \cT_{u^{(1)},v^{(1)}}\rightarrow \dots \rightarrow \cT_{u^{(N)},v^{(N)}}
    \end{equation}
    such that $\cT_{u^{(N)},v^{(N)}}$ is a rotated Richardson variety and $\pi_k(\cT_{u^{(p)},v^{(p)}}) = \pi_k(\cT_{u^{(p+1)},v^{(p+1)}})$ for all $p<N$ as in \Cref{prop:projrecur}. 
    
    ($\implies$): Notice that by  \Cref{prop:projrecur}, $\dim(\cT_{u^{(p)},v^{(p)}}) \geq \dim(\cT_{u^{(p+1)},v^{(p+1)}})$ and that by the construction in \eqref{eqn:chainofTuv}, equality holds if and only if one of the following holds:
    \begin{enumerate}
        \item $(u^{(p+1)}, v^{(p+1)}) = (u^{(p)}, v^{(p)})$;
        \item $(u^{(p+1)}, v^{(p+1)}) = (u^{(p)}s_i, v^{(p)}s_i)$ where $i\in \Des_{\ba^{(p)}}(v^{(p)})\cap \Des_{\ba^{(p)}}(u^{(p)})$ and $i\neq k$.
    \end{enumerate}
    Therefore $\pi_k$ is birational on $\cT_{u,v}$ if and only if each step in the chain of tilted Richardsons must be one of the two cases above and $u^{(N)}\leq_{\ba^{(N)}}^k v^{(N)}$. By (2) of \Cref{lemma:ktiltedBruhat}, $\pi_k$ is birational implies $u\leq_\ba^k v$. 

    ($\impliedby$): 
    For any step in \eqref{eqn:Feb25aaa} where $(u^{(p+1)}, v^{(p+1)}) \neq  (u^{(p)}, v^{(p)})$, we have $v^{(p+1)} = v^{(p)}s_i$ with $i\in \Des_{\ba^{(p)}}(v^{(p)})$. By (1) of \Cref{lemma:ktiltedBruhat}, $i\in \Des_{\ba^{(p)}}(u^{(p)})$ and thus $u^{(p+1)} = u^{(p)}s_i$. In particular, $\dim(\cT_{u^{(p)},v^{(p)}}) = \dim(\cT_{u^{(p+1)},v^{(p+1)}})$ and $\pi_k$ is birational on $\cT_{u^{(p)},v^{(p)}}$ if and only if it is birational on $\cT_{u^{(p+1)},v^{(p+1)}}$.  Therefore by (2) of \Cref{lemma:ktiltedBruhat}, $u^{(p)} \leq_{\ba^{(p)}}^k v^{(p)}$ implies $u^{(p+1)} \leq_{\ba^{(p+1)}}^k v^{(p+1)}$. Since $\leq_{\ba}^k$ is independent of the choice of $\ba$,  we also have $u^{(p)} \leq_{\ba^{(p)}}^k v^{(p)}$ implies $u^{(p+1)} \leq_{\ba^{(p+1)}}^k v^{(p+1)}$ when $(u^{(p+1)}, v^{(p+1)}) = (u^{(p)}, v^{(p)})$. 
    Since $u = u^{(0)}\leq_\ba^k v^{(0)} = v$, we can conclude that $u^{(N)}\leq_{\ba^{(N)}}^{k} v^{(N)}$. Therefore $\pi_k$ is birational on $\cT_{u,v}$. 
\end{proof}

\begin{theorem}
    If $\pi_k:\cT_{u,v}\rightarrow \Pi_{u,v}$ is birational, then $\pi_k:\cT_{u,v}^\circ \rightarrow \Pi_{u,v}^\circ$ is an isomorphism.
\end{theorem}
\begin{proof}
    We will show that if $\pi_k:\cT_{u,v}\rightarrow \Pi_{u,v}$ is birational, then there exists a rotated Richardson variety $\cT_{u^{(N)},v^{(N)}}$ such that 
    \[\pi_k(\cT_{u,v}^\circ) = \pi_k(\cT_{u^{(N)},v^{(N)}}^\circ) = \Pi_{u,v}^\circ.\]
    
    By the ``$\implies$'' part in the proof of \Cref{thm:birational}, there exists a chain of tilted Richardson varieties of the form \eqref{eqn:Feb25aaa} such that 
    \begin{itemize}
        \item $\cT_{u^{(N)},v^{(N)}}$ is a rotated Richardson variety,
        \item $\pi_k(\cT_{u^{(N)},v^{(N)}}) = \Pi_{u,v}$,
    \end{itemize}
    and, for all $p<N$, either
    \begin{enumerate}
        \item $(u^{(p+1)}, v^{(p+1)}) = (u^{(p)}, v^{(p)})$; or
        \item $(u^{(p+1)}, v^{(p+1)}) = (u^{(p)}s_i, v^{(p)}s_i)$ where $i\in \Des_{\ba^{(p)}}(v^{(p)})\cap \Des_{\ba^{(p)}}(u^{(p)})$ and $i\neq k$.
    \end{enumerate}
    Notice that in case (1), we automatically have $\pi_k(\cT_{u^{(p+1)}, v^{(p+1)}}^\circ) = \pi_k(\cT_{u^{(p)}, v^{(p)}}^\circ)$. In case (2), by \Cref{prop:recur}, $\cT_{u^{(p+1)}, v^{(p+1)}}^\circ \cong \cT_{u^{(p)}s_i, v^{(p)}s_i}^\circ$ under the map $gB\mapsto g\dot{s}_iB$. Since $i\neq k$, we have $\pi_k(\cT_{u^{(p+1)}, v^{(p+1)}}^\circ) = \pi_k(\cT_{u^{(p)}, v^{(p)}}^\circ)$. Therefore $\pi_k(\cT_{u^{(p+1)}, v^{(p+1)}}^\circ) = \pi_k(\cT_{u^{(p)}, v^{(p)}}^\circ)$ for all $p<N$ and thus $\pi_k(\cT_{u,v}^\circ) = \pi_k(\cT_{u^{(N)}, v^{(N)}}^\circ)$. 

    Since $\cT_{u^{(N)}, v^{(N)}}$ is a rotated Richardson variety, $u^{(N)}\leq^k_{\ba}v^{(N)}$ where $\ba = (a,a,\dots,a)$ for some $a$. Since the quantum Bruhat graph $\Gamma_n$ is invariant under left multiplication by the long cycle $\cyclic = 23\cdots 1$ (\Cref{lemma:cyclic-symmetry}), we have $\cyclic^{-a+1}u^{(N)}\leq \cyclic^{-a+1}v^{(N)}$ in $k$-Bruhat order. Since the cyclic rotation $\cyclic$ commutes with $\pi_k$, by \cite[Theorem~5.9]{KLSjuggling}, $\pi_k(\cT_{u^{(N)}, v^{(N)}}^\circ) = \Pi_{u,v}^\circ$. We can then conclude that $\pi_k(\cT_{u,v}^\circ) =\Pi_{u,v}^\circ$ as desired. 
\end{proof}